\documentclass[reqno, 10pt, centertags]{amsart}
\usepackage{amsmath,amsthm,amscd,amssymb,latexsym,upref,stmaryrd}
\usepackage{mathrsfs}

\numberwithin{equation}{section}
\newtheorem{theorem}{Theorem}[section]
\newtheorem{proposition}[theorem]{Proposition}
\newtheorem{corollary}[theorem]{Corollary}
\newtheorem{lemma}[theorem]{Lemma}
\newtheorem{remark}[theorem]{Remark}
\newtheorem{definition}[theorem]{Definition}
\newtheorem{example}[theorem]{Example}

\DeclareMathOperator{\col}{col}

\DeclareMathOperator{\card}{card}

\DeclareMathOperator{\diag}{diag}
\DeclareMathOperator{\diam}{diam}
\DeclareMathOperator{\codiag}{codiag}
\DeclareMathOperator{\dom}{dom}

\DeclareMathOperator{\range}{ran}

\DeclareMathOperator{\const}{const}
\DeclareMathOperator*{\esssup}{ess\,sup}

\renewcommand{\le}{\leqslant}
\renewcommand{\ge}{\geqslant}
\renewcommand{\Im}{{\rm Im\,}}
\renewcommand{\Re}{{\rm Re\,}}
\newcommand{\ol}{\overline}
\newcommand{\wt}{\widetilde}
\newcommand{\wh}{\widehat}
\renewcommand{\(}{\left(}
\renewcommand{\)}{\right)}

\newcommand{\eps}{\varepsilon}
\renewcommand{\a}{\alpha}
\renewcommand{\b}{\beta}
\newcommand{\g}{\gamma}
\renewcommand{\l}{\lambda}
\renewcommand{\L}{\Lambda}

\def\cB{\mathcal{B}}

\def\cE{\mathcal{E}}
\def\cF{\mathcal{F}}
\def\cG{\mathcal{G}}

\def\cI{\mathcal{I}}
\def\cJ{\mathcal{J}}
\def\cK{\mathcal{K}}
\def\cL{\mathcal{L}}
\def\cN{\mathcal{N}}
\def\cR{\mathcal{R}}
\def\cS{\mathcal{S}}

\def\cX{\mathcal{X}}

\def\sF{\mathscr{F}}
\def\sR{\mathscr{R}}

\def\fB{\mathfrak{B}}
\def\fC{\mathfrak{C}}
\def\fF{\mathfrak{F}}
\def\fG{\mathfrak{G}}
\def\fH{\mathfrak{H}}
\def\fI{\mathfrak{I}}

\def\bC{\mathbb{C}}
\def\bD{\mathbb{D}}
\def\bN{\mathbb{N}}
\def\bQ{\mathbb{Q}}
\def\bR{\mathbb{R}}
\def\bT{\mathbb{T}}
\def\bU{\mathbb{U}}
\def\bZ{\mathbb{Z}}

\newcommand{\VectorSpace}[2]{{#1}({#2}; \bC^2)}
\newcommand{\MatrixSpace}[2]{{#1}({#2}; \bC^{2 \times 2})}
\newcommand{\LL}[1]{\MatrixSpace{L^{#1}}{[0,1]}}
\newcommand{\LLV}[1]{\VectorSpace{L^{#1}}{[0,1]}}
\newcommand{\CC}[1]{\MatrixSpace{C^{#1}}{[0,1]}}
\newcommand{\XX}[1]{\MatrixSpace{X_{#1}}{\Omega}}
\newcommand{\XXZ}[1]{\MatrixSpace{X_{#1}^{0}}{\Omega}}

\newcommand{\bigabs}[1]{\bigl|{#1}\bigr|}
\newcommand{\abs}[1]{\left|{#1}\right|}
\newcommand{\ceil}[1]{\left\lceil{#1}\right\rceil}
\newcommand{\norm}[1]{\left\|{#1}\right\|}
\newcommand{\bignorm}[1]{\bigl\|{#1}\bigr\|}

\begin{document}

\sloppy

\title[Stability of spectral characteristics of BVP for $2 \times 2$ Dirac type
systems]{Stability of spectral characteristics and Bari basis property \\
of boundary value problems for $2 \times 2$ Dirac type systems}

\author{Anton~A.~Lunyov}
\address{
Facebook, Inc.,
1 Hacker Way, Menlo Park,
California, 94025,
United States of America}
\email{A.A.Lunyov@gmail.com}

\author{Mark~M.~Malamud}
\address{
Peoples Friendship University of Russia (RUDN University),
6 Miklukho-Maklaya St.
Moscow, 117198,
Russian Federation}
\email{malamud3m@gmail.com}

\date{}

\begin{abstract}
The paper is concerned with the stability property under perturbation $Q \to
\widetilde{Q}$ of different spectral characteristics of a boundary value problem
associated in $L^2([0,1]; \mathbb{C}^2)$ with the following $2 \times 2$ Dirac
type equation
\begin{equation} \label{eq:Ly.abstract}
 L_U(Q) y = -i B^{-1} y' + Q(x) y = \lambda y , \quad
 B = \begin{pmatrix} b_1 & 0 \\ 0 & b_2 \end{pmatrix}, \quad b_1 < 0 < b_2,\quad
 y = {\rm col}(y_1, y_2),
\end{equation}
with a potential matrix $Q \in L^p([0,1]; \mathbb{C}^{2 \times 2})$ and subject
to the regular boundary conditions $Uy :=\{U_1, U_2\}y=0$. If $b_2 = -b_1 =1$
this equation is equivalent to one dimensional Dirac equation. Our approach to
the spectral stability relies on the existence of the triangular transformation
operators for system~\eqref{eq:Ly.abstract} with $Q \in L^1$, which was
established in our previous works. The starting point of our investigation is
the Lipshitz property of the mapping $Q \to K_Q^{\pm}$, where $K_Q^{\pm}$ are
the kernels of transformation operators for system~\eqref{eq:Ly.abstract}.
Namely, we prove the following uniform estimate:
\begin{equation} \label{eq:K-wtK<Q-wtQ_Abstract}
 \|K_Q^{\pm} - K_{\widetilde{Q}}^{\pm}\|_{X_{\infty,p}^2}
 + \|K_Q^{\pm} - K_{\widetilde{Q}}^{\pm}\|_{X_{1,p}^2}
 \le C \cdot \|Q - \widetilde{Q}\|_p, \qquad
 Q, \widetilde{Q}\in \mathbb{U}_{p,r}^{2 \times 2}, \quad p \in [1, \infty],
\end{equation}
on balls $\mathbb{U}_{p,r}^{2 \times 2}$ in $L^p([0,1];
\mathbb{C}^{2 \times 2})$. It is new even for $\widetilde{Q} = 0$. Here
$X_{\infty,p}^2$, $X_{1,p}^2$ are the special Banach
spaces naturally arising in such problems. We also obtained similar estimates
for Fourier transforms of $K_Q^{\pm}$. Both of these estimates are of
independent interest and play a crucial role in the proofs of all spectral
stability results discussed in the paper. For instance, as an immediate
consequence of these estimates we get the Lipshitz property of the mapping
$Q \to \Phi_Q(\cdot, \lambda)$, where $\Phi_Q(x, \lambda)$ is the fundamental
matrix of the system~\eqref{eq:Ly.abstract}.

Assuming the spectrum $\Lambda_{Q} = \{\lambda_{Q,n}\}_{n \in \mathbb{Z}}$ of
$L_U(Q)$ to be asymptotically simple, denote by $F_Q = \{f_{Q,n}\}_{|n| > N}$ a
sequence of corresponding normalized eigenfunctions, $L_U(Q) f_{Q,n} =
\lambda_{Q,n}f_{Q,n}$. Assuming \emph{boundary conditions (BC) to be strictly
regular}, we show that the mapping $Q \to \Lambda_Q - \Lambda_0$ sends
$L^p([0,1]; \mathbb{C}^{2 \times 2})$ either into $l^{p'}$ or into the weighted
$l^p$-space $l^p(\{(1+|n|)^{p-2}\})$; we also establish its Lipshitz property on
compacts in $L^p([0,1]; \mathbb{C}^{2 \times 2})$, $p \in [1, 2]$. The proof of
the second estimate involves as an important ingredient inequality that
generalizes classical Hardy-Littlewood inequality for Fourier coefficients. It
is also shown that the mapping $Q \to F_Q - F_0$ sends
$L^p([0,1]; \mathbb{C}^{2 \times 2})$ into the space $l^{p'}(\mathbb{Z};
C([0,1]; \bC^2)$ of sequences of continuous vector-functions, and has the
Lipshitz property on compacts in $L^p([0,1]; \mathbb{C}^{2 \times 2})$,
$p \in [1, 2]$.

Certain modifications of these spectral stability results are also proved for
balls $\mathbb{U}_{p,r}^{2 \times 2}$ in $L^p([0,1]; \mathbb{C}^{2 \times 2})$,
$p \in [1, 2]$.

Note also that the proof of the Lipshitz property of the mapping $Q \to
F_Q-F_0$ involves the deep Carleson-Hunt theorem for maximal
Fourier transform, while the proof of this property for the mapping $Q \to
\Lambda_Q - \Lambda_0$ relies on the estimates of the classical Fourier
transform and is elementary in character.

In the case of $Q \in L^2$ we establish a criterion for the system of root
vectors of $L_U(Q)$ \emph{to form a Bari basis in} $L^2([0,1]; \mathbb{C}^2)$.
Under a simple additional assumption this system forms a Bari basis if and only
if BC are self-adjoint.
\end{abstract}

\maketitle{}

\renewcommand{\contentsname}{Contents}
\tableofcontents

\section{Introduction} \label{sec:Intro}
In this paper we continue our investigation~\cite{LunMal16JMAA} of the following
first order system of differential equations
\begin{equation} \label{eq:systemIntro_1}
 \cL y = -i B^{-1} y' + Q(x) y=\l y, \qquad y = \col(y_1,y_2),
 \qquad x \in [0,1],
\end{equation}
where
\begin{equation} \label{eq:BQ}
 B = \begin{pmatrix} b_1 & 0 \\ 0 & b_2 \end{pmatrix},
 \quad b_1 < 0 < b_2 \quad \text{and}\quad
 Q = \begin{pmatrix} 0 & Q_{12} \\ Q_{21} & 0 \end{pmatrix} \in \LL{1}.
\end{equation}

If $B = \(\begin{smallmatrix} -1 & 0 \\ 0 & 1 \end{smallmatrix}\)$
system~\eqref{eq:systemIntro_1} is equivalent to the Dirac system (see the classical
monographs~\cite{LevSar88},~\cite{Mar77}).

Let us associate linearly independent boundary conditions (BC)
\begin{equation} \label{eq:BC.intro}
 U_j(y) := a_{j 1}y_1(0) + a_{j 2}y_2(0) + a_{j 3}y_1(1) + a_{j 4}y_2(1)= 0,
 \quad j \in \{1,2\},
\end{equation}
with system~\eqref{eq:systemIntro_1}, and denote as $L(Q) := L_U(Q)$ an
operator, associated in $\LLV{2}$ with the boundary value problem
(BVP)~\eqref{eq:systemIntro_1}--\eqref{eq:BC.intro}. It is defined by differential
expression $\cL$ on the domain
\begin{equation}
 \dom(L_U(Q)) = \{f \in AC([0,1]; \bC^2) :\ \cL f \in \LLV{2}, \
 U_1(f) = U_2(f) = 0\}.
\end{equation}

The completeness property of the root vectors system of BVP for general
$n \times n$ system of the form~\eqref{eq:systemIntro_1} with a nonsingular
diagonal $n\times n$ matrix $B$ with complex entries and a potential matrix
$Q(\cdot)$ of the form
\begin{equation} \label{1.2}
 B = \diag(b_1, b_2, \ldots, b_n) \in \bC^{n\times n} \quad \text{and}\quad
 Q(\cdot) =: (q_{jk}(\cdot))_{j,k=1}^n \in L^1([0,1]; \bC^{n\times n}).
\end{equation}
was established in~\cite{MalOri12} for a wide class of BVPs, although for
$2 \times 2$ Dirac system with $Q\in C([0,1]; \bC^{2 \times 2})$ it was proved
earlier in~\cite[Chapter 1.3]{Mar77}. In~\cite{MalOri12,LunMal14JST} the authors
also found completeness conditions for non-regular and even degenerate BCs. In
these papers it was also established the Riesz basis property of the root
vectors systems (with and without parentheses) for different classes of BVPs for
$n \times n$ system with arbitrary $B$~\eqref{1.2} and $Q \in L^{\infty}([0,1];
\bC^{n\times n})$. Note also that BVP for $2m \times 2m$ Dirac equation
($B=\diag(-I_m, I_m)$) were investigated in~\cite{MykPuy13} (Bari-Markus
property for Dirichlet BVP with $Q \in L^2([0,1]; \bC^{2m \times 2m})$ and
in~\cite{KurAbd18, KurGad20} (Bessel and Riesz basis properties on abstract
level).

The Riesz basis property in $\LLV{2}$ of
BVP~\eqref{eq:systemIntro_1}--\eqref{eq:BC.intro}, i.e. of the operator
$L_U(Q)$, for $2 \times 2$ Dirac system with various assumptions on the
potential matrix $Q$ was investigated in numerous papers (see~\cite{TroYam01,
TroYam02, Mit03, Mit04, HasOri09, DjaMit10, Bask11, DjaMit12UncDir,
DjaMit12TrigDir, DjaMit12Crit, DjaMit13CritDir, LunMal14IEOT, LunMal14Dokl,
SavShk14, LunMal16JMAA} and references therein). At that time the most strong
result was obtained by P.~Djakov and B.~Mityagin~\cite{DjaMit10, DjaMit12UncDir}
and A.~Baskakov, A.~Derbushev, A.~Shcherbakov~\cite{Bask11} who proved under the
assumption $Q \in \LL{2}$ that the root vectors system of the
BVP~\eqref{eq:systemIntro_1}--\eqref{eq:BC.intro} \emph{with strictly regular
BCs} forms a Riesz basis and forms \emph{a Riesz basis with parentheses}
whenever BC are only regular. Note however that the methods of these papers are
substantially rely on $L^2$-technique (like Parseval equality, Hilbert-Schmidt
operators, etc.) and cannot be applied to $L^1$-potentials.

Later the case $Q \in \LL{1}$ was treated independently and with different
methods by the authors~\cite{LunMal14Dokl,LunMal16JMAA}, on the one hand, and by
A.M.~Savchuk and A.A.~Shkalikov~\cite{SavShk14}, on the other hand. Namely, it
was proved that BVP~\eqref{eq:systemIntro_1}--\eqref{eq:BC.intro} with
$Q \in \LL{1}$ and \emph{strictly regular boundary conditions} has the Riesz
basis property while BVP with only regular BC has \emph{the property of Riesz
basis with parentheses}.

Recall in this connection that BC~\eqref{eq:BC.intro} are \emph{regular}, if and
only if they are \emph{equivalent} to the following conditions
\begin{equation} \label{eq:BC.Reg_B_Con_Intro}
 \wh{U}_{1}(y) = y_1(0) + b y_2(0) + a y_1(1) = 0,\quad
 \wh{U}_{2}(y) = d y_2(0) + c y_1(1) + y_2(1) = 0,
\end{equation}
with certain $a,b,c,d \in \bC$ satisfying $ad-bc \ne 0$. Recall also that
regular BC~\eqref{eq:BC.intro} are called \textbf{strictly regular}, if the
sequence $\l_0 = \{\l_n^0\}_{n \in \bZ}$ of the eigenvalues of the unperturbed
($Q=0$) BVP~\eqref{eq:systemIntro_1}--\eqref{eq:BC.intro} (of the operator
$L_U(0)$), is asymptotically separated. In particular, the eigenvalues
$\{\l_n^0\}_{|n| > n_0}$ are geometrically and algebraically simple.

It is well known that \emph{non-degenerate separated} BC are always
\emph{strictly regular}. Moreover, conditions~\eqref{eq:BC.Reg_B_Con_Intro}
\emph{are strictly regular for Dirac operator if and only if} $(a-d)^2 \ne
-4bc$. In particular, antiperiodic (periodic) BC \emph{are regular but not
strictly regular} for Dirac system while they \emph{may be strictly regular for
Dirac type system with certain} $b_1 \not= - b_2.$

To describe our approach to Riesz basis property used
in~\cite{LunMal14Dokl,LunMal16JMAA} let us denote by $e_{\pm}(\cdot,\l)$
solution to system~\eqref{eq:systemIntro_1} satisfying the initial conditions
$e_\pm(0,\l)=\binom{1}{\pm1}$. Our investigation
in~\cite{LunMal14Dokl,LunMal16JMAA} substantially relies on the following
representation of $e_{\pm}(\cdot,\l)$ by means of triangular transformation
operators:
\begin{equation} \label{eq:e=(I+K)e0_Intro}
 e_{\pm}(x,\l) = (I+K_Q^{\pm})e^0_{\pm}(x,\l)
 = e^0_{\pm}(x,\l) + \int^x_0 K_Q^{\pm}(x,t) e^0_{\pm}(t,\l)dt,
 \qquad \text{where} \qquad
 e^0_{\pm}(x,\l)=\binom{e^{ib_1\l x}}{\pm e^{ib_2\l x}},
\end{equation}
and the kernels $K^{\pm}= K_Q^{\pm} = \bigl(K^{\pm}_{jk}\bigr)^2_{j,k=1} \in
\XXZ{1,1} \cap \XXZ{\infty,1}$
(see~\eqref{eq:X1p.norm.def},~\eqref{eq:Xinfp.norm.def} for definitions of these
spaces).

In turn, representation~\eqref{eq:e=(I+K)e0_Intro} immediately leads to the
following key formula for the characteristic determinant $\Delta(\cdot)$ of the problem~\eqref{eq:systemIntro_1}--\eqref{eq:BC.intro}:
\begin{align}
\label{eq:Delta=Delta0+_Intro}
 \Delta_Q(\l) &= \Delta_0(\l) + \int^1_0 g_{1,Q}(t) e^{i b_1 \l t} dt
 + \int^1_0 g_{2,Q}(t) e^{i b_2 \l t} dt,
\end{align}
Here $\Delta_0(\cdot)$ is the characteristic determinant of
problem~\eqref{eq:systemIntro_1}--\eqref{eq:BC.intro} with $Q_0=0$ and
$g_{l,0}(\cdot)\in L^1[0,1], \ l \in \{1,2\}$.

Formula~\eqref{eq:Delta=Delta0+_Intro} yields estimate of the difference
$\Delta_Q(\l) - \Delta_0(\l)$ from above. Combining this estimate with the
classical estimate of $\Delta_0(\cdot)$ from below and applying the Rouch\'e
theorem one arrives at the asymptotic formula
\begin{equation} \label{eq:l.n=l.n0+o(1)_Intro_1}
 \l_n = \l_n^0 + o(1), \quad\text{as}\quad n \to \infty, \quad n \in \bZ,
\end{equation}
relating the eigenvalues $\L = \{\l_n\}_{n \in \bZ}$ and $\L_0 =
\{\l_n^0\}_{n \in \bZ}$ of the operators $L_U(Q)$ and $L_U(0)$ (with
\emph{regular BC}), respectively, (see~\cite{LunMal14Dokl, LunMal16JMAA} for
details and also~\cite{SavShk14} where formula~\eqref{eq:l.n=l.n0+o(1)_Intro_1}
was obtained by using different method). Note also that
representation~\eqref{eq:Delta=Delta0+_Intro} for the determinant
$\Delta_Q(\cdot)$ was substantially used in a recent papers by
A.~Makin~\cite{Mak19},~\cite{Mak20}.

In~\cite{LunMal14Dokl, LunMal16JMAA} we also applied
representation~\eqref{eq:e=(I+K)e0_Intro} to obtain asymptotic formulas for
solutions to equation~\eqref{eq:systemIntro_1} as well as eigenfunctions of the
problem~\eqref{eq:systemIntro_1}--\eqref{eq:BC.intro}.

In the present paper we continue our preceding investigation
from~\cite{LunMal14Dokl,LunMal16JMAA} of
BVPs~\eqref{eq:systemIntro_1}--\eqref{eq:BC.intro} and transformation operators
for systems~\eqref{eq:systemIntro_1}. In the first chapter we prove
\emph{Lipshitz property} of the mappings $Q \to K^{\pm}$ on the balls
\begin{equation}
 \bU_{p,r}^{2 \times 2} := \left\{F \in \LL{p} :
 \|F\|_p := \|F\|_{\LL{p}} \le r\right\}, \quad r > 0,
\end{equation}
in $\LL{p}$. Namely, our first main result reads as follows: there is a constant
$C = C(B, p, r)$, not dependent on $Q, \wt{Q}\in \bU_{p,r}^{2 \times 2}$, such
that for any $p \in [1, \infty)$ the following uniform estimate holds
\begin{equation} \label{eq:K-wtK<Q-wtQ_Intro}
 \|K_Q^{\pm} - K_{\wt{Q}}^{\pm}\|_{\XX{\infty,p}}
 + \|K_Q^{\pm} - K_{\wt{Q}}^{\pm}\|_{\XX{1,p}}
 \le C \cdot \|Q - \wt{Q}\|_p, \qquad Q, \wt{Q}\in \bU_{p,r}^{2 \times 2}.
\end{equation}
Here ${K}_{\wt{Q}}^{\pm}$ are the kernels from
representation~\eqref{eq:e=(I+K)e0_Intro} for $\wt{e}$ being a solution
of~\eqref{eq:systemIntro_1} with $\wt{Q}$ in place of ${Q}$ and the spaces
$X_{\infty,p}$, $X_{1,p}$ are introduced in~\eqref{eq:X1p.norm.def},
\eqref{eq:Xinfp.norm.def}.

Combining uniform estimate~\eqref{eq:K-wtK<Q-wtQ_Intro} with
representation~\eqref{eq:Delta=Delta0+_Intro} we obtain the following statement
meaning the \emph{Lipshitz property} of the map $Q\to g_{l,Q}$ on the
$L^p$-balls and playing a crucial role in our approach to subsequent estimates.
\begin{lemma} \label{lem:dif-ce_of_determ-s_Four_repres-n}
Let $Q, \wt{Q}\in \bU_{p,r}^{2 \times 2}$ with $p \in [1, \infty]$. Then the
difference of characteristic determinants $\Delta_{Q}(\l) = \Delta_{Q,U}(\l)$
and $\Delta_{\wt{Q}}(\l)= \Delta_{\wt{Q},U}(\l)$ of the
problem~\eqref{eq:systemIntro_1}--\eqref{eq:BC.intro} admits the following
representation
\begin{equation} \label{eq:Delta-wt.Delta_Intro}
 \Delta_{Q}(\l) - \Delta_{\wt{Q}}(\l) = \int^1_0 g_1(t) e^{i b_1 \l t} dt
 + \int^1_0 g_2(t) e^{i b_2 \l t} dt.
\end{equation}
with $g_j := g_{j,Q} - g_{j,\wt{Q}} \in L^p[0,1]$, $j \in \{1, 2\}$.
Moreover, there is a constant $C = C(p, r, B) > 0$ such that
\begin{equation} \label{eq:whg1+whg2<C.whQ_Intro}
 \|g_1\|_p + \|g_2\|_p = \|g_{1,Q} - g_{1,\wt{Q}}\|_p
 + \|g_{2,Q} - g_{2,\wt{Q}}\|_p \le C \cdot \|Q-\wt{Q}\|_p.
\end{equation}
\end{lemma}
To demonstrate our first applications of estimate~\eqref{eq:K-wtK<Q-wtQ_Intro}
and Lemma~\ref{lem:dif-ce_of_determ-s_Four_repres-n} let us denote by $\L_Q :=
\L_{U,Q} = \{\l_{Q,n}\}_{n \in \bZ}$ the spectrum of the operator $L_U(Q)$
assuming it to be asymptotically simple.

As an immediate application of Lemma~\ref{lem:dif-ce_of_determ-s_Four_repres-n}
we complete formula~\eqref{eq:l.n=l.n0+o(1)_Intro_1} by establishing
$c_0$-Lipshitz property on compacts: \emph{for each compact} $\cK (\subset
\bU_{1,r}^{2 \times 2})$ \emph{and any} $\eps > 0$ \emph{there exists}
$N_{\eps} > 0$, \emph{not dependent on} $Q\in \cK$, \emph{such that}
\begin{equation} \label{eq:ln-ln0<eps_Intro}
 \sup_{|n| > N_{\eps}}\bigabs{\l_{Q,n} - \l_n^0} \le \eps, \qquad Q \in \cK.
\end{equation}

In the case of Dirac system ($-b_1= b_2=1$) this result was established earlier
in~\cite[Theorem 3]{Sad16}.

Starting with Lemma~\ref{lem:dif-ce_of_determ-s_Four_repres-n} and assuming
\emph{boundary conditions~\eqref{eq:BC.intro} to be strictly regular} we establish the
Lipshitz property of the mapping $Q \to \L_Q$ in different norms: there exists an
enumeration of the spectra $\{\l_{Q,n}\}_{n \in\bZ}$ and $\{\l_{\wt{Q},n}\}_{n \in \bZ}$
of the operators $L_U(Q)$ and $L_U({\wt{Q}})$, and a set $\cI_{Q,\wt{Q}} \subset \bZ$,
such that with certain constants $C = C(B,p,r), N = N(B,p,r) > 0$, not dependent on $Q,
\wt{Q} \in \bU_{p,r}^{2 \times 2}$ with $p\in (1,2]$, the following estimates hold:
\begin{align}
\label{eq:unif_estim_I(q,wtQ).intro}
 \card\(\bZ \setminus \cI_{Q,\wt{Q}}\) & \le N, \\
\label{eq:sum.ln-wtln.intro}
 \sum_{n \in \cI_{Q, \wt{Q}}} \bigabs{\l_{Q,n} - \l_{\wt{Q},n}}^{p'}
 & \le C \cdot \|Q - \wt{Q}\|_p^{p'}, \qquad 1/p' + 1/p = 1, \\
\label{eq:weight.ln-wtln<Q-wtQ.intro}
 \sum_{n \in \cI_{Q, \wt{Q}}} \(1+|n|\)^{p-2}
 \bigabs{\l_{Q,n} - \l_{\wt{Q},n}}^p & \le C \cdot \|Q - \wt{Q}\|_p^{p}.
\end{align}

On a compact set $\cK$ in $L^p([0,1]; \bC^{2 \times 2})$ subsets $\cI_{Q,\wt{Q}}
\subset \bZ$ can be chosen independent of the pair $\{Q,\wt{Q}\}$ and in view
of~\eqref{eq:unif_estim_I(q,wtQ).intro} the summation
in~\eqref{eq:sum.ln-wtln.intro}--\eqref{eq:weight.ln-wtln<Q-wtQ.intro} takes the
form: $\sum_{|n|\ge N_1}$. Here $N_1 (\in \bN)$ does not depend on $Q, \wt{Q}
\in \cK$.

Relation~\eqref{eq:ln-ln0<eps_Intro} is also valid for \emph{regular BC} and
extends Theorem 3 from~\cite{Sad16} to the case of Dirac type system
$(b_1 \ne -b_2)$. When $\wt{Q} = 0$
estimates~\eqref{eq:sum.ln-wtln.intro}--\eqref{eq:weight.ln-wtln<Q-wtQ.intro}
give (uniform on balls) $l^p$-estimates of the remainder in the asymptotic
formula~\eqref{eq:l.n=l.n0+o(1)_Intro_1} for the eigenvalues of the
\emph{regular} problem~\eqref{eq:BC.intro} for Dirac type system.
For Dirac operator $(b_1=-b_2=1)$ estimate~\eqref{eq:sum.ln-wtln.intro} with
$\wt{Q}=0$ generalizes the corresponding result obtained firstly
in~\cite[Theorems 4.3, 4.5]{SavShk14} for compact $\cK=\{Q,0\}$ consisting of
two entries. For Dirac operator $(-b_1=b_2=1)$
estimate~\eqref{eq:sum.ln-wtln.intro} with $\wt{Q}=0$ generalizes the
corresponding result obtained firstly in~\cite[Theorems 4.3, 4.5]{SavShk14} with
a constant $C$ that depends on $Q$ (i.e. for two points compact $\cK = \{Q,0\}$)
and later on in~\cite{SavSad18} for arbitrary compacts $\cK$ in $\LL{1}$.
Note in this connection, that in the very recent papers, A.~Gomilko and
L.~Rzepnicki~\cite{GomRze20} and A.~Gomilko~\cite{Rzep20} obtained new, sharp
asymptotic formulas for eigenfunctions of Sturm–Liouville operators with
singular potentials, and for eigenvalues and eigenfunctions of Dirichlet BVP for
Dirac system with $Q \in \LL{p}$, $1 \le p < 2$, respectively.

Weighted estimate~\eqref{eq:weight.ln-wtln<Q-wtQ.intro} is new even for
Dirac system with $Q\in \bU_{p,r}^{2 \times 2}$ and $\wt{Q}=0$ and even for
trivial compact $\cK=\{Q,0\}$.

\textbf{(II)} In Sections~\ref{sec:stability.eigenvalue}
and~\ref{sec:eigenfunction.stabil} we investigate Fourier transform of
transformation operators kernels $K^{\pm}_Q$ from
representations~\eqref{eq:e=(I+K)e0_Intro}. Our main result in this direction
plays a key role in the sequel and reads as follows. Its proof relies on a new
estimation technique for underlying integral equation on the kernels of the
transformation operators.
\begin{theorem} \label{th:int.wh.K<int.wh.Q.intro}
Let $Q, \wt{Q} \in \bU_{1,r}^{2 \times 2}$ for some $r > 0$ and let $K^{\pm}_Q$
and $K_{\wt{Q}}^{\pm}$ be the kernels from integral
representation~\eqref{eq:e=(I+K)e0}. Then there is a constant $C = C(r, B)> 0$,
not dependent on $Q$, $\wt{Q}$, such that the following uniform estimate holds:
\begin{multline} \label{eq:Kjk<sup.Q.sup_Intro}
 \sum_{j,k=1}^2 \abs{\int_0^x \Bigl(K^{\pm}_Q -
 K^{\pm}_{\wt{Q}}\Bigr)_{jk}(x, t) e^{i b_k \l t} \,dt}
 \le C e^{2 (b_2-b_1) |\Im \l| x} \cdot \sum_{j \ne k}
 \(\sup_{s \in [0,x]} \abs{\int_0^s \(Q_{jk}(t) - \wt{Q}_{jk}(t)\)
 e^{i (b_k-b_j) \l t} dt} \right. \\
 \left. + \,\|Q - \wt{Q}\|_1 \sup_{s \in [0,x]}
 \abs{\int_0^s \wt{Q}_{jk}(t) e^{i (b_k-b_j) \l t} dt}\),
 \qquad x \in [0,1], \quad \l \in \bC.
\end{multline}
\end{theorem}
First, we apply estimate~\eqref{eq:Kjk<sup.Q.sup_Intro} to the fundamental
matrix solution $\Phi_Q(x,\l)$ of system~\eqref{eq:systemIntro_1} by
establishing Lipshitz property of the mapping $Q \to \Phi_Q(x,\l)$ on the
$L^1$-balls $\bU_{1,r}^{2 \times 2}$. Namely, we show that there exists a
constant $C = C(r, B) > 0$, not dependent on $Q, \wt{Q} \in
\bU_{1,r}^{2 \times 2}$, such that the following uniform estimate holds
\begin{multline} \label{eq:Phi.Fourier_Intro}
 \abs{\Phi_Q(x,\l) - \Phi_{\wt{Q}}(x,\l)}
 \le C e^{2 (b_2-b_1) |\Im \l| x} \cdot \sum_{j \ne k}
 \(\sup_{s \in [0,x]} \abs{\int_0^s \(Q_{jk}(t) - \wt{Q}_{jk}(t)\)
 e^{i (b_k-b_j) \l t} dt} \right. \\
 \left. + \|Q - \wt{Q}\|_1 \sup_{s \in [0,x]}
 \abs{\int_0^s \wt{Q}_{jk}(t) e^{i (b_k-b_j) \l t} dt}\),
 \qquad x \in [0,1], \quad \l \in \bC.
\end{multline}

To state the next application of Theorem~\ref{th:int.wh.K<int.wh.Q.intro}
let us assume the spectrum $\L_{U,Q} = \{\l_{Q,n}\}_{n \in \bZ}$ of $L_U(Q)$ to
be asymptotically simple, and introduce a sequence $\{f_{Q,n}\}_{|n| > N}$ of
corresponding normalized eigenfunctions, $L_U(Q) f_{Q,n} = \l_{Q,n}f_{Q,n}$.

Now we are ready to state $l^p$-stability properties of eigenfunctions of the
operators $L_U(Q)$.
\begin{theorem} \label{th:eigenvalue.holes_Intro}
Let $Q, \wt{Q} \in \bU_{p,r}^{2 \times 2}$, $p \in (1,2]$, $r > 0$, and let BC
$\{U_j\}_1^2$ of the form~\eqref{eq:BC.intro} be strictly regular. Then there
exists an enumeration of the spectra $\{\l_{Q,n}\}_{n \in\bZ}$ and
$\{\l_{\wt{Q},n}\}_{n \in \bZ}$ of the operators $L_U(Q)$ and $L_U({\wt{Q}})$,
and the set $\cI_{Q,\wt{Q}} \subset \bZ$, such that with some constants
$C, N > 0$, not dependent on matrices $\{Q, \wt{Q}\}$, the following estimates
hold
\begin{align}
\label{eq:fn-wt.fn<Q.Lp.intro}
 \sum_{n \in \cI_{Q, \wt{Q}}}
 \bignorm{f_{Q,n} - f_{\wt{Q},n}}_{C([0,1]; \bC^2)}^{p'}
 & \le C \cdot \|Q - \wt{Q}\|_p^{p'}, \\
\label{eq:fn-wt.fn<Q.Hardy.intro}
 \sum_{n \in \cI_{Q, \wt{Q}}} (1+|n|)^{p-2}
 \bignorm{f_{Q,n} - f_{\wt{Q},n}}_{C([0,1]; \bC^2)}^{p}
 & \le C \cdot \|Q - \wt{Q}\|_p^{p}.
\end{align}
\end{theorem}
On compact sets $\cK$ in $L^p$
estimates~\eqref{eq:fn-wt.fn<Q.Lp.intro}--\eqref{eq:fn-wt.fn<Q.Hardy.intro} are
simplified because subsets $\cI_{Q,\wt{Q}} \subset \bZ$ can be chosen to be
independent of the pair $\{Q,\wt{Q}\}$, and in view
of~\eqref{eq:unif_estim_I(q,wtQ).intro} the summation
in~\eqref{eq:sum.ln-wtln.intro}--\eqref{eq:fn-wt.fn<Q.Hardy.intro} can be
replaced with the following: $\sum_{|n|\ge N_1}$. Here $N_1 (\in \bN)$ does not
depend on $Q$ and $\wt{Q}$. Inequality~\eqref{eq:fn-wt.fn<Q.Hardy.intro}
generalizes the classical Hardy-Littlewood
inequality~\cite[Theorem XII.3.19]{Zig59_v2} for Fourier coefficients (see
Example~\ref{ex:stabil.equiv}).

The proof of estimates~\eqref{eq:fn-wt.fn<Q.Lp.intro}--\eqref{eq:fn-wt.fn<Q.Hardy.intro}
substantially relies on the uniform
estimate~\eqref{eq:Phi.Fourier_Intro}. Moreover, for efficient estimate of the
``maximal'' Fourier transform (see definition~\eqref{eq.Fg.sup.Gx_Intro})
we use classical theorems by Carleson-Hunt(\cite[Theorem 6.2.1]{Gra09}),
Hausdorff-Young, and Hardy-Littlewood (see~\cite[Chapter XII]{Zig59_v2}).

Estimates~\eqref{eq:sum.ln-wtln.intro}--\eqref{eq:weight.ln-wtln<Q-wtQ.intro} with $\wt{Q} = 0$ give
(uniform on balls) $l^p$-estimates of the remainder in the asymptotic formula for the
eigenvalues of the problem~\eqref{eq:BC.intro} for Dirac type system. In this case
$(\wt{Q} = 0)$ for Dirac system $(b_1 = -b_2)$ estimate~\eqref{eq:sum.ln-wtln.intro}
was recently and by another method obtained by A. Savchuk~\cite{Sav19}.

Note in conclusion that \emph{periodic and antiperiodic (necessarily
non-strictly regular) BVP} for $2 \times 2$ Dirac and Sturm-Liouville equations
have also attracted certain attention during the last decade. For\ instance, a
criterion for the system of root functions of the \emph{periodic} BVP for
$2 \times 2$ Dirac equation to contain a Riesz basis (without parentheses!) was
obtained by P. Djakov and B. Mityagin in~\cite{DjaMit12Crit} (see also recent
survey~\cite{DjaMit20UMNper} and recent papers~\cite{Mak19},~\cite{Mak20} by
A.S.~Makin and the references therein). It is also worth mentioning that
F.~Gesztesy and V.A.~Tkachenko~\cite{GesTka09,GesTka12} for $q \in L^2[0,\pi]$
and P.~Djakov and B.S.~Mityagin~\cite{DjaMit12Crit} for $q \in W^{-1,2}[0,\pi]$
established by different methods a \emph{criterion} for the system of root
functions to contain a Riesz basis for Sturm-Liouville operator
$-\frac{d^2}{dx^2} + q(x)$ on $[0,\pi]$ (see also survey~\cite{Mak12}).

Finally, we mention that the Riesz basis property for abstract operators is
investigated in numerous papers. Due to the lack of space we only
mention~\cite{Katsn67,MarMats84,Markus88,Agran99}, the recent
survey~\cite{Shk10}, and the references therein.

The paper is organized as follows.

In Section~\ref{sec:X1p} we introduce the Banach spaces $X_{1,p}:= X_{1,p}(\Omega)$
and $X_{\infty,p} := X_{\infty,p}(\Omega)$ as well as their separable subspaces
$X_{\infty,p}^0(\Omega)$ and $X_{1,p}^0(\Omega)$ being the closures of $C(\Omega)$ in
respective norms, and investigate Volterra type operators of the form
\eqref{eq:e=(I+K)e0_Intro} with kernels from these spaces. Besides, we show that for
each $a\in[0,1]$ the trace mappings $N(x,t)\to N(a,t)$ and $N(x,t)\to N(x,a)$
originally defined on $C(\Omega)$ are extended as continuous mappings
$X_{\infty,p}^0(\Omega)\to L^p[0,a]$ and $X_{1,p}^0(\Omega) \to L^p[a,1]$, respectively.
So, the functions $g_l$ in~\eqref{eq:Delta-wt.Delta_Intro}, being traces of
$K^{\pm}_Q,$ are well defined.

In Section~\ref{sec:Transform} we show that $K^{\pm}_Q \in
X_{\infty,p}^0(\Omega) \cap X_{1,p}^0(\Omega)$ and prove the main
estimate~\eqref{eq:K-wtK<Q-wtQ_Intro} for the difference of transformation
operators kernels, i.e. establish \emph{the Lipshitz property} of the mapping
$Q \to K^{\pm}_Q$ on the $L^p$-balls $\bU_{p,r}^{2 \times 2}$. In passing we
prove several auxiliary statements on kernels $K^{\pm}_Q$ useful in
Subsection~\ref{subsec:maxim.Fourier} for proving important
estimate~\eqref{eq:Kjk<sup.Q.sup_Intro}.

In Section~\ref{sec:general} we apply the main uniform estimate~\eqref{eq:K-wtK<Q-wtQ_Intro}
 to characteristic determinants $\Delta_Q$. In particular,
 we prove here Lemma~\ref{lem:dif-ce_of_determ-s_Four_repres-n} as well as clarify asymptotic
formula~\eqref{eq:l.n=l.n0+o(1)_Intro_1}.
 We also indicate here (see Remark~\ref{rem:cond.examples}) certain classes of
\emph{strictly regular BC} in purely algebraic terms. Emphasize that as distinguished
from the case of Dirac operators such a description \emph{is non-trivial for Dirac type
$(b_1\not = -b_2)$ operators}. For\ instance, if $bc=0$ and $ad\not = 0$, then
BC~\eqref{eq:BC.Reg_B_Con_Intro} are strictly regular whenever $b_1 \ln |d| + b_2 \ln
|a| \ne 0.$ Besides, \textbf{antiperiodic boundary
conditions}~\eqref{eq:BC.Reg_B_Con_Intro} $(a=d=1$, $b=c=0)$ \textbf{are strictly
regular} provided that
\begin{equation} \label{eq:cond_on_b_1,b_2}
b_1= -n_1\beta, \quad b_2 = n_2\beta,\quad \text{with} \quad n_1, n_2\in \bN, \ \beta\in \bC, \quad \text{and}\quad
n_1 - n_2 =2p+1\in \bZ.
\end{equation}
Therefore the previous results imply the following surprising statement.
\begin{corollary} \label{cor:basis.Perid_Cond}
Let $Q, \wt{Q} \in \bU_{p,r}^{2 \times 2}$, $p \in (1,2]$.
Assume also that $b_1, b_2$ satisfy~\eqref{eq:cond_on_b_1,b_2}. Then
\textbf{antiperiodic BC} are \textbf{strictly regular}, hence \textbf{the operator $L_{U}(Q)$ has Riesz basis property}. Moreover, the corresponding eigenvalues and eigenvectors satisfy uniform Lipshitz
type estimates~\eqref{eq:sum.ln-wtln.intro}--\eqref{eq:weight.ln-wtln<Q-wtQ.intro} and ~\eqref{eq:fn-wt.fn<Q.Lp.intro}--\eqref{eq:fn-wt.fn<Q.Hardy.intro}, respectively.
\end{corollary}
This result demonstrates substantial difference between Dirac and Dirac type operators.

In Section~\ref{sec:fourier.transform} we investigate Fourier transform and
\emph{``maximal'' Fourier transform}
\begin{equation} \label{eq.Fg.sup.Gx_Intro}
 F[g](\l) := \int_0^1 g(t) e^{i \l t}\, dt \quad\text{and}\quad \sF[g](\l) := \sup_{x \in [0,1]}
 \abs{\int_0^x g(t) e^{i \l t}\, dt}, \qquad \l \in \bC,
\end{equation}
in order to generalize certain classical Hausdorff-Young and Hardy-Littlewood
theorems for Fourier coefficients (\cite[Chapter XII]{Zig59_v2}). In particular,
we prove here that for any incompressible sequence $\L = \{\mu_n\}_{n \in \bZ}$
in the strip $\Pi_h$, there is a constant $C = C(p, h, d) > 0$, not dependent on
$\L$, such that the following estimate holds uniformly in $g$ and $\L$:
\begin{equation} \label{eq:sum.int.nu.g<g_Intro}
\sum_{n \in \bZ} (1+|n|)^{p-2} |F[g]|^{p}(\mu_n)
 \le \sum_{n \in \bZ} (1+|n|)^{p-2} \sF[g]^{p}(\mu_n)
 \le C \cdot \|g\|_p^p, \qquad g \in L^p[0,1], \quad p \in (1, 2].
\end{equation}
Here and throughout the paper $\Pi_h := \{z \in \bC : |\Im z| \le h\}$ denotes a
symmetrical horizontal strip of semi-width $h \ge 0$.

Inequality~\eqref{eq:sum.int.nu.g<g_Intro} generalizes Hardy-Littlewood theorem (\cite[Theorem XII.3.19]{Zig59_v2})
and coincides with it for ordinary Fourier transform and
$\mu_n = 2 \pi n$. In turn, this inequality is an important ingredient in
proving the estimate~\eqref{eq:weight.ln-wtln<Q-wtQ.intro}.

We also consider ``maximal'' Fourier transform $\cF[K^{\pm}_Q]$ given
by~\eqref{eq.Fg.sup.Gx_Intro} with kernels $K^{\pm}_Q(x,t)$ in place of $g$ and
applying the main estimate~\eqref{eq:K-wtK<Q-wtQ_Intro} establish the following
uniform Lipshitz estimate on $L^p$-balls:
\begin{equation} \label{eq:FK-FwtK_intro}
 \|\cF[K^{\pm}_Q] - \cF[K^{\pm}_{\wt{Q}}]\|_{C(\Pi_h)} \le e^{|b|h} \cdot
 C(B, p, r) \cdot \|Q - \wt{Q}\|_{L^p},
 \qquad Q, \wt{Q} \in \bU_{p, r}^{2 \times 2}, \quad p \in [1, \infty).
\end{equation}

In Section~\ref{sec:stability.eigenvalue} we prove
estimates~\eqref{eq:ln-ln0<eps_Intro}--\eqref{eq:weight.ln-wtln<Q-wtQ.intro}.
Here we substantially use the following important statement that can be
extracted from Lemma~\ref{lem:dif-ce_of_determ-s_Four_repres-n}: assuming the
spectra $\L_{U,Q} = \{\l_{Q,n}\}_{n \in \bZ}$ and $\L_{U,\wt{Q}} =
\{\l_{\wt{Q},n}\}_{n \in \bZ}$ of operators $L_U(Q)$ and $L_U(\wt{Q})$,
respectively, to be asymptotically simple, we show that for any compact subset
$\cK$ of $\bU_{p, r}^{2 \times 2}$ there exist constants $C$ and $N > 0$, not
dependent on $\{Q, \wt{Q}\}\subset \cK$, such that the following \emph{uniform}
two-sided estimates hold
\begin{equation} \label{eq:ln-wtln<C.Delta.intro}
 C^{-1} \cdot |\Delta_{\wt{Q}}(\l_{Q,n})| \le |\l_{Q,n} - \l_{\wt{Q},n}| \le
 C \cdot |\Delta_{\wt{Q}}(\l_{Q,n})|, \qquad |n| > N, \quad Q, \wt{Q} \in \cK.
\end{equation}
Emphasize that in
proving~\eqref{eq:sum.ln-wtln.intro}--\eqref{eq:weight.ln-wtln<Q-wtQ.intro} we
use only evaluation of the ordinary Fourier transform and \emph{do not use deep
Carleson-Hunt result}. In particular, the proof
of~\eqref{eq:weight.ln-wtln<Q-wtQ.intro} relies on the uniform estimate between
the first and third terms in~\eqref{eq:sum.int.nu.g<g_Intro} only with classical
$F$ (not $\cF$). This fact makes the proof of
estimates~\eqref{eq:sum.ln-wtln.intro}--\eqref{eq:weight.ln-wtln<Q-wtQ.intro}
elementary in character.

In Section~\ref{sec:eigenfunction.stabil} we
prove Theorems~\ref{th:int.wh.K<int.wh.Q.intro}
and~\ref{th:eigenvalue.holes_Intro}. We also establish uniform estimates similar
to~\eqref{eq:fn-wt.fn<Q.Lp.intro}--\eqref{eq:fn-wt.fn<Q.Hardy.intro}, where
the ball $\bU_{p,r}^{2 \times 2}$ is replaced with a compact $\cK \subset L^p$,
and summation is over $|n| > N$ (see Theorem~\ref{th:eigenfunc.compact}).
Proposition~\ref{prop:root.Q12=0} shows that in the case $Q_{12} = 0$ summation
in~\eqref{eq:fn-wt.fn<Q.Lp.intro}--\eqref{eq:fn-wt.fn<Q.Hardy.intro} can be
extended to the entire $\bZ$ and the estimate remains uniform for all $Q_{21}
\in L^p[0,1]$, $p \in (1,2]$.

In Section~\ref{sec:bari} we apply results of
Section~\ref{sec:eigenfunction.stabil} in the case of $Q \in L^2$ to establish a
criterion for the system of root vectors of
BVP~\eqref{eq:systemIntro_1}--\eqref{eq:BC.intro} \emph{to form a Bari basis in}
$\LLV{2}$. To the best of our knowledge this problem was not investigated
before. Inequality~\eqref{eq:fn-wt.fn<Q.Lp.intro} reduces the general case to a
simplest case of $Q=0$. In Proposition~\ref{prop:crit.bari.b.ne.0} we establish
a criterion in terms of the eigenvalues of the unperturbed operator $L_0$. If
either $b_1 b_2^{-1} \in \bQ$ or $abcd=0$, then the criterion can be simplified
as follows: the system of root vectors of such
BVP~\eqref{eq:systemIntro_1},~\eqref{eq:BC.Reg_B_Con_Intro} forms a Bari basis
if and only if BC~\eqref{eq:BC.Reg_B_Con_Intro} are self-adjoint.
\section{The Banach spaces $X_{1,p}$ and $X_{\infty,p}$} \label{sec:X1p}
Let $p \in [1, \infty]$. Following~\cite{Mal94} denote by $X_{1,p} :=
X_{1,p}(\Omega)$ and $X_{\infty,p} := X_{\infty,p}(\Omega)$ the linear spaces
composed of (equivalent classes of) measurable functions defined on $\Omega :=
\{(x,t) : 0 \le t \le x \le 1\}$ satisfying
\begin{align}
\label{eq:X1p.norm.def}
 \|f\|_{X_{1,p}}^p & :=
 \esssup_{t \in [0,1]} \int_t^1 |f(x,t)|^p \,dx < \infty, \qquad p < \infty, \\
\label{eq:Xinfp.norm.def}
 \|f\|_{X_{\infty,p}}^p & :=
 \esssup_{x \in [0,1]} \int_0^x |f(x,t)|^p \,dt < \infty, \qquad p < \infty,
\end{align}
respectively, and $\|f\|_{X_{1,\infty}} = \|f\|_{X_{\infty,\infty}}
= \esssup_{(x,t) \in \Omega} |f(x,t)|$. It can easily be
shown that the spaces $X_{1,p}$ and $X_{\infty,p}$ equipped with the
norms~\eqref{eq:X1p.norm.def} and~\eqref{eq:Xinfp.norm.def} form Banach spaces that
are not separable. Denote by $X_{1,p}^0 := X_{1,p}^0(\Omega)$ and
$X_{\infty,p}^0 := X_{\infty,p}^0(\Omega)$ the subspaces of $X_{1,p}(\Omega)$
and $X_{\infty,p}(\Omega)$, respectively, obtained by taking the closure of
continuous functions $f \in C(\Omega)$. Clearly, the set $C^1(\Omega)$ of smooth
functions is also dense in both spaces $X_{1,p}^0$ and $X_{\infty,p}^0$. Note
also that the following embeddings hold and are continuous
\begin{equation} \label{eq:Xqp_embedding}
 X_{1,p_1} \subset X_{1,p_2} \subset X_{1,1} \qquad \text{and}\qquad
 X_{\infty,p_1} \subset X_{\infty,p_2} \subset X_{\infty,1},
 \qquad p_1 > p_2 \ge 1.
\end{equation}
The following simple properties of the class $X_{\infty,p}^0(\Omega)$ will be
important in the sequel.
\begin{lemma} \label{lem:trace.X}
Let $p \ge 1$. For each $a \in [0,1]$, the trace mappings
\begin{equation}
 i_{a,\infty}:\ C(\Omega) \to C[0,a], \quad
 i_{a,\infty}\bigl(N(x,t)\bigr):=N(a,t), \quad
 \text{and}\quad i_{a,1}:\ C(\Omega) \to C[a,1], \quad
 i_{a,1}\bigl(N(x,t)\bigr):=N(x,a),
\end{equation}
originally defined on $C(\Omega)$ admit continuous extensions (also denoted by
$i_{a,\infty}$ and $i_{a,1}$, respectively) as mappings from
$X_{\infty,p}^0(\Omega)$ onto $L^p[0,a]$ and $X_{1,p}^0(\Omega)$ onto
$L^p[a,1]$, respectively.
\end{lemma}
\begin{proof}
The proof is immediate by combining the corresponding results for
$X_{\infty,1}^0(\Omega)$ and $X_{1,1}^0(\Omega)$
(see~\cite[Lemma 2.2]{LunMal16JMAA}) with the continuity of
embeddings~\eqref{eq:Xqp_embedding}.
\end{proof}
Going over to the vector case let us recall some notations. For $u = \col(u_1,
\ldots, u_n) \in \bC^n$ we denote
\begin{equation}
 |u|_\a^\a := |u_1|^\a + \ldots + |u_n|^\a, \quad 0 < \a < \infty,
 \qquad |u|_{\infty} = \max\{|u_1|, \ldots, |u_n|\}.
\end{equation}
When there is no confusion, we will denote
$|u| = |u|_\a$. $\bC^n$ equipped with the norm $\|\cdot\|_\a$ will be denoted
as $\bC^n_\a$. We will also use notation $u^* := \begin{pmatrix} \ol{u_1} &
\ldots & \ol{u_n} \end{pmatrix}$. Clearly, the product $u_1 u_2^* \in \bC$.

Further, for $A = (a_{jk})_{j,k=1}^n \in \bC^{n \times n}$ we
denote as $|A|_{\a \to \b}$ a norm of the linear operator from $\bC^n_\a$ to
$\bC^n_\b$ with a matrix $A$, where $\a, \b \in [1, \infty]$,
\begin{equation} \label{eq:|A|.def}
 |A|_{\a \to \b} = \sup\{|A u|_\b : u \in \bC^n, |u|_\a = 1\}.
\end{equation}
Similarly, when there is no confusion, we will denote $|A| = |A|_{\a \to \b}$.
We will use the following well-known inequalities in the sequel:
\begin{align}
 & |u_1 u_2^*| \le |u_1|_{\a} \cdot |u_2^*|_{\a'}, \qquad 1/\a + 1/\a' = 1,
 \qquad u_1, u_2 \in \bC^n, \\
 & |Au|_\b \le |A|_{\a \to \b} \cdot |u|_\a, \qquad \a, \b \in [1, \infty],
 \qquad A \in \bC^{n \times n}, \quad u \in \bC^n, \\
 & |A_1 A_2|_{\a \to \b} \le |A_1|_{\a \to \b} \cdot |A_2|_{\a \to \b},
 \qquad 1 \le \a \le \b \le \infty, \qquad A_1, A_2 \in \bC^{n \times n}.
\end{align}

Now we are ready to introduce the Banach spaces
\begin{equation} \label{eq:Xhpnxn.def}
 X_{1,p}^n := X_{1,p}(\Omega; \bC^{n \times n}) :=
 X_{1,p}(\Omega) \otimes \bC^{n \times n} \qquad\text{and}\qquad
 X_{\infty, p}^n := X_{\infty,p}(\Omega; \bC^{n \times n}) :=
 X_{\infty,p}(\Omega) \otimes \bC^{n \times n},
\end{equation}
consisting of $n \times n$ matrix-functions $F = (F_{jk})_{j,k=1}^n$ with
$X_{1,p}$- and $X_{\infty,p}$-entries, respectively, equipped with the norms
\begin{align}
\label{eq:X1p.nxn.norm.def}
 \|F\|_{X_{1,p}^n}^p & :=
 \esssup_{t \in [0,1]} \int_t^1 |F(x,t)|_{1 \to p}^p \,dx < \infty,
 \qquad p \in [1, \infty), \\
\label{eq:Xinfp.nxn.norm.def}
 \|F\|_{X_{\infty,p}^n}^p & :=
 \esssup_{x \in [0,1]} \int_0^x |F(x,t)|_{p' \to \infty}^p \,dt < \infty,
 \qquad p \in [1, \infty), \qquad 1/p + 1/p' = 1,
\end{align}
and $\|F\|_{X_{1,\infty}^n} = \|F\|_{X_{\infty,\infty}^n} =
\esssup_{(x,t) \in \Omega} |F(x,t)|_{1 \to \infty}$.
Besides, we introduce subspaces
\begin{equation} \label{eq:Zhpnxn.def}
 X_{1,p}^{0,n} := X_{1,p}^0(\Omega; \bC^{n \times n}) :=
 X_{1,p}^0(\Omega) \otimes \bC^{n \times n} \quad \text{and} \quad
 X_{\infty,p}^{0,n} := X_{\infty,p}^0(\Omega; \bC^{n \times n}) :=
 X_{\infty,p}^0(\Omega) \otimes \bC^{n \times n},
\end{equation}
being separable parts of $X_{1,p}^n$ and $X_{\infty,p}^n$, respectively.

For brevity, throughout the section we denote
\begin{equation}
 L^s := L^s([0,1]; \bC^n), \qquad C := C([0,1]; \bC^n),
 \qquad s \in [1, \infty],
\end{equation}
Equip the space $L^s$ of vector functions with the following norm
\begin{equation}
 \|f\|_s^s := \|\col(f_1, \ldots, f_n)\|_s^s :=
 \|f_1\|_s^s + \ldots + \|f_n\|_s^s :=
 \|f_1\|_{L^s[0,1]}^s + \ldots + \|f_n\|_{L^s[0,1]}^s, \qquad s \in [1, \infty),
\end{equation}
and $\|f\|_{\infty} = \max\{\|f_1\|_{\infty}, \ldots, \|f_n\|_{\infty}$. It is
clear, that $\|f\|_s^s = \int_0^1 |f(x)|_s^s \,dx$, $s \in [1, \infty)$.
With each measurable matrix kernel $N(x,t) = \bigl(N_{jk}(x, t)\bigr)^n_{j,k=1}$
one associates a Volterra type operator
\begin{equation} \label{eq:cN.def}
 \cN: f \to \int^x_0 N(x,t) f(t) dt.
\end{equation}
Denote by $\|\cN\|_{\a \to \b} := \|N\|_{L^\a \to L^\b}$, $\a, \b \in
[1,\infty]$, the norm of the operator $\cN$ acting from $L^\a$ to $L^\b$,
provided that it is bounded, $\cN \in \cB(L^\a, L^\b)$. The following result
motivates introduction of the spaces $X_{1,1}^n$ and $X_{\infty,1}^n$. In
particular, the second statement sheds light on the interpolation role of these
spaces (cf.~\cite{Mal94}). This result substantially completes Lemma 2.3
from~\cite{LunMal16JMAA}.

Recall that a Volterra operator on a Banach space is \emph{a compact operator
with zero spectrum}.
\begin{lemma} \label{lem:volterra_oper}
Let $\cN$ be a Volterra type operator given by~\eqref{eq:cN.def} for a
measurable matrix-function $N(\cdot, \cdot)$.

\textbf{(i)} Let either $q=1$ or $q=\infty$. Then $\cN \in \cB(L^q, L^q)$ if and
only if $N \in X_{q,1}^n$\,, in which case
\begin{equation} \label{eq:|cN|1-1.inf-inf}
 \|\cN\|_{q \to q} = \|N\|_{X_{q,p}^n}.
\end{equation}
If $N \in X_{q,1}^{0,n}$, then the operator $\cN$ is a Volterra operator in
$L^q$, and the inverse $(I + \cN)^{-1}$ is given by
\begin{equation} \label{eq:1+cN.inv}
 (I + \cN)^{-1} = I + \cS : \ f \to f + \int^x_0 S(x,t) f(t) \,dt
 \quad \text{with}\quad S \in X_{q,1}^{0,n}.
\end{equation}

\textbf{(ii)} Let $N \in X_{1,1}^n \cap X_{\infty,1}^n$. Then
$\cN \in \cB(L^s, L^s)$ for each $s \in [1,\infty]$, and
\begin{equation} \label{eq:|cN|s-s}
 \|\cN\|_{s \to s} \le \|N\|_{X_{1,1}^n}^{1/s} \cdot
 \|N\|_{X_{\infty,1}^n}^{1-1/s}.
\end{equation}
Moreover, if $N \in X_{1,1}^{0,n} \cap X_{\infty,1}^{0,n}$, then $\cN$ is a
Volterra operator in $L^s$ for each $s \in [1,\infty]$.
\end{lemma}
\begin{proof}
\textbf{(i)} Relation~\eqref{eq:|cN|1-1.inf-inf} was proved in Lemma 2.3
from~\cite{LunMal16JMAA} (see also~\cite{Mal94}). In particular, it implies
equivalence of inclusions $\cN \in \cB(L^q) := \cB(L^q, L^q)$ and $N \in
X_{q,1}^n$. Let us prove representation~\eqref{eq:1+cN.inv}. Since $N \in
X_{q,1}^{0,n}$, one finds a sequence $N_k \in C(\Omega; \bC^{n \times n})$ that
approaches $N$ in $X_{q,1}^n$ as $k \to \infty$. It is well-known, that the
corresponding operators $\cN_k$ are Volterra operators in $L^q$, such that
$(I + \cN_k)^{-1} = I + \cS_k$, where $\cS_k$ is the Volterra operator in $L^q$
of the form~\eqref{eq:cN.def} with $S_k \in C(\Omega; \bC^{n \times n})$.
Identity~\eqref{eq:|cN|1-1.inf-inf} implies that
\begin{equation} \label{eq:Nk-N}
 \lim\limits_{k \to \infty} \|\cN_k - \cN\|_{q \to q} =
 \lim\limits_{k \to \infty} \|N_k - N\|_{X_{q,1}^n} = 0,
\end{equation}
which yields that $\cN$ is a Volterra operator in $L^q$, and existence of
the inverse $(I + \cN)^{-1}$. Further, continuity of the mapping $T \to T^{-1}$
in the group of invertible operators in the Banach algebra $\cB(L^q)$,
relation~\eqref{eq:Nk-N} and identity~\eqref{eq:|cN|1-1.inf-inf} imply
\begin{align}
\label{eq:I+Nk-I+N}
 & I + \cS_k = (I + \cN_k)^{-1} \to (I + \cN)^{-1}
 \quad \text{as} \quad k \to \infty \quad \text{in} \quad L^q, \\
\label{eq:Sk-Sm}
 & \|S_k - S_m\|_{X_{q,1}^n} = \|\cS_k - \cS_m\|_{q \to q} =
 \|(I + \cN_k)^{-1} - (I + \cN_m)^{-1}\|_{q \to q} \to 0
 \quad \text{as} \quad k, m \to \infty.
\end{align}
Thus, there exists $S \in X_{q,1}^{0,n}$ such that $\lim\limits_{k \to \infty}
\|S_k - S\|_{X_{q,1}^n} = 0$. As it has been already proved for the kernel
$N \in X_{q,1}^{0,n}$, the kernel $S$ generates Volterra operator $\cS \in
\cB(L^q)$ of the form~\eqref{eq:cN.def}. Moreover, $\lim\limits_{k \to \infty}
\|\cS_k - \cS\|_{q \to q} = \lim\limits_{k \to \infty} \|S_k - S\|_{X_{q,1}^n} =
0$. Combining this with~\eqref{eq:I+Nk-I+N} implies
equality~\eqref{eq:1+cN.inv}.

\textbf{(ii)} The proof is immediate from relations~\eqref{eq:|cN|1-1.inf-inf}
combined with Riesz-Torin theorem (see~\cite{LunMal16JMAA} and~\cite{Mal94}).
\end{proof}
The following result demonstrates natural occurrence of the spaces
$X_{1,p}^n$ and $X_{\infty,p}^n$ in the study of the integral operators acting
from $L^\a$ to $L^\b$ with the special $\a, \b$.
\begin{proposition} \label{prop:volterra_oper.Xp}
Let $\cN$ be a Volterra type operator given by~\eqref{eq:cN.def} for a
measurable matrix-function $N(\cdot, \cdot)$. Let also $p \in [1, \infty]$ and
$1/p + 1/p' = 1$. Then:

\textbf{(i)} The inclusion $\cN \in \cB(L^1, L^p)$ holds if and only if
$N \in X_{1,p}^n$, in which case
\begin{equation} \label{eq:|cN|1-p}
 \|\cN\|_{1 \to p} = \|N\|_{X_{1,p}^n}.
\end{equation}
Moreover, if $N \in X_{1,p}^{0,n}$, then the operator $\cN$ is compact from
$L^1$ to $L^p$,
and the following relation holds:
\begin{equation} \label{eq:cS.X1p}
 -\cN \cdot (I + \cN)^{-1} = \cS \in \cB(L^1, L^p), \quad \text{where} \quad
 \cS : \ f \to \int^x_0 S(x,t) f(t) \,dt
 \quad \text{with} \quad S \in X_{1,p}^{0,n},
\end{equation}
where operator $(I + \cN)^{-1}$ is treated as an operator from $\cB(L^1, L^1)$,
and exists due to Lemma~\ref{lem:volterra_oper}(i).

\textbf{(ii)} The inclusion $\cN \in \cB(L^{p'}, L^{\infty})$ holds if and only
if $N \in X_{\infty,p}^n$\,, in which case
\begin{equation} \label{eq:|cN|p'-inf}
 \|\cN\|_{p' \to \infty} = \|N\|_{X_{\infty,p}^n},
\end{equation}
If $N \in X_{\infty,p}^{0,n}$ then the operator $\cN$ sends $L^{p'}$ to
$C = C([0,1]; \bC^n)$ and is compact. Let $\cN_C : C \to C$, be a restriction of
the operator $\cN$. Then the inverse operator $(I + \cN_C)^{-1} \in
\cB(C, C)$, and the following relation holds:
\begin{equation} \label{eq:cS.Xinfp}
 -(I + \cN_C)^{-1} \cdot \cN = \cS \in \cB(L^{p'}, C), \quad \text{where} \quad
 \cS : \ f \to \int^x_0 S(x,t) f(t) \,dt
 \quad \text{with} \quad S \in X_{\infty,p}^{0,n}.
\end{equation}

\textbf{(iii)} Let $N \in X_{1,p}^{0,n} \cap X_{\infty,p}^{0,n}$. Then a
triangular operator $\cN$ is a Volterra operator in $L^s$ for each $s \in
[1, \infty]$ and the inverse operator $(I + \cN)^{-1}$ is given by
\begin{equation} \label{eq:I+cN.inv=I+cS}
 (I + \cN)^{-1} = I + \cS: \ f \to f + \int^x_0 S(x,t)f(t) \,dt,
 \qquad S \in X_{1,p}^{0,n} \cap X_{\infty,p}^{0,n}.
\end{equation}
\end{proposition}
\begin{proof}
\textbf{(i)} Let $f \in L^1$, $g \in L^{p'}$, and $N \in X_{1,p}^n$. Then
applying H\"older's inequality one gets
\begin{align}
\nonumber
 |(\cN f,g)_{p,p'}| & = \abs{\int_0^1 \(\int_0^x N(x,t)f(t) \,dt\)g^*(x)\,dx}
 \le \int_0^1 \(\int_0^x \abs{N(x,t) f(t)}_p \,dt\) |g^*(x)|_{p'} \,dx \\
\label{eq:estim_(Nf,g)_f_in_L_1}
 & \le \int_0^1 |f(t)|_1 \,dt \int_t^1 |N(x,t)|_{1 \to p} \cdot
 |g^*(x)|_{p'} \,dx \le \|f\|_1\cdot \|N\|_{X_{1,p}^n} \cdot \|g\|_{p'}.
\end{align}
It follows that $\cN \in \cB(L^1, L^{p})$ and $\|\cN\|_{{1}\to p} \le
\|N\|_{X_{1,p}^n}$. Besides, estimate~\eqref{eq:estim_(Nf,g)_f_in_L_1} yields
$\cN^* \in \cB(L^{p'}, L^{\infty})$, i.e. $\cN^* g = \int_t^1 N(x,t)^* g(x) \,dx
\in L^{\infty}$ for each $g \in L^{p'}$. From this fact one
extracts the opposite inequality $\|\cN\|_{{1}\to p} \ge \|N\|_{X_{1,p}^n}$,
which yields equality~\eqref{eq:|cN|1-p}.

Now let $N \in X_{1,p}^{0,n}$. Due to the inclusion~\eqref{eq:Xqp_embedding}
one has $N \in X_{1,1}^{0,n}$. By Lemma~\ref{lem:volterra_oper}(ii), $\cN$ is a
Volterra operator in $L^1$ and the inverse $(I + \cN)^{-1} \in \cB(L^1, L^1)$ is
given by~\eqref{eq:1+cN.inv}, with $S \in X_{1,1}^{0,n}$. Hence $\cS = -\cN
(I + \cN)^{-1} \in \cB(L^1, L^{p})$ because $\cN \in \cB(L^1, L^p)$. This
yields the desired relation~\eqref{eq:cS.X1p}, but only with $S \in
X_{1,1}^{0,n} \cap X_{1,p}^n$.

Let us show that $S \in X_{1,p}^{0,n}$. Since $N \in X_{1,p}^{0,n}$, one finds a
sequence $N_k \in C(\Omega; \bC^{n \times n})$ that approaches $N$ in
$X_{1,p}^{0,n}$ as $k \to \infty$. Clearly, it also approaches $N$ in
$X_{1,1}^{0,n}$. With account of~\eqref{eq:Sk-Sm} and
identity~\eqref{eq:|cN|1-p}, this implies,
\begin{equation} \label{eq:Nk-Nm}
 \norm{N_k - N_m}_{X_{1,p}^n} = \norm{\cN_k - \cN_m}_{1 \to p} \to 0
 \qquad \text{and} \qquad
 \norm{(I + \cN_k)^{-1} - (I + \cN_m)^{-1}}_{1 \to 1} \to 0
 \quad \text{as}\quad k, m \to \infty.
\end{equation}
Clearly, $\cN_k$ is a sequence of compact operators from $L^1$ to $L^p$
approaching $\cN$ in $\|\cdot\|_{1 \to p}$-norm. Hence $\cN$ is a compact
operator from $L^1$ to $L^p$. It is clear that $\cS_k = -\cN_k
(I + \cN_k)^{-1}$, $k \in \bZ$. Hence, identity~\eqref{eq:|cN|1-p} and
relation~\eqref{eq:Nk-Nm} imply
\begin{multline} \label{eq:Sk-Sm.1p}
 \norm{S_k - S_m}_{X_{1,p}^n} =
 \norm{\cS_k - \cS_m}_{1 \to p} =
 \norm{\cN_k (I + \cN_k)^{-1} - \cN_m (I + \cN_m)^{-1}}_{1 \to p} \\
 \le \norm{\cN_k - \cN_m}_{1 \to p} \norm{(I + \cN_k)^{-1}}_{1 \to 1} +
 \norm{\cN_m}_{1 \to p} \norm{(I + \cN_k)^{-1} - (I + \cN_m)^{-1}}_{1 \to 1}
 \to 0 \quad \text{as}\quad k, m \to \infty.
\end{multline}
Thus, there exists $\wt{S} \in X_{1,p}^{0,n}$ such that
$\lim\limits_{k \to \infty} \|S_k - \wt{S}\|_{X_{1,p}^n} = 0$.
As it has been already proved for the kernel $N \in X_{1,p}^{0,n}$, the kernel
$\wt{S}$ generates a compact operator $\wt{\cS} \in \cB(L^1, L^p)$ of the
form~\eqref{eq:cN.def}. Moreover,
$\lim\limits_{k \to \infty} \|\cS_k - \wt{\cS}\|_{1 \to p} =
\lim\limits_{k \to \infty} \|S_k - \wt{S}\|_{X_{1,p}^n} = 0$.
Relation~\eqref{eq:Sk-Sm.1p} also implies that $\cS_k = -\cN_k (I + \cN_k)^{-1}
\to -\cN (I + \cN)^{-1} = \cS$ as $k \to \infty$ in the norm
$\|\cdot\|_{1 \to p}$, which implies that $\cS = \wt{\cS}$ and $S = \wt{S}
\in X_{1,p}^{0,n}$

\textbf{(ii)} Inequality $\|\cN\|_{p' \to \infty} \le \|N\|_{X_{\infty,p}^n}$ is
immediate from H\"older's inequality. The opposite one is proved as in the
previous step. The inclusion $\range(\cN) \subset C = C([0,1]; \bC^n)$ is
obvious for $N \in C(\Omega; \bC^n)$. Given an arbitrary $N \in
X_{\infty,p}^{0,n}$ one choose a sequence of continuous functions $N_k \subset
C(\Omega; \bC^n)$ approaching $N$ in $X_{\infty,p}^n$ as $k \to \infty$. Then
for any $f \in L^{p'}$ the sequence $\cN_k f \in C$ approaches $\cN f$
uniformly, hence $\cN f \in C$. It follows from the equality
$\lim\limits_{k \to \infty} \|\cN - \cN_k\|_{p' \to \infty} =
\lim\limits_{k \to \infty} \|N - N_k\|_{X_{\infty,p}^n} = 0$ that $\cN$ is
compact from $L^{p'}$ to $C$.

Now let $N \in X_{\infty,p}^{0,n} \subset X_{\infty,1}^{0,n}$.
Lemma~\ref{lem:volterra_oper}(i) implies that $\cN$ is a Volterra operator in
$L^{\infty}$ and, thus, $\cN_C$ is a Volterra operator in $C$ with
$(I + \cN_C)^{-1} \in \cB(C, C)$. It follows that $\cS := - (I + \cN_C)^{-1} \cN
\in \cB(L^{p'}, C)$. The proof is finished in the same way as in part (i), where
all the norms $\|\cdot\|_{X_{1,p}^n}$, $\|\cdot\|_{1 \to p}$ and
$\|\cdot\|_{1 \to 1}$ are replaced with $\|\cdot\|_{X_{\infty,p}^n}$,
$\|\cdot\|_{p' \to \infty}$ and $\|\cdot\|_{\infty \to \infty}$, respectively.

\textbf{(ii)} The statement is immediate by combining parts (i) and (ii) with
Lemma~\ref{lem:volterra_oper}(ii).
\end{proof}
\begin{remark}
In connection with Proposition~\ref{prop:volterra_oper.Xp}, let us recall
Theorems~XI.1.5 and~XI.1.6 from~\cite{KantAkil77}, concerning integral
representations of bounded linear operators. Namely, let $p \in (1, \infty]$,
$p' \in [1, \infty)$ and $1/p + 1/p' = 1$, and let $\cR$ and $\cS$ be bounded
linear operators from $L^1[0,1]$ to $L^p[0,1]$ and $L^{p'}[0,1]$ to $C[0,1]$,
respectively. Then they admit the following integral representations:
\begin{align}
 (\cR f)(x) = \int_0^1 R(x,t)f(t) \,dt, \quad \text{and for} \ p < \infty
 \ \text{its norm is} \quad \|\cR\|_{1 \to p}^p =
 \esssup_{t \in [0,1]} \int_0^1 |R(x,t)|^p \,dx < \infty, \\
 (\cS f)(x) = \int_0^1 S(x,t)f(t) \,dt, \quad \text{and for} \ p < \infty
 \ \text{its norm is} \quad \|\cS\|_{p' \to \infty}^{p} =
 \esssup_{x \in [0,1]} \int_0^1 |S(x,t)|^{p} \,dt < \infty.
\end{align}
\end{remark}
Denote by $\cB_{1,p}(r)\ (\cB_{\infty,p}(r))$ the ball centered at zero of
radius $r$ in $X_{1,p}^{0,n}\ (X_{\infty,p}^{0,n})$.
\begin{lemma} \label{lem:I+S}
Let either $q=1$ or $q=\infty$ and $p \in [2, \infty)$. Let $N \in
X_{q,p}^{0,n}$ and let $S$ be the kernel from~\eqref{eq:cS.X1p}
or~\eqref{eq:cS.Xinfp} if $q=1$ or $q=\infty$, respectively, i.e. the kernel
of ``inverse'' operator $\cS = -\cN (I + \cN)^{-1}$ or
$\cS = -(I + \cN_C)^{-1} \cN$.

\textbf{(i)} The kernel $S$ satisfies the following inequality
\begin{equation} \label{eq:|S|.Xqp}
 \|S\|_{X_{q,p}^n} \le
 2^{1-1/p} \cdot \|N\|_{X_{q,p}^n} \cdot
 \exp\left(p^{-1} 2^{p-1} \|N\|^p_{X_{q,p}^n}\right).
\end{equation}

\textbf{(ii)} Moreover, if $N, \wt{N} \in \cB_{q,p}(r) \subset X_{q,p}^{0,n}$,
and $\wt{S}$ is the kernel defined for $\wt{N}$ similarly as above
from~\eqref{eq:cS.X1p} or~\eqref{eq:cS.Xinfp}, then the following uniform
estimate holds on the ball $\cB_{q,p}(r)$:
\begin{equation} \label{eq:wh.S<wh.N}
 \|S - \wt S\|_{X_{q,p}^n} \le 3^{1-1/p} \exp\bigl(p^{-1} 3^p r^p\bigr)
 \cdot \|N - \wt N\|_{X_{q,p}^n}.
\end{equation}
\end{lemma}
\begin{proof}
\textbf{(i)} Let first $q=1$, i.e. $N \in X_{1,p}^{0,n}$. By
Proposition~\ref{prop:volterra_oper.Xp}(i), $S \in X_{1,p}^{0,n}$, and
corresponding Volterra type operators are related by $\cS = -\cN
(I + \cN)^{-1}$, where $\cN, \cS \in \cB(L^1, L^p)$ and inverse operator
$(I + \cN)^{-1}$ is treated as an operator in $L^1$. Since $\cN$ acts from $L^1$
to $L^p$, we also have $\cS = -(I + \cN_p)^{-1} \cN$, where $\cN_p$ is a
restriction of $\cN$ acting in $L^p$. Hence $\cN + \cS + \cN_p \cS = 0$ on
$L^1$. This implies that the kernels $N$ and $S$ are related by the equation
\begin{equation} \label{eq:N+S+int.NS}
 N(x,t) + S(x,t) + \int_t^x N(x,\xi) S(\xi,t)\,d\xi = 0,
 \qquad \text{for almost all} \quad (x,t) \in \Omega.
\end{equation}
Since $p \ge 2 \ge p' := \frac{p}{p-1}$, combining H\"older's inequality with
power-mean inequality gives for almost all $(x,t) \in \Omega$,
\begin{equation} \label{eq:int.NS}
 \abs{\int_t^x N(x,\xi) S(\xi,t) \,d\xi}^p \le
 \(\int_t^x |S(x,\xi)|^{p'} \,d\xi \)^{p/p'} \int_t^x |N(\xi,t)|^p \,d\xi \le
 \|N\|^{p}_{X_{1,p}^n} \cdot \int_t^x |S(x,\xi)|^p \,d\xi.
\end{equation}
Here and throughout the rest of the proof $|A| = |A|_{1 \to p}$ for $A \in
\bC^{n \times n}$. Since $N, S \in X_{1,p}^{0,n}$, then, according to
Lemma~\ref{lem:trace.X}, for any $t \in [0,1]$ the traces $N(\cdot,t)$ and
$S(\cdot,t)$ are well defined and belong to $L^p[t,1]$. Therefore, fixing
$t \in [0,1]$, integrating~\eqref{eq:N+S+int.NS} over $x \in [t,1]$, applying
inequality $2^{1-p} |a+b|^p \le |a|^p + |b|^p$ and estimate~\eqref{eq:int.NS}
yields
\begin{align}
\nonumber
 2^{1 - p} S_p(t) & := 2^{1 - p} \int_t^1 |S(x,t)|^p \,dx
 \le \int_t^1 |N(x,t)|^p \,dx + \|N\|^{p}_{X_{1,p}^n}
 \int_t^1 \(\int_t^x |S(x,\xi)|^p \,d\xi\) \,dx \\
\label{eq:Ft<int.Fxi}
 & \le \|N\|_{X_{1,p}^n}^p + \|N\|_{X_{1,p}^n}^p
 \int_t^1 \(\int_{\xi}^1 |S(x,\xi)|^p \,dx\) \,d\xi
 = \|N\|_{X_{1,p}^n}^p \(1 + \int_t^1 S_p(\xi) \,d\xi\), \qquad t \in [0,1].
\end{align}
Applying Gr\"onwall's inequality here implies
\begin{equation} \label{eq:Ft<N}
 S_p(t) \le 2^{p-1} \cdot \|N\|_{X_{1,p}^n}^p \cdot
 \exp\left( (1-t) 2^{p-1} \|N\|^p_{X_{1,p}^n}\right), \qquad t \in [0,1].
\end{equation}
Noting that $\|S\|_{X_{1,p}^n}^p = \sup_{t \in [0,1]} S_p(t)$, we arrive
at~\eqref{eq:|S|.Xqp} with $q=1$.

The case $q = \infty$ is treated similarly but using identity $\cN + \cS +
\cS \cN = 0$ instead of~\eqref{eq:N+S+int.NS}.

\textbf{(ii)} Let again $q=1$. Setting $\wh S(x,t) := S(x,t) - \wt{S}(x,t)$ and
$\wh{N}(x,t) := N(x,t) - \wt{N}(x,t)$ we easily derive from
equation~\eqref{eq:N+S+int.NS} and similar equation related $\wt{N}(x,t)$ and
$\wt{S}(x,t)$, that
\begin{equation}\label{eq:for_differ-ce_kernels_of_inverse_oper-r}
 |\wh{S}(x,t)| \le |\wh{N}(x,t)| +
 \int_t^x |\wh{N}(x,\xi)| \cdot |S(\xi,t)| \,d\xi +
 \int_t^x |\wt{N}(x,\xi)| \cdot |\wh{S}(x,t)| \,d\xi.
\end{equation}
Using inequality~\eqref{eq:int.NS} and performing the same transformations
as in~\eqref{eq:Ft<int.Fxi} we arrive at
\begin{align} \label{eq:whFt<int.whFxi}
 3^{1 - p} \wh{S}_p(t) := 3^{1 - p} \int_t^1 |\wh{S}(x,t)|^p \,dx
 \le \|\wh{N}\|_{X_{1,p}^n}^p + \|\wh{N}\|_{X_{1,p}^n}^p
 \int_t^1 S_p(\xi) \,d\xi +
 \|\wt{N}\|_{X_{1,p}^n}^p \int_t^1 \wh{S}_p(\xi) \,d\xi, \qquad t \in [0,1].
\end{align}
Note that $\|N\|_{X_{1,p}^n} \le r$, hence, in accordance with
estimate~\eqref{eq:Ft<N},
\begin{equation} \label{eq:1+int.F}
 1 + \int_t^1 S_p(\xi) \,d\xi \le 1 + \int_t^1 \mu e^{(1-\xi) \mu} \,d\xi
 = e^{(1-t) \mu} \le e^{\mu} = \exp(2^{p-1} r^p),
 \qquad \text{where} \quad \mu := 2^{p-1} r^p.
\end{equation}
With account of~\eqref{eq:1+int.F} and inequality $\|\wt{N}\|_{X_{1,p}^n}
\le r$, estimate~\eqref{eq:whFt<int.whFxi} turns into
\begin{align} \label{eq:whFt<int.whFxi.simple}
 \wh{S}_p(t) \le 3^{p-1} \exp(2^{p-1} r^p) \cdot \|\wh{N}\|_{X_{1,p}^n}^p +
 3^{p-1} r^p \int_t^1 \wh{S}_p(\xi) \,d\xi, \qquad t \in [0,1].
\end{align}
Now Gr\"onwall's inequality applies and gives
\begin{equation}
 \wh{S}_p(t) \le 3^{p-1} \exp(2^{p-1} r^p) \cdot \|N - \wt{N}\|_{X_{1,p}^n}^p
 \cdot \exp\( (1-t) 3^{p-1} r^p\)
 \le 3^{p-1} \exp(3^p r^p) \cdot \|N - \wt{N}\|_{X_{1,p}^n}^p,
 \qquad t \in [0,1].
\end{equation}
In turn, this inequality yields~\eqref{eq:wh.S<wh.N}.
\end{proof}
\begin{lemma} \label{lem:Banach_algebras}
For each $p\in [1,\infty]$ the spaces $X_{1,p}^n$ and $X_{\infty,p}^n$ form
Banach algebras with respect to the product (``composition of kernels'')
\begin{equation} \label{eq:composition_of_kernels}
 N(x,t) = (N_1*N_2)(x,t) := \int_t^x N_1(x,\xi) N_2(\xi,t)\,d\xi.
\end{equation}
\end{lemma}
\begin{proof}
\textbf{(i)} Let $N_1, N_2 \in X_{1,p}^n$. Assuming that $p \in (1, \infty)$ and
applying H\"older's inequality yields
\begin{align}
\nonumber
 \int_t^1 |N(x,t)|^{p}\,dx & \le \int_t^1 \left((x-t)^{p/p'}\int_t^x
 |N_1(x,\xi)|^{p} |N_2(\xi,t)|^{p} \,d\xi\right)\,dx \\
 & \le \int_t^1 |N_2(\xi,t)|^{p} \(\int_\xi^1|N_1(x,\xi)|^{p}\,dx\) \,d\xi \le
 \|N_1\|^{p}_{X_{1,p}^n} \cdot \|N_2\|^{p}_{X_{1,p}^n}.
\end{align}
It follows that $N_1*N_2 \in X_{1,p}^n$ and the inequality
$\|N\|^{p}_{X_{1,p}^n} \le \|N_1\|^{p}_{X_{1,p}^n} \cdot
\|N_2\|^{p}_{X_{1,p}^n}$ holds. This proves the statement for $p>1$. The cases
$p=1$ and $p=\infty$ are treated simpler without using H\"older's inequality.

\textbf{(ii)} The case of the space $X_{\infty,p}^n$ is treated similarly.
\end{proof}
\begin{lemma} \label{lem:inverse_oper_for_small_N}
Let $N \in X_{\infty,p}^{0,n}$\,, $p \ge 1$, let $N_p(x) :=
\sup_{\xi \in [0,x]} \int_0^\xi |N(\xi,t)|^p\,dt$, and let $S$ be given
by~\eqref{eq:I+cN.inv=I+cS}. Then the following implication
holds
\begin{equation}\label{eq:estimate_for_S(x)}
 N_p(x) < (2x)^{1-p} \quad \Longrightarrow \quad S_p(x)
 \le \frac{2^{p-1} N_p(x)}{1 - (2x)^{p-1} N_p(x)}\,.
\end{equation}
In particular, if $\|N\|_{X_{\infty,p}^n} \le r =: 2^{1/p-1} R$ where
$R < 1$, then $\|S\|_{X_{\infty,p}^n} \le R \cdot (1 - R^p)^{-1/p}$.
\end{lemma}
\begin{proof}
Applying H\"older's inequality and changing order of integration, yields
\begin{equation}\label{eq:Minkovsk_ineq}
 \int_0^x \left|\int_t^x S(x,\xi) N(\xi,t)\,d\xi \right|^p \,dt
 \le x^{p-1} \int_0^x \int_0^\xi |S(x,\xi) N(\xi,t)|^p\,dt \,d\xi
\end{equation}
It easily follows from relation~\eqref{eq:N+S+int.NS} with account
of~\eqref{eq:Minkovsk_ineq} that
\begin{align} \label{eq:estimate_for_kernel_of_inverse_oper_New}
 2^{1-p} \int_0^x |S(x,t)|^p \,dt \le \int_0^x |N(x,t)|^p \,dt +
 x^{p-1} \int_0^x |S(x,\xi)|^p \,d\xi \int_0^\xi |N(\xi,t)|^p \,dt
 \le N_p(x) \(1 + x^{p-1} \int_0^x |S(x,t)|^p\,dt\).
\end{align}
This estimate implies~\eqref{eq:estimate_for_S(x)}.
\end{proof}
\section{Triangular transformation operators} \label{sec:Transform}
Consider $2 \times 2$ Dirac type system~\eqref{eq:systemIntro_1}
\begin{equation} \label{eq:systemIntro}
 L(Q)y := -i B^{-1} y'+Q(x)y=\l y, \qquad y={\rm col}(y_1,y_2), \qquad x\in[0,1],
\end{equation}
with the diagonal matrix $B$ and potential matrix $Q(\cdot)$ given
by~\eqref{eq:BQ}.

The existence of a triangular transformation operator for
system~\eqref{eq:systemIntro} with summable potential matrix $Q \in L^1([0,1];
\bC^{2 \times 2})$ was established in our previous paper~\cite{LunMal16JMAA}
(the case $Q \in L^\infty([0,1]; \bC^{2 \times 2})$ was treated earlier
in~\cite{Mal99}). The purpose of this section is to prove Lipshitz dependence
(in respective norms) of the kernels of transformation operators on $Q \in L^p([0,1];
\bC^{2 \times 2})$.
 We start with the following result from~\cite{LunMal16JMAA}.
\begin{theorem} \label{th:Trans}~\cite[Theorem 2.5]{LunMal16JMAA}
Let $Q = \codiag(Q_{12}, Q_{21})\in \LL{1}$. Assume that $e_{\pm}(\cdot,\l)$ are the solutions to the
system~\eqref{eq:systemIntro} satisfying the initial conditions
$e_\pm(0,\l)=\binom{1}{\pm1}$. Then $e_{\pm}(\cdot,\l)$ admits the following
representation by means of the triangular transformation operator
\begin{equation} \label{eq:e=(I+K)e0}
 e_{\pm}(x,\l) = (I+K^{\pm})e^0_{\pm}(x,\l)
 = e^0_{\pm}(x,\l) + \int^x_0 K^{\pm}(x,t) e^0_{\pm}(t,\l)dt,
\end{equation}
where
\begin{equation} \label{eq:e=e0}
 e^0_{\pm}(x,\l)=\binom{e^{ib_1\l x}}{\pm e^{ib_2\l x}}, \quad\text{and}\quad
 K^{\pm}=\bigl(K^{\pm}_{jk}\bigr)^2_{j,k=1} \in \XXZ{1,1} \cap \XXZ{\infty,1}.
\end{equation}
\end{theorem}
Moreover, it is shown in~\cite{Mal99} that if $Q = \codiag(Q_{12}, Q_{21}) \in
\CC{1}$, then the matrix kernel in triangular
representation~\eqref{eq:e=(I+K)e0} is smooth, $K^{\pm} =
\bigl(K_{jk}^{\pm}\bigr)_{j,k=1}^2 \in C^1(\Omega, \bC^{2 \times 2})$, and it is
the unique solution of the following boundary value problem
\begin{equation} \label{1op_JMAA}
B^{-1}D_x K^{\pm}(x,t) + D_t K^{\pm}(x,t)B^{-1} + iQ(x)K^{\pm}(x,t)=0,
\end{equation}
\begin{equation} \label{2op_JMAA}
K^{\pm}(x,x)B^{-1} - B^{-1}K^{\pm}(x,x) = iQ(x),\quad x\in [0,1],
\end{equation}
\begin{equation} \label{3op_JMAA}
K^{\pm}(x,0)B^{-1} \binom{1}{\pm 1} = 0, \quad x\in [0,1].
\end{equation}
The proof of this result in~\cite{Mal99} is divided in two steps. At first it is proved
that there exists the smooth unique solution
$R(x,t) = \bigl(R_{jk}(x,t)\bigr)^2_{j,k=1} \in C^1(\Omega, \bC^{2 \times 2})$ of
the problem~\eqref{1op_JMAA}-\eqref{2op_JMAA} satisfying instead of~\eqref{3op_JMAA} the
following conditions:
\begin{equation} \label{6op_JMAA}
R_{11}(x,0)= R_{22}(x,0) =0, \quad x\in [0,1].
\end{equation}
 Using this result it is proved in~\cite{LunMal16JMAA} the
following result being a starting point of our investigation here.
\begin{proposition} \label{prop:K.props}~\cite{LunMal16JMAA}
Let $Q \in \LL{1}$ and let $K^{\pm}$ be the kernels of the corresponding transformation
operators from representation~\eqref{eq:e=(I+K)e0}. Then there exist
\begin{equation} \label{eq:P.R.in.X0}
 R=(R_{jk})^2_{j,k=1} \in \XXZ{1,1} \cap \XXZ{\infty,1} \quad\text{and}\quad
 P^{\pm} = \diag(P^{\pm}_1, P^{\pm}_2) \in \LL{1},
\end{equation}
such that
\begin{equation} \label{2.51op}
 K^{\pm}(x,t) = R(x,t) + P^{\pm}(x-t) + \int^x_t R(x,s) P^{\pm}(s-t)ds,
 \qquad 0 \le t \le x \le 1.
\end{equation}
Moreover, $R(\cdot,\cdot)$ is a unique solution of the following system of integral
equations for $0 \le t \le x \le 1$,
\begin{align}
\label{eq:Rkk=int.Qkj.Rjk}
 R_{kk}(x,t) &= - \frac{i}{a_k}\int^x_{x-t}Q_{kj}(\xi)R_{jk}
 \bigl(\xi, \xi -x + t \bigr) \,d\xi, \quad j=2/k, \quad k \in \{1,2\}, \\
\label{eq:Rjk=Qjk-int.Qjk.Rkk}
 R_{jk}(x,t) &= \frac{i}{a_k - a_j} Q_{jk}(\a_k x + \a_j t)
 - \frac{i}{a_j} \int^x_{\a_k x + \a_j t}
 Q_{jk}(\xi)R_{kk}\bigl(\xi,\g_k(\xi-x)+t\bigr)d\xi,
 \quad j=2/k, \quad k \in \{1,2\}.
\end{align}
\end{proposition}
Here and in what follows the following notations are used:
\begin{equation} \label{eq:def.ak.alphak}
 a_k := b_k^{-1}, \quad \g_{k} := \frac{a_k}{a_j} = \frac{b_j}{b_k},
 \quad \a_k := \frac{a_k}{a_k-a_j} = \frac{b_j}{b_j-b_k},
 \qquad j = 2/k, \quad k \in \{1,2\}.
\end{equation}
Note that $\a_1 + \a_2 = 1$. Since $b_1 < 0 < b_2$, it is clear that
$\a_k \in (0, 1)$ and $\g_k < 0$, $k \in \{1, 2\}$.

Note also that for smooth $Q (\in \CC{1})$
system~\eqref{eq:Rkk=int.Qkj.Rjk}-\eqref{eq:Rjk=Qjk-int.Qjk.Rkk} is equivalent
to the system~\eqref{1op_JMAA}--\eqref{2op_JMAA},~\eqref{6op_JMAA}.

Alongside equation~\eqref{eq:systemIntro} we consider similar Dirac type
equation with a potential matrix $\wt{Q} \in \LL{1}$ and denote by
$\wt{e}_{\pm}(\cdot,\l)$ the solutions of the system~\eqref{eq:systemIntro}
satisfying the initial conditions $\wt{e}_\pm(0, \l) = \binom{1}{\pm 1}$. In
accordance with Theorem~\ref{th:Trans} it admits the representation
\begin{equation} \label{eq:wt_e=(I+K)e0}
 \wt{e}_{\pm}(x,\l) = (I+ \wt{K}^{\pm})e^0_{\pm}(x,\l)
 = e^0_{\pm}(x,\l) + \int^x_0 \wt{K}^{\pm}(x,t) e^0_{\pm}(t,\l)dt,
\end{equation}
The main result of this section shows the Lipshitz of the mapping $Q\to K^{\pm} :=
K^{\pm}_Q$ on the balls in $L^p\bigl([0,1]; \bC^{2 \times 2}\bigr)$ and reads as
follows.
\begin{theorem} \label{th:K-wtK<Q-wtQ}
Let $Q= \codiag(Q_{12}, Q_{21}), \ \wt{Q} = \codiag(\wt{Q}_{12}, \wt{Q}_{21}) \in \bU_{p,
r}^{2 \times 2}$ for some $p \in [1, \infty)$ and $r > 0$. Let also $K^{\pm}_Q$
and $K^{\pm}_{\wt{Q}}$
be the kernels of the corresponding transformation
operators from representations~\eqref{eq:e=(I+K)e0} and~\eqref{eq:wt_e=(I+K)e0}. Then
\begin{equation} \label{eq:K.wtK.in.X0}
 K^{\pm} := K^{\pm}_{Q},\ \wt{K}^{\pm} := K^{\pm}_{\wt{Q}} \in \XXZ{1,p} \cap \XXZ{\infty,p},
\end{equation}
and there exists a constant $C = C(B, p, r)$ that does not depend on $Q$ and
$\wt{Q}$ such that the following estimate holds
\begin{equation} \label{eq:K-wtK<Q-wtQ}
 \|K^{\pm} - \wt{K}^{\pm}\|_{\XX{\infty,p}} + \|K^{\pm} - \wt{K}^{\pm}\|_{\XX{1,p}}
 \le C \cdot \|Q - \wt{Q}\|_p.
\end{equation}
\end{theorem}

Let $\wt{R}(\cdot, \cdot)$ be the kernel satisfying the system of integral
equations~\eqref{eq:Rkk=int.Qkj.Rjk}--\eqref{eq:Rjk=Qjk-int.Qjk.Rkk} with
$\wt{Q}(\cdot)$ in place of $Q(\cdot)$. Alongside definition~\eqref{2.51op} the
kernels $\wt{K}^{\pm}(\cdot, \cdot)$ and $\wt{R}(\cdot, \cdot)$ are related by the
similar equation
\begin{equation} \label{eq:connec-n_of_wtK_and_wtR}
 \wt{K}^{\pm}(x,t) = \wt{R}(x,t) + \wt{P}^{\pm}(x-t)
 + \int^x_t \wt{R}(x,s) \wt{P}^{\pm}(s-t)ds
\end{equation}
with a certain diagonal matrix function $\wt{P}^{\pm} = \diag \bigl(\wt{P}^{\pm}_1,
\wt{P}^{\pm}_2\bigr) \in \LL{1}$. For brevity we will systematically use the following
notations below,
\begin{equation} \label{eq:wh.Q.K.R.P.def}
 \wh{Q} := Q - \wt{Q}, \quad \wh{K}^{\pm} := K^{\pm} - \wt{K}^{\pm}, \quad
 \wh{R} := R - \wt{R}, \quad \wh{P}^{\pm} := P^{\pm} - \wt{P}^{\pm}.
\end{equation}
In the following two auxiliary results we show that the (non-linear) mappings
\begin{align}
 \cR_{Q,\infty}: \LL{p} \to \XX{\infty,p},
 & \qquad \cR_{Q,\infty}: Q \to R = R_Q, \\
 \cR_{Q,1}: \LL{p} \to \XX{1,p}, & \qquad \cR_{Q,1}: Q \to R = R_Q,
\end{align}
are uniformly Lipshitz on each ball of the radius $r$ in $\LL{p}$.
\begin{lemma} \label{lem:wh.estim.Xinf}
Let $Q, \wt{Q} \in \bU_{p, r}^{2 \times 2}$ for some $p \in [1, \infty)$ and
$r > 0$. Let also $R = R_Q$ be the solution of
problem~\eqref{eq:Rkk=int.Qkj.Rjk}--\eqref{eq:Rjk=Qjk-int.Qjk.Rkk} and let
$\wt{R} = R_{\wt{Q}}$ be the solution of
problem~\eqref{eq:Rkk=int.Qkj.Rjk}--\eqref{eq:Rjk=Qjk-int.Qjk.Rkk} with $\wt{Q}$
in place of $Q$. Then $R, \wt{R} \in \XXZ{\infty,p}$ and there exists a constant
$C^{(\infty)} = C^{(\infty)}(B, p, r)$ not dependent on $Q$, $\wt{Q}$, and such
that the following uniform estimate holds
\begin{equation} \label{2.12op}
 \|R - \wt{R}\|_{X_{\infty,p}(\Omega; \bC^{2 \times 2})}
 = \|R_Q - R_{\wt{Q}}\|_{X_{\infty,p}(\Omega; \bC^{2 \times 2})}
 \le C^{(\infty)} \|Q - \wt{Q}\|_p.
\end{equation}
In particular, kernels $R_Q$ are uniformly bounded on each $L^p$-ball, i.e. \ \
$\|R_Q\|_{X_{\infty,p}(\Omega; \bC^{2 \times 2})} \le C^{(\infty)}\|Q\|_{p}$
\ \ \ for \ \ $Q \in \bU_{p, r}^{2 \times 2}$.
\end{lemma}
\begin{proof}
We put
\begin{equation} \label{14op}
 \wh{R}_{jk}= R_{jk} - \wt{R}_{jk}, \quad \wh{Q}_{jk} = Q_{jk} - \wt{Q}_{jk},
 \qquad j,k \in \{1,2\}.
\end{equation}
and
\begin{equation} \label{15op}
 \wh{\fI}_{jk}(p,x) := \int_0^x |\wh{R}_{jk}(x,t)|^p \,dt,
 \qquad \fI_{jk}(p,x) := \int_0^x |R_{jk}(x,t)|^p \,dt, \qquad j,k\in\{1,2\}.
\end{equation}
It is easily seen that for any $f \in L^p[0,1] \subset L^1[0,1]$ and $0 \le t_1
\le t_2 \le 1$, one has due to H\"older's inequality
\begin{equation} \label{eq:holder.ineq}
 \abs{\int_{t_1}^{t_2} f(\xi) \,d\xi}^p
 \le (t_2-t_1)^{p/p'} \int_{t_1}^{t_2} \abs{f(\xi)}^p \,d\xi
 \le \int_{t_1}^{t_2} \abs{f(\xi)}^p \,d\xi, \qquad 1/p' + 1/p = 1.
\end{equation}

Writing down the difference $R_{11} - \wt{R}_{11}$ as a sum of two terms
containing $\wh{Q}_{12}$ and $\wh{R}_{21}$ as the factors, then evaluating each
summand with the help of inequality~\eqref{eq:holder.ineq}, and making use the
change of variables $\xi = u$, \ $\xi - x + t = v$, we obtain
\begin{align} \label{16op}
\wh{\fI}_{11}(p,x) &:= \int^x_0|\wh{R}_{11}(x,t)|^p\,dt \nonumber \\
&\le 2^{p-1}|b_1|^p \(\int^x_0\,dt
\int^x_{x-t}|\wh{Q}_{12}(\xi)
 R_{21}(\xi,\xi-x+t)|^p\, d\xi
+ \ \int^x_0\,dt \int^x_{x-t}|\wt{Q}_{12}(\xi) \wh
R_{21}(\xi,\xi-x+t)|^p\, d\xi \) \nonumber \\
&\le 2^{p-1} |b_1|^p \(\int^x_0 |\wh{Q}_{12}(u)|^p\,du \int^u_0 |R_{21}(u, v)|^p\, dv
+ \int^x_0 |\wt{Q}_{12}(u)|^p\,du \int^u_0 | \wh{R}_{21}(u,v)|^p\, dv\) \nonumber \\
&= 2^{p-1}|b_1|^p \int^x_0 |\wh{Q}_{12}(u)|^p \fI_{21}(p,u)\,du + 2^{p-1} |b_1|^p \int^x_0 |\wt
Q_{12}(u)|^p \wh{\fI}_{21}(p,u)\,du.
\end{align}
Similarly, it follows from~\eqref{eq:Rjk=Qjk-int.Qjk.Rkk} and~\eqref{15op}
\begin{align}
\nonumber
 \wh{\fI}_{21}(p,x) = \int^x_0|\wh{R}_{21}(x,t)|^p\,dt
 & \le \frac{3^{p-1}}{|a_1-a_2|^p}
 \int^x_0 \abs{\wh{Q}_{21}(\a_1 x + \a_2 t)}^p \,dt \\
\nonumber
 & + \frac{3^{p-1}}{|a_2|^p} \int^x_0 dt \int^x_{\a_1 x + \a_2 t}
 |\wh{Q}_{21}(\xi)R_{11} (\xi,\g_{1}(\xi-x)+t\bigr)|^p\, d\xi \\
 & + \frac{3^{p-1}}{|a_2|^p} \int^x_0 dt \int^x_{\a_1 x + \a_2 t}
 |\wt{Q}_{21}(\xi) \wh{R}_{11}(\xi,\g_{1}(\xi-x) + t)|^p \,d\xi.
\end{align}
Making use the change of variables $\xi=u,$\ $\frac{a_1}{a_2}(\xi-x)+t=v$, in
the two last integrals, noting that $\frac{a_1 x}{a_1-a_2}>0$ and
$\frac{a_1}{a_2}(u-x)>0$, and setting
\begin{equation} \label{eq:constant_C_1(B,p)}
C_1(B,p):= \max \left\{\frac{3^{p-1}}{|a_2|\cdot|a_1 - a_2|^{p-1}},\
\frac{3^{p-1}}{|a_2|^p} \right\}
\end{equation}
we obtain
\begin{align} \label{17op}
\wh{\fI}_{21}(p,x) & \le \frac{3^{p-1}}{|a_2|\cdot|a_1-a_2|^{p-1}}
\int^x_{\frac{a_1 x}{a_1-a_2}}|\wh{Q}_{21}(u)|^p\,du +
\frac{3^{p-1}}{|a_2|^p}\int^x_{\frac{a_1 x}{a_1-a_2}}|\wh
Q_{21}(u)|^p \,du\int^u_{\frac{a_1}{a_2}(u-x)}|R_{11}(u,v)|^p\, dv \nonumber \\
& + \frac{3^{p-1}}{|a_2|^p} \int^x_{\frac{a_1 x}{a_1-a_2}}|\wt{Q}_{21}(u)|^p\,
du\int^u_{\frac{a_1}{a_2}(u-x)}|\wh{R}_{11}(u,v)|^p \,dv
\le C_1(B,p) \int^x_0 |\wh{Q}_{21}(u)|^p \,du \nonumber \\
& + C_1(B,p)\left\{ \int^x_0 |\wh{Q}_{21}(u)|^p \,du\int^u_0 |R_{11}(u,v)|^p\, dv +
 \int^x_0 |\wt{Q}_{21}(u)|^p \,du \int^u_0 |\wh{R}_{11}(u,v)|^p\, dv \right\} \nonumber \\
&= C_1(B,p) \left\{\int^x_0|\wh{Q}_{21}(t)|^p\, dt + \int^x_0 \fI_{11}(p,
u)|\wh{Q}_{21}(u)|^p \,du + \int^x_0 \wh{\fI}_{11}(p,u)|\wt{Q}_{21}(u)|^p \,du
\right\}.
\end{align}

Now we are ready to evaluate the functions $R_{11}(\cdot,\cdot)$ and
$R_{21}(\cdot,\cdot)$ in $X_{\infty,p}(\Omega)$-norm. We show that these norms
depend only on radius $r$ of the ball in $L^p([0,1];\bC^{2 \times 2})$ contained
$Q$ and do not depend on $Q$ itself. This fact plays a crucial role in what
follows.

If $\wt{Q} = 0$ the estimates~\eqref{17op} and~\eqref{16op} are simplified to
\begin{equation} \label{eq:est-te_fI_21_via_fI_11}
{\fI}_{21}(p,x) = \int^x_0|R_{21}(x,t)|^p\,dt \le C_1(B,p)
\left\{\int^x_0|Q_{21}(t)|^p\, dt + \int^x_0 \fI_{11}(p, u)|Q_{21}(u)|^p \,du \right\}
\end{equation}
and
\begin{equation} \label{eq:est-te_fI_11_via_fI_21}
{\fI}_{11}(p,x) = \int^x_0|R_{11}(x,t)|^p\,dt \le |b_1|^p \int^x_0 |Q_{12}(u)|^p
\fI_{21}(p,u)\,du,
\end{equation}
respectively. Inserting the second estimate into~\eqref{eq:est-te_fI_21_via_fI_11} one
gets
\begin{align} \label{eq:est-te_fI_21_via_integ_of_fI_21}
{\fI}_{21}(p,x) & \le C_1(B,p) \(\|Q_{21}\|^p_{L^p[0,1]} + |b_1|^p \int^x_0
|Q_{21}(u)|^p\,du \int^u_0 |Q_{12}(s)|^p \fI_{21}(p,s)\,ds \) \nonumber \\
 & \le C_1(B,p) \(\|Q_{21}\|^p_{L^p[0,1]} + |b_1|^p \int^1_0
|Q_{21}(u)|^p\,du \int^x_0 |Q_{12}(s)|^p \fI_{21}(p,s)\,ds \)
\end{align}
Now the Gr\"onwall's inequality applies and leads to the first desired estimate
\begin{align} \label{eq:evaluation_of_fI_21}
{\fI}_{21}(p,x) \le C_1(B,p)\|Q_{21}\|^p_{L^p[0,1]} \cdot
\exp\(|b_1|^p\|Q_{21}\|^p_{L^p[0,1]}\int^x_0 |Q_{12}(s)|^p\,ds\)
\end{align}
In turn, inserting this estimate in~\eqref{eq:est-te_fI_11_via_fI_21} yields the
second desired estimate
\begin{align}
{\fI}_{11}(p,x) \le |b_1|^p C_1(B,p)\|Q_{21}\|^p_{L^p[0,1]} \cdot \int^x_0 &
|Q_{12}(u)|^p \cdot \exp\(|b_1|^p\|Q_{21}\|^p_{L^p[0,1]}\int^u_0
 |Q_{12}(s)|^p\,ds\)\,du \nonumber \\
& = C_1(B,p) \(\exp\(|b_1|^p\|Q_{21}\|^p_{L^p[0,1]}\int^x_0 |Q_{12}(s)|^p\,ds
\) -1 \)
\end{align}
For potential matrices $Q$ belonging the ball of $\LLV{p}$ of the radius $r$
these estimates imply
\begin{align} \label{eq:main_unif_est-s_for_R_11_and_R_21}
\|R_{11}(\cdot,\cdot)\|^p_{X_{\infty, p}(\Omega)}
&\le C_1(B,p)\(\exp\(|b_1|^p \cdot r^{2p}\)-1\) \le |b_1|^p r^{2p}
C_1(B,p)\exp\(|b_1|^p \cdot r^{2p}\), \nonumber \\
& \|R_{21}(\cdot,\cdot)\|^p_{X_{\infty, p}(\Omega)} \le C_1(B,p)r^p \exp\(|b_1|^p
\cdot r^{2p}\).
\end{align}
We set
\begin{equation} \label{18op}
C_1(B, p, r) := C_1(B,p) r^p \cdot \exp\(|b_1|^p \cdot r^{2p}\)\cdot \max\{1,
|b_1|^pr^{p}\}.
\end{equation}
Next using estimates~\eqref{eq:main_unif_est-s_for_R_11_and_R_21} we return to
evaluation of the functions $\wh{\fI}_{11}(p, \cdot)$ and $\wh{\fI}_{21}(p,
\cdot)$. Combining estimate~\eqref{16op} with estimate
\eqref{eq:main_unif_est-s_for_R_11_and_R_21} (for ${\fI}_{21}(p, \cdot)$) and using
definition~\eqref{18op} yields
\begin{equation} \label{19op}
\wh{\fI}_{11}(p, x) \le 2^{p-1}|b_1|^p C_1(p,r) \|\wh{Q}_{12}\|_{L^p}^p +
2^{p-1}|b_1|^p \int^x_0\wh{\fI}_{21}(p,t)|\wt{Q}_{12}(t)|^p\,dt.
\end{equation}
Inserting this inequality into~\eqref{17op} and taking into account estimate
\eqref{eq:main_unif_est-s_for_R_11_and_R_21} (for ${\fI}_{11}(p, \cdot)$), using
definition~\eqref{18op}, and noting that $\|\wt{Q}_{21}\|_{L^p} \le r$ we derive
\begin{align} \label{eq:estimate_for_fI_21}
\wh{\fI}_{21}(p,x) &\le C_1(B,p)\(1 + C_1(p,r)\) \int^x_0|\wh
Q_{21}(u)|^p\,du +
 C_1(B,p) \int^x_0\wh{\fI}_{11}(p,u)|\wt{Q}_{21}(u)|^p\,du \nonumber \\
&\le C_1(B,p)\(1 + C_1(p,r)\) \|\wh{Q}_{21}\|_{L^p}^p + 2^{p-1}|b_1|^p
C_1(p,r)\|\wh
Q_{12}\|^p_{L^p}\cdot\|\wt{Q}_{21}\|^p_{L^p} \nonumber \\
& + 2^{p-1}|b_1|^p C_1(B,p) \int^x_0|\wt{Q}_{21}(u)|^p\,du \int^u_0\wh{\fI}_{21}(p,t)|\wt{Q}_{12}(t)|^p\,dt \nonumber \\
&\le C_2(B,p,r) \(\|\wh{Q}_{21}\|^p_{L^p} + \|\wh{Q}_{12}\|^p_{L^p}\) +
2^{p-1}|b_1|^p C_1(B,p) \int^x_0\wh{\fI}_{21}(p,t)|\wt
Q_{12}(t)|^p\,dt \int^x_t|\wt{Q}_{21}(u)|^p\,du \nonumber \\
&\le C_2(B,p,r) \(\|\wh{Q}_{21}\|^p_{L^p} + \|\wh{Q}_{12}\|^p_{L^p}\) +
2^{p-1}|b_1|^p C_1(B,p) \cdot \|\wt{Q}_{21}\|^p_{L^p} \int^x_0\wh
{\fI}_{21}(p,t)|\wt{Q}_{12}(t)|^pdt,
\end{align}
where
$$
C_2(B,p,r) = \max\left\{ C_1(B,p)(1 + C_1(p,r)),\ 2^{p-1}|b_1|^p C_1(p,r) r
\right\}.
$$
Applying Gr\"onwall's inequality to~\eqref{eq:estimate_for_fI_21} and noting
that $\|\wt{Q}_{21}\|_{L^p}, \|\wt{Q}_{12}\|_{L^p} \le r$, one easily deduces
\begin{align}
\nonumber
 \wh{\fI}_{21}(p,x) \le C_2(B,p,r) \(\|\wh{Q}_{12}\|^p_{L^p}
 + \|\wh{Q}_{21}\|^p_{L^p}\) \exp\(2^{p-1}|b_1|^p C_1(B,p)\cdot
 \|\wt{Q}_{21}\|^p_{L^p}\int^x_0 |\wt{Q}_{12}(t)|^pdt\) \\
\label{2.29op}
 \le C_2(B,p,r) \(\|\wh{Q}_{12}\|^p_{L^p}
 + \|\wh{Q}_{21}\|^p_{L^p}\) \exp\(2^{p-1}|b_1|^p C_1(B,p)\cdot r^{2p}\).
\end{align}
Implementing this inequality into~\eqref{19op} we arrive at the estimate
\begin{align} \label{2.30op}
\wh{\fI}_{11}(p,x) \le \(\|\wh{Q}_{12}\|^p_{L^p} + \|\wh
Q_{21}\|^p_{L^p}\) \(C_1(p,r) + C_2(B,p,r)\|\wt
Q_{12}\|^p_{L^p}\exp\(2^{p-1}|b_1|^p C_1(B,p)\cdot r^{2p}\)\) \nonumber \\
 \le 2^{p-1}|b_1|^p \|\wh{Q}\|^p_{\LL{p}} \(C_1(p,r) + C_2(B,p,r)r^p\cdot\exp\(2^{p-1}|b_1|^p C_1(B,p)\cdot r^{2p}\)\).
\end{align}
Similar reasoning leads to similar estimates for
$\wh{\fI}_{12}$ and $\wh{\fI}_{22}$. Combining
these estimates with~\eqref{2.29op} and~\eqref{2.30op} we arrive
at~\eqref{2.12op}.
\end{proof}
\begin{lemma} \label{lem:R.diff.X1}
Let $Q, \wt{Q} \in \bU_{p, r}^{2 \times 2}$ for some $p \in [1, \infty)$ and
$r > 0$. Assume also that $R = R_Q$ and $\wt{R} = R_{\wt{Q}}$ are the (unique)
solutions of the
problem~\eqref{eq:Rkk=int.Qkj.Rjk}--\eqref{eq:Rjk=Qjk-int.Qjk.Rkk} with $Q$
and $\wt{Q}$ in place of $Q$, respectively. Then $R, \wt{R} \in \XXZ{1,p}$ and
there exists a constant $C^{(1)} = C^{(1)}(B, p, r)$ not dependent on $Q$ and
$\wt{Q}$ and such that the following uniform estimate holds
\begin{equation} \label{eq:wh.R.X1p<wh.Q}
 \|R - \wt{R}\|_{X_{1,p}(\Omega; \bC^{2 \times 2})} = \|R_Q - R_{\wt{Q}}\|_{X_{1,p}(\Omega; \bC^{2 \times 2})}
 \le C^{(1)} \|Q - \wt{Q}\|_p.
\end{equation}
In particular, one has the uniform estimate: $\|R_Q\|_{X_{1,p}(\Omega; \bC^{2 \times
2})} \le C^{(1)}\|Q\|_{p}$\ \ for \ \ $Q \in \bU_{p, r}^{2 \times 2}$.
\end{lemma}
\begin{proof}
\textbf{(i)} As in the proof of Lemma~\ref{lem:wh.estim.Xinf} we put
(see~\eqref{14op}) $\wh{R}_{jk} = R_{jk} - \wt{R}_{jk}$ and $\wh{Q}_{jk} =
Q_{jk} - \wt{Q}_{jk}$,\ $j,k\in\{1,2\}$.

First we prove estimate~\eqref{eq:wh.R.X1p<wh.Q} for the case
$j \ne k$. Since $\a_1, \a_2 \in (0,1)$ and $\a_1 + \a_2 = 1$, it
follows that
\begin{equation} \label{eq:estimate_of_xi_jk}
 0 \le t \le \a_k + \a_j t \le 1.
\end{equation}
Making a change of variable $u = \a_k x + \a_j t$ and taking into
account~\eqref{eq:estimate_of_xi_jk} and~\eqref{eq:holder.ineq}, one has
\begin{align} \label{eq:estimate_Q_{jk,n}}
 \abs{\frac{i}{a_k - a_j} \int_t^1 \wh{Q}_{jk}(\a_k x+ \a_j t)\,dx}^p
 = \abs{\frac{1}{a_k} \int^{\a_k + \a_j t}_t \wh{Q}_{jk}(u)\,du}^p
 \le C_b^p \int^{\a_k + \a_j t}_t \abs{\wh{Q}_{jk}(u)}^p\,du,
\end{align}
where $C_b := \max\{|b_1|, |b_2|\} = 1 / \min\{|a_1|, |a_2|\}$.
Further, note that alongside $R_{jk}$ the kernel $\wt{R}_{jk}$ satisfies
equation~\eqref{eq:Rjk=Qjk-int.Qjk.Rkk} with $\wt{Q}_{jk}$ in place of $Q_{jk}$,
$j,k \in \{1,2\}.$ Taking difference of these equations, then integrating the
result with respect to $x\in [t,1],$ making use the change of variables $u =
\xi,$ $v = (\xi-x)\g_{lk} + t = (\xi-x)\frac{a_k}{a_j} + t$, and using
inequalities~\eqref{eq:estimate_Q_{jk,n}},~\eqref{eq:holder.ineq} we obtain
\begin{align}
\nonumber
 3^{1-p} C_b^{-p} \int^1_t|\wh{R}_{jk}(x,t)|^p\,dx
 & = 3^{1-p} C_b^{-p} \int_t^1 |R_{jk}(x,t) - \wt{R}_{jk}(x,t)|^p\,dx \\
\nonumber
 & \le \int^{\a_k + \a_j t}_t |\wh{Q}_{jk}(u)|^p\,du
 + \int^1_t\,dx \int^x_{\a_k x + \a_j t} \abs{\wt{Q}_{jk}(\xi)
 \wh{R}_{kk}(\xi,\g_k(\xi-x)+ t)}^p\,d\xi \nonumber \\
\nonumber
 & \qquad \qquad \qquad \qquad \quad \ \,
 + \int^1_t dx\int^x_{\a_k x + \a_j t}
 |\wh{Q}_{jk}(\xi)R_{kk}(\xi,\g_k(\xi-x)+t)|^p\,d\xi \\
\nonumber
 & = \int^{\a_k + \a_j t}_t |\wh{Q}_{jk}(u)|^p\,du
 + \int_t^{\a_k + \a_j t}\,dv
 \int^{(v-t)\frac{a_j}{a_k} + 1}_v|\wt{Q}_{jk}(u) \wh{R}_{kk}(u,v)|^p\,du \\
\nonumber
 & \qquad \qquad \qquad \qquad \quad \ \,
 + \int_t^{\a_k + \a_j t}
dv\int^{(v-t)\frac{a_j}{a_k}+1}_v|\wh
Q_{jk}(u) R_{kk}(u,v)|^p\,du \nonumber \\
\nonumber
 & \le \int^1_t |\wh{Q}_{jk}(u)|^p\,du
 + \int_t^1\,dv\int^1_v|\wt{Q}_{jk}(u) \wh{R}_{kk}(u,v)|^p\,du \\
\nonumber
 & \qquad \qquad \qquad \quad \ \ \,
 + \int_t^1\,dv \int^1_v|\wh{Q}_{jk}(u) R_{kk}(u,v)|^p\,du \\
\nonumber
 & = \int^1_t |\wh{Q}_{jk}(u)|^p\,du
 + \int_t^1\,|\wt{Q}_{jk}(u)|^p\,du\int^u_t|\wh{R}_{kk}(u,v)|^p\,dv \\
\nonumber
 & \qquad \qquad \qquad \quad \ \ \,
 + \int_t^1 |\wh{Q}_{jk}(u)|^p\,du \int^u_t |R_{kk}(u,v)|^p\,dv \\
\label{2.39op_New}
 & \le \|\wh{Q}_{jk}\|_p^p
 + \|\wt{Q}_{jk}\|_p^p \cdot \|\wh{R}_{kk}\|^p_{X_{\infty, p}}
 + \|\wh{Q}_{jk}\|_p^p \cdot \|R_{kk}\|^p_{X_{\infty, p}}.
\end{align}
Here we use simple inequalities $\a_k + \a_j t\le 1$ and ${(v-t)\frac{a_j}{a_k} + 1}\le
1$. The latter holds because $t\le v$ and ${a_j}{a_k}<0$. It follows from
\eqref{2.39op_New} with account of definitions~\eqref{eq:X1p.norm.def} and
\eqref{eq:Xinfp.norm.def} that
\begin{equation} \label{2.40op}
 \|\wh{R}_{jk}\|_{X_{1,p}} \le 3 C_b \(\|\wh{Q}_{jk}\|_p +
 \|\wt{Q}_{jk}\|_{p} \cdot \|\wh{R}_{kk}\|_{X_{\infty,p}(\Omega)} +
 \|\wh{Q}_{jk}\|_{p} \cdot \|R_{kk}\|_{X_{\infty,p}(\Omega)}\), \quad j \ne k.
\end{equation}
Further, in accordance with Lemma~\ref{lem:wh.estim.Xinf} the following
estimates hold
$$
\|\wh{R}_{kk}\|_{X_{\infty,p}(\Omega)} \le C^{(\infty)} \|\wh{Q}\|_p \quad \text{and}\quad \|R_{kk}\|_{X_{\infty,p}(\Omega)} \le
C^{(\infty)} \|Q\|_p \le C^{(\infty)} r,
$$
with the constant $C^{(\infty)} = C^{(\infty)}(r, B, p)$ not dependent on $Q$ and $\wt{Q}$ running through the ball $\bU_{p,r}^{2 \times 2}$. Inserting these
estimates into~\eqref{2.40op} yields
\begin{equation} \label{2.40op_second}
 \|\wh{R}_{jk}\|_{X_{1,p}} \le 3 C_b \(\|\wh{Q}_{jk}\|_p
 \bigl(1 + \|R_{kk}\|_{X_{\infty,p}}\bigr) + r \|\wh{R}_{kk}\|_{X_{\infty,p}} \)
 \le 3 C_b \(1 + 2r C^{(\infty)}\)\|\wh{Q}\|_p,
 \quad j \ne k.
\end{equation}

(ii) Going over to the case $j = k$ we start with
equation~\eqref{eq:Rkk=int.Qkj.Rjk} and similar equation for $\wt{R}_{kk}$ which
holds with $\wt{Q}_{jk}$ in place of $Q_{jk}$, $j,k \in \{1,2\}$. Taking
difference of~\eqref{eq:Rkk=int.Qkj.Rjk} and this equation, then integrating
the difference with respect to $x\in [t,1],$ and then making use the change of
variables $\xi=u, \xi-x+t=v$, applying inequality~\eqref{eq:holder.ineq} one
obtains
\begin{align}
2^{1-p} |a_j|^p\int^1_t|\wh{R}_{jj}(x,t)|^p\,dx & = 2^{1-p}|a_j|^p\int^1_t|R_{jj}(x,t)- \wt{R}_{jj}(x,t)|^p\,dx \nonumber \\
&\le \int^1_t dx\int^x_{x-t} \abs{\wt{Q}_{jk}(\xi)\wh
R_{kj}(\xi,\xi-x+t)}^p\, d\xi +
\int^1_t dx\int^x_{x-t} \abs{\wh{Q}_{jk}(\xi)R_{kj}(\xi,\xi-x+t)}^p\,d\xi \nonumber \\
&= \int^t_0 dv\int^{v-t+1}_v \abs{\wt{Q}_{jk}(u) \wh{R}_{kj}(u,v)}^p\,du +
\int^t_0 dv\int^{v-t+1}_v \abs{\wh{Q}_{jk}(u) R_{kj}(u,v)}^p\,du \nonumber \\
&\le \int^1_0 dv\int^{1}_v \abs{\wt{Q}_{jk}(u) \wh{R}_{kj}(u,v)}^p\,du +
\int^1_0 dv\int^{1}_v \abs{\wh{Q}_{jk}(u) R_{kj}(u,v)}^p\,du \nonumber \\
&= \int^1_0 \abs{\wt{Q}_{jk}(u)}^p\,du\int^u_0 \abs{\wh
R_{kj}(u,v)}^p\,dv + \int^1_0 \abs{\wh{Q}_{jk}(u)}^p\,du\int^u_0
\abs{R_{kj}(u,v)}^p\,dv.
\end{align}
It follows with account of definitions~\eqref{eq:X1p.norm.def} and
\eqref{eq:Xinfp.norm.def} that
\begin{equation} \label{estim_for_widehat_R_jj_1}
 \|R_{jj} - \wt{R}_{jj}\|_{X_{1,p}(\Omega)} \le 2|b_j|
 \(\|\wt{Q}_{jk}\|_p \cdot \|\wh{R}_{kj}\|_{X_{\infty,p}(\Omega)} +
 \|\wh{Q}_{jk}\|_p \cdot \|R_{kj}\|_{X_{\infty,p}(\Omega)}\),
 \quad j \in \{1,2\}.
\end{equation}
Further, in accordance with Lemma~\ref{lem:wh.estim.Xinf} the following
estimates hold
$$
\|\wh{R}_{kj}\|_{X_{\infty,p}(\Omega)} \le C^{(\infty)} \|\wh{Q}\|_p \quad \text{and}\quad \|R_{kj}\|_{X_{\infty,p}} \le C^{(\infty)}
\|Q\|_p \le C^{(\infty)} r,
$$
with the constant $C^{(\infty)} = C^{(\infty)}(r, b_1, b_2, p)$ not dependent on $Q$ and
$\wt{Q}$ from the ball in $L^p$ of the radius $r$. Inserting these estimates into
\eqref{2.40op} yields
\begin{equation} \label{estim_for_widehat_R_jj_2}
 \|\wh{R}_{jj}\|_{X_{1,p}(\Omega)} \le 2|b_j| \(r \|\wh{R}_{kj}\|_{X_{\infty,p}}
 + \|\wh{Q}\|_p \|R_{kj}\|_{X_{\infty,p}}\) \le 4 C_b r C^{(\infty)} \cdot
 \|\wh{Q}\|_p, \quad j \in \{1,2\}.
\end{equation}
Combining~\eqref{2.40op_second} with~\eqref{estim_for_widehat_R_jj_2}
one arrives at the desired estimate~\eqref{eq:wh.R.X1p<wh.Q}.
\end{proof}
Combining Lemma~\ref{lem:wh.estim.Xinf} with Lemma~\ref{lem:R.diff.X1} we arrive
at a part of the following result.
\begin{proposition} \label{prop2.5}
Let $Q= \codiag(Q_{12}, Q_{21}), \wt{Q} = \codiag(\wt{Q}_{12}, \wt{Q}_{21}) \in \bU_{p, r}^{2 \times 2}$
for some $p \ge 1$ and $r > 0$. Then:

(i) The system of integral equations~\eqref{eq:Rkk=int.Qkj.Rjk}--\eqref{eq:Rjk=Qjk-int.Qjk.Rkk} has the unique
solution $R = R_Q = (R_{jk})_{j,k=1}^2$ belonging to $X_{1,p}^0\(\Omega;\bC^{2\times
2}\) \cap X_{\infty,p}^0\(\Omega;\bC^{2 \times 2}\)$. Moreover, this
solution is unique in the space $X_{\infty,p}^0\(\Omega;\bC^{2 \times 2}\)$.

(ii) Let also $\wt{R} = R_{\wt{Q}}$ be the (unique) solution of the
problem~\eqref{eq:Rkk=int.Qkj.Rjk}--\eqref{eq:Rjk=Qjk-int.Qjk.Rkk} with $\wt{Q}$ in
place of $Q$. Then $\wt{R} \in \XXZ{1,p}\cap \XXZ{\infty,p}$ and the following uniform
estimate holds
\begin{equation} \label{eq:wh.R.X1infp<wh.Q}
 \|R - \wt{R}\|_{\XX{1,p}} + \|R - \wt{R}\|_{\XX{\infty,p}}
 \le C^{(1,\infty)} \|Q - \wt{Q}\|_p,
\end{equation}
where the constant $C^{(1,\infty)} := C^{(1)}(B, p, r) + C^{(\infty)}(B, p, r)$ does
not depend on $Q, \wt{Q} \in \bU_{p,r}^{2 \times 2}$.

(iii) The following uniform estimate holds (with $b:= \max\{|b_1|, |b_2|\}$)
 $$
\|R_Q(\cdot,\cdot)\|^p_{X_{\infty,p}(\Omega; \bC^{2 \times 2})} +
\|R_Q(\cdot,\cdot)\|^p_{X_{1,p}(\Omega; \bC^{2 \times 2})} \le C_1(B,p)
r^{p}\cdot \exp\(|b|^p \cdot r^{2p}\)\cdot \max\{1, |b|^p r^{p}\}.
 $$

(iv) The operator
\begin{equation} \label{2.39op}
 \cR_Q: \binom{f_1}{f_2}\to \int^x_0 R_Q(x,t)\binom{f_1(t)}{f_2(t)}dt =
 \int^x_0 \begin{pmatrix} R_{11}(x,t) & R_{12}(x,t) \\ R_{21}(x,t) & R_{22}(x,t)
 \end{pmatrix} \binom{f_1(t)}{f_2(t)}\,dt
\end{equation}
is a Volterra operator in the scale $L^s([0,1],\bC^2)$, $s \in [1,\infty]$.
Moreover, there exists a constant $C_3 = C_3(B,p,r)>0$ not dependent on $s \in
[1, \infty]$ and $Q, \wt{Q} \in \bU_{p,r}^{2 \times 2}$, and such that the
following uniform estimate holds
\begin{equation} \label{estimate_for_diference_inverses}
 \|(I + \cR_Q)^{-1} - (I + \cR_{\wt{Q}})^{-1}\|_{s \to s} \le C_3
 \|\wt{Q} - Q\|_p,
\end{equation}
\end{proposition}
\begin{proof}
(i), (ii). These statements are immediate by combining
Lemma~\ref{lem:wh.estim.Xinf} with Lemma~\ref{lem:R.diff.X1}.

(iii) This statement is immediate from
estimates~\eqref{eq:main_unif_est-s_for_R_11_and_R_21} and
definition~\eqref{18op}.

(iv) In accordance with (i) the kernel $R = (R_{jk})_{j,k=1}^2 \in
\XXZ{\infty,p} \cap \XXZ{1,p}$. Therefore, Lemma~\ref{lem:volterra_oper}(ii)
ensures that $\cR := \cR_Q$ is Volterra operator in each $\LLV{s}$ and
$(I + \cR)^{-1} = I + \cS$, where $(\cS f)(x) = \int_0^x S(x,t) f(t)\,dt$,
$f \in \LLV{1}$, and $\cS$ is also a Volterra operator, where
$S \in \XXZ{\infty,1} \cap \XXZ{1,1}$.

Let us first assume that $Q \in \CC{1}$. Now the kernel $R_Q$ is $C^1$-smooth,
$R_Q = (R_{jk})_{j,k=1}^2 \in C^1(\Omega; \bC^{2 \times 2})$, hence so is the
kernel $S = S_Q$, i.e. $(S_{jk})_{j,k=1}^2 \in C^1(\Omega; \bC^{2 \times 2})$.
Moreover, according to~\cite{Mal99} the operator $I + \cR$ intertwines the
operators $L_0(Q)$ and $L_0(0)$ generated by equation~\eqref{eq:systemIntro}
with potential matrices $Q$ and $0$, respectively, subject to the Cauchy
boundary condition $y(0)=0$, i.e. $L_0(Q) (I + \cR) = (I + \cR) L_0(0)$.
It follows that $(I + \cS) L_0(Q) = L_0(0) (I + \cS)$. Starting with this
equation and repeating the proof of~\cite[Theorem 1.2]{Mal99} one
rewrites it
as an equation on the matrix kernel $S = (S_{jk})_{j,k=1}^2)$ of the operator
$\cS$, similar to~\eqref{1op_JMAA}--\eqref{3op_JMAA}, which in turn leads
to the system of integral equations similar
to~\eqref{eq:Rkk=int.Qkj.Rjk}--\eqref{eq:Rjk=Qjk-int.Qjk.Rkk}.

Setting $J_{jk}(t) := \int_t^1 |S_{jk}(x,t)|^p dx$, $p \in [1, \infty)$,
and following the reasoning of Lemma~\ref{lem:wh.estim.Xinf} with $X_{1,p}$-norm
instead of $X_{\infty,p}$-norm one derives that for $Q \in
\bU_{p, r}^{2 \times 2} \cap \CC{1}$ the following uniform estimate holds
\begin{equation} \label{eq:S.X1}
 \|S\|_{\XX{1,1}} \le \|S\|_{\XX{1,p}} \le C_1 \|Q\|_p.
\end{equation}
Performing similar computations as in Lemma~\ref{lem:R.diff.X1} and using
estimate~\eqref{eq:S.X1} we arrive at the similar estimate
\begin{equation} \label{eq:S.Xinf}
 \|S\|_{\XX{\infty,1}} \le \|S\|_{\XX{\infty,p}} \le C_{\infty} \|Q\|_p.
\end{equation}
As usual, both constants $C_1 = C_1(B,p,r)$ and $C_{\infty} = C_{\infty}(B,p,r)$
do not depend on $Q$.

Going over to the case $Q = \codiag(Q_{12}, Q_{21}) \in \LL{p}$, choose
a sequence $Q_n = \codiag(Q_{12,n}, Q_{21,n}) \in \CC{1}$ approaching $Q$ in
$\LLV{p}$-norm. So, we have estimates~\eqref{eq:S.X1}--\eqref{eq:S.Xinf} with
$Q_n$ instead of $Q$. Passing in these
inequalities to the limit as $n \to \infty$ we arrive at the required
estimates~\eqref{eq:S.X1}--\eqref{eq:S.Xinf} with $Q \in \LL{p}$. Using
relations~\eqref{eq:|cN|1-1.inf-inf} one rewrites these estimates in the
operator form
\begin{align}
\label{eq:1+R-1+wtR.1}
 \|(I + \cR)^{-1}\|_{1 \to 1} & = \|I + \cS\|_{1 \to 1} \le 1 + C_1 \|Q\|_p, \\
\label{eq:1+R-1+wtR.inf}
 \|(I + \cR)^{-1}\|_{\infty \to \infty} &= \|I + \cS\|_{\infty \to \infty}
 \le 1 + C_{\infty} \|Q\|_p.
\end{align}
Combining~\eqref{eq:1+R-1+wtR.1}--\eqref{eq:1+R-1+wtR.inf}
and~\eqref{eq:wh.R.X1infp<wh.Q} with estimate~\eqref{eq:|cN|s-s} from
Lemma~\ref{lem:volterra_oper} one has
\begin{align}
\label{eq:norm.I+R}
 \qquad \qquad \qquad \|(I + \cR)^{-1}\|_{s \to s} & = \|I + \cS\|_{s \to s}
 \le 1 + C_1^{1/s} C_{\infty}^{1-1/s} \|Q\|_p \le C_2,
 \qquad & s \in [1, \infty), \qquad \\
\label{eq:norm.R-wtR}
 \qquad \qquad \qquad \|\cR_Q - \cR_{\wt{Q}}\|_{s \to s}
 &\le C_1^{1/s} C_{\infty}^{1-1/s} \|Q - \wt{Q}\|_p,
 \qquad & s \in [1, \infty), \qquad
\end{align}
where $C_2 := 1 + r \max\{C_1, C_{\infty}\}$.
Combining~\eqref{eq:norm.I+R}--\eqref{eq:norm.R-wtR} we arrive at
\begin{multline} \label{eq:norm.IQ-QwtQ}
 \norm{(I + \cR_Q)^{-1} - (I + \cR_{\wt{Q}})^{-1}}_{s \to s} =
 \norm{(I + \cR_Q)^{-1} (\cR_{\wt{Q}} - \cR_Q)
 (I + R_{\wt{Q}})^{-1}}_{s \to s} \\
 \le \norm{(I + \cR_Q)^{-1}}_{s \to s} \norm{\cR_{\wt{Q}} - \cR_Q}_{s \to s}
 \norm{(I + R_{\wt{Q}})^{-1}}_{s \to s}
 \le C_2^2 C_1^{1/s} C_{\infty}^{1-1/s} \|Q - \wt{Q}\|_p.
\end{multline}
Setting $C_3 := C_2^2 \max\{C_1, C_{\infty}\}$ we arrive
at~\eqref{estimate_for_diference_inverses}.
\end{proof}
\begin{proof}[Proof of Theorem~\ref{th:K-wtK<Q-wtQ}]
Let $Q = \codiag(Q_{12}, Q_{21}), \wt{Q} = \codiag(\wt{Q}_{12}, \wt{Q}_{21})$
and let $Q, \wt{Q} \in \bU_{p, r}^{2 \times 2}$ for some $p \in [1, \infty)$ and
$r > 0$. In accordance with Theorem~\ref{th:Trans}, for any $Q =
\codiag(Q_{12}, Q_{21}) \in \LL{1}$
representation~\eqref{eq:e=(I+K)e0}--\eqref{eq:e=e0} holds with the matrix
kernel $K^{\pm} \in \XXZ{1,1} \cap \XXZ{\infty,1}$. To evaluate the difference
$K^{\pm} - \wt{K}^{\pm}$ in respective norms we start with
representations~\eqref{2.51op} and~\eqref{eq:connec-n_of_wtK_and_wtR} for
$K^{\pm}$ and $\wt{K}^{\pm}$, respectively.

First, one should find $P^{\pm}= \diag(P^{\pm}_1, P^{\pm}_2) \in \LL{p}$ so that
$K^{\pm}$ satisfies condition~\eqref{3op_JMAA}, i.e.
\begin{equation}
 a_1 K^{\pm}_{j1}(x,0) \pm a_2 K^{\pm}_{j2}(x,0) = 0, \quad j \in \{1,2\}.
\end{equation}
Inserting representation~\eqref{2.51op} for $K^{\pm}$ in these relations leads
to the following system of Volterra type integral equations
\begin{equation} \label{2.52op}
\begin{cases}
a_1 P_1^{\pm}(x) + \int^x_0[a_1 R_{11}(x,t)P^{\pm}_1(t) \pm a_2 R_{12}(x,t)
P^{\pm}_2(t)]dt=\mp a_2 R_{12}(x,0) =: g_1^{\pm}(x), \\
\pm a_2 P^{\pm}_2(x) + \int^x_0[a_1 R_{21}(x,t)P^{\pm}_1(t) \pm a_2
R_{22}(x,t)P^{\pm}_2(t)]dt = - a_1 R_{21}(x,0) =: g_2^{\pm}(x),
\end{cases} \qquad g^{\pm} := \binom{g_{1}^{\pm}}{g_{2}^{\pm}}.
\end{equation}

By Lemma~\ref{lem:R.diff.X1}, $R = R_Q \in \XXZ{1,p}$. Therefore,
Lemma~\ref{lem:trace.X} applies and ensures that the trace mapping $i_{0,1}:
X_{1,p}^0(\Omega) \to L^p[0,1]$, $i_{0,1}\bigl(N(x,t)\bigr) := N(x,0)$, is a
contraction, hence the functions $R_{jk}(x,0)$, $j,k\in \{1,2\}$, are well
defined and $g_j^{\pm}(\cdot) \in L^p[0,1]$, $j\in \{1,2\}.$

On the other hand, by Proposition~\ref{prop2.5}(iv), the operator $\cR$ of the
form~\eqref{2.39op} is a Volterra operator in $\LL{p}$. Therefore,
system~\eqref{2.52op}, being a second kind system of Volterra equations in
$\LLV{p}$ with respect to $\col\{a_1 P_1^{\pm}(\cdot),
\pm a_2 P_2^{\pm}(\cdot)\}$, has the unique solution in $\LLV{p}$.

Similar conclusion is valid with respect to $\wt{R}= R_{\wt{Q}}$ and
$\wt{g}^{\pm}(x) := \col\(\wt{g}_{1}^{\pm}, \wt{g}_{2}^{\pm}\) :=
\col\( \mp a_2 \wt{R}_{12}(x,0), -a_1 \wt{R}_{21}(x,0)\)$.
Taking the difference of~\eqref{2.51op}
and~\eqref{eq:connec-n_of_wtK_and_wtR}, then taking $p$-th power with account
of estimate~\eqref{eq:holder.ineq}, and integrating the obtained inequality
with respect to $t \in [0,x]$ one gets
\begin{multline} \label{eq:estimate_for_dif-ce_wtK-K}
 4^{-p+1}\int^x_0 |K^{\pm}(x,t) - \wt{K}^{\pm}(x,t)|^p\,dt
 \le \int^x_0 |R(x,t) - \wt{R}(x,t)|^p \,dt
 + \int^x_0 |P^{\pm}(t) - \wt{P}^{\pm}(t)|^p\, dt \\
 + \int^x_0 |R(x,s) - \wt{R}(x,s)|^p\,ds \int^s_0 |\wt{P}^{\pm}(t)|^p\,dt
 + \int^x_0 |R(x,s)|^p\,ds \int^s_0 |P^{\pm}(s-t) - \wt{P}^{\pm}(s-t)|^p\,dt \\
 \le \|R - \wt{R}\|^p_{\XX{\infty,p}} \(1 + \|\wt{P}^{\pm}\|_p^p\)
 + \|P^{\pm} - \wt{P}^{\pm}\|_p^p \(1 + \|R\|^p_{\XX{\infty,p}}\).
\end{multline}
Combining Lemma~\ref{lem:R.diff.X1} (see~\eqref{eq:wh.R.X1p<wh.Q}) with
Lemma~\ref{lem:trace.X} yields the following estimate
\begin{align}
\nonumber
 \|g^{\pm} - \wt{g}^{\pm}\|_{L^p([0,1]; \bC^2)}
 & \le \|g_{1}^{\pm} - \wt{g}^{\pm}_1\|_{L^p[0,1]}
 + \|g_{2}^{\pm} - \wt{g}^{\pm}_2\|_{L^p[0,1]} \\
\nonumber
 & = |a_2| \cdot \|R_{12}(x,0) - \wt{R}_{12}(x,0)\|_{L^p[0,1]}
 + |a_1| \cdot \|R_{21}(x,0) - \wt{R}_{21}(x,0)\|_{L^p[0,1]} \\
\label{eq:L-p-estim_differ_(wt_g-g)}
 & \le |a_2|\cdot \|R_{12} - \wt{R}_{12}\|_{X_{1,p}(\Omega)}
 + |a_1| \cdot \|R_{21} - \wt{R}_{21}\|_{X_{1,p}(\Omega)}
 \le C_1 \|Q - \wt{Q}\|_p.
\end{align}
where $C_1$
does not depend on $Q, \wt{Q} \in \bU_{p,r}^{2 \times 2}$.

Further, it follows from the system~\eqref{2.52op} and similar system for
a vector $\wt{P}$ that
\begin{equation} \label{eq:for-la_for_P^{pm}}
 \binom{a_1 P_{1}^{\pm}(\cdot)}{\pm a_2 P_{2}^{\pm}(\cdot)}
 = (I + \cR)^{-1}\binom{g_{1}^{\pm}}{g_{2}^{\pm}}
 \qquad \text{and} \qquad
 \binom{a_1 \wt{P}_{1}^{\pm}(\cdot)}{\pm a_2 \wt{P}_{2}^{\pm}(\cdot)}
 = (I + \wt{\cR})^{-1}\binom{\wt{g}_{1}^{\pm}}{\wt{g}_{2}^{\pm}}.
\end{equation}
It follows from~\eqref{eq:L-p-estim_differ_(wt_g-g)} with $\wt{Q}=0$ that
\begin{equation} \label{eq:estim_for_g_j}
 \|g^{\pm}\|_{\LLV{p}} \le C_1 \|Q\|_p \le r C_1, \qquad
 Q \in \bU_{p, r}^{2 \times 2}.
\end{equation}
Combining~\eqref{eq:for-la_for_P^{pm}} with
estimates~\eqref{eq:L-p-estim_differ_(wt_g-g)},~\eqref{eq:estim_for_g_j} and
estimates~\eqref{estimate_for_diference_inverses},~\eqref{eq:norm.I+R} from
Proposition~\ref{prop2.5}(iv), implies
\begin{multline} \label{eq:estim_wtP-P<(wtQ-Q)}
 |a_k| \cdot \norm{P_k^{\pm} - \wt{P}_k^{\pm}}_{p}
 \le \norm{(I + \cR)^{-1} g^{\pm} - (I + \wt{\cR})^{-1} \wt{g}^{\pm}}_{p}
 \le \norm{ \((I + \cR)^{-1} - (I + \wt{\cR})^{-1}\)g^{\pm}}_{p}
 + \norm{(I + \wt{\cR})^{-1} (g^{\pm} - \wt{g}^{\pm})}_{p} \\
 \le \norm{(I + \cR)^{-1} - (I + \wt{\cR})^{-1}}_{p \to p}
 \norm{g^{\pm}}_{p}
 + \norm{(I + \wt{\cR})^{-1}}_{p \to p} \norm{g^{\pm} - \wt{g}^{\pm}}_{p}
 \le C_4 \|Q - \wt{Q}\|_p, \qquad k \in \{1,2\},
\end{multline}
where $C_4 = C_1 \cdot (C_2 + r C_3)$ and does not depend on $Q, \wt{Q} \in
\bU_{p, r}^{2 \times 2}$. Given that $b_k = a_k^{-1}$, this in turn yields
estimates
\begin{equation} \label{eq:norm.whP.P}
 \norm{P^{\pm} - \wt{P}^{\pm}}_{p} \le C_5 \|Q - \wt{Q}\,\|_p, \qquad
 \|P^{\pm}\|_{L^p} \le C_5 r, \qquad Q, \wt{Q} \in \bU_{p, r}^{2 \times 2},
\end{equation}
where $C_5 = (|b_1|^p + |b_2|^p)^{1/p} C_4$.
Inserting estimates~\eqref{eq:norm.whP.P} and~\eqref{eq:wh.R.X1infp<wh.Q}
into~\eqref{eq:estimate_for_dif-ce_wtK-K}, we arrive at the estimate
\begin{equation} \label{eq:K-wtK<Q-wtQ_New_1}
 \|K^{\pm} - \wt{K}^{\pm}\|_{\XX{\infty,p}} \le C \cdot \|Q - \wt{Q}\|_p
\end{equation}
being a part of the required estimate~\eqref{eq:K-wtK<Q-wtQ}. The second
inequality in~\eqref{eq:K-wtK<Q-wtQ} is proved similarly.
\end{proof}
\begin{remark}
(i) For Dirac $2 \times 2$ system $(B = \diag(-1,1))$ with continuous $Q$ the
triangular transformation operators have been constructed
in~\cite[Ch.10.3]{LevSar88} and~\cite[Ch.1.2]{Mar77}. For $Q \in (L^1[0,1];
\bC^{2 \times 2})$ it is proved in~\cite{AlbHryMyk05} by an appropriate
generalization of Marchenko's method.

(ii) Let $J: f\to \int_0^xf(t)dt$ denote the Volterra integration operator on
$L^p[0,1]$. Note that the similarity of integral Volterra operators given
by~\eqref{eq:cN.def} to the simplest Volterra operator of the form $B \otimes J$
acting in the spaces $L^p([0,1]; \bC^2)$ has been investigated
in~\cite{Mal99, Rom08}. The technique of investigation of integral equations for
the kernels of transformation operators in the spaces $X_{\infty,1}(\Omega)$ and
$X_{1,1}(\Omega)$ goes back to the paper~\cite{Mal94}.

(iii) The proof of Proposition~\ref{prop2.5}(iv) can be significantly
simplified in multiple cases without the use of the intricate intertwining
property. Namely, if $p \ge 2$, then
inequalities~\eqref{eq:S.X1}--\eqref{eq:S.Xinf} are immediate by combining
Lemma~\ref{lem:I+S}(i) with Proposition~\ref{prop2.5}(ii). Similarly, if radius
$r$ is sufficiently small, then Lemma~\ref{lem:inverse_oper_for_small_N} implies
these estimates.

Finally, let us consider a compact set $\cK$ instead of the ball
$\bU_{p,r}^{2 \times 2}$. Now, estimate~\eqref{eq:wh.R.X1infp<wh.Q} and
Lemma~\ref{lem:volterra_oper}(ii) ensure the compactness of the set of operators
$\{I + \cR_Q\}_{Q \in \cK}$ in each $\cB(L^s) := \cB(L^s, L^s)$. It remains to
note that the set $\{(I + \cR_Q)^{-1}\}_{Q \in \cK}$ is compact in $\cB(L^s)$,
as an image of a compact set under the continuous mapping $T \to T^{-1}$,
defined on the open set of invertible elements of the Banach algebra $\cB(L^s)$.
The proof is finished in the same way using~\eqref{eq:norm.IQ-QwtQ}.
\end{remark}

\section{General properties of a $2 \times 2$ Dirac-type BVP} \label{sec:general}
Here we consider $2 \times 2$ Dirac-type equation~\eqref{eq:systemIntro_1}
\begin{equation} \label{eq:system}
 -i B^{-1} y'+Q(x)y=\l y, \qquad y={\rm col}(y_1,y_2), \qquad x \in [0,1],
\end{equation}
subject to the following general boundary conditions
\begin{equation} \label{eq:BC}
 U_j(y) := a_{j 1}y_1(0) + a_{j 2}y_2(0) + a_{j 3}y_1(1) + a_{j 4}y_2(1)= 0,
 \quad j \in \{1,2\}.
\end{equation}
Let us also set
\begin{equation} \label{eq:A.Ajk.Jjk}
 A := \begin{pmatrix} a_{11} & a_{12} & a_{13} & a_{14} \\
 a_{21} & a_{22} & a_{23} & a_{24} \end{pmatrix}, \qquad
 A_{jk} = \begin{pmatrix} a_{1j} & a_{1k} \\ a_{2j} & a_{2k} \end{pmatrix},
 \quad J_{jk} = \det (A_{jk}), \quad j,k\in\{1,\ldots,4\}.
\end{equation}
With the system~\eqref{eq:system} one associates, in a natural way, the maximal
operator $L_{\max} = L_{\max}(Q)$ defined in $L^2([0,1]; \bC^n)$ by the differential
expression~\eqref{eq:system} on the domain
\begin{equation} \label{Max_oper_Intro}
\dom(L_{\max}) = \{y \in W^1_1([0,1]; \bC^n) : L_{\max}y \in L^2([0,1]; \bC^n)\}.
\end{equation}
Next we denote by $L := L(Q, U_1, U_2)$ the operator associated in $\LLV{2}$
with the BVP~\eqref{eq:system}--\eqref{eq:BC}. It is defined as the restriction
of the maximal operator $L_{\max} = L_{\max}(Q)$~\eqref{Max_oper_Intro} to the
domain
\begin{equation} \label{eq:dom}
 \dom(L) = \dom(L(Q, U_1, U_2)) = \{y \in \dom(L_{\max}) : U_1(y) = U_2(y) = 0\}.
\end{equation}
Let
\begin{equation} \label{eq:Phi.def}
 \Phi(\cdot, \l) =
 \begin{pmatrix} \varphi_{11}(\cdot, \l) & \varphi_{12}(\cdot, \l)\\
 \varphi_{21}(\cdot,\l) & \varphi_{22}(\cdot,\l)
 \end{pmatrix} =: \begin{pmatrix} \Phi_1(\cdot, \l) & \Phi_2(\cdot, \l)
 \end{pmatrix}, \qquad \Phi(0, \l) = I_2 :=
 \begin{pmatrix} 1 & 0 \\ 0 & 1 \end{pmatrix},
\end{equation}
be a fundamental matrix solution of the
system~\eqref{eq:system}. Here $\Phi_k(\cdot, \l)$ is the
$k$th column of $\Phi(\cdot, \l)$.

The eigenvalues of the problem~\eqref{eq:system}--\eqref{eq:BC} counting multiplicity
are the zeros (counting multiplicity) of the characteristic determinant
\begin{equation} \label{eq:Delta.def}
 \Delta_Q(\l) := \det
 \begin{pmatrix}
 U_1(\Phi_1(\cdot,\l)) & U_1(\Phi_2(\cdot,\l)) \\
 U_2(\Phi_1(\cdot,\l)) & U_2(\Phi_2(\cdot,\l))
 \end{pmatrix}.
\end{equation}
Inserting~\eqref{eq:Phi.def} and~\eqref{eq:BC} into~\eqref{eq:Delta.def}, setting
$\varphi_{jk}(\l) := \varphi_{jk}(1,\l)$,
 and taking notations~\eqref{eq:A.Ajk.Jjk} into account we arrive at the following expression for
the characteristic determinant
\begin{equation} \label{eq:Delta}
 \Delta_Q(\l) = J_{12} + J_{34}e^{i(b_1+b_2)\l}
 + J_{32}\varphi_{11}(\l) + J_{13}\varphi_{12}(\l)
 + J_{42}\varphi_{21}(\l) + J_{14}\varphi_{22}(\l).
\end{equation}
Alongside the problem~\eqref{eq:system}--\eqref{eq:BC} we consider the same problem with
$\wt{Q}$ in place of $Q$. Denote the corresponding fundamental matrix solution, its
entries, and the corresponding characteristic determinant as $\wt{\Phi}(\cdot, \l)$,
$\wt{\varphi}_{jk}(\cdot, \l)$, $j, k \in \{1, 2\}$, and $\wt{\Delta}(\l)$,
respectively.
If $Q=0$ we denote a fundamental matrix solution as $\Phi^0(\cdot, \l)$. Clearly
\begin{equation} \label{eq:Phi0.def}
 \Phi^0(x, \l)
 = \begin{pmatrix} e^{i b_1 x \l} & 0 \\ 0 & e^{i b_2 x \l} \end{pmatrix}
 =: \begin{pmatrix}
 \varphi_{11}^0(x, \l) & \varphi_{12}^0(x, \l)\\
 \varphi_{21}^0(x,\l) & \varphi_{22}^0(x,\l)
 \end{pmatrix}
 =: \begin{pmatrix} \Phi_1^0(x, \l) & \Phi_2^0(x, \l) \end{pmatrix},
 \qquad x \in [0,1], \quad \l \in \bC.
\end{equation}
Here $\Phi_k^0(\cdot, \l)$ is the $k$th column of $\Phi^0(\cdot, \l)$. In
particular, the characteristic determinant $\Delta_0(\cdot)$ becomes
\begin{equation} \label{eq:Delta0}
 \Delta_0(\l) = J_{12} + J_{34}e^{i(b_1+b_2)\l}
 + J_{32}e^{ib_1\l} + J_{14}e^{ib_2\l}.
\end{equation}
In the case of Dirac system $(B =\diag (-1,1))$ this formula is
simplified to
\begin{equation} \label{eq:Delta0_Dirac}
 \Delta_0(\l) = J_{12} + J_{34} + J_{32}e^{-i\l} + J_{14}e^{i\l}.
\end{equation}
\subsection{Representation of the characteristic determinant}
\label{subsec:det}
Our investigation of the perturbation determinant relies on the following result
clarifying our Proposition 3.1 from~\cite{LunMal16JMAA} and coinciding with it
for $Q \in \LL{1}$.
\begin{lemma} \label{lem:phi.jk=e+int}
Let $Q \in \LL{p}$ for some $p \in [1,\infty)$. Then
the functions $\varphi_{jk}(\cdot, \l)$, $j,k \in \{1,2\}$, admit the following
representations
\begin{equation} \label{eq:phi.jkx}
 \varphi_{jk}(x,\l) = \delta_{jk} e^{i b_k \l x}
 + \int_0^x K_{j1,k}(x,t) e^{i b_1 \l t}dt
 + \int_0^x K_{j2,k}(x,t) e^{i b_2 \l t}dt,
 \qquad x \in [0,1], \quad \l \in \bC,
\end{equation}
where
\begin{equation} \label{eq:Kjlk}
 K_{jl,k} := 2^{-1} \(K_{jl}^+ + (-1)^{l+k} K_{jl}^-\)
 \in X_{1,p}^0(\Omega) \cap X_{\infty,p}^0(\Omega), \qquad j,k,l \in \{1,2\}.
\end{equation}
\end{lemma}
\begin{proof}
Comparing initial conditions and applying the Cauchy uniqueness theorem one
easily gets $2 \Phi_1(\cdot,\l) = 2 \begin{pmatrix} \varphi_{11}(\cdot, \l) \\
\varphi_{21}(\cdot,\l) \end{pmatrix} = e_{+}(\cdot, \l) + e_{-}(\cdot, \l)$.
Inserting in place of $e_{+}(\cdot, \l)$ and $e_{-}(\cdot, \l)$ their
expressions from~\eqref{eq:e=(I+K)e0} one arrives
at~\eqref{eq:phi.jkx}--\eqref{eq:Kjlk} for $k=1$. Case $k=2$ is treated
similarly.
\end{proof}
\begin{remark}
Let $Q_{12} = 0$ and $Q_{21} \in L^1[0,1]$. Straightforward calculations show
that in this case
\begin{equation} \label{eq:Phi.Q12=0}
 \Phi(x,\l) = \begin{pmatrix}
 e^{i b_1 \l x} & 0 \\
 -i b_2 e^{i b_2 \l x} \int_0^x Q_{21}(t) e^{i (b_1 - b_2) \l t} dt
 & e^{i b_2 \l x}
 \end{pmatrix}.
\end{equation}

Let us demonstrate formulas~\eqref{eq:phi.jkx} in this model example. Set for brevity
$q(t) := -i b_2 Q_{21}(t)$ and recall that $\a_1 = \frac{b_2}{b_2-b_1}$ and $\a_2 =
\frac{-b_1}{b_2-b_1}$. It is easy to verify
using~\eqref{eq:Rkk=int.Qkj.Rjk}--\eqref{eq:Rjk=Qjk-int.Qjk.Rkk},~\eqref{2.52op}
and~\eqref{2.51op} that
\begin{align}
\label{eq:R.P.Q12=0}
 R(x,t) &= \begin{pmatrix} 0 & 0 \\ \a_2 q(\a_1 x + \a_2 t) & 0
 \end{pmatrix}, \qquad
 P^{\pm}(x) = \begin{pmatrix} 0 & 0 \\ 0 & \pm \a_1 q(\a_1 x)\end{pmatrix}, \\
\label{eq:K.Q12=0}
 K^{\pm}(x,t) &= \begin{pmatrix} 0 & 0 \\ \a_2 q(\a_1 x + \a_2 t) &
 \pm \a_1 q(\a_1 x - \a_1 t) \end{pmatrix}.
\end{align}
Relation~\eqref{eq:Kjlk} now implies that
\begin{equation} \label{eq:Kjlk.Q12=0}
 K_{21,1}(x,t) = \a_2 q(\a_1 x + \a_2 t), \quad
 K_{22,1}(x,t) = \a_1 q(\a_1 x - \a_1 t),
\end{equation}
while $K_{jl,k} = 0$ for other triples $(j,k,l) \in \{1,2\}^3$.
Hence~\eqref{eq:phi.jkx} implies that $\varphi_{jk}(x,\l) = \delta_{jk}
e^{i b_k \l x}$ for all pairs $(j,k) \in \{1,2\}^2$ but $(2,1)$ and that
\begin{equation} \label{eq:phi21x}
 \varphi_{21}(x,\l) =
 \int_0^x \a_2 q(\a_1 x + \a_2 t) e^{i b_1 \l t}dt
 + \int_0^x \a_1 q(\a_1 x - \a_1 t) e^{i b_2 \l t}dt,
 \qquad x \in [0,1], \quad \l \in \bC.
\end{equation}
Making changes of variables $u = \a_1 x + \a_2 t$, $v = \a_1 x - \a_1 t$ in two
integrals in~\eqref{eq:phi21x} we arrive at~\eqref{eq:Phi.Q12=0}. Namely, first
integral turns into integral $e^{i b_2 \l x} \int_{\a_1 x}^x q(t)
e^{i (b_1 - b_2) \l t} dt$, while the second one turns into similar integral
with integration over interval $[0, \a_1 x]$.
\end{remark}
\begin{lemma} \label{lem:phi-wt.phi}
Let $Q, \wt{Q} \in \bU_{p, r}^{2 \times 2}$ for some $p \in [1, \infty)$ and
$r > 0$. Then the following representation takes place
\begin{align} \label{eq:phi-wt.phi.jkx}
 \varphi_{jk}(x,\l) - \wt{\varphi}_{jk}(x,\l) =
 \int_0^x \wh{K}_{j1,k}(x,t) e^{i b_1 \l t}dt
 + \int_0^x \wh{K}_{j2,k}(x,t) e^{i b_2 \l t}dt,
 \qquad x \in [0,1], \quad \l \in \bC,
\end{align}
where
\begin{equation} \label{eq:whK.def}
 \wh{K}_{jl,k} := K_{jl,k} - \wt{K}_{jl,k} \in X_{1,p}^0(\Omega) \cap
 X_{\infty,p}^0(\Omega), \quad j,k,l \in \{1,2\},
\end{equation}
and for some $C = C(p, r, B)$ the following \emph{uniform estimate} holds
\begin{equation} \label{eq:whK<C.whQ}
 \|\wh{K}_{jl,k}\|_{X_{\infty, p}(\Omega)} + \|\wh{K}_{jl,k}\|_{X_{1,p}(\Omega)}
 \le C \cdot \|Q-\wt{Q}\|_p, \quad j,k,l \in \{1,2\}.
\end{equation}
\end{lemma}
\begin{proof}
Subtracting formula~\eqref{eq:phi.jkx} for $\wt{Q}$ from the same formula for
$Q$ we arrive at~\eqref{eq:phi-wt.phi.jkx}. Taking into account
formulas~\eqref{eq:Kjlk} and applying Theorem~\ref{th:K-wtK<Q-wtQ} (estimate~\eqref{eq:K-wtK<Q-wtQ}) we
arrive at uniform estimate~\eqref{eq:whK<C.whQ}.
\end{proof}
\begin{lemma} \label{lem:Delta=Delta0+}
Let $Q \in \LL{p}$ for some $p \in [1,\infty)$. Then the characteristic determinant
$\Delta_Q(\cdot)$ of the problem~\eqref{eq:system}--\eqref{eq:BC} is an entire function
of exponential type and admits the following representation:
\begin{align}
\label{eq:Delta=Delta0+}
 \Delta_Q(\l) &= \Delta_0(\l) + \int^1_0 g_{1,Q}(t) e^{i b_1 \l t} dt
 + \int^1_0 g_{2,Q}(t) e^{i b_2 \l t} dt, \\
\label{eq:gl=J.R}
 g_{l,Q}(\cdot) &= J_{32} K_{1l,1}(1, \cdot) + J_{42} K_{2l,1}(1, \cdot)
 + J_{13} K_{1l,2}(1, \cdot) + J_{14} K_{2l,2}(1, \cdot) \in L^p[0,1],
 \quad l \in \{1,2\}.
\end{align}
\end{lemma}
\begin{proof}
Consider representations~\eqref{eq:phi.jkx} for $\varphi_{jk}(\cdot,\l)$, $j,k
\in \{1,2\}$. By Lemma~\ref{lem:phi.jk=e+int}, $K_{jl,k}(\cdot,\cdot) \in
X_{1,p}^0(\Omega) \cap X_{\infty,p}^0(\Omega),$ \ $j,k,l \in \{1,2\}$. Therefore
by Lemma~\ref{lem:trace.X}, the trace functions $K_{jl,k}(1,\cdot)$ are well
defined and $K_{jl,k}(1,\cdot) \in L^p[0,1]$, $j,k,l \in \{1,2\}$. Therefore
inserting $x=1$ in~\eqref{eq:phi.jkx} we obtain special
representations for $\varphi_{jk}(\cdot)$, $j,k \in \{1,2\}$,
\begin{equation}
 \varphi_{jk}(\l) := \varphi_{jk}(1,\l) = \delta_{jk} e^{i b_k \l} +
 \int_0^1 K_{j1,k}(1,t) e^{i b_1 \l t}dt +
 \int_0^1 K_{j2,k}(1,t) e^{i b_2 \l t}dt.
\end{equation}
Inserting these expressions in~\eqref{eq:Delta} and taking
formula~\eqref{eq:Delta0} for $\Delta_0(\cdot)$ into account we arrive
at~\eqref{eq:Delta=Delta0+}--\eqref{eq:gl=J.R}.
\end{proof}
\begin{lemma} \label{lem:Delta-wt.Delta}
Let $Q, \wt{Q} \in \bU_{p, r}^{2 \times 2}$ for some $p \in [1, \infty)$ and
$r > 0$. Then the following representation takes place
\begin{equation} \label{eq:Delta-wt.Delta}
 \Delta_Q(\l) - \Delta_{\wt{Q}}(\l) = \int^1_0 \wh{g}_1(t) e^{i b_1 \l t} dt
 + \int^1_0 \wh{g}_2(t) e^{i b_2 \l t} dt,
\end{equation}
where $\wh{g}_l := g_{Q,l} - g_{{\wt{Q}},l} \in L^p[0,1]$, $l\in\{1,2\}$, and
for some $\wh{C} = \wh{C}(p, r, B, A)$, the following \emph{uniform estimate}
holds
\begin{equation} \label{eq:whg1+whg2<C.whQ}
 \|\wh{g}_1\|_p + \|\wh{g}_2\|_p = \|{g}_{Q,1} - g_{{\wt{Q}},1}\|_p
 + \|g_{Q,2} - g_{{\wt{Q}},2}\|_p \le \wh{C} \cdot \|Q-\wt{Q}\|_p,
 \qquad Q, \wt{Q} \in \bU_{p, r}^{2 \times 2}.
\end{equation}
\end{lemma}
\begin{proof}
Subtracting formula~\eqref{eq:Delta=Delta0+} for $\wt{Q}$ from the same formula
for $Q$ we arrive at~\eqref{eq:Delta-wt.Delta} with
\begin{equation} \label{eq:whgl=sum}
 \wh{g}_l(\cdot) = g_l(\cdot) - \wt{g}_l(\cdot) =
 J_{32} \wh{K}_{1l,1}(1, \cdot) + J_{42} \wh{K}_{2l,1}(1, \cdot) +
 J_{13} \wh{K}_{1l,2}(1, \cdot) + J_{14} \wh{K}_{2l,2}(1, \cdot),
 \quad l \in \{1, 2\}.
\end{equation}
Combining formulas~\eqref{eq:whK.def} with estimate\eqref{eq:whK<C.whQ}, and
Lemma~\ref{lem:trace.X} implies the following uniform estimate
\begin{equation} \label{eq:whKp<C.whQp}
 \|\wh{K}_{jl,k}(1, \cdot)\|_p \le
 \|\wh{K}_{jl,k}\|_{X_{\infty, p}(\Omega)} \le
 C(p, r, B) \cdot \|Q-\wt{Q}\|_p, \quad j,k,l \in \{1,2\}, \qquad Q, \wt{Q} \in \bU_{p, r}^{2 \times 2}.
\end{equation}
In turn, combining formulas~\eqref{eq:whKp<C.whQp} and~\eqref{eq:whgl=sum} we arrive
at \emph{uniform estimate}~\eqref{eq:whg1+whg2<C.whQ} with
\begin{equation} \label{eq:C=C.rpBA}
 \wh{C} = \wh{C}(p, r, B, A) = 2 \cdot (|J_{32}| + |J_{42}| + |J_{13}| +
 |J_{14}|) \cdot C(p, r, B).
\end{equation}
This completes the proof.
\end{proof}
\begin{lemma} \label{lem:Delta<Exp}
Let $Q \in \bU_{p, r}^{2 \times 2}$ for some $p \in [1, \infty)$ and $r > 0$.
Then there exists a constant $C = C(r, B, A)>0$ not dependent on $Q$ and
$p$ and such that the following uniform estimate holds
\begin{equation} \label{eq:Delta.l<C.exp}
 |\Delta_Q(\l)| \le C \cdot \bigl(e^{-b_1 \Im \l} + e^{-b_2 \Im \l}\bigr),
 \quad \l \in \bC.
\end{equation}
\end{lemma}
\begin{proof}
Since $b_1 < 0 < b_2$ then
\begin{align}
\label{eq:e1+e2>1}
 e^{-b_1 \Im \l} + e^{-b_2 \Im \l} > 1, \qquad
 e^{-b_1 \Im \l} + e^{-b_2 \Im \l} > \abs{e^{i (b_1 + b_2) \l}},
 \qquad \l \in \bC.
\end{align}
Inserting these inequalities into~\eqref{eq:Delta0} implies
estimate~\eqref{eq:Delta.l<C.exp} for $\Delta_0$ with $C_0 = |J_{12}| + |J_{34}|
+ |J_{32}| + |J_{14}|$. Since $\bU_{p, r}^{2 \times 2} \subset
\bU_{1, r}^{2 \times 2}$, it follows from~\eqref{eq:Delta=Delta0+},
\eqref{eq:whg1+whg2<C.whQ}, and estimate~\eqref{eq:Delta.l<C.exp} for $\Delta_0$
that
\begin{equation}
 |\Delta_Q(\l)| \le |\Delta_0(\l)| + \|g_1\|_1 \cdot \max\{e^{-b_1 \Im \l}, 1\}
 + \|g_2\|_1 \cdot \max\{e^{-b_2 \Im \l}, 1\} \le
 (C_0 + \wh{C} \|Q\|_1) \bigl(e^{-b_1 \Im \l} + e^{-b_2 \Im \l}\bigr),
 \quad \l \in \bC.
\end{equation}
Since $\|Q\|_1 \le \|Q\|_p \le r$, this implies~\eqref{eq:Delta.l<C.exp} with
$C = C_0 + \wh{C} r$.
\end{proof}
\subsection{Regular and strictly regular boundary conditions}
\label{subsec:regular}
Let us recall the definition of regular boundary conditions.
\begin{definition} \label{def:regular}
Boundary conditions~\eqref{eq:BC} are called \textbf{regular} if
\begin{equation} \label{eq:J32J14ne0}
 J_{14} J_{32} \ne 0.
\end{equation}
\end{definition}
Let us recall one more definition (cf.~\cite{Katsn71}).
\begin{definition} \label{def:incompressible}
Let $\L := \{\l_n\}_{n \in \bZ}$ be a sequence of complex numbers. It is
called \textbf{incompressible} if for some $d \in \bN$ every rectangle
$[t-1,t+1] \times \bR \subset \bC$ contains at most $d$ entries of the sequence,
i.e.
\begin{equation} \label{eq:card.incomp}
 \card\{n \in \bZ : |\Re \l_n - t| \le 1 \} \le d, \quad t \in \bR.
\end{equation}
To emphasize parameter $d$ we will sometimes call $\L$ an incompressible
sequence of density $d$.
\end{definition}
Recall that $\bD_r(z) \subset \bC$ denotes the disc of radius $r$ with a
center $z$.

Let us recall the following simple property of incompressible sequences
that improves Lemma 4.3 from~\cite{LunMal16JMAA}.
\begin{lemma} \label{lem:incompress.connect}
Let $\L = \{\l_n\}_{n \in \bZ}$ be an incompressible sequence of density
$d$. Then for any $\eps \le (2 d)^{-1}$ every connected component of the union
of discs $\cup_{n \in \bZ} \bD_{\eps}(\l_n)$ has at most $d$ discs
$\bD_{\eps}(\l_n)$.
\end{lemma}
\begin{proof}
Let $2 d \eps \le 1$ and assume that some connected component $\fC$ of
$\cup_{n \in \bZ} \bD_{\eps}(\l_n)$ has more than $d$ discs $\bD_{\eps}(\l_n)$.
Let $D_0 = \bD_{\eps}(\l_{n_0})$, $n_0 \in \bZ$, be one of the discs in
$\fC$.

Consider the sequence $\fB$ of all discs $\bD_{\eps}(\l_n)$ in $\fC$ that have
graph distance at most $d$ from $D_0$. Let us show that $\fB$ has more than $d$
discs. Indeed, if no disc in $\fC$ has graph distance more than $d$ from $D_0$,
then $\fB$ contains all discs from $\fC$ which has more than $d$ discs.
Otherwise, $\fB$ has some disc $D$ with distance $d$ from $D_0$. All discs on
the path from $D_0$ to $D$ belong to $\fB$ and hence $\fB$ has at least $d+1$
discs.

Let $D = \bD_{\eps}(\l_n)$ be a disc in $\fB$. Since graph distance from
$D$ to $D_0$ is at most $d$ and disc radii are $\eps$ we have $|\l_n - \l_{n_0}|
< 2 d \eps \le 1$. Thus centers of all discs in $\fB$ lie in the rectangle
$[t_0 - 1, t_0 + 1] \times \bR$, where $t_0 = \Re \l_{n_0}$. Since there are
more than $d$ discs in $\fB$ it contradicts the fact that $\L$ is an
incompressible sequence of density $d$.
\end{proof}
In Section~\ref{sec:stability.eigenvalue} we will also need the following
additional property of incompressible sequences.
\begin{lemma} \label{lem:incompress.M}
Let $\L_0 = \{\l_n^0\}_{n \in \bZ}$ be an incompressible sequence of
density $d_0$. Let $\L = \{\l_n\}_{n \in \bZ}$ be a sequence of complex
numbers such that $|\l_n - \l_n^0| < M$, $n \in \bZ$, for some $M > 0$. Then
$\L$ is an incompressible sequence of density $d = d_0 \ceil{M+1}$. Here
$\ceil{x}$ denotes the smallest integer $k$ such that $x \le k$.
\end{lemma}
\begin{proof}
Let $t \in \bR$ and let
\begin{equation}
 \cN_t := \{n \in \bZ : \l_n \in [t-1, t+1] \times \bR\}.
\end{equation}
We need to show that $\card(\cN_t) \le d = d_0 \ceil{M+1}$. Since
$|\l_n - \l_n^0| < M$, $n \in \bZ$, it follows that
\begin{equation} \label{eq:ln0.in.RtM}
 \l_n^0 \in R_{t,M} := [t-1-M, t+1+M] \times \bR, \quad n \in \cN_t,
\end{equation}
which in turn yields
\begin{equation} \label{eq:card<card}
 \card(\cN_t) \le \card\{n \in \bZ : \l_n^0 \in R_{t,M}\}.
\end{equation}
Let $t_k := t - M + 2k$, $k \in \bN$. It is clear from the definition of
$\ceil{x}$ that
\begin{equation} \label{eq:RtM.in.union}
 R_{t,M} \subset \bigcup_{k=0}^{\ceil{M}}
 \Bigl([t_k-1, t_k+1] \times \bR\Bigr).
\end{equation}
Combining~\eqref{eq:card.incomp} with~\eqref{eq:RtM.in.union} implies that
$\card\{n \in \bZ : \l_n^0 \in R_{t,M}\} \le d_0 \ceil{M+1}$.
Relation~\eqref{eq:card<card} now yields desired inequality
$\card(\cN_t) \le d$.
\end{proof}

Let us recall certain important properties from~\cite{LunMal16JMAA} of the
characteristic determinant $\Delta(\cdot)$ in the case of regular boundary
conditions.
\begin{proposition}~\cite[Proposition 4.6]{LunMal16JMAA} \label{prop:sine.type}
Let the boundary conditions~\eqref{eq:BC} be regular and let
$\Delta_Q(\cdot)$ be the characteristic determinant of the
problem~\eqref{eq:system}--\eqref{eq:BC} given by
\eqref{eq:Delta}. Then the following statements hold:

\textbf{(i)}
The function $\Delta_Q(\cdot)$ has infinitely many zeros $\L :=
\{\l_n\}_{n \in \bZ}$ counting multiplicities and $\L \subset \Pi_h$ for
some $h \ge 0$.

\textbf{(ii)} The sequence $\L$ is incompressible.

\textbf{(iii)} For any $\eps > 0$ there exists $C_{\eps} > 0$ such that the
determinant $\Delta_Q(\cdot)$ admits the following estimate from below
\begin{equation} \label{eq:Delta>=}
 |\Delta_Q(\l)| \ge C_{\eps}(e^{-b_1 \Im \l} + e^{-b_2 \Im \l}),
 \quad \l \in \bC \setminus \bigcup_{n \in \bZ} \bD_{\eps}(\l_n).
\end{equation}
\end{proposition}
Clearly, the conclusions of Proposition~\ref{prop:sine.type} are valid for the
characteristic determinant $\Delta_0(\cdot)$ given by~\eqref{eq:Delta0}. Let
$\L_0 = \{\l_n^0\}_{n \in \bZ}$ be the sequence of its zeros counting
multiplicity. Let us order the sequence $\L_0$ in a (possibly non-unique)
way such that $\Re \l_n^0 \le \Re \l_{n+1}^0$, $n \in \bZ$.

Let us recall an important result from~\cite{LunMal14Dokl,LunMal16JMAA}
and~\cite{SavShk14} concerning asymptotic behavior of eigenvalues.
\begin{proposition}~\cite[Proposition 4.7]{LunMal16JMAA}
\label{prop:Delta.regular.basic}
Let $Q \in \LL{1}$ and let boundary conditions~\eqref{eq:BC} be regular. Then
the sequence $\L = \{\l_n\}_{n \in \bZ}$ of zeros of $\Delta_Q(\cdot)$ can be ordered in
such a way that the following asymptotic formula holds
\begin{equation} \label{eq:l.n=l.n0+o(1)}
 \l_n = \l_n^0 + o(1), \quad\text{as}\quad |n| \to \infty, \quad n \in \bZ.
\end{equation}
\end{proposition}
Let us refine this ordering to have some additional important properties.
\begin{proposition} \label{prop:lambdan.ordering}
Let $Q \in \LL{1}$ and let boundary conditions~\eqref{eq:BC} be regular. Then:

(i) For any $\eps > 0$ there exist constants $M_{\eps} = M_{\eps}(Q, B, A)> 0$
and $C_{\eps} = C_{\eps}(B, A) > 0$ such that
\begin{align}
\label{eq:Rouche_estimate}
 & |\Delta_Q(\l) - \Delta_0(\l)| < |\Delta_0(\l)|, \qquad \l \notin
 \wt{\Omega}_{\eps}, \\
\label{eq:Delta_estimate}
 & |\Delta_Q(\l)| > C_{\eps} \(e^{-b_1 \Im \l} + e^{-b_2 \Im \l}\),
 \qquad \l \notin \wt{\Omega}_{\eps}.
\end{align}
where
\begin{equation} \label{eq:Omega.eps}
 \wt{\Omega}_{\eps} := \bD_{M_{\eps}}(0) \cup \Omega_{\eps}^0, \quad
 \Omega_{\eps}^0 := \bigcup\limits_{n \in \bZ} \bD_{\eps}(\l_n^0).
\end{equation}

(ii) The sequence $\L = \{\l_n\}_{n \in \bZ}$ can be ordered in such a way that
for any $\eps > 0$ and $n \in \bZ$ numbers $\l_n$ and $\l_n^0$ belong to the
same connected component of \ $\wt{\Omega}_{\eps}$. In addition,
relation~\eqref{eq:l.n=l.n0+o(1)} also holds for this ordering.
\end{proposition}
\begin{proof}
\textbf{(i)} By Proposition~\ref{prop:sine.type}(ii), the sequence $\L_0 =
\{\l_n^0\}_{n \in \bZ}$ of zeros of $\Delta_0$ is incompressible, is of density
$d_0$ and lies in the strip $\Pi_{h_0}$ for some $d_0 = d_0(B, A) \in \bN$ and
$h_0 = h_0(B, A) > 0$. Set $\eps_0 := (2 d_0)^{-1} = \eps_0(B, A)$ and let
$0 < \eps \le \eps_0$. By Proposition~\ref{prop:sine.type}(iii), there exists
$C_{\eps}^0 = C_{\eps}^0(B, A) > 0$ such that the estimate~\eqref{eq:Delta>=}
for $\Delta_0$ holds,
\begin{equation} \label{eq:Delta0_estimate}
 |\Delta_0(\l)| > C_{\eps}^0 \(e^{-b_1 \Im \l} + e^{-b_2 \Im \l}\),
 \qquad \l \notin \Omega_{\eps}^0.
\end{equation}
Since functions $g_1$ and $g_2$ from the representation~\eqref{eq:Delta=Delta0+}
are summable, it follows from~\cite[Lemma 3.5]{LunMal16JMAA} with $\delta =
C_{\eps}^0/6 = \delta(B, A)$, Lemma~\ref{lem:Delta=Delta0+} and
estimate~\eqref{eq:e1+e2>1} that
\begin{align}
\nonumber
 |\Delta_0(\l) - \Delta_Q(\l)|
 & \le \abs{\int_0^1 g_1(t) e^{i b_1 \l t} dt}
 + \abs{\int_0^1 g_2(t) e^{i b_2 \l t} dt} \\
\nonumber
 & \le 6^{-1} C_{\eps}^0
 \(e^{-b_1 \Im \l} + 1 + e^{-b_2 \Im \l} + 1\) \\
\label{eq:Delta-Delta0}
 & < 2^{-1} C_{\eps}^0
 \(e^{-b_1 \Im \l} + e^{-b_2 \Im \l}\), \qquad |\l| > M_{\eps},
\end{align}
with certain $M_{\eps} = M_{\eps}(Q, B, A) > 0$. Combining
estimate~\eqref{eq:Delta0_estimate} with~\eqref{eq:Delta-Delta0}
yields~\eqref{eq:Rouche_estimate} and~\eqref{eq:Delta_estimate} with $C_\eps =
C_{\eps}^0/2$. Estimate~\eqref{eq:Delta_estimate} yields that all zeros of
$\Delta(\cdot)$ belong to $\wt{\Omega}_{\eps}$, $0 < \eps \le \eps_0$.

\textbf{(ii)} Let $\fC_{\eps,k}^0$, $k \in \bZ$, be the sequence of all
connected components of $\Omega_{\eps}^0$. Since $\L_0$ is an incompressible
sequence of density $d_0$ and $\eps \le (2 d_0)^{-1}$,
Lemma~\ref{lem:incompress.connect} implies that each connected component of
$\Omega_{\eps}^0$ contains at most $d_0$ discs $\bD_{\eps}(\l_n^0)$. Hence
\begin{equation} \label{eq:diam.Cepsk}
 \diam\(\fC_{\eps,k}^0\) \le 2 \eps d_0 \le 1, \quad k \in \bZ.
\end{equation}
Let $\fC_{\eps}$ be a connected component of $\wt{\Omega}_{\eps}$ that contains
$\bD_{M_{\eps}}(0)$. Clearly, $\fC_{\eps}$ is the union of the disc
$\bD_{M_{\eps}}(0)$ and all connected components $\fC_{\eps,k}^0$ that intersect
with this disc. Hence inequality~\eqref{eq:diam.Cepsk} implies that
\begin{equation} \label{eq:Ceps.in.DM}
 \fC_{\eps} \subset \bD_{M_{\eps}+1}(0).
\end{equation}
We can cover $\bD_{M_{\eps}+1}(0)$ with $\ceil{M_{\eps} + 1}$ rectangles
$[t-1, t+1] \times \bR$. Hence, inequality~\eqref{eq:card.incomp} yields that
$\fC_{\eps}$ contains at most $d_0 \cdot \ceil{M_{\eps} + 1}$ entries of $\L_0$.
Due to estimate~\eqref{eq:Rouche_estimate} the Rouch\'e theorem applies and
ensures that in every (necessary bounded) connected component of
$\wt{\Omega}_{\eps}$ the functions $\Delta_0$ and $\Delta_Q = \Delta_0 +
(\Delta_Q - \Delta_0)$ have the same number of zeros counting multiplicity.
Hence, $\Delta_Q(\cdot)$ has at most $d_0 \cdot \ceil{M_{\eps} + 1}$ zeros in
$\fC_{\eps}$ counting multiplicity and at most $d_0$ zeros in any other
connected component of $\wt{\Omega}_{\eps}$.

Due to this for each $\eps$ we can choose ordering of the sequence
$\{\l_n\}_{n \in \bZ}$, in such a way that for each $n \in \bZ$, numbers $\l_n$
and $\l_n^0$ belong to the same connected component of $\wt{\Omega}_{\eps}$.
Note that this ordering is different for different $\eps$. Let's analyze how
this ordering changes when $\eps$ tends to zero. Without loss of generality we
can assume that $M_{\eps}$ is a non-decreasing function of $\eps$. We have
$\eps = \eps_0$ as our initial state of the ordering. Each connected component
has finite number of $\l_n^0$ and ordering is fixed up to a permutation within
the component. With $\eps$ tending to zero each connected component of
$\Omega_{\eps}^0$ which is not yet part of $\fC_{\eps}$ will have a finite
number of events happening. One type of event is a split. This is when some of
the discs $\bD_{\eps}(\l_n^0)$ and $\bD_{\eps}(\l_m^0)$ no longer intersect each
other. For each split we will refine the ordering after the split based on the
newly formed components and will track the ``lifetime'' of each new component
independently. Note that equal eigenvalues $\l_n^0$ and $\l_m^0$ will always
stay in one component together for all $\eps > 0$. The second type of event is
when the disc $\bD_{M_{\eps}}(0)$ ``reaches'' one of the discs
$\bD_{\eps}(\l_n^0)$ from the component we track. In this case this component is
consumed by $\fC_{\eps}$ and we no longer need to change ordering. It is clear
that after consumption ordering will satisfy the property of $\l_n$ and $\l_n^0$
being in $\fC_{\eps}$ simultaneously. Note, that for each $n \in \bZ$ this
process will end in a finite number of steps: either component will split into
individual discs $\bD_{\eps}(\l_n^0)$ that no longer will split and ordering
will be fixed for all multiple eigenvalues equal to $\l_n^0$; or component
$\fC_{\eps}$ will consume the component of $\bD_{\eps}(\l_n^0)$ and so ordering
will not change after that. Hence this process defines a single ordering of the
sequence $\L$ for which $\l_n$ and $\l_n^0$ are in the same connected
component of $\wt{\Omega}_{\eps}$ for all $\eps > 0$ and $n \in \bZ$. Combining
this with the fact that for connected components $\fC_{\eps,k}^0$ we have
$\diam(\fC_{\eps,k}^0) \to 0$ as $\eps \to 0$ uniformly at $k$, yields
relation~\eqref{eq:l.n=l.n0+o(1)}.
\end{proof}
\begin{definition} \label{def:canon.order}
Let $\L = \{\l_n\}_{n \in \bZ}$ be the sequence of zeros of the
characteristic determinant $\Delta_Q(\cdot)$ of the Dirac-type operator $L(Q)$
with summable potential and regular boundary conditions. Let
$\wt{\Omega}_{\eps}$ be defined in~\eqref{eq:Omega.eps}. The ordering of
$\L$ for which $\l_n$ and $\l_n^0$ belong to the same connected component
of $\wt{\Omega}_{\eps}$ for all $\eps > 0$ and $n \in \bZ$, will be called a
\textbf{canonical ordering}. Proposition~\ref{prop:lambdan.ordering} implies
its existence.
\end{definition}
In the sequel we need the following definitions.
\begin{definition} \label{def:sequences}
\textbf{(i)} A sequence $\L := \{\l_n\}_{n \in \bZ}$ of complex numbers is
said to be \textbf{separated} if for some positive $\tau > 0,$
\begin{equation} \label{separ_cond}
 |\l_j - \l_k| > 2 \tau \quad \text{whenever}\quad j \ne k.
\end{equation}
In particular, all entries of a separated sequence are distinct.

\textbf{(ii)} The sequence $\L$ is said to be \textbf{asymptotically
separated} if for some $N \in \bN$ the subsequence $\{\l_n\}_{|n| > N}$ is
separated.
\end{definition}
Let us recall a notion of strictly regular boundary conditions.
\begin{definition} \label{def:strictly.regular}
Boundary conditions~\eqref{eq:BC} are called \textbf{strictly regular}, if they
are regular, i.e. $J_{14} J_{32} \ne 0$, and the sequence of zeros $\l_0 =
\{\l_n^0\}_{n \in \bZ}$ of the characteristic determinant $\Delta_0(\cdot)$ is
asymptotically separated. In particular, there exists $n_0$ such that zeros
$\{\l_n^0\}_{|n| > n_0}$ are geometrically and algebraically simple.
\end{definition}
It follows from Proposition~\ref{prop:Delta.regular.basic} that the sequence
$\L = \{\l_n\}_{n \in \bZ}$ of zeros of $\Delta_Q(\cdot)$ is asymptotically
separated if the boundary conditions are strictly regular.

Assuming boundary conditions~\eqref{eq:BC} to be regular, let us rewrite them in
a more convenient form. Since $J_{14} \ne 0$, the inverse matrix $A_{14}^{-1}$
exists. Therefore writing down boundary conditions~\eqref{eq:BC} as the vector
equation $\binom{U_1(y)}{U_2(y)} = 0$ and multiplying it by the matrix
$A_{14}^{-1}$ we transform these conditions as follows
\begin{equation} \label{eq:BC.new}
\begin{cases}
 \wh{U}_{1}(y) = y_1(0) + b y_2(0) + a y_1(1) = 0, \\
 \wh{U}_{2}(y) = d y_2(0) + c y_1(1) + y_2(1) = 0,
\end{cases}
\end{equation}
with some $a,b,c,d \in \bC$. Now $J_{14} = 1$ and the boundary conditions
~\eqref{eq:BC.new} are regular if and only if $J_{32} = ad-bc \ne 0$. Thus, the
characteristic determinants $\Delta_0(\cdot)$ and $\Delta(\cdot)$ take the form
\begin{align}
\label{eq:Delta0.new}
 \Delta_0(\l) &= d + a e^{i (b_1+b_2) \l} + (ad-bc) e^{i b_1 \l}
 + e^{i b_2 \l}, \\
\label{eq:Delta.new}
 \Delta(\l) &= d + a e^{i (b_1+b_2) \l} + (ad-bc) \varphi_{11}(\l)
 + \varphi_{22}(\l) + c \varphi_{12}(\l) + b \varphi_{21}(\l).
\end{align}
Recall that $a_n \asymp b_n$, $n \in \bZ$, means that there exists
$C_2 > C_1 > 0$ such that $C_1 |b_n| \le |a_n| \le C_2 |b_n|$, $n \in \bZ$. We
will need the following simple property of zeros of $\Delta_0(\cdot)$.
\begin{lemma} \label{lem:ln0.exp.asymp}
Let boundary conditions~\eqref{eq:BC.new} be strictly regular and $\L_0 =
\{\l_n^0\}_{n \in \bZ}$ be the sequence of zeros of $\Delta_0(\cdot)$. Then the
following asymptotic relation holds
\begin{equation} \label{eq:|1+de|+|1+ae|.asymp.1}
 \abs{1 + d e^{-i b_2 \l_n^0}}^2 +
 \abs{1 + a e^{ i b_1 \l_n^0}}^2 \asymp 1, \quad n \in \bZ.
\end{equation}
Moreover, if either $b_1/b_2 \in \bQ$ or $bc=0$, then the sequence $\L_0$ is
separated. I.e. for some $\tau > 0$ we have
\begin{equation} \label{eq:ln0-lm0}
 |\l_n^0 - \l_m^0| > 2 \tau, \qquad n \ne m, \quad n, m \in \bZ.
\end{equation}
\end{lemma}
\begin{proof}
Note that
\begin{equation}
 \Delta_0(\l)
 = \(d + e^{i b_2 \l}\) \(1 + a e^{i b_1 \l}\) - b c \cdot e^{i b_1 \l}
 = e^{i b_2\l} \(1 + d e^{-i b_2\l}\) \(1 + a e^{i b_1 \l}\)
 - b c \cdot e^{i b_1 \l}, \quad \l \in \bC.
\end{equation}
Hence
\begin{equation} \label{eq:Delta_0_in_roots}
 \(1 + d e^{-i b_2 \l_n^0}\) \(1 + a e^{i b_1 \l_n^0}\)
 = b c \cdot e^{i (b_1-b_2) \l_n^0}, \quad n \in \bZ.
\end{equation}
According to Proposition~\ref{prop:sine.type}(i), there exists $h \ge 0$ such
that $\{\l_n^0\}_{n \in \bZ} \subset \Pi_h$. Hence
\begin{equation} \label{eq:e.bj.ln.asymp.1}
 e^{i b_j \l_n^0} \asymp 1, \quad n \in \bZ, \quad j \in \{1, 2\}.
\end{equation}
First assuming that $bc \ne 0$ and combining~\eqref{eq:Delta_0_in_roots} with
of~\eqref{eq:e.bj.ln.asymp.1} yields the following estimate with some $C_2 > C_1
> 0$
\begin{equation} \label{eq:||+||>=bc.e}
 C_2 > \abs{1 + d e^{-i b_2 \l_n^0}}^2 + \abs{1 + a e^{ i b_1 \l_n^0}}^2
 \ge 2 \abs{\(1 + d e^{-i b_2 \l_n^0}\)\(1 + a e^{ i b_1 \l_n^0}\)}
 = 2 |bc| \cdot \abs{e^{i (b_1-b_2) \l_n^0}} > C_1, \quad |n| \in \bZ,
\end{equation}
which proves~\eqref{eq:|1+de|+|1+ae|.asymp.1} in this case.

Now let $b c = 0$. In this case $a d \ne 0$ and $\Delta_0(\l) = e^{i b_2\l}
\(1 + d e^{-i b_2\l}\) \(1 + a e^{i b_1 \l}\)$. Let $\L_0^1 =
\{\l_{1,n}^{0}\}_{n \in \bZ}$ and $\L_0^2 = \{\l_{2,n}^{0}\}_{n \in \bZ}$ be the
sequences of zeros of the first and second factor, respectively. Clearly, these
sequences constitute arithmetic progressions lying on the lines, parallel to the
real axis. More precisely,
\begin{equation} \label{eq:l1n.l2n.bc=0}
 \l_{1,n}^{0} = \frac{\arg(-d) + 2 \pi n}{b_2} - i\frac{\ln|d|}{b_2}, \qquad
 \l_{2,n}^{0} = \frac{\arg(-a^{-1}) + 2 \pi n}{b_1} + i\frac{\ln|a|}{b_1},
 \qquad n \in \bZ.
\end{equation}
Since boundary conditions~\eqref{eq:BC.new} are strictly regular, the arithmetic
progressions $\L_0^1$ and $\L_0^2$ are asymptotically separated. Let us prove
relation~\eqref{eq:ln0-lm0}. If $b_1 / b_2 \in \bQ$, then the entire sequence
$\{\l_n^0\}_{n \in \bZ}$ is periodic, and hence it will be separated not just
asymptotically (this is valid not only if $bc=0$, see
Remark~\ref{rem:cond.examples}, case 3, below). If $b_1/b_2 \notin \bQ$, then
the set $\{2 \pi n / b_2 - 2 \pi m / b_1\}_{m, n \in \bZ}$ is everywhere dense
on $\bR$. Hence the arithmetic progressions $\L_0^1$ and $\L_0^2$ should lie on
a different lines, parallel to the real axis, and will be separated as well.
Relation~\eqref{eq:ln0-lm0} implies the following asymptotic relations
\begin{equation} \label{eq:1+de.1+ae.asymp.1}
 1 + d e^{-i b_2 \l_{2,n}^0} \asymp 1, \quad
 1 + a e^{i b_1 \l_{1,n}^0} \asymp 1, \qquad n \in \bZ,
\end{equation}
Relations~\eqref{eq:1+de.1+ae.asymp.1} trivially
imply~\eqref{eq:|1+de|+|1+ae|.asymp.1}.
\end{proof}
\begin{remark} \label{rem:cond.examples}
Let us list some types of \emph{strictly regular} boundary
conditions~\eqref{eq:BC.new}. In all of these cases except 4b the set of zeros
of $\Delta_0$ is a union of finite number of arithmetic progressions.

\begin{enumerate}

\item Regular BC~\eqref{eq:BC.new} for Dirac operator ($-b_1 = b_2 = 1$) are
strictly regular if and only if $(a-d)^2 \ne -4bc$.

\item Separated BC ($a=d=0$, $bc \ne 0$) are always strictly regular.

\item Let $b_1 / b_2 \in \bQ$, i.e. $b_1 = -n_1 \b$, $b_2 = n_2 \b$,
$n_1, n_2 \in \bN$, $\b > 0$ and $\gcd(n_1,n_2)=1$. Since $ad \ne bc$,
$\Delta_0(\cdot)$ is a polynomial at $e^{i \b \l}$ of degree $n_1 + n_2$.
Hence, BC~\eqref{eq:BC.new} are strictly regular if and only if this polynomial
does not have multiple roots. Let us list some cases with explicit conditions.

\begin{enumerate}

\item~\cite[Lemma 5.3]{LunMal16JMAA} Let $ad \ne 0$ and $bc=0$. Then
BC~\eqref{eq:BC.new} are strictly regular if and only if
\begin{equation} \label{eq:bc=0.crit.rat}
 b_1 \ln |d| + b_2 \ln |a| \ne 0 \quad\text{or}\quad
 n_1 \arg(-d) - n_2 \arg(-a) \notin 2 \pi \bZ.
\end{equation}

\item In particular, antiperiodic BC ($a=d=1$, $b=c=0$) are strictly regular if
and only if $n_1 - n_2$ is odd. Note that these BC are not strictly regular in
the case of a Dirac system.

\item~\cite[Proposition 5.6]{LunMal16JMAA} Let $a=0$, $bc \ne 0$. Then
BC~\eqref{eq:BC.new} are strictly regular if and only if
\begin{equation} \label{eq:a=0.crit.rat}
 n_1^{n_1} n_2^{n_2} (-d)^{n_1 + n_2} \ne (n_1 + n_2)^{n_1 + n_2} (-b c)^{n_2}.
\end{equation}

\end{enumerate}

\item Let $\a := -b_1 / b_2 \notin \bQ$. Then the problem of strict regularity of BC is
generally much more complicated. Let us list some known cases:

\begin{enumerate}

\item~\cite[Lemma 5.3]{LunMal16JMAA} Let $ad \ne 0$ and $bc=0$. Then
BC~\eqref{eq:BC.new} are strictly regular if and only if
\begin{equation} \label{eq:bc=0.crit.irrat}
 b_1 \ln |d| + b_2 \ln |a| \ne 0.
\end{equation}

\item~\cite[Proposition 5.6]{LunMal16JMAA} Let $a=0$ and $bc, d \in \bR
\setminus \{0\}$. Then BC~\eqref{eq:BC.new} are strictly regular if and only if
\begin{equation} \label{eq:a=0.crit}
 d \ne -(\a+1)\(|bc| \a^{-\a}\)^{\frac{1}{\a+1}}.
\end{equation}

\end{enumerate}

\end{enumerate}
\end{remark}

\section{Fourier transform estimates} \label{sec:fourier.transform}
\subsection{Generalization of Hausdorff-Young and Hardy-Littlewood theorems}
\label{subsec:HausYoungHardyLittle}
To evaluate deviations of eigenvalues of operators $L(Q)$ and $L(\wt{Q})$ we
extend here classical Hausdorff-Young and Hardy-Littlewood interpolation
theorems for Fourier coefficients (see~\cite[Theorem XII.2.3]{Zig59_v2}
and~\cite[Theorem XII.3.19]{Zig59_v2}, respectively) to the case of arbitrary
incompressible sequence $\L = \{\mu_n\}_{n \in \bZ}$ instead of
$\L = \{2\pi n\}_{n \in\bZ}$.

For efficient estimate of \emph{eigenvectors deviations} in
Section~\ref{sec:eigenfunction.stabil} we will use the following (sublinear)
Carleson transform (the maximal version of the classical Fourier transform)
\begin{equation} \label{eq:cE.f.def}
 \cE[f](\l) := \sup_{N >0}
 \abs{\int_{-N}^N F[f](t) e^{-i \l t}\, dt}, \qquad \l\in \bR.
\end{equation}
where $F[f]$ denotes the classical Fourier transform, $F[f](\l) =
\lim_{N \to \infty} \int_{-N}^N f(t) e^{i \l t}\, dt$.
The most important property of such transforms is the following Carleson-Hunt
theorem (see~\cite[Theorems 6.2.1, 6.3.3]{Gra09}).
\begin{theorem} \label{th:Carleson-Hunt}
For any $p\in (1,\infty)$ the Carleson operator $\cE$ is a bounded operator
from $L^p$ to itself, i.e. there exists a constant $C_p>0$ such that
$$
\|\cE[f]\|_{L^p} \le C_p\|f\|_{L^p}, \qquad f\in L^p(\bR).
$$
\end{theorem}
For our considerations it is more convenient to consider the following version
(with $g$ in place of $F[g]$) of $\cE$:
\begin{equation} \label{eq:sF.f.def}
\sF[g](\l) := \sup_{x \in [0,1]}
 \abs{\int_0^x g(t) e^{i \l t}\, dt}, \qquad g \in L^p[0,1], \quad \l \in \bC.
\end{equation}
We also put for brevity $\sF[g]^{\theta}(\l) := \(\sF[g](\l)\)^{\theta}$.

Combining Carleson-Hunt Theorem~\ref{th:Carleson-Hunt} with Hausdorff-Young
theorem leads to the following result (see e.g.~\cite{Sav19} for details).
\begin{proposition} \label{prop:CarlHuntHausYoung}
For any $p \in (1, 2]$ the maximal Fourier transform $\sF$ maps $L^p[0,1]$ into
$L^{p'}[0,1]$ boundedly, i.e. the following estimate holds
\begin{equation} \label{eq:carl.hunt}
 \int_{-\infty}^{\infty} \sF[g]^{p'}(x) \,dx
 \le \g_{p} \cdot \|g\|_p^{p'}, \qquad g \in L^p[0,1], \quad 1/p + 1/p' = 1,
\end{equation}
where $\g_{p} > 0$ does not depend on $g \in L^p[0,1]$.
\end{proposition}
In the sequel we will need the following lemma which proof substantially relies
on estimate~\eqref{eq:carl.hunt}.
\begin{lemma} \label{lem:fourier.coef.Lp}
Let $g \in L^p[0,1]$ for some $p \in (1, 2]$ and $h \ge 0$. Let us set
\begin{equation} \label{eq:gn.def}
 g_n := \sup\bigl\{\sF[g](\l) : \l \in \Pi_{h,n} \bigr\},
 \qquad \Pi_{h,n} := [n, n+1] \times [-h, h] \subset \bC.
 \qquad n \in \bZ.
\end{equation}
Then the following inequality holds
\begin{equation} \label{eq:sum.gn.p}
 \sum_{n \in \bZ} g_n^{p'} \le C_{p,h} \cdot \|g\|_p^{p'},
 \qquad C_{p,h} := \g_p \cdot e^{p' (h+1)}, \quad 1/p + 1/p' = 1.
\end{equation}
\end{lemma}
\begin{proof}
By definition~\eqref{eq:sF.f.def}, $\sF[g](\l) = \sup_{x \in [0,1]} |G_x(\l)|$,
where $G_x(\l) := \int_0^x g(t) e^{i \l t} dt$, $x \in [0,1]$, is an entire
function. Hence the function $|G_x(\cdot)|^{p'}$ is subharmonic, $x \in [0,1]$.

Let $n \in \bZ$ be fixed. It is clear that $\bD_1(\l) \subset [n-1,n+2] \times
[-h-1, h+1]$ for $\l \in \Pi_{h,n}$. Combining
subharmonic property with definition~\eqref{eq:sF.f.def} and this inclusion
yields
\begin{equation} \label{eq:G.l<int}
 |G_x(\l)|^{p'} \le \frac{1}{\pi} \iint\limits_{|t + iy - \l| \le 1}
 |G_x(t + i y)|^{p'} dt\, dy
 \le \frac{1}{\pi} \int_{-h-1}^{h+1} \int_{n-1}^{n+2}
 \sF[g]^{p'}(t + i y) \,dt\, dy, \qquad x \in [0,1], \quad \l \in \Pi_{h,n}.
\end{equation}
Taking supremum over $x$ and $\l$ in~\eqref{eq:G.l<int} with account
of~\eqref{eq:gn.def}, and then summing up resulting inequalities, and taking
into account that the union of intervals $[n-1, n+2)$, $n \in \bZ$, covers every
point of $\bR$ exactly 3 times, we get
\begin{equation} \label{eq:sum.sn<int}
 \sum_{n \in \bZ} g_n^{p'} = \sum_{n \in \bZ}
 \sup_{\genfrac{}{}{0pt}{2}{x \in [0,1]}{\l \in \Pi_{h,n}}}
 |G_x(\l)|^{p'} \le
 \frac{1}{\pi} \sum_{n \in \bZ} \int_{-h-1}^{h+1} \int_{n-1}^{n+2}
 \sF[g]^{p'}(t + i y) \,dt\, dy =
 \frac{3}{\pi} \int_{-h-1}^{h+1}
 \int_{-\infty}^{\infty} \sF[g]^{p'}(t + i y) dt\, dy.
\end{equation}
It is clear from the definition~\eqref{eq:sF.f.def} that
\begin{equation} \label{eq:sFg.t+iy}
 \sF[g](t + iy) = \sF[\wt{g}_y](t), \qquad \wt{g}_y(x) := g(x) e^{-yx},
 \qquad t, y \in \bR, \quad x \in [0,1].
\end{equation}
Hence inequality~\eqref{eq:carl.hunt} implies
\begin{equation} \label{eq:carl.hunt.y}
 \int_{-\infty}^{\infty} \sF[g]^{p'}(t + iy) \,dt
 = \int_{-\infty}^{\infty} \sF[\wt{g}_y]^{p'}(t) \,dt
 \le \g_{p} \cdot \|\wt{g}_y\|_p^{p'}
 \le \g_{p} \cdot e^{p'|y|} \cdot \|g\|_p^{p'}, \qquad y \in \bR.
\end{equation}
Inserting~\eqref{eq:carl.hunt.y} into~\eqref{eq:sum.sn<int} we arrive at the inequality
\begin{align} \label{eq:sum.gnp}
 \sum_{n \in \bZ} g_n^{p'}
 \le \frac{3 \g_{p}}{\pi} \cdot \|g\|_p^{p'}
 \int_{-h-1}^{h+1} e^{p'|y|} dy = C \cdot \|g\|_p^{p'}, \qquad
 C := \frac{6 \g_{p}}{\pi p'} \(e^{p'(h+1)}-1\) < \g_p e^{p'(h+1)} = C_{p,h}.
\end{align}
To estimate $C$ we used inequalities $\pi > 3$ and $p' \ge 2$. This finishes
the proof of the desired estimate~\eqref{eq:sum.gn.p}.
\end{proof}
Now we are ready to state the main result of this section being a generalization of the Hausdorff-Young and
Hardy-Littlewood theorems to the case of
 non-harmonic series with exponents forming an incompressible sequence $\L = \{\mu_n\}_{n \in \bZ}$ of density $d (\in \bN)$ instead of $\L
= \{2\pi n\}_{n \in \bZ}$.
\begin{theorem} \label{th:p.bessel}
Let $p \in (1, 2]$ and let $\L = \{\mu_n\}_{n \in \bZ}$ be an incompressible
sequence of density $d \in \bN$ lying in the strip $\Pi_h$. Then there exists
$C = C(p, h, d) > 0$ that does not depend on $\L$ and such that the following
estimates hold uniformly with respect to $g$ and $\L$
\begin{align}
\label{eq:sum.int.g<g}
 \sum_{n \in \bZ} \sF[g]^{p'}(\mu_n)
 & \le C \cdot \|g\|_p^{p'}, \qquad g \in L^p[0,1], \quad 1/p' + 1/p = 1. \\
\label{eq:sum.int.nu.g<g}
 \sum_{n \in \bZ} (1+|n|)^{p-2} \sF[g]^{p}(\mu_n)
 & \le C \cdot \|g\|_p^p, \qquad g \in L^p[0,1].
\end{align}
\end{theorem}
\begin{proof}
For brevity let us set
\begin{equation}
 \sF_f := \sF[f], \qquad \sF_f^{\a}(\l) := \Bigl(\sF[f](\l)\Bigr)^{\a},
 \qquad f \in L^1[0,1], \quad \l \in \bC, \quad \a > 0.
\end{equation}

\textbf{(i)} Set $m_n := [\Re \mu_n]$, $n \in \bZ$, where $[x]$ denotes the integer part
of $x$. Since $\{\mu_n\}_{n \in \bZ}$ is an incompressible sequence of density $d$, it
follows that for every $k \in \bZ$ there are at most $d$ indexes $n \in \bZ$ with
$m_n = k$. It is clear from the definition~\eqref{eq:gn.def} of $g_n$ that $\sF_g(\mu_n)
\le g_{m_n}$, $n \in \bZ$. Summing up the $p'$th powers of these inequalities, setting
$C := d \cdot C_{p,h}$, then applying Lemma~\ref{lem:fourier.coef.Lp} (inequality
\eqref{eq:sum.gn.p}) and taking observation on multiplicities of $m_n$ into account, we
arrive at
\begin{equation}
 \sum_{n \in \bZ} \sF_g^{p'}(\mu_n) \le \sum_{n \in \bZ} g_{m_n}^{p'}
 \le d \sum_{k \in \bZ} g_{k}^{p'} \le C \cdot \|g\|_p^{p'}.
\end{equation}

\textbf{(ii)} Now let us prove~\eqref{eq:sum.int.nu.g<g}. Consider a mapping
\begin{equation} \label{eq:oper_T_for_cFourier}
 T:\ f \to \{n \sF_f(\mu_n)\}_{n \in \bZ}
\end{equation}
defined on $L^1[0,1]$. Equip the set of integers $\bZ$ with a (finite) discrete
measure $\nu$ by setting $\nu(n) = (1+|n|)^{-2}$, $n \in \bZ$.
Inequality~\eqref{eq:sum.int.g<g} applied with $g=f$ and $p=2$ implies that
\begin{equation} \label{eq:L2-l2-weight_estimate}
 \|T f\|^2_{l^2\bigl(\bZ;\nu\bigr)}
 = \sum_{n\in \bZ} \frac{|(T f)(n)|^2}{(1+|n|\bigr)^{2}}
 = \sum_{n\in\bZ}\frac{n^2}{(1+|n|)^2} \sF_f^2(\mu_n)
 \le \sum_{n\in\bZ} \sF_f^2(\mu_n) \le N_2^2 \|f\|^2_2,
\end{equation}
where $N_2 = C(2, h, d)$, i.e. the mapping $T$ boundedly maps $L^2[0,1]$ into
the weighted space $l^2(\bZ; \nu)$ and the norm of this mapping is at most
$N_2$.

Next denote by $l^1_w(\bZ; \nu)$ a weak $l^1(\bZ; \nu)$ space. We show that the
mapping
\begin{equation}
 T: L^1[0,1] \to l^1_w(\bZ; \nu)
\end{equation}
is well defined. To this end we set $E_t :=
\{n \in \bZ:\ |n| \sF_f(\mu_n) > t\}$. Since $|\Im \mu_n| \le h$, $ n \in \bZ$,
it follows from~\eqref{eq:sF.f.def} that $\sF_f(\mu_n) \le e^h \cdot \|f\|_1$,
$n \in \bZ$. Therefore one gets
\begin{equation} \label{eq:weak-type_estimate}
 \nu(E_t) = \sum_{|n| \sF_f(\mu_n) > t} (1+|n|)^{-2}
 \le \sum_{|n| > t / (e^h \|f\|_1)}
 (1+|n|)^{-2} \le \frac{2 e^h \|f\|_1}{t}.
\end{equation}
This estimate means that $Tf \in l^1_w(\bZ; \nu)$ for each $f \in L_1[0,1]$,
i.e. the mapping $T$ has a weak type $(1,1)$ with a norm not exceeding
$N_1 := 2 e^h$. Since for each $\l \in \bC$, the functional $\sF_f(\l)$ is
sublinear,
\begin{equation}
 \sF_{f_1+f_2}(\l) \le \sF_{f_1}(\l) + \sF_{f_2}(\l),
 \qquad f_1, f_2 \in L^1[0,1],
\end{equation}
it follows that $T$ is a quasilinear operator with parameter $\kappa \le 2$ in
both cases we considered. Combining estimates~\eqref{eq:L2-l2-weight_estimate}
and~\eqref{eq:weak-type_estimate} and applying the Marcinkiewicz theorem
(\cite[Theorem 1.3.1]{BergLof76},~\cite[Theorem XII.4.6]{Zig59_v2})
we conclude that the mapping $T$ is of type $(p,p)$ for each $p \in (1,2]$ with
a norm not exceeding $c_{\kappa,p} N_1^{2/p-1} N_2^{2-2/p}$ for some constant
$c_{\kappa,p} > 0$, which proves~\eqref{eq:sum.int.nu.g<g}.
\end{proof}
\begin{corollary} \label{cor:p.bessel_with_sine_type_null_set}
Let $\L = \{\mu_n\}_{n\in \bZ}$ be sequence of zeros of a sine-type function
$\Phi(\cdot)$ with the width of indicator diagram $1$. Then for any $p \in
(1, 2]$ estimates~\eqref{eq:sum.int.g<g} and~\eqref{eq:sum.int.nu.g<g} hold
uniformly in $g \in L^p[0,1]$ and $\L$.
\end{corollary}
\begin{proof}
The proof is immediate from Theorem~\ref{th:p.bessel} if one notes that the null
set of sine-type function $\Phi(\cdot)$ is always incompressible
(see~\cite{Lev61},~\cite{Katsn71}, and Proposition~\ref{prop:sine.type}(ii)).
\end{proof}
Next we present a version of Bessel type inequalities, where the maximal version
$\cF$ of Fourier transform is replaced by the classical one. It is an immediate
consequence of Theorem~\ref{th:p.bessel}. However, we present a direct proof
which is elementary in character because it does not involve Carleson-Hunt
Theorem~\ref{th:Carleson-Hunt}.
\begin{proposition} \label{th:p.bessel_for_ordinary_Fourier}
Let $p \in (1, 2]$, $1/p' + 1/p = 1$, let $\L = \{\mu_n\}_{n \in \bZ}$ be an
incompressible sequence of density $d \in \bN$ lying in the strip $\Pi_h$, and $G(\l) =
G_g(\l) := \int_0^1 g(t) e^{i b \l t}\,dt$. Then there exists $C = C(b, p, h, d) > 0$
that does not depend on $\L$ and such that the following estimates hold \emph{uniformly
with respect to $g \in L^p[0,1]$ and $\L$}
\begin{align}
\label{eq:sum.int.g<g_Fourier}
 \sum_{n \in \bZ} |G(\mu_n)|^{p'} = \sum_{n \in \bZ} \abs{\int_0^1 g(t) e^{i b \mu_n t}\,dt}^{p'}
 & \le C \cdot \|g\|_p^{p'}, \qquad g \in L^p[0,1]. \\
\label{eq:sum.int.nu.g<g_Fourier}
 \sum_{n \in \bZ} (1+|n|)^{p-2} |G(\mu_n)|^{p}
 & \le C \cdot \|g\|_p^p, \qquad g \in L^p[0,1].
\end{align}
\end{proposition}
\begin{proof}
\textbf{(i)} It follows from the definition of $G(\cdot)$ that it is an entire
function of exponential type not exceeding $|b|$. Moreover, since $p \in
(1, 2]$, the Hausdorff-Young inequality for the Fourier transform (see
e.g.~\cite[Theorem 1.2.1.]{BergLof76},~\cite[Chapter 5.1]{SteinWei71}) ensures
that $G(\cdot) \in L^{p'}({\bR})$, and
\begin{equation} \label{eq:Gp'<C1.gp}
 \|G\|_{p'} := \(\int_{-\infty}^{+\infty} |G(\l)|^{p'} d\l\)^{1/p'}
 \le C_1(b,p) \cdot \|g\|_p,\qquad p \in (1, 2],
\end{equation}
where $C_1(b) > 0$ does not depend on $p$ and $g$. Therefore $G \in
L^{p'}_{\sigma}$, with $\sigma := |b|$.

To evaluate the left-hand side of~\eqref{eq:sum.int.g<g_Fourier} we repeat the
reasoning of~\cite[Lemma~2]{Katsn71}(see also~\cite[Setion 20.1]{Lev96}) paying
attention to the dependence of the constant $C$ on parameters involved.
Since $\L$ is an incompressible sequence of the density $d$ lying in the strip
$\Pi_h$, then every point of $\bC$ is covered by at most $d$ closed discs
$\ol{\bD_{1}(\mu_n)}$, $n \in \bZ$. Combining this fact with the subharmonicity
of $|G(z)|^{p'}$ implies
\begin{equation} \label{eq:sum.G.lk<int}
 \sum_{n \in \bZ} |G(\mu_n)|^{p'} \le \frac{1}{\pi}\sum_{n \in \bZ}
 \ \iint\limits_{|z - \l_n| \le 1}|G(x + i y)|^{p'} dx\, dy
 \le \frac{d}{\pi} \
 \iint\limits_{|\Im z| \le h + 1} |G(x + i y)|^{p'} dx\, dy
 \le \frac{2d}{\pi} \int_{0}^{h+1} e^{p'|b| y}\, dy \cdot \|G\|_{p'}^{p'}.
\end{equation}
Combining estimate~\eqref{eq:Gp'<C1.gp} with~\eqref{eq:sum.G.lk<int} and setting
\begin{equation}
 C_2(b, p, h, d) := \frac{2d}{\pi} \int_{0}^{h+1} e^{p' |b| y} dy =
 \frac{2d(e^{p'|b|(h+1)}-1)}{\pi p'|b|} \qquad \text{and} \qquad
 C(b, p, h, d) = C_2(b, p, h, d)\cdot C_1(b,p)^{p'},
\end{equation}
we arrive at~\eqref{eq:sum.int.g<g_Fourier}.

\textbf{(ii)} Now we define the operator $T$ by
formula~\eqref{eq:oper_T_for_cFourier} but with $F$ instead of $\cF$, i.e. we
put
\begin{equation} \label{eq:oper_T_for_ordinary_Fourier}
 T:\ f \to \{n G_f(\mu_n)\}_{n \in \bZ}.
\end{equation}
The rest of reasoning is just a repetition of that in the proof of
Theorem~\ref{th:p.bessel} while instead of~\eqref{eq:sum.int.g<g} we apply
estimate~\eqref{eq:sum.int.g<g_Fourier} with $p=2$.
\end{proof}
Next we also extend inverse statements of the Hausdorff-Young and
Hardy-Littlewood theorems to the case of non-harmonic exponentials series with
exponents $\L = \{\mu_n\}_{n \in \bZ}$ forming the null set of a sine type
entire function instead of $\L = \{2\pi n\}_{n \in \bZ}$.
\begin{proposition} \label{prop:inverse.sine.type}
Let $p \in (1, 2]$ and $p' = p/(p-1)$. Let $F(\cdot)$ be a sine-type function
with the width of indicator diagram $1$. Assume also that a sequence of its
zeros $\L = \{\mu_n\}_{n\in \bZ}$ is separated.
Then there exists $C = C(p, \L) > 0$ such that the following statements hold:

(i) For any sequence $\{a_n\}_{n \in \bZ} \in l^p(\bZ)$ the series
\begin{equation} \label{eq:exponent_series}
 \sum_{n \in \bZ} a_n e^{i \mu_n x} (=: f(x))
\end{equation}
converges in $L^{p'}[0,1]$ to a certain function $f \in L^{p'}[0,1]$ and the
following estimate holds
\begin{equation} \label{eq:Hausd-Yang_type_estimate}
 \|f\|_{L^{p'}[0,1]} \le C \cdot \|a\|_{l^p(\bZ)}.
\end{equation}

(ii) For any sequence $\{a_n\}_{n \in \bZ} \in l^{p'}(\bZ; (1+|n|)^{{p'}-2})$ the
series~\eqref{eq:exponent_series} converges in $L^{p'}[0,1]$ to a certain
function $f \in L^{p'}[0,1]$ and the following estimate holds
\begin{equation} \label{eq:Hausd-Yang_type_estimate_2}
 \|f\|_{L^{p'}[0,1]}^{p'} \le C \cdot \sum_{n \in \bZ} (1+|n|)^{{p'}-2} |a_n|^{p'}.
\end{equation}
\end{proposition}
\begin{proof}
(i) Consider a mapping
\begin{equation}
 a = \{a_n\}_{n \in \bZ} \to f(x) = \sum_{n \in \bZ} a_n e^{i \mu_n x}.
\end{equation}
Since $\L$ is a separated set of zeros of a sine-type function, Levin's theorem
(see~\cite[Theorem 23.2]{Lev96}) ensures that the sequence of exponentials
$\{e^{i\mu_n t}\}_{n\in\bZ}$ forms a Riesz basis in $L^2[0,1]$. It implies, in
particular, that for each sequence $\{a_n\}_{n\in\bZ}\in l^2(\bZ)$ the
series~\eqref{eq:exponent_series} converges in $L^2[0,1]$ and the following
estimate holds
\begin{equation} \label{eq:L2_estimate}
 \|f\|_{L^{2}[0,1]} \le C(2, \L) \(\sum_{n\in\bZ} |a_n|^2\)^{1/2}
 = C(2, \L)\|a\|_{l^2(\bZ)}.
\end{equation}

Further, let $p=1$ and $\{a_n\}\in l^1(\bZ)$. Since $F(\cdot)$ is s sine-type
function, the sequence $\L$ of its zeros lie in a horizontal strip $\Pi_h$ for
some $h \ge 0$. Therefore one gets that the series~\eqref{eq:exponent_series}
converges absolutely (hence uniformly) and determines a bounded (in fact,
continuous) function satisfying
\begin{equation} \label{eq:l1-to-L1_estimate}
|f(x)| \le \exp{(hx)}\cdot \|a\|_{l^1(\bZ)} \le \exp{(h)}\cdot \|a\|_{l^1(\bZ)}, \quad
x\in [0,1].
\end{equation}
Thus estimate~\eqref{eq:Hausd-Yang_type_estimate} holds with $p=1$ and $p=2$. It
remains to apply Riesz-Torin theorem (see~\cite[Theorem XII.1.11]{Zig59_v2},
\cite[Theorem 1.1.1]{BergLof76},~\cite[Theorem V.1.3]{SteinWei71})

(ii)
By Levin-Golovin theorem (\cite[Theorem 23.2]{Lev96}), the system of exponentials
$\{e_n\}_{n\in \bZ} = \{e^{i \mu_n x}\}_{n\in \bZ}$ forms a Riesz basis in
$L^2[0,1]$.
Noting that $\{a_n\}_{n \in \bZ} \in l^{p'}(\bZ; (1+|n|)^{{p'}-2})\subset l^2(\bZ)$
one gets that the series $\sum_{n\in \bZ}a_n e^{i \mu_n x}$
converges in $L^2[0,1]$. Setting $s_m = \sum_{|n|\le m} a_n e_n $ we
show that, in fact, this series converges in $L^q[0,1]$ for $q= p' \in [2, \infty)$.

 Let also $\{\chi_j\}_{j\in \bZ}$ be biorthogonal in $L^2[0,1]$ to the
sequence $\{e^{i \mu_n x}\}_{n\in \bZ}$, i.e.
$$
(e_n,\chi_j) = \int_0^1 e^{i \mu_n t}\chi_j(t)\, dt = F_j(\mu_n) = \delta_{nj}, \qquad
F_j(z):= \frac{F(z)}{F'(\mu_n)(z - \mu_n)}\, .
$$
This sequence forms a Riesz basis in $L^2[0,1]$ alongside with $\{e_n\}$.
Therefore, any $g \in L^2[0,1]$ admits a decomposition $g = \sum_{n \in \bZ}
b_n \chi_n$ with $\{b_n\}_{n \in \bZ}\in l^2(\bZ)$, and applying H\"older's
inequality and Theorem~\ref{th:p.bessel} (ii) one gets
\begin{align}
\nonumber
 \abs{\int_0^1 s_m(t) \ol{g(t)} \,dt}
 & = \abs{\(\sum_{|n|\le m} a_n e_n, \sum_{j \in \bZ} b_j \chi_j \)}
 = \abs{\sum_{|n| \le m} a_n b_n} \\
\nonumber
 & \le \(\sum_{|n|\le m} |a_n|^q (1+ |n|)^{q-2}\)^{1/q} \cdot
 \(\sum_{|n|\le m} |b_n|^p (1+ |n|)^{p-2} \)^{1/p} \\
 & \le C_p||g\|_p \(\sum_{|n|\le m} |a_n|^q (1+|n|)^{q-2}\)^{1/q}.
\end{align}
Since $L^2[0,1] \subset L^p[0,1]$ for $p\in (1,2]$, this inequality is extended to
$g\in L^p[0,1]$. Taking then supremum over $g$ running through the unit ball in
$L^p[0,1]$ one derives $\|s_m\|_q^q \le C_p \sum_{n\le m} |a_n|^q (1+ |n|)^{q-2}$.
It follows with account of the condition $\{a_n\}_{n \in \bZ} \in l^{q}(\bZ; (1+|n|)^{{q}-2})$ that
\begin{equation} \label{eq:Hardi-Littl-d_estimate_s_k}
\|s_k - s_m\|_q^q \le C_p \sum_{k\le |n|\le m+1} |a_n|^q (1+ |n|)^{q-2} \to 0 \quad \text{as} \quad k,m\to \infty,
\end{equation}
i.e. the sequence $\{s_m\}_{n \in \bZ}$ is a Cauchy sequence. Thus, there exists $f\in L^q[0,1]$
such that $\lim_{m\to\infty}\|s_m - f\|_q =0$. Passing to the limit in~\eqref{eq:Hardi-Littl-d_estimate_s_k}
as $m\to\infty$ we arrive at~\eqref{eq:Hausd-Yang_type_estimate_2}.
\end{proof}
\begin{remark}
The proof of Lemma~\ref{lem:fourier.coef.Lp} extends the classical reasoning on
estimates of Hardy space functions and $L^p_\sigma$-classes of entire functions
(see~\cite[Lectures 20-21]{Lev96},\cite[Lemma~2]{Katsn71}) to the case of
maximal Fourier transform.

In the proof of Theorem~\ref{th:p.bessel}(ii)
(inequality~\eqref{eq:sum.int.nu.g<g}) and
Proposition~\ref{prop:inverse.sine.type}(ii) we generalize the reasoning of the
proof of~\cite[Theorem 12.3.19]{Zig59_v2}. When the paper was almost ready we
found out that inequality~\eqref{eq:sum.int.g<g} was proved in the recent
paper~\cite{Sav19}. We have kept the proof for reader's convenience.
\end{remark}
\subsection{Uniform versions of Riemann-Lebesgue Lemma}
The following result will be needed in the sequel and easily follows by
combining Lemma~\ref{lem:fourier.coef.Lp} with Chebyshev's inequality. It can be
treated as a uniform (with respect to $g \in \bU_{p,r}$) version of the
classical Riemann-Lebesgue Lemma.
\begin{lemma} \label{lem:max.fourier.tail.Lp}
Let $g \in \bU_{p,r}$ for some $p \in (1, 2]$ and $r>0$. Let also $b \in \bR
\setminus \{0\}$ and $h \ge 0$. Then for any $\delta > 0$ there exists a set
$\cI_{g,\delta} \subset \bZ$ such that the following inequalities \emph{hold uniformly
with respect to} $g \in \bU_{p,r}$,
\begin{align}
\label{eq:max.card.Ngd}
 \card(\bZ \setminus \cI_{g,\delta}) \le N_{\delta}
 := C \cdot (r/\delta)^{p'}, & \qquad 1/p' + 1/p = 1, \\
\label{eq:sup.int.g.e<delta}
 \sF[g](b\l) = \sup_{x \in [0,1]} \abs{\int_0^x g(t) e^{i b \l t} dt}
 < \delta, & \qquad \l \in \bigcup_{n \in \cI_{g, \delta}}
 [n, n+1] \times [-h, h].
\end{align}
Here $C = C(p, h, b) > 0$ does not depend on $g$, $r$ and $\delta$.
\end{lemma}
\begin{proof}
Note that the proof of Lemma~\ref{lem:fourier.coef.Lp} remains the same if
we replace $\sF[g](x+iy)$ with $\sF[g](b(x+iy))$ in the
definition~\eqref{eq:gn.def} of $g_n$. The only change is that the constant
$C_{p,h}$ depends also on $b$ now. Below we assume this adjusted definition
of $g_n$.

Let $\delta > 0$. Let us prove that $\cI_{g,\delta} := \{n \in \bZ : g_n < \delta\}$
satisfies~\eqref{eq:max.card.Ngd}--\eqref{eq:sup.int.g.e<delta}.
Inequality~\eqref{eq:sup.int.g.e<delta} immediately follows from the definition of $g_n$,
$\sF[g]$, and $\cI_{g,\delta}$. Applying inequality~\eqref{eq:sum.gn.p} to estimate the
cardinality of $\cI_{g,\delta}$ we derive
\begin{equation} \label{eq:Crp>N.delta}
 C r^{p'} \ge C \|g\|_p^{p'} \ge \sum_{n \in \bZ} g_n^{p'}
 \ge \sum_{n \notin \cI_{g,\delta}} g_n^{p'}
 \ge \sum_{n \notin \cI_{g,\delta}} \delta^{p'}
 = \card\bigl(\bZ \setminus \cI_{g,\delta}\bigr) \cdot \delta^{p'}.
\end{equation}
In turn, this yields~\eqref{eq:max.card.Ngd}.
\end{proof}
Let us emphasize that the term ``uniformity'' in
Lemma~\ref{lem:max.fourier.tail.Lp} does not relate to the set $\cI_{g,\delta}$
that depends on $g$, but only to the ``size'' of its complement
(see~\eqref{eq:max.card.Ngd}).
Note also that Lemma~\ref{lem:max.fourier.tail.Lp} for regular Fourier
transforms can be proved easier without the use of the deep Carleson-Hunt
theorem.
\begin{lemma} \label{lem:fourier.tail.Lp}
Let $g \in \bU_{p,r}$ for some $p \in (1, 2]$ and $r>0$. Let also $b \in \bR
\setminus \{0\}$ and $h \ge 0$. Then for any $\delta > 0$ there exists a set
$\cI_{g,\delta} \subset \bZ$ such that the following inequalities \emph{hold uniformly
with respect to} $g \in \bU_{p,r}$,
\begin{align}
\label{eq:card.Ngd}
 \card(\bZ \setminus \cI_{g,\delta}) \le N_{\delta} := C \cdot (r/\delta)^{p'},
 & \qquad 1/p' + 1/p = 1, \\
\label{eq:int.g.e<delta}
 F[g](b \l) := \abs{\int_0^1 g(t) e^{i b \l t} dt} < \delta,
 & \qquad \l \in \bigcup_{n \in \cI_{g, \delta}} [n, n+1] \times [-h, h].
\end{align}
Here $C = C(p, h, b) > 0$ does not depend on $g$, $r$ and $\delta$.
\end{lemma}
\begin{proof}
First we can prove the version of Lemma~\ref{lem:fourier.coef.Lp} with
$F[g](t+iy)$ in place of $\sF[g](t+iy)$ and use classical Hausdorff-Young
theorem~\cite[Chapter 5.1]{SteinWei71} for Fourier transforms instead of
the deep inequality~\eqref{eq:carl.hunt} in transition~\eqref{eq:carl.hunt.y}.
After that the proof follows the proof of Lemma~\ref{lem:max.fourier.tail.Lp}.
\end{proof}
Next we investigate the maximal version of Fourier transform defined on the
space $X_{\infty,1}(\Omega)$ by:
\begin{equation} \label{eq:int.Kxt.e<delta.e}
\cF[G](\l) := \sup_{x \in [0,1]} \abs{\int_0^x G(x,t) e^{i b \l t} dt}, \qquad \l\in
\bC.
\end{equation}
First we present the following ``uniform'' version of the Riemann-Lebesgue lemma
for the space $X_{\infty,1}^0(\Omega)$. To this end for any $h\ge 0$ we let
$$
C_0(\Pi_h) := \{\varphi\in C(\Pi_h): \lim_{t\to \pm\infty} \varphi(t\pm iy)=0\
\text{uniformly in}\ \ y\in [-h,h] \}.
$$
\begin{proposition} \label{prop:K.Xinf.Fourier}
Let $h\ge 0$ and let $\cF$ be given by~\eqref{eq:int.Kxt.e<delta.e}. Then:

(i) The nonlinear mapping $\cF: X_{\infty,1}^0(\Omega) \to C(\Pi_h)$ is well
defined and it is Lipshitz, i.e.
\begin{equation} \label{eq:FG-FwtG}
\|\cF[G] - \cF[\wt{G}]\|_{C(\Pi_h)} \le e^{|b|h} \cdot \|G - \wt{G}\|_{X_{\infty,1}(\Omega),}\qquad
G, \wt{G} \in X_{\infty,1}^0(\Omega).
\end{equation}

(ii) For any $h\ge 0$ the mapping $\cF$ continuously maps
$X_{\infty,1}^0(\Omega)$ into $C_0(\Pi_h)$.

(iii) For any compact set $\cX$ in $X_{\infty,1}^0(\Omega)$ the following relation holds
\begin{equation} \label{eq:int.Fourier_of_Kxt_to_0_unif}
 \lim_{\l\to\infty} \cF[G](\l) = 0 \quad \text{uniformly in}\quad G\in \cX\quad
 \text{and}\quad \l\in \Pi_h.
\end{equation}
\end{proposition}
\begin{proof}
(i) Let $G \in X_{\infty,1}^0(\Omega)$. By
Proposition~\ref{prop:volterra_oper.Xp}(ii)
clearly valid in the scalar case too,
the operator $f \to \int_0^x G(x,t)f(t)\, dt$ maps $L^{\infty}[0,1]$ into
$C[0,1]$. To prove the continuity of $\cF[G](\cdot)$, given an $\eps > 0$ one
finds $\delta>0$, such that $\max_{t\in [0,1]}|1- \exp(ib\mu t)| < \eps$
whenever $|\mu| < \delta$. Clearly,
\begin{equation}
 \max_{x \in [0,1]}\abs{\int_0^x G(x,t) e^{i b \l_2 t}\, dt} \le
 \max_{x \in [0,1]}\abs{\int_0^x G(x,t) e^{i b \l_1 t}\, dt} +
 \max_{x \in [0,1]}\abs{\int_0^x G(x,t) (e^{i b \l_2 t} - e^{i b \l_1 t})\, dt}.
\end{equation}
Interchanging $\l_1$ and $\l_2$, combining both inequalities, and setting
$C(\l_1, \l_2) := \min\{e^{-b \Im \l_1},e^{-b \Im \l_2}\}$, one arrives at
\begin{equation}
 |\cF[G](\l_2) - \cF[G](\l_1)| \le \max_{x \in [0,1]}
 \abs{\int_0^x G(x,t) (e^{i b \l_2 t} - e^{i b \l_1 t})\, dt} \le
 \eps C(\l_1, \l_2)\cdot \|G\|_{X_{\infty,1}(\Omega)},
\end{equation}
whenever $|\l_2 - \l_1| < \delta$. Thus, the function $\cF[G](\cdot)$ belongs to
$C(\Pi_h)$ and is uniformly continuous.

Next, let us establish the Lipshitz property of $\cF: X_{\infty,1}^0(\Omega) \to
C(\Pi_h)$. Given $G, \wt{G} \in X_{\infty,1}^0(\Omega)$, one gets
\begin{equation} \label{eq:Lipshitz_est_for_X_{1,1},infty}
|\cF[G](\l) - \cF[\wt{G}](\l)| \le \max_{x \in [0,1]}\abs{\int_0^x\( G(x,t) - \wt{G}(x,t) \)e^{i b \l t}\, dt}
\le e^{|b|h} \cdot \|G - \wt{G}\|_{X_{\infty,1}(\Omega)}.
\end{equation}

(ii) Recall, that in accordance with the definition of $X_{\infty,1}^0(\Omega)$, the set $C(\Omega)$,
hence the space $C^1(\Omega)$, is dense in $X_{\infty,1}^0(\Omega)$. Fix $G_1\in C^1(\Omega)$ and
integrating by parts one gets
\begin{align}
\nonumber
 \abs{\int_0^x G_1(x,t) e^{i b \l t} dt}
 & \le \frac{1}{|b \l|} \(|G_1(x,x)| \cdot |e^{i b \l x}| + |G_1(x,0)| +
 \int_0^x \abs{D_t G_1(x, t) e^{i b \l t}} dt\) \\
\nonumber
 & \le \frac{1}{|b| \cdot |\l|} \(|G_1(x,x)| + |G_1(x,0)| +
 \int_0^x |D_t G_1(x, t)| dt\) \cdot \(e^{-b \Im \l} + 1\) \\
\label{eq:int.g.delta.e<delta.e_C_0}
 & \le \frac{3 \|G_1\|_{C^1(\Omega)}}{|b| \cdot |\l|} \cdot \(e^{-b \Im \l} + 1\)
 \le \frac{3 \|G_1\|_{C^1(\Omega)}}{|\Re\l|} \cdot\frac{\(e^{|b|h} + 1\)}{|b|}, \qquad x \in [0,1].
\end{align}
Thus, $\lim_{t\to\infty}\cF[G_1](t + iy) = 0$
uniformly in $y\in [-h,h]$, and $\cF[G_1](\cdot)\in C_0(\Pi_h)$. Combining this relation
with estimate~\eqref{eq:Lipshitz_est_for_X_{1,1},infty} yields similar relation for $\cF[G]$ with any $G \in X_{\infty,1}^0(\Omega)$, i.e.
$\cF[G](\cdot)\in C_0(\Pi_h)$ for any $G \in X_{\infty,1}^0(\Omega)$.

(iii)
By (ii), the mapping $\cF$ continuously maps $X_{\infty,1}^0(\Omega)$ into
$C_0(\Pi_h)$. Therefore the image $\cF(\cX)$ of a compact set $\cX$ is also
compact in $C_0(\Pi_h)$. To derive uniform
relation~\eqref{eq:int.Fourier_of_Kxt_to_0_unif} it remains to apply the
necessary condition of compactness in $C_0(\Pi_h)$ (uniform smallness of ``tails'').
\end{proof}
Proposition~\ref{prop:K.Xinf.Fourier}(iii) contains as a special case the
following ``uniform'' version of the classical Riemann-Lebesgue Lemma. Namely,
for any compact $\cK$ in $L^1[0,1]$ one has:
\begin{equation} \label{eq:riem.leb.uniform}
 \sup_{g \in \cK} \abs{\int_0^1 g(t) e^{i \l t} dt} = o(1)
 \quad\text{as}\quad \l \to \infty \quad\text{uniformly in}\quad
 g\in \cK \quad\text{and}\quad \l \in \Pi_h.
\end{equation}

Next we complete Proposition~\ref{prop:K.Xinf.Fourier} by evaluating the
``maximal'' Fourier transform $\cF[G](\cdot)$ in the plane instead of a strip.
This statement will be useful in Section~\ref{sec:stability.eigenvalue} when
applying the Rouch\'e theorem.
\begin{lemma} \label{lem:fourier.coef}
Let $\cX$ be a compact set in $X_{\infty,1}^0(\Omega)$, $b \in \bR \setminus \{0\}$ and
$\delta > 0$. Then there exists a constant $C = C(\cX, b, \delta) > 0$ such that the
following uniform in $G \in \cX$ estimate takes place
\begin{equation} \label{eq:int.Kxt.e<delta.e_2}
\cF[G](\l)
 \le \delta (e^{-b \Im \l} + 1), \qquad |\l| > C, \quad G \in \cX.
\end{equation}
\end{lemma}
\begin{proof}
Since $C^1(\Omega)$ is dense in $X_{\infty,1}^0(\Omega)$ and $\cX \subset
X_{\infty,1}^0(\Omega)$ is compact, there exists a finite $\delta/2$-net
$\{G_1, \ldots, G_n\}$ for $\cX$, such that $G_j \in C^1(\Omega)$. Let $G \in
\cX$, then for some $j \in \{1, \ldots, n\}$ we have
$\|G - G_j\|_{X_{\infty,1}(\Omega)} < \delta/2$. It is clear that
\begin{equation} \label{eq:int.g-g.delta.e<delta.e}
 \abs{\int_0^x (G(x,t) - G_j(x,t)) e^{i b \l t} dt}
 \le \|G - G_j\|_{X_{\infty,1}(\Omega)} \max_{t \in [0,1]} \abs{e^{i b \l t}}
 < \frac{\delta}{2} (e^{-b \Im \l} + 1), \qquad x \in [0,1], \quad \l \in \bC.
\end{equation}
Repeating the deduction of estimate~\eqref{eq:int.g.delta.e<delta.e_C_0} one
derives
\begin{equation} \label{eq:int.g.delta.e<delta.e}
 \abs{\int_0^x G_j(x,t) e^{i b \l t} dt}
 \le \frac{3 \|G_j\|_{C^1(\Omega)}}{|b \l|}
 \(e^{-b \Im \l} + 1\) < \frac{\delta}{2}
 \(e^{-b \Im \l} + 1\), \qquad x \in [0,1], \quad |\l| > C,
\end{equation}
with $C = \frac{6}{|b| \delta} \max \left\{\|G_k\|_{C^1(\Omega)} : k \in
\{1, \ldots, n\}\right\}$. Clearly, $C$ only depends on $\cX$, $b$, and
$\delta$. Combining~\eqref{eq:int.g-g.delta.e<delta.e}
with~\eqref{eq:int.g.delta.e<delta.e} we arrive
at~\eqref{eq:int.Kxt.e<delta.e_2}.
\end{proof}
Finally we apply Proposition~\ref{prop:K.Xinf.Fourier} and
Lemma~\ref{lem:fourier.coef} to transformation operators.
\begin{corollary}
 Let $K^{\pm}_Q$ be the kernel of transformation operator from representation~\eqref{eq:e=(I+K)e0}.
Then the composition $Q\to K_Q^{\pm}\to \cF[K_Q^{\pm}]$ continuously maps $\LL{p}$ into
$C_0(\Pi_h;\bC^{2 \times 2})$, $h\ge 0$, and is a Lipshitz mapping on balls in
$\LL{p}$, $p\in [1, \infty)$, i.e.
\begin{equation} \label{eq:FK-FwtK}
\|\cF[K^{\pm}_Q] - \cF[K^{\pm}_{\wt{Q}}]\|_{C(\Pi_h)} \le e^{|b|h} \cdot C(B, p, r) \cdot \|Q - \wt
Q\|_{L^p}, \qquad Q, \wt{Q} \in \bU_{p, r}^{2 \times 2}.
\end{equation}
\end{corollary}
\begin{proof}
The proof is immediate by combining Proposition~\ref{prop:K.Xinf.Fourier}(i)
with Theorem~\ref{th:K-wtK<Q-wtQ}.
\end{proof}
\begin{lemma} \label{lem:fourier.coef.alter}
Let $\cK$ be a compact set in $\LL{1}$ and $Q \in \cK$. Let also $K^{\pm} = K^{\pm}_Q$ be the kernel of
the transformation operator from representation~\eqref{eq:e=(I+K)e0}. Then for any
$\delta > 0$ there exists a constant $M = M(\cK, B, \delta) > 0$ such that the following estimate
takes place uniformly in $Q \in \cK$
\begin{equation} \label{eq:int.Kjk.e<delta.e}
 \cF[K_{jk}^{\pm}](\l) = \sup_{x \in [0,1]} \abs{ \int_0^x K_{jk}^{\pm}(x, t) e^{i b_k \l t} dt } \le
 \delta (e^{-b_k \Im \l} + 1),
 \qquad |\l| > M, \quad j,k \in \{1, 2\}.
\end{equation}
In particular, for any $h \ge 0$ one has: $\sup_{Q \in \cK}
\cF[K_{jk}^{\pm}](\l) \to 0$ as $|\l| \to \infty$ and $\l \in \Pi_h$.
\end{lemma}
\begin{proof}
By Theorem~\ref{th:K-wtK<Q-wtQ}, the mapping $T^{\pm}: Q \to K_{Q}^{\pm}(\cdot,
\cdot)$ continuously maps $\LL{1}$ into $X_{\infty,1}^0(\Omega;\bC^{2 \times 2})$.
Hence the image $\{K_{Q}^{\pm}(\cdot, \cdot) : Q \in \cK\}$ is compact in
$X_{\infty,1}^0(\Omega;\bC^{2 \times 2})$.
 Lemma~\ref{lem:fourier.coef} completes the proof.
\end{proof}
\begin{corollary} \label{cor:compact.examples}
Let $\cK$ be a ball either in the Sobolev spaces $W_1^s[0,1]$ with $s \in \bR_+$
or in the Lipshitz space $\L_{\a}[0,1]$, with $\a \in (0,1]$, or in the space
$V[0,1]$ of functions of bounded variation. Then
relations~\eqref{eq:int.Fourier_of_Kxt_to_0_unif}
and~\eqref{eq:int.Kjk.e<delta.e} hold true uniformly in $Q \in \cK$.
\end{corollary}
\begin{proof}
It well known that the balls in $W_1^s[0,1]$ and $\L_{\a}[0,1]$ are relatively
compact in $L^p[0,1]$ because of compact embedding $W_1^s[0,1] \hookrightarrow
L^p[0,1]$ and $\L_{\a}[0,1] \hookrightarrow L^p[0,1]$. Besides, balls in
$V[0,1]$ are relatively compact in $L^p[0,1]$ due to the second Helly's theorem
and Lebesgue's dominated convergence theorem. It remains to apply
Proposition~\ref{prop:K.Xinf.Fourier} and Lemma~\ref{lem:fourier.coef}.
\end{proof}
\begin{remark} \label{rem:fourier.coef.Lp}
\textbf{(i)} Let us present a simple example of the non-compact set in $L^p[0,1]$ for
which uniform relation~\eqref{eq:int.Fourier_of_Kxt_to_0_unif} is violated.
Consider the following set of functions
\begin{equation} \label{eq:cX.def}
 \cG := \{g_{\mu}(x) := g_0(x) e^{-i \mu x}: \ \mu \in \bR\}, \quad g_0 \in
 L^p[0,1], \quad c_0 := \int_0^1 g_0(t) dt > 0.
\end{equation}
It is clear that
\begin{equation}
 \cF[g_{\mu}](\mu) \ge \abs{\int_0^1 g_0(t) e^{-i \mu t} e^{i \mu t} dt}
 = c_0 \ne 0 \qquad \text{and} \qquad
 \lim_{|\l| \to \infty} \cF[g_{\mu}](\l) = 0.
\end{equation}
The last relation is satisfied \emph{not uniformly} on $\cG$. Moreover,
inequality~\eqref{eq:sup.int.g.e<delta} holds on sets $\cI_{g_{\mu}, \delta} =
\bZ \setminus (\mu - N_{\delta}, \mu + N_{\delta})$ that depends on $g_{\mu}$,
and their complements ``tend to infinity'' when $\mu \to \infty$, but have
uniformly bounded ``sizes'', $\card(\bZ \setminus \cI_{g_{\mu}, \delta}) \le
2 N_{\delta}$. We are indebted to V.P. Zastavnyi who has informed us about this
example.

\textbf{(ii)} One can complete the scalar version of
Proposition~\ref{prop:volterra_oper.Xp}(ii) by proving that for any
$G(\cdot,\cdot)\in X_{\infty,1}(\Omega)$ the function $\int_0^x G(x,t)
e^{i b \l t} dt$ is continuous in two variables $(x,\l)$, in particular,
$\cF[G](\cdot) \in C(\Pi_h)$ for any $h>0$.

\end{remark}
\section{Stability property of eigenvalues} \label{sec:stability.eigenvalue}
\subsection{Uniform localization of spectrum}
\label{subsec:stability.compact}
In this subsection we will obtain uniform with respect to $Q \in \cK$ version of
the asymptotic formula~\eqref{eq:l.n=l.n0+o(1)}, where $\cK$ is either a compact
in $\LL{1}$ or $\cK = \bU_{p,r}^{2 \times 2}$, $p \in (1, 2]$.

First, we enhance Proposition~\ref{prop:Delta.regular.basic} to get uniform
estimates for $Q \in \cK$, where $\cK$ is compact in $L^{1}$.
The following result generalizes~\cite[Theorem~3]{Sad16} to the case of
Dirac-type system and regular boundary conditions.
\begin{proposition} \label{prop:incompress.uniform}
Let $\cK$ be compact in $\LL{1}$ and $Q \in \cK$. Let boundary
conditions~\eqref{eq:BC} be regular, let $\Delta(\cdot) := \Delta_Q(\cdot)$ be
the corresponding characteristic determinant, and let $\L := \L_Q :=
\{\l_{Q,n}\}_{n \in \bZ}$, be canonically ordered sequence of its zeros. Let
also $\L_0 = \{\l_n^0\}_{n \in \bZ}$ be the sequence of zeros $\Delta_0$. Then
the following estimates hold:

(i) There exists a constant $M = M(\cK, B, A) > 0$ that does not depend on $Q$
and such that
\begin{equation} \label{eq:ln-ln0<M}
 \sup_{n \in \bZ}\abs{\l_{Q,n} - \l_n^0} \le M, \qquad Q \in \cK.
\end{equation}
In particular, there exist constants $h = h(\cK, B, A) > 0$ and $d =
d(\cK, B, A)$ that do not depend on $Q$ and such that $\L_Q$ is an
incompressible sequence of density $d$ and lying in the strip $\Pi_h$.

(ii) For any $\eps > 0$ there exists a constant $N_{\eps} = N_{\eps}(\cK, B, A)
\in \bN$ such that
\begin{equation} \label{eq:ln-ln0<eps}
 \sup_{|n| > N_{\eps}}\abs{\l_{Q,n} - \l_n^0} \le \eps, \qquad Q \in \cK.
\end{equation}
If in addition boundary conditions~\eqref{eq:BC} are strictly regular,
then there exists $\eps_0 = \eps_0(B, A)$ such that for any $\eps \in (0,
\eps_0]$ the discs $\bD_{2\eps}(\l_{Q,n})$, $|n| > N_{\eps}$, are disjoint, and
there exists a constant $\wt{C}_{\eps} = \wt{C}_{\eps}(B, A) > 0$ such that
\begin{equation} \label{eq:Delta>=C.eps}
 \min_{|\l - \l_{Q,n}| = 2 \eps}|\Delta_Q(\l)| \ge \wt{C}_{\eps},
 \qquad |n| > N_{\eps}, \quad Q \in \cK.
\end{equation}
\end{proposition}
\begin{proof}
Recall, that the sequence $\L_0 = \{\l_n^0\}_{n \in \bZ}$ of zeros of $\Delta_0$
is incompressible, is of density $d_0$ and lies in the strip $\Pi_{h_0}$ for
some $d_0 = d_0(B, A) \in \bN$ and $h_0 = h_0(B, A) > 0$. Let $0 < \eps \le
\eps_0 := (2 d_0)^{-1} = \eps_0(B, A)$. Note that functions $g_1$ and $g_2$ from
representation~\eqref{eq:Delta=Delta0+} are linear combinations of kernel traces
$K_{jk}^{\pm}(1, \cdot)$, $j,k \in \{1,2\}$. Repeating the proof of
estimate~\eqref{eq:Delta-Delta0} but using Lemma~\ref{lem:fourier.coef.alter}
with $x=1$ instead of~\cite[Lemma 3.5]{LunMal16JMAA}, we see that this estimate
is valid with a constant $M_{\eps} = M_{\eps}(\cK, B, A) > 0$ that does not
depend on $Q$.

\textbf{(i)} For brevity we set $\l_n := \l_{Q,n}$, $n \in \bZ$. Recall that
for each $n \in \bZ$ and $\eps > 0$ numbers $\l_n$ and $\l_n^0$ belong to the
same connected component of $\wt{\Omega}_{\eps}$ defined
in~\eqref{eq:Omega.eps}. Also recall that connected components
of~\eqref{eq:Omega.eps} satisfy
properties~\eqref{eq:diam.Cepsk}--\eqref{eq:Ceps.in.DM}. Hence for
$\l_n, \l_n^0 \in \fC_{\eps}$ we have $|\l_n - \l_n^0| < 2(M_{\eps} + 1)$. All
other components have diameter at most 1 and hence $|\l_n - \l_n^0| < 1$ for
such $n$. By fixing $\eps = \eps_0$ this yields~\eqref{eq:ln-ln0<M} with $M =
2(M_{\eps_0} + 1)$ that does not depend on $Q$. In turn,
Lemma~\ref{lem:incompress.M} implies that $\L_Q$ is an incompressible
sequence of density $d = d_0 \ceil{M+1}$. Since $\L_0 \subset \Pi_{h_0}$,
inequality~\eqref{eq:ln-ln0<M} yields that $\L_Q \subset \Pi_h$ for
$h = h_0 + M$.

\textbf{(ii)} Set $\wt{\eps} := \min\{\eps/(2 d_0), \eps_0\}$ and $N_{\eps} :=
\sup\{|n| : |\l_n^0| < M_{\wt{\eps}} + 1\}$. Since $M_{\wt{\eps}}$ does not
depend on $Q$, then $N_{\eps}$ also does not depend on $Q$. Let
$|n| > N_{\eps}$. In this case, due to the choice of $N_{\eps}$ and
inclusion~\eqref{eq:Ceps.in.DM}, numbers $\l_n$ and $\l_n^0$ do not belong to
$\fC_{\wt{\eps}}$. Recall that the sequence $\L = \{\l_n\}_{n \in \bZ}$ is
canonically ordered. It means that $\l_n, \l_n^0 \in \fC_{\wt{\eps},k}^0$ for
some $k \in \bZ$. Inequality~\eqref{eq:diam.Cepsk} now yields that
$|\l_n - \l_n^0| < 2 \wt{\eps} d_0 \le \eps$, $|n| > N_{\eps}$, and finishes the
proof of this part.

Finally, let us prove estimate~\eqref{eq:Delta>=C.eps}. Since boundary
conditions~\eqref{eq:BC} are strictly regular, the sequence $\L_0$ is
asymptotically separated. Namely, for some $\tau_0 = \tau_0(B, A) > 0$ and
$N_0 = N_0(B, A) \in \bN$ we have
\begin{equation} \label{eq:ln0-lm0>2tau0}
 |\l_n^0 - \l_m^0| > 2 \tau_0, \quad n \ne m, \qquad |n|, |m| > N_0.
\end{equation}
Let $0 < \eps \le \eps_0 := \tau_0/3$ and $N_{\eps} := \max\{N_0, \sup\{|n| :
|\l_n^0| < M_{\eps} + 3\eps\}\}$. Here we redefined values of $\eps_0$ and
$N_{\eps}$ that we set above. From definition~\eqref{eq:Omega.eps} of
$\wt{\Omega}_{\eps}$, inequality~\eqref{eq:ln0-lm0>2tau0} and the choice of
$N_{\eps}$ it follows that the discs $\bD_{3\eps}(\l_n^0)$, $|n| > N_{\eps}$,
are disjoint and do not intersect with $\bD_{M_{\eps}}(0)$. Hence rings
$\bD_{3\eps}(\l_n^0) \setminus \bD_{\eps}(\l_n^0)$ do not intersect with
$\wt{\Omega}_{\eps}$,
\begin{equation} \label{eq:D3e-De}
 \bD_{3\eps}(\l_n^0) \setminus \bD_{\eps}(\l_n^0) \subset
 \bC \setminus \wt{\Omega}_{\eps}, \quad |n| > N_{\eps}.
\end{equation}
In particular, $\bT_{\eps}(\l_n^0) \subset \bC \setminus \wt{\Omega}_{\eps}$,
$|n| > N_{\eps}$. Combining~\eqref{eq:Rouche_estimate} and the Rouch\'e theorem,
implies that each disc $\bD_{\eps}(\l_n^0)$, $|n| > N_{\eps}$, contains exactly
one (simple) zero of $\Delta_Q(\cdot)$, i.e.,
\begin{equation} \label{eq:ln-ln0<tau0}
 |\l_n - \l_n^0| < \eps \le \tau_0/3, \quad |n| > N_{\eps}.
\end{equation}
It follows from~\eqref{eq:ln-ln0<tau0} that
\begin{equation} \label{eq:De.in.D2e.in.D3e}
 \bD_{\eps}(\l_n^0) \subset \bD_{2 \eps}(\l_n) \subset \bD_{3\eps}(\l_n^0),
 \quad |n| > N_{\eps}.
\end{equation}
Since the discs $\bD_{3\eps}(\l_n^0)$, $|n| > N_{\eps}$,
are disjoint, inclusion~\eqref{eq:De.in.D2e.in.D3e} implies that the discs
$\bD_{2\eps}(\l_{Q,n})$, $|n| > N_{\eps}$, are also disjoint.
Inclusions~\eqref{eq:D3e-De} and~\eqref{eq:De.in.D2e.in.D3e} imply that
$\bT_{2 \eps}(\l_n) \subset \bC \setminus \wt{\Omega}_{\eps}$. Hence
combining~\eqref{eq:Delta_estimate} with~\eqref{eq:e1+e2>1}
yields~\eqref{eq:Delta>=C.eps}.
\end{proof}
Next we extend Proposition~\ref{prop:incompress.uniform} to the case
$\cK = \bU_{p,r}^{2 \times 2}$, $p \in (1, 2]$. Part (i) remains valid but we
substantially rely on the fact that $p > 1$. Part (ii) only remains valid if
we relax inequality $|n| > N_{\eps}$ to an inclusion $n \in \cI_{Q, \eps}$,
where complements of the sets $\cI_{Q, \eps}$ have uniformly bounded
cardinalities over $Q \in \bU_{p,r}^{2 \times 2}$.
\begin{proposition} \label{prop:incompress.holes}
Let $Q \in \bU_{p, r}^{2 \times 2}$ for some $p \in (1, 2]$ and $r > 0$.
Let boundary conditions~\eqref{eq:BC} be regular, let $\Delta(\cdot) :=
\Delta_Q(\cdot)$ be the corresponding characteristic determinant, and let $\L :=
\L_Q = \{\l_{Q,n}\}_{n \in \bZ}$ be a canonically ordered sequence of its zeros.
Then the following statements hold true:

(i) There exists a constant $M = M(p, r, B, A) > 0$, not dependent on
$Q$, such that
\begin{equation} \label{eq:ln-ln0<M.Lp}
 \sup_{n \in \bZ}\abs{\l_{Q,n} - \l_n^0} \le M,
 \qquad Q \in \bU_{p, r}^{2 \times 2}.
\end{equation}
In particular, there exist constants $h = h(p, r, B, A) \ge 0$ and $d =
d(p, r, B, A) > 0$, not dependent on $Q$, such that $\L_Q$ is an incompressible
sequence of density $d$ and lying in the strip $\Pi_h$.

(ii) For any $\eps > 0$ there exists $N_{\eps} = N_{\eps}(p, r, B, A) \in \bN$
that do not depend on $Q$, and a set $\cI_{Q, \eps} \subset \bZ$ such that,
\begin{align}
\label{eq:card.IQ}
 \card\(\bZ \setminus \cI_{Q, \eps}\) & \le N_{\eps}, \\
\label{eq:ln-ln0<eps.hole}
 |\l_n - \l_n^0| & < \eps, \quad n \in \cI_{Q,\eps}.
\end{align}

(iii) If boundary conditions~\eqref{eq:BC} are strictly regular then there
exists $\eps_0 = \eps_0(B, A)$ such that for any $\eps \in (0, \eps_0]$ the
discs $\bD_{2\eps}(\l_n)$, $n \in \cI_{Q, \eps}$, are disjoint, and there
exists $\wt{C}_{\eps} = \wt{C}_{\eps}(B, A) > 0$, not dependent on $Q$, $p$ and
$r$, such that
\begin{equation}
\label{eq:Delta>=C.eps.hole}
 \min_{|\l - \l_n| = 2 \eps}|\Delta(\l)| \ge \wt{C}_{\eps},
 \quad n \in \cI_{Q,\eps},
\end{equation}
where $\cI_{Q, \eps} \subset \bZ$
satisfies~\eqref{eq:card.IQ}--\eqref{eq:ln-ln0<eps.hole}.
\end{proposition}
\begin{proof}
\textbf{(i)} By Lemma~\ref{lem:Delta=Delta0+}, the difference $\Delta_Q(\l) -
\Delta_0(\l)$ admits representation~\eqref{eq:Delta=Delta0+}, and hence
\begin{equation} \label{eq:Delta.dif<g1+g2}
 |\Delta_Q(\l) - \Delta_0(\l)|
 \le \abs{\int_0^1 g_1(t) e^{i b_1 \l t} dt} +
 \abs{\int_0^1 g_2(t) e^{i b_2 \l t} dt}, \qquad \l \in \bC.
\end{equation}
In accordance with~\eqref{eq:whg1+whg2<C.whQ} the uniform estimate
$\|g_1\|_p + \|g_2\|_p \le \wh{C} \|Q\|_p \le \wh{C}r =: R$ \ holds,
where $\wh{C} = \wh{C}(p, r, B, A) > 0$. Since $p > 1$, it follows
from~\eqref{eq:Delta.dif<g1+g2} and H\"older's inequality that
\begin{align} \label{eq:Delta.Holder}
 |\Delta_Q(\l) - \Delta_0(\l)|
 \le \sum_{j=1}^2 \|g_j\|_p \(
 \frac{e^{-b_j p' \Im \l} - 1}{-b_j p' \Im \l}\)^{1/p'}
 \le R \cdot \frac{e^{-b_1 \Im \l} + e^{-b_2 \Im \l}}{|b_0 p' \Im \l|^{1/p'}},
 \qquad \Re \l \ne 0, \quad 1/p + 1/p' = 1,
\end{align}
where $b_0 := \min\{|b_1|, |b_2|\}$. Recall, that the sequence $\L_0 =
\{\l_n^0\}_{n \in \bZ}$ of zeros of $\Delta_0$ is incompressible, has density
$d_0$ and lies in the strip $\Pi_{h_0}$ for some $d_0 = d_0(B, A) \in \bN$ and
$h_0 = h_0(B, A) > 0$.

Let $0 < \eps \le \eps_0 := (2 d_0)^{-1} = \eps_0(B, A)$. By
Proposition~\ref{prop:sine.type}(iii), there exists $C_{\eps}^0 =
C_{\eps}^0(B, A) > 0$ such that the estimate~\eqref{eq:Delta0_estimate} for
$\Delta_0(\cdot)$ holds:
\begin{equation} \label{eq:Delta0>C}
 |\Delta_0(\l)| > C_{\eps}^0 \(e^{-b_1 \Im \l} + e^{-b_2 \Im \l}\)
 > C_{\eps}^0, \qquad \l \notin \Omega_{\eps}^0
 = \bigcup\limits_{n \in \bZ} \bD_{\eps}(\l_n^0).
\end{equation}
We can also assume that $C_{\eps}^0$ is a non-decreasing function
of $\eps$. Combining~\eqref{eq:Delta.Holder} with estimate~\eqref{eq:Delta0>C}
for $\eps = \eps_0$ and setting
\begin{equation} \label{eq:h.prBA.def}
 h = h(p, r, B, A) := \max\Bigl\{h_0 + \eps_0,
 \bigl(2 R / C_{\eps_0}^0\bigr)^{p'}(b_0 p')^{-1}\Bigr\},
\end{equation}
one has
\begin{equation} \label{eq:DeltaQ>0.Pih}
 |\Delta_Q(\l)| \ge |\Delta_0(\l)| - |\Delta_Q(\l) - \Delta_0(\l)|
 > 2^{-1} C_{\eps_0}^0 \(e^{-b_1 \Im \l} + e^{-b_2 \Im \l}\) > 0,
 \qquad |\Im \l| \ge h.
\end{equation}
This implies the inclusion $\L_Q \subset \Pi_h$ for all $Q \in
\bU_{p, r}^{2 \times 2}$.

Set $\delta = C_{\eps}^0/4 (< C_{\eps_0}^0/2)$. As established above,
$g_j \in \bU_{p,R}^{2 \times 2}$, $j \in \{1, 2\}$, $R = r \cdot
\wh{C}(p, r, B, A)$. Hence by Lemma~\ref{lem:fourier.tail.Lp} there exists
$C_j = C(p, h, b_j) > 0$ such that uniform
inequalities~\eqref{eq:card.Ngd}--\eqref{eq:int.g.e<delta} hold for $g := g_j$,
$j \in \{1, 2\}$, i.e.
\begin{align}
\label{eq:card.Ngd.gj}
 \card(\bZ \setminus \cI_{g_{j},\delta}) & \le N_{j,\delta}
 := C_j \cdot (R/\delta)^{p'}, \\
\label{eq:int.gj.e<delta}
 \abs{\int_0^1 g_j(t) e^{i b_j \l t} dt}
 & < \delta = C_{\eps}^0/4, \qquad \l \in \bigcup_{n \in \cI_{g_j, \delta}}
 [n, n+1] \times [-h, h].
\end{align}
Set $\wh{\cI}_{Q,\eps} := \cI_{g_1,\delta} \cap \cI_{g_2,\delta}$.
Combining~\eqref{eq:Delta.dif<g1+g2},~\eqref{eq:Delta0>C},
\eqref{eq:DeltaQ>0.Pih} and~\eqref{eq:int.gj.e<delta} we arrive at
\begin{equation} \label{eq:DeltaQ>C}
 |\Delta_Q(\l)| \ge |\Delta_0(\l)| - |\Delta_Q(\l) - \Delta_0(\l)|
 > C_{\eps}^0/4, \qquad \l \not \in \Omega_{h,Q,\eps} := \Omega_{\eps}^0
 \cup \Pi_{h,Q,\eps}, \quad \Pi_{h,Q,\eps} :=
 \bigcup_{n \notin \wh{\cI}_{Q,\eps}} (n, n+1) \times (-h, h).
\end{equation}
It is clear from~\eqref{eq:card.Ngd.gj} that $\card\bigl(\bZ \setminus
\wh{\cI}_{Q,\eps}\bigr) \le \wh{N}_{\eps} := N_{g_1,\delta} + N_{g_2,\delta}$,
which implies that
\begin{align}
\label{eq:Omega.h.Q.eps}
 & \Pi_{h,Q,\eps} := \bigcup_{k=1}^{m_{\eps}} \Bigl((\a_k, \b_k) \times
 (-h, h)\Bigr) \subset \Pi_h, \qquad m := m_{\eps} \le \wh{N}_{\eps}, \\
\label{eq:ak.bk}
 & \a_k, \b_k \in \bZ, \quad k \in \{1, \ldots, m\}, \qquad
 \a_1 < \b_1 < \a_2 < \b_2 < \ldots < \a_m < \b_m, \qquad
 \sum_{k=1}^m (\b_k-\a_k) \le \wh{N}_{\eps},
\end{align}
i.e. segments $[\a_k, \b_k]$, $k \in \{1, \ldots, m\}$, with integer ends are
disjoint with a total length not exceeding $\wh{N}_{\eps}$, that does not depend
on $Q$.

As in the proof of Proposition~\ref{prop:lambdan.ordering}, let
$\fC_{\eps,j}^0$, $j \in \bZ$, be the sequence of all connected components of
$\Omega_{\eps}^0$. Recall, that since $\eps \le (2 d_0)^{-1}$, then each
such connected component contains at most $d_0$ discs $\bD_{\eps}(\l_n^0)$,
and according to~\eqref{eq:diam.Cepsk} has diameter at most one,
$\diam\bigl(\fC_{\eps,j}^0\bigr) \le 2 d_0 \eps \le 1$, $j \in \bZ$.

Let $k \in \{1, \ldots, m\}$. Consider a connected component $\fC_k$ of
$\Omega_{h,Q,\eps}$ that contains the rectangle $(\a_k, \b_k) \times (-h, h)$.
Clearly, $\fC_k$ is the union of this rectangle and a finite set of connected
components $\fC_{\eps,j}^0$. Since $\diam\bigl(\fC_{\eps,j}^0\bigr) \le 1$,
$j \in \bZ$, it follows that $\fC_k \subset (\a_k - 1, \b_k + 1) \times (-h,h)$.
And, thus,
\begin{equation}
 \diam\(\fC_k\) \le \b_k - \a_k + 2 + 2h \le \wh{N}_{\eps} + 2 + 2h.
\end{equation}
This implies that all connected components of $\Omega_{h,Q,\eps}$ has uniformly
bounded diameters for all $Q \in \bU_{p,r}^{2 \times 2}$.

Due to~\eqref{eq:DeltaQ>C} the Rouch\'e theorem applies and ensures that in
every connected component of $\Omega_{h,Q,\eps}$ the functions $\Delta_0$ and
$\Delta_Q = \Delta_0 + (\Delta_Q - \Delta_0)$ have the same number of zeros
counting multiplicity. At this point, we need to enhance the canonical ordering
to satisfy the following property: for any $\eps > 0$ and $n \in \bZ$ numbers
$\l_n$ and $\l_n^0$ belong to the same connected component of
$\Omega_{h,Q,\eps}$. It is clear from~\eqref{eq:int.gj.e<delta} and from the
fact that $C_{\eps}^0$ is non-decreasing function of $\eps$, that the sets
$\cI_{g_j, C_{\eps}^0/4}$ monotonically increase as $\eps$ tends to 0, i.e.,
$\cI_{g_j, C_{\eps_1}^0/4} \subset \cI_{g_j, C_{\eps_2}^0/4}$ if $\eps_1 >
\eps_2$. Hence, the same is true for $\wh{\cI}_{Q,\eps}$. Now the proof finished
in the same way as in Proposition~\ref{prop:lambdan.ordering}(ii) by tracking
the ``lifetime'' of connected components of $\Omega_{h,Q,\eps}$. This property,
in particular, implies
\begin{equation} \label{eq:ln-ln0<diam}
 |\l_n - \l_n^0| \le \diam(\fC_{\eps,n}), \qquad n \in \bZ, \quad \eps > 0,
\end{equation}
where $\fC_{\eps,n}$ is a connected component of $\Omega_{h,Q,\eps}$ that
contains $\l_n^0$ (and, thus, also contains $\l_n$). Now set $\eps = \eps_0$.
Since connected components have uniformly bounded diameters,
inequality~\eqref{eq:ln-ln0<diam} implies inequality~\eqref{eq:ln-ln0<M.Lp}. In
turn, Lemma~\ref{lem:incompress.M} yields that $\L_Q$ is an impressible sequence
of density $d$, that does not depend on $Q$.

\textbf{(ii)} Set $\wt{\eps} := \min\{\eps/(2 d_0), \eps_0\}$. Let
$\cI_{Q,\eps}$ be the set of integers $n$, for which connected component
$\fC_{\wt{\eps},j_n}^0$ of $\Omega_{\wt{\eps}}$ that contains
$\bD_{\wt{\eps}}(\l_n^0)$ does not intersect with $\Pi_{h,Q,\wt{\eps}}$.
Let $n \in \cI_{Q,\eps}$. In this case, $\fC_{\wt{\eps},j_n}^0$ is also a
connected component of $\Omega_{h,Q,\eps}$. Hence
inequality~\eqref{eq:ln-ln0<diam} yields that
\begin{equation} \label{eq:ln-ln0<diam<eps}
 |\l_n - \l_n^0| < \diam\(\fC_{\wt{\eps},j_n}^0\)
 \le 2 \cdot \wt{\eps} \cdot d_0 \le \eps, \qquad n \in \cI_{Q, \eps},
\end{equation}
and proves inequality~\eqref{eq:ln-ln0<eps.hole}.

Let us estimate $\card\bigl(\bZ \setminus \cI_{Q,\eps}\bigr)$. Let $n \notin
\cI_{Q,\eps}$. Then the connected component $\fC_{\wt{\eps},j_n}^0$ intersects
with $\Pi_{h,Q,\wt{\eps}}$. Since $\diam\bigl(\fC_{\wt{\eps},j_n}^0\bigr)
\le 1$, it follows from~\eqref{eq:Omega.h.Q.eps} that $\fC_{\wt{\eps},j_n}^0
\subset (\a_k - 1, \b_k + 1) \times (-h, h)$ for some
$k \in \{1, \ldots, m_{\wt{\eps}}\}$. Thus,
\begin{equation} \label{eq:ln0.in.PQih}
 \l_n^0 \in \wh{\Pi}_{h,Q,\wt{\eps}} := \bigcup_{k=1}^{m_{\eps}}
 \Bigl((\a_k - 1, \b_k + 1) \times (-h, h)\Bigr),
 \qquad n \in \bZ \setminus \cI_{Q,\eps}.
\end{equation}
Since $\L_0 = \{\l_n^0\}_{n \in \bZ}$ is an incompressible sequence of density
$d_0$, then for each $k \in \{1, \ldots, m_{\wt{\eps}}\}$ the rectangle
$(\a_k - 1, \b_k + 1) \times (-h, h)$ contains at most $d_0 \ceil{(\b_k - \a_k
+ 2)/2}$ entries of $\L_0$. Inclusion~\eqref{eq:ln0.in.PQih} and
inequality~\eqref{eq:ak.bk} now implies
\begin{equation}
 \card\bigl(\bZ \setminus \cI_{Q,\eps}\bigr) \le
 \sum_{k=1}^{m_{\wt{\eps}}} d_0 \ceil{(\b_k - \a_k + 2)/2} \le
 \frac{d_0}{2} \sum_{k=1}^{m_{\wt{\eps}}} (\b_k - \a_k) +
 \frac{3 d_0 m_{\wt{\eps}}}{2} \le 2 d_0 \wh{N}_{\wt{\eps}} =: N_{\eps},
\end{equation}
which finishes the proof since $\wh{N}_{\wt{\eps}}$ does not depend on $Q$.

\textbf{(iii)} Since boundary conditions~\eqref{eq:BC} are strictly regular,
then relation~\eqref{eq:ln0-lm0>2tau0} holds. Let us redefine $\eps_0$ defined
above as $\eps_0 := \tau_0 / 3$. Let $0 < \eps \le \eps_0$. Let us redefine the
set $\cI_{Q,\eps}$ defined in the proof of part (ii) above:
\begin{equation} \label{eq:cIQe.def}
 \cI_{Q,\eps} := \left\{n \in \bZ \ : \ |n| > N_0, \ \
 \bD_{3 \eps}(\l_n^0) \cap \Pi_{h,Q,\eps} = \varnothing \right\},
\end{equation}
where $\Pi_{h,Q,\eps}$ is defined in~\eqref{eq:Omega.h.Q.eps}. Using the same
reasoning as in the proof of part (ii) we can prove relation~\eqref{eq:card.IQ}.

Let $n \in \cI_{Q,\eps}$. Since $|n| > N_0$, then discs $\bD_{\eps}(\l_n^0)$ are
disjoint due to~\eqref{eq:ln0-lm0>2tau0}. Hence each such a disc is a standalone
connected component of $\Omega_{\eps}^0$ that does not intersect with
with $\Pi_{h,Q,\eps}$. Hence the new set $\cI_{Q,\eps}$ is a subset of
previously defined set $\cI_{Q,\eps}$, which implies~\eqref{eq:ln-ln0<eps.hole}
due to the proof of part (ii). Due to the construction~\eqref{eq:cIQe.def} of
$\cI_{Q,\eps}$, it is clear that the inequality~\eqref{eq:DeltaQ>C} holds
for $\l \in \bD_{3\eps}(\l_n^0) \setminus \bD_{\eps}(\l_n^0)$. The proof of the
estimate~\eqref{eq:Delta>=C.eps.hole} is now finished in the same way as
in the proof of Proposition~\ref{prop:incompress.uniform}(ii).
\end{proof}
\begin{proposition} \label{prop:ln-wtln<Delta}
Let $\cK$ be compact in $\LL{1}$ and $Q, \wt{Q} \in \cK$. Let boundary
conditions~\eqref{eq:BC} be strictly regular, and let $\L_Q =
\{\l_{Q,n}\}_{n \in \bZ}$ and $\L_{\wt{Q}} = \{\l_{\wt{Q},n}\}_{n \in \bZ}$
be canonically ordered sequences of zeros of characteristic determinants
$\Delta := \Delta_Q$ and $\wt{\Delta} := \Delta_{\wt{Q}}$ respectively. Then
there exists constants $N \in \bN$, $C \ge 1$ that do not depend on $Q$ and
$\wt{Q}$ and depend on $\cK$, $A$, $B$ only and such that the following uniform
estimate holds
\begin{equation} \label{eq:ln-wtln<C.Delta}
 C^{-1} \cdot |\Delta_{\wt{Q}}(\l_{Q,n})| \le |\l_{Q,n} - \l_{\wt{Q},n}| \le
 C \cdot |\Delta_{\wt{Q}}(\l_{Q,n})|, \qquad |n| > N, \quad Q, \wt{Q} \in \cK.
\end{equation}
\end{proposition}
\begin{proof}
For brevity set $\l_n := \l_{Q,n}$, $\wt{\l}_n := \l_{\wt{Q},n}$, and
$\wt{\Delta} := \Delta_{\wt{Q}}$, $\Delta := \Delta_{Q}$. Let $\L_0 =
\{\l_n^0\}_{n \in \bZ}$ be the sequence of zeros of the characteristic
determinant $\Delta_0$. Since boundary conditions~\eqref{eq:BC} are strictly
regular then the sequence $\L_0$ is asymptotically separated, i.e. the
estimate~\eqref{eq:ln0-lm0>2tau0} holds for some $\tau_0 > 0$ and $N_0 \in \bN$.
Applying Proposition~\ref{prop:incompress.uniform}(ii) with $\eps =
\min\{\eps_0, \tau_0/3\} = \eps(\cK, B, A)$ we see that the discs
$\bD_{2 \eps}(\wt{\l}_n)$, $|n| > N_{\eps}$, are disjoint, and
\begin{equation} \label{eq:wtln-ln0<eps}
 |\l_n - \l_n^0| < \eps, \quad
 |\wt{\l}_n - \l_n^0| < \eps, \qquad |n| > N_{\eps}.
\end{equation}
In particular, $\l_n \in \bD_{2 \eps}(\wt{\l}_n)$, $|n| > N_{\eps}$. Hence for
each $|n| > N_{\eps}$ the function $f(z) := \wt{\Delta}(z)/(z - \wt{\l}_n)$ is
non-zero holomorphic in $\bD_{2 \eps}(\wt{\l}_n)$ with $f(\wt{\l}_n) :=
\wt{\Delta}'(\wt{\l}_n) \ne 0$. If $\l_n = \wt{\l}_n$ then
relation~\eqref{eq:ln-wtln<C.Delta} is trivial as all parts are equal to zero.
If $\l_n \ne \wt{\l}_n$ then combining the Minimum Principle for analytic
functions with~\eqref{eq:Delta>=C.eps} we have
\begin{equation} \label{eq:Delta.wtln/ln-wtln>C}
 \frac{|\wt{\Delta}(\l_n)|}{|\l_n - \wt{\l}_n|} \ge
 \min_{|z - \wt{\l}_n| = 2 \eps}\frac{|\wt{\Delta}(z)|}{|z - \wt{\l}_n|} \ge
 \frac{C_{2 \eps}}{2 \eps}, \qquad |n| > N_{\eps}, \quad Q, \wt{Q} \in \cK.
\end{equation}
Relation~\eqref{eq:Delta.wtln/ln-wtln>C} now yields the second inequality
in~\eqref{eq:ln-wtln<C.Delta} with $C = \frac{2\eps}{C_{2\eps}}$ and $N =
N_{\eps}$.

On the other hand, by Proposition~\ref{prop:incompress.uniform}(i), $\L_Q
\subset \Pi_h$ with $h:= h(\cK, B, A)$, not dependent on $Q \in \cK$, i.e.
$|\Im \l_n| \le h$ for $n \in \bZ$ and $Q\in \cK$. Moreover,
Lemma~\ref{lem:Delta<Exp} (see estimate~\eqref{eq:Delta.l<C.exp}), ensures the
\emph{uniform} estimate $|\wt{\Delta}(\l)| = |\Delta_{\wt{Q}}(\l)| \le C_h$,
$\l \in \Pi_h$, for any ${\wt{Q}} \in \bU_{1, r}^{2 \times 2}$. Applying the
Maximum Principle similarly to~\eqref{eq:Delta.wtln/ln-wtln>C} yields the first
\emph{uniform} inequality in~\eqref{eq:ln-wtln<C.Delta}.
\end{proof}
Recall that notation $x_n \asymp y_n$ as $|n| \to \infty$, means that there
exists $N \in \bN$ and $C_2 > C_1 > 0$ such that two-sided estimate $C_1 |y_n|
\le |x_n| \le C_2 |y_n|$,\ \ $|n| > N$, holds.
\begin{corollary} \label{cor:lambda.vs.Delta}
Let $Q, \wt{Q} \in \LL{1}$ and let boundary conditions~\eqref{eq:BC} be strictly
regular. Let $\{\l_n\}_{n \in \bZ}$ and $\{\wt{\l}_n\}_{n \in \bZ}$ be
canonically ordered sequences of zeros of characteristic determinants $\Delta :=
\Delta_Q$ and $\wt{\Delta} := \Delta_{\wt{Q}}$ respectively. Then the following
estimates hold
\begin{equation} \label{eq:lambda.vs.wtDelta}
 |\l_n - \wt{\l}_n| \asymp
 |\wt{\Delta}(\l_n)|, \quad
 |\l_n - \l_n^0| \asymp
 |\Delta_0(\l_n)| \quad\text{as}\quad |n| \to \infty.
\end{equation}
\end{corollary}
\subsection{Stability property of eigenvalues for $Q \in L^p$}
Here we apply abstract results from Section~\ref{sec:fourier.transform} to
establish stability of the mapping $Q\to \L_Q := \{\l_{Q,n}\}_{n \in \bZ}$ in
different norms. Proposition~\ref{prop:ln-wtln<Delta} shows that to this
end it suffices to evaluate the sequences $\{\wt{\Delta}(\l_n)\}_{n \in \bZ} =
\{\Delta_{\wt{Q}}\bigl(\l_{Q,n}\bigr)\}_{n \in \bZ}$ when $Q$ runs through
either the ball $\bU_{p, r}^{2 \times 2}$ or a compact $\cK \subset L^1$. Our
main result in this direction reads as follows.
\begin{lemma} \label{lem:l_p_estimates_ln<Q-wtQ}
Let $Q, \wt{Q} \in \LL{1}^{2 \times 2}$, let boundary
conditions~\eqref{eq:BC} be regular, and let $\L_Q = \{\l_{Q,n}\}_{n \in \bZ}$
be a sequence of zeros of characteristic determinant $\Delta_{Q}$.

\textbf{(i)} Let $p \in (1, 2]$ and $r > 0$. Then there exists
$C = C(p, r, B) > 0$ such that the following inequalities hold:
\begin{align}
\label{eq:l-p.estimate_New}
 \sum_{n \in \bZ}\abs{\Delta_{\wt{Q}}\bigl(\l_{Q,n}\bigr)}^{p'}
 & \le C \cdot \|Q-\wt{Q}\|_p^{p'}, \qquad
 Q, \wt{Q} \in \bU_{p, r}^{2 \times 2}, \\
\label{eq:weight.ln-wtln<Q-wtQ.hole_New}
 \sum_{n \in \bZ} \(1+|n|\)^{p-2}
 \abs{\Delta_{\wt{Q}}\bigl(\l_{Q,n}\bigr)}^p
 & \le C \cdot \|Q - \wt{Q}\|_p^{p}\ ,
 \qquad Q, \wt{Q} \in \bU_{p, r}^{2 \times 2}.
\end{align}

\textbf{(ii)} Let $\cK \subset \LL{1}$ be a compact. Then the following
holds:
\begin{align}
\label{eq:sup.Delta.ln}
 \sup_{n \in \bZ} \abs{\Delta_{\wt{Q}}\bigl(\l_{Q,n}\bigr)}
 & \le C \cdot \|Q - \wt{Q}\|_1, \qquad
 Q, \wt{Q} \in \cK, \\
\label{eq:lim.Delta.ln}
 \sup_{Q, \wt{Q} \in \cK} \abs{\Delta_{\wt{Q}}\bigl(\l_{Q,n}\bigr)}
 & \to 0 \quad\text{as}\quad n \to \infty.
\end{align}
In other words, the set of sequences $\left\{
\{\Delta_{\wt{Q}}\bigl(\l_{Q,n}\bigr)\}_{n \in \bZ}\right\}_{Q, \wt{Q} \in \cK}$
\ forms a compact set in $c_0(\bZ)$.
\end{lemma}
\begin{proof}
In accordance with Lemma~\ref{lem:Delta-wt.Delta} the difference of determinants
admits representation~\eqref{eq:Delta-wt.Delta}, i.e.
\begin{equation} \label{eq:Delta.Q-Delta.wtQ}
 \Delta_Q(\l) - \Delta_{\wt{Q}}(\l) = \int_0^1 \wh{g}_1(t) e^{i b_1 \l t} \,dt
 + \int_0^1 \wh{g}_2(t) e^{i b_2 \l t} \,dt,
\end{equation}
where $\wh{g}_j = g_{Q,j} - g_{\wt{Q},j} \in L^p[0,1]$, $j\in \{1,2\}$,
and satisfies uniform estimate~\eqref{eq:whg1+whg2<C.whQ}.

Since $\Delta_Q(\cdot)$ is an entire sine-type function, its zero set $\L_Q =
\{\l_{Q,n}\}_{n \in \bZ}$ is incompressible (see
Proposition~\ref{prop:sine.type}(ii)). Moreover, in accordance with
Proposition~\ref{prop:incompress.uniform}(i), all null sequences
$\{\l_{Q,n}\}_{n \in \bZ}$, $Q\in \cK$, are incompressible of density $d =
d(\cK, B, A) \in \bN$ and lie in the strip $\Pi_h$ with $h = h(\cK, B, A) > 0$,
where neither $d$ nor $h$ depends on $Q\in \cK$.

Let us prove estimate~\eqref{eq:l-p.estimate_New}. Applying Bessel type
inequality~\eqref{eq:sum.int.g<g_Fourier} with $g = \wh{g}_j$, $j \in \{1,2\}$,
and $\mu_n =\l_n = \l_{{Q},n}$, $n \in \bZ$, and taking uniform (with respect to
$Q, \wt{Q} \in \bU_{p, r}^{2 \times 2}$) estimate~\eqref{eq:whg1+whg2<C.whQ}
into account, yields
\begin{align}
\nonumber
 \sum_{n \in \bZ} \abs{\Delta_{\wt{Q}}\bigl(\l_{n}\bigr)}^{p'}
 & = \sum_{n \in \bZ} \abs{\Delta_Q \bigl(\l_n\bigr) -
 \Delta_{\wt{Q}}\bigl(\l_n\bigr)}^{p'}
 \le 2^{p'-1} \sum_{n \in \bZ} \(\abs{\int_0^1 \wh{g}_1(t) e^{i b_1 \l_n t}
 \,dt}^{p'} + \abs{\int_0^1 \wh{g}_2(t) e^{i b_2 \l_n t}\,dt}^{p'}\) \\
\label{eq:Delta.Holder_for_balls}
 & \le 2^{p'-1} C_1
 \(\|\wh{g}_1\|_p^{p'} + \|\wh{g}_2\|_p^{p'}\)
 \le 2^{p'-1} C_1 \wh{C}^{p'} \|Q - \wt{Q}\|_p^{p'},
 \qquad Q, \wt{Q} \in \bU_{p, r}^{2 \times 2}.
\end{align}
Here the constants $C_1 = C_1(b, p, h, d) > 0$ and $\wh{C} = \wh{C}(b, p, h, d)
> 0$ are taken from the uniform estimates~\eqref{eq:sum.int.g<g_Fourier}
and~\eqref{eq:whg1+whg2<C.whQ}, respectively. Setting $C = 2^{p'-1} C_1
\wh{C}^{p'}$ we arrive at~\eqref{eq:l-p.estimate_New}.

Weighted estimate~\eqref{eq:weight.ln-wtln<Q-wtQ.hole_New} is proved similarly
but using inequality~\eqref{eq:sum.int.nu.g<g_Fourier}.
Estimate~\eqref{eq:sup.Delta.ln} follows from~\eqref{eq:Delta.Q-Delta.wtQ} and
the estimate
\begin{equation}
 \abs{\int_0^1 \wh{g}_j(t) e^{i b_j \l_n t}\,dt}
 \le e^{|b_j \Im \l_n|} \cdot \|\wh{g}_j\|_1
 \le \wh{C} e^{|b_j| h} \cdot \|Q - \wt{Q}\|_1,
 \qquad j \in \{1, 2\}, \quad \l \in \Pi_h,
 \quad Q, \wt{Q} \in \bU_{1,r}^{2 \times 2}.
\end{equation}
Further, according to~\eqref{eq:gl=J.R} and~\eqref{eq:Kjlk} functions
$\wh{g}_1(t)$ and $\wh{g}_2(t)$ are linear combinations of sixteen well-defined
summable trace functions $K_{Q,jk}^{\pm}(1,t)$ and $K_{\wt{Q},jk}^{\pm}(1,t)$,
$j, k \in \{1, 2\}$. Lemma~\ref{lem:fourier.coef.alter} now implies
relation~\eqref{eq:lim.Delta.ln} due to inclusion $\l_n \in \Pi_h$. The last
statement of the lemma follows from the well-known criteria of compactness in
$c_0(\bZ)$.
\end{proof}
Next we enhance and complete Proposition~\ref{prop:Delta.regular.basic} in the
case of $Q \in \LL{p}$ with $p \in [1, 2]$. Our first result restricts the set
$\cK$ of potentials matrices to be a compact.
\begin{theorem} \label{th:ln-wtln<Q-wtQ}
Let $\cK$ be compact in $\LL{p}$ for some $p \in [1, 2]$, and $Q, \wt{Q} \in
\cK$. Let boundary conditions~\eqref{eq:BC} be strictly regular, and let
$\L_Q := \{\l_{Q,n}\}_{n \in \bZ}$ and $\L_{\wt{Q}} :=
\{\l_{\wt{Q},n}\}_{n \in \bZ}$ be canonically ordered sequences of zeros of
characteristic determinants $\Delta(\cdot) := \Delta_Q(\cdot)$ and
$\wt{\Delta}(\cdot) := \Delta_{\wt{Q}}(\cdot)$, respectively. Then there exist
constants $N = N(\cK, A, B) \in \bN$ and $C = C(p, \cK, A, B) > 0$ that do not
depend on $Q$ and $\wt{Q}$ and such that the following estimates hold:
\begin{align}
\label{eq:sum.ln-wtln}
 \sum_{|n| > N} \abs{\l_{Q,n} - \l_{\wt{Q},n}}^{p'}
 & \le C \cdot \|Q - \wt{Q}\|_p^{p'}, \qquad Q, \wt{Q} \in \cK,
 \quad p \in (1, 2], \quad 1/p' + 1/p = 1, \\
\label{eq:weight.ln-wtln<Q-wtQ}
 \sum_{|n| > N} \(1+|n|\)^{p-2} \abs{\l_{Q,n} - \l_{\wt{Q},n}}^p
 & \le C \cdot \|Q - \wt{Q}\|_p^{p},
 \qquad Q, \wt{Q} \in \cK, \quad p \in (1, 2].
\end{align}
If $p=1$ then the following holds:
\begin{align}
\label{eq:sup.ln-wtln}
 \sup_{|n| > N} \abs{\l_{Q,n} - \l_{\wt{Q},n}}
 & \le C \cdot \|Q - \wt{Q}\|_1, \qquad Q, \wt{Q} \in \cK, \\
\label{eq:lim.ln-wtln}
 \sup_{Q, \wt{Q} \in \cK} \abs{\l_{Q,n} - \l_{\wt{Q},n}}
 & \to 0 \quad\text{as}\quad n \to \infty.
\end{align}
In other words, the set of sequences $\left\{\left\{\abs{\l_{Q,n} -
\l_{\wt{Q},n}}\right\}_{n \in \bZ}\right\}_{Q, \wt{Q} \in \cK}$ \ forms a
compact set in $c_0(\bZ)$.
\end{theorem}
\begin{proof}
Clearly, $\cK$ is compact in $\LL{1}$ as well, and $\cK \subset
\bU_{p, r}^{2 \times 2}$ for some $r = r(p, \cK) > 0$. First, note that
relation~\eqref{eq:lim.ln-wtln} was proved in
Proposition~\ref{prop:incompress.uniform}(ii). Further, in accordance with
Proposition~\ref{prop:ln-wtln<Delta} there exist constants $C, N > 0$, not
dependent on $Q$, $\wt{Q}$, such that the uniform
estimate~\eqref{eq:ln-wtln<C.Delta} holds, i.e.
\begin{equation}
 \bigabs{\l_{Q,n} - \l_{\wt{Q},n}} \le C \cdot
 \bigabs{\Delta_{\wt{Q}}(\l_{Q,n})}, \qquad |n| > N, \quad Q, \wt{Q} \in \cK.
\end{equation}
Combining this estimate with all the
statements of Lemma~\ref{lem:l_p_estimates_ln<Q-wtQ} finishes the proof.
\end{proof}
Applying Theorem~\ref{th:ln-wtln<Q-wtQ} with a two-point compact $\cK =
\{Q, 0\}$ we can complete Proposition~\ref{prop:Delta.regular.basic} as follows.
\begin{corollary} \label{prop:lambda.n.in.lp}
Let $Q \in \LL{p}$ for some $p \in (1, 2]$. Let boundary
conditions~\eqref{eq:BC} be strictly regular, and let $\Delta(\cdot)$ be the
corresponding characteristic determinant. Then the sequence $\L =
\{\l_n\}_{n \in \bZ}$ of its zeros can be ordered in such a way that the
following inequalities take place
\begin{align}
\label{eq:l.n-l.n0.in.lp}
 \sum_{n \in \bZ} \abs{\l_n - \l_n^0}^{p'} & < \infty, \qquad 1/p' + 1/p = 1, \\
\label{eq:l.n-l.n0.in_weight_lp}
 \sum_{n \in \bZ} (1+|n|)^{p-2} \abs{\l_n - \l_n^0}^p \, & < \infty.
\end{align}
\end{corollary}
Note also that relation~\eqref{eq:l.n-l.n0.in.lp} in the case of $2 \times 2$
Dirac system was obtained firstly in~\cite[Theorems 4.3, 4.5]{SavShk14}.
\begin{remark}
Here we will show that inequalities~\eqref{eq:sum.ln-wtln}
and~\eqref{eq:weight.ln-wtln<Q-wtQ} generally can not be derived from each
other. Let $p = 3/2$ and, thus, $p' = 3$. First we assume that
\begin{equation*}
 \a_n := \l_n - \wt{\l}_n =
 \((1 + |n|)\ln^2 (1 + |n|))\)^{-1/3}.
\end{equation*}
It is clear that
\begin{equation*}
 \{\a_n\}_{n \in \bZ} \in l^{3}(\bZ) \quad \text{while} \quad
 \left\{\(1 + |n|\)^{-1/3} \a_n\right\}_{n \in \bZ}
 \notin l^{3/2}(\bZ),
\end{equation*}
which shows that the inequality~\eqref{eq:sum.ln-wtln} holds
while~\eqref{eq:weight.ln-wtln<Q-wtQ} is not true. Now let
\begin{equation*}
 \a_n = \begin{cases}
 k^{-1/3}, & n = k^2 \ \ \text{for some}\ \ k \in \bN, \\
 0, & n \ne k^2.
 \end{cases}
\end{equation*}
In this case it is clear that the opposite relations hold,
\begin{equation*}
 \{\a_n\}_{n \in \bZ} \notin l^{3}(\bZ) \quad \text{while} \quad
 \left\{\(1 + |n|\)^{-1/3} \a_n\right\}_{n \in \bZ}
 \in l^{3/2}(\bZ),
\end{equation*}
Note also that under rather general condition $\a_n = o(n^{-1/p'})$ as
$n \to \infty$, inequality~\eqref{eq:sum.ln-wtln} does
imply~\eqref{eq:weight.ln-wtln<Q-wtQ}. Indeed, in this case $|\a_n|^{p'}
= \frac{\b_n}{1+|n|}$ and $\b_n<1$ for $|n| > N$. Then since $p' =
p/(p-1) > p$,
\begin{equation}
 (1+|n|)^{p-2} |\a_n|^p = \frac{\b_n^{p/p'}}{1+|n|}
 \le \frac{\b_n}{1+|n|} = |\a_n|^{p'}, \quad |n| > N.
\end{equation}
Hence, if $\{\a_n\}_{n \in \bZ} \in l^{p'}(\bZ)$ then $\left\{
\(1+|n|\)^{1-2/p} \a_n\right\}_{n \in \bZ} \in l^{p}(\bZ)$.
\end{remark}
Next we extend Theorem~\ref{th:ln-wtln<Q-wtQ} to the case $\cK =
\bU_{p,r}^{2 \times 2}$. Similarly to Proposition~\ref{prop:incompress.holes} we
cannot select a universal constant $N$ serving all potentials. Instead, we need
to sum over the sets of integers, the complements of which have uniformly
bounded cardinality.
\begin{theorem} \label{th:ln-wtln<Delta.Q-wtQ.Lp}
Let $Q, \wt{Q} \in \bU_{p, r}^{2 \times 2}$ for some $p \in (1, 2]$ and $r > 0$.
Let boundary conditions~\eqref{eq:BC} be strictly regular, and let $\L_Q =
\{\l_{Q,n}\}_{n \in \bZ}$ and $\L_{\wt{Q}} = \{\l_{\wt{Q},n}\}_{n \in \bZ}$ be
canonically ordered sequences of zeros of characteristic determinants $\Delta :=
\Delta_{Q}$ and $\wt{\Delta} := \Delta_{\wt{Q}}$, respectively. Then the
following holds:

\textbf{(i)} There exists constants $N \in \bN$, $C_1, C_2, C > 0$, not
dependent on $Q$, $\wt{Q}$, and a set $\cI := \cI_{Q, \wt{Q}} \subset \bZ$, such
that the following estimates hold
\begin{align}
\label{eq:card.IQ.wtQ}
 \card\(\bZ \setminus \cI_{Q, \wt{Q}}\) & \le N,
 \qquad Q, \wt{Q} \in \bU_{p, r}^{2 \times 2}, \\
\label{eq:ln-wtln<C.Delta.hole}
 C_1 \cdot \abs{\Delta_{\wt{Q}}\bigl(\l_{Q,n}\bigr)}
 \le |\l_{Q,n} - \l_{\wt{Q},n}|
 & \le C_2 \cdot \abs{\Delta_{\wt{Q}}\bigl(\l_{Q,n}\bigr)},
 \qquad n \in \cI_{Q, \wt{Q}},
 \quad Q, \wt{Q} \in \bU_{p, r}^{2 \times 2}, \\
\label{eq:ln-wtln<Q-wtQ.hole}
 \sum_{n \in \cI_{Q, \wt{Q}}} \abs{\l_{Q,n} - \l_{\wt{Q},n}}^{p'}
 & \le C \cdot \|Q - \wt{Q}\|_p^{p'},
 \qquad Q, \wt{Q} \in \bU_{p, r}^{2 \times 2}, \quad 1/p + 1/p' = 1, \\
\label{eq:weight.ln-wtln<Q-wtQ.hole}
 \sum_{n \in \cI_{Q, \wt{Q}}} \(1+|n|\)^{p-2}
 \abs{\l_{Q,n} - \l_{\wt{Q},n}}^p
 & \le C \cdot \|Q - \wt{Q}\|_p^{p},
 \qquad Q, \wt{Q} \in \bU_{p, r}^{2 \times 2}.
\end{align}

\textbf{(ii)} For any $\eps > 0$ there exist a set $\cI_{\eps} := \cI_{Q,
\wt{Q}, \eps} \subset \bZ$ and a constant $N_{\eps} = N_{\eps}(p, r, A, B) \in
\bN$ that does not depend on $Q$ and $\wt{Q}$, such that the following uniform
estimates hold
\begin{align}
\label{eq:card.IQ.wtQ.eps}
 \card\(\bZ \setminus \cI_{Q, \wt{Q}, \eps}\) & \le N_{\eps},
 \qquad Q, \wt{Q} \in \bU_{p, r}^{2 \times 2}, \\
\label{eq:ln-ln0<eps.whQ}
 \sup_{n \in \cI_{Q, \wt{Q}, \eps}}\abs{\l_{Q,n} - \l_{\wt{Q},n}} & \le
 \eps \|Q - \wt{Q}\|_{p}, \qquad Q, \wt{Q} \in \bU_{p, r}^{2 \times 2}.
\end{align}
\end{theorem}
\begin{proof}
As in the proof of Proposition~\ref{prop:ln-wtln<Delta} we set for
brevity $\l_n := \l_{Q,n}$, $\wt{\l}_n := \l_{\wt{Q},n}$, and $\wt{\Delta} :=
\Delta_{\wt{Q}}$, $\Delta := \Delta_{Q}$. Recall, that $\L_0 =
\{\l_n^0\}_{n \in \bZ}$ is a sequence of zeros of the characteristic determinant
$\Delta_0$.

\textbf{(i)} Applying Proposition~\ref{prop:incompress.holes}(iii) with $\eps =
\eps_0 = \eps(\cK, B, A)$ we see that the discs $\bD_{2 \eps}(\wt{\l}_n)$,
$n \in \cI_{\wt{Q}, \eps}$, are disjoint, and
\begin{equation} \label{eq:wtln-ln0<eps.holes}
 |\l_n - \l_n^0| < \eps, \quad
 |\wt{\l}_n - \l_n^0| < \eps, \qquad n \in \cI_{Q, \wt{Q}} :=
 \cI_{Q, \eps} \cap \cI_{\wt{Q}, \eps}.
\end{equation}
In particular, $\l_n \in \bD_{2 \eps}(\wt{\l}_n)$, $n \in \cI_{Q, \wt{Q}}$. The
proof of inequality~\eqref{eq:ln-wtln<C.Delta.hole} is finished in the same way
as in Proposition~\ref{prop:ln-wtln<Delta}. Further, note that
\begin{equation}
 \card\bigl(\bZ \setminus \cI_{Q, \wt{Q}}\bigr) \le
 \card\bigl(\bZ \setminus \cI_{Q, \eps}\bigr) +
 \card\bigl(\bZ \setminus \cI_{\wt{Q}, \eps}\bigr) \le
 2 N_{\eps} =: N(p, r, B, A) =: N,
\end{equation}
which proves~\eqref{eq:card.IQ.wtQ}. Combining
inequality~\eqref{eq:ln-wtln<C.Delta.hole} with
Lemma~\ref{lem:l_p_estimates_ln<Q-wtQ}(i) implies inequalities
\eqref{eq:ln-wtln<Q-wtQ.hole}--\eqref{eq:weight.ln-wtln<Q-wtQ.hole}.

\textbf{(ii)} If $Q = \wt{Q}$ then inequality~\eqref{eq:ln-ln0<eps.whQ} is
trivial. Assume that $Q \ne \wt{Q}$. We will apply Chebyshev's inequality as it
was done in Lemma~\ref{lem:max.fourier.tail.Lp}. Let $\eps > 0$ and let us set
\begin{equation} \label{eq:cI.Q.wtQ.eps.def}
 \cI_{Q, \wt{Q}, \eps} := \{n \in \bZ \ : \
 \abs{\l_{Q,n} - \l_{\wt{Q},n}} \le \eps \|Q - \wt{Q}\|_{p}\}
 \cap \cI_{Q, \wt{Q}}.
\end{equation}
It is clear that~\eqref{eq:ln-ln0<eps.whQ} follows
from~\eqref{eq:cI.Q.wtQ.eps.def}. Further,
inequality~\eqref{eq:ln-wtln<Q-wtQ.hole} implies
\begin{align}
\nonumber
 C \cdot \|Q - \wt{Q}\|_p^{p'}
 & \ge \sum_{n \in \cI_{Q, \wt{Q}}} \abs{\l_{Q,n} - \l_{\wt{Q},n}}^{p'}
 \ge \sum_{n \in \cI_{Q, \wt{Q}} \setminus \cI_{Q, \wt{Q}, \eps}}
 \abs{\l_{Q,n} - \l_{\wt{Q},n}}^{p'} \\
\label{eq:C.whQ>card}
 & \ge \sum_{n \in \cI_{Q, \wt{Q}} \setminus \cI_{Q, \wt{Q}, \eps}}
 \eps^{p'} \|Q - \wt{Q}\|_{p}^{p'}
 = \eps^{p'} \|Q - \wt{Q}\|_{p}^{p'}
 \card\(\cI_{Q, \wt{Q}} \setminus \cI_{Q, \wt{Q}, \eps}\).
\end{align}
Since $Q \ne \wt{Q}$ and $\cI_{Q, \wt{Q}, \eps} \subset \cI_{Q, \wt{Q}}$,
inequalities~\eqref{eq:C.whQ>card} and~\eqref{eq:card.IQ.wtQ} imply
\begin{equation}
 \card\(\bZ \setminus \cI_{Q, \wt{Q}, \eps}\) =
 \card\(\cI_{Q, \wt{Q}} \setminus \cI_{Q, \wt{Q}, \eps}\) +
 \card\(\bZ \setminus \cI_{Q, \wt{Q}}\) \le
 C \eps^{-p'} + N =: N_{\eps},
\end{equation}
which proves~\eqref{eq:card.IQ.wtQ.eps}.
\end{proof}
Building up on the example in Remark~\ref{rem:fourier.coef.Lp}(i) we will show
that Proposition~\ref{prop:incompress.uniform}(ii) is not valid for balls in
$L^p$ and demonstrate significance of introducing subsets $\cI_{Q,\eps}$ in
Proposition~\ref{prop:incompress.holes}(ii). We also show that the constant $C$
in~\eqref{eq:sup.ln-wtln} can not be arbitrary small in the case of compacts in
$L^1$.
\begin{proposition} \label{prop:eigenv.Q12=0}
Let $Q_{12} = 0$. Let boundary conditions~\eqref{eq:BC.new} be regular (and thus
$\{\l_n^0\}_{n \in \bZ} \subset \Pi_h$ for some $h \ge 0$) and let
$\L_Q = \{\l_{Q,n}\}_{n \in \bZ}$ be canonically ordered sequence of zeros of
the characteristic determinant $\Delta := \Delta_{Q}$. Let also
\begin{equation} \label{eq:Kgmux}
 \cK = \left\{G_{\mu} := \begin{pmatrix} 0 & 0 \\ g_{\mu} & 0 \end{pmatrix} :
 \mu \in \Pi_h \right\},
 \qquad g_{\mu}(x) = g_0(x) e^{-i \mu (b_1 - b_2) x},
 \quad x \in [0,1], \quad \mu \in \Pi_h,
\end{equation}
where $g_0 \in \bU_{p,r}$, $c_0 := \int_0^1 g_0(t) dt \ne 0$, for some $p \ge 1$
and $r > 0$.

\textbf{(i)} Let $b \ne 0$. Then relation~\eqref{eq:ln-ln0<eps} is not valid for $\cK$.
More precisely, there exists $\eps_0 > 0$ such that
\begin{equation} \label{eq:lQnn}
|\l_{Q_n,m} - \l_n^0| \ge \eps_0, \quad m,n \in \bZ, \quad\text{where}\quad
Q_n = G_{\l_n^0}.
\end{equation}

\textbf{(ii)} Let again $b \ne 0$. Then for $Q = G_{\mu} \in \cK$
relations~\eqref{eq:card.IQ}--\eqref{eq:ln-ln0<eps.hole} are valid with
$\cI_{Q, \eps} = \{n_\mu + 1, \ldots, n_\mu + N_{\eps}\}$, where $n_{\mu} \in \bZ$ and $N_{\eps} = N_{\eps}(g_0, B, A) \in \bN$ does not depend on $\mu$.

\textbf{(iii)} Let $b=0$. Then $\l_{Q,n} = \l_n^0$ for all $n \in \bZ$ and
$Q \in \LL{1}$ with $Q_{12} = 0$.

\textbf{(iv)} Let $\cX = \left\{G_{\mu}/\mu :
\mu \in \Pi_h, |\mu| \ge 1\right\} \cup \{0\}$. Then $\cX$ is compact in
$\LL{p}$. Let $b \ne 0$. Then there exists $N_0, \eps_0 > 0$ such that
\begin{equation} \label{eq:lQnn-ln0>eps.Qn}
 |\l_{Q_n,n} - \l_n^0| \ge \eps_0 \|Q_n\|_p, \quad |n| > N_0,
 \quad\text{where}\quad Q_n = G_{\l_n^0}/\l_n^0 \in \cX.
\end{equation}
\end{proposition}
\begin{proof}
Since $Q_{12} = 0$, then explicit formula~\eqref{eq:Phi.Q12=0} holds. Inserting it
into~\eqref{eq:Delta.new} we arrive at
\begin{align} \label{eq:Delta.a=d=0}
 \Delta_Q(\l) = \Delta_0(\l)
 - i b_2 b e^{i b_2 \l} \int_0^1 Q_{21}(t) e^{i (b_1 - b_2) \l t} dt.
\end{align}
Hence $\Delta_Q(\l) = 0$ is equivalent to
\begin{equation} \label{eq:Delta0=Jl}
 F_0(\l) = J(\l, Q), \quad\text{where}\quad
 F_0(\l) := e^{-i b_2 \l} \Delta_0(\l) \quad\text{and}\quad
 J(\l, Q) := i b_2 b \int_0^1 Q_{21}(t) e^{i (b_1 - b_2) \l t} dt.
\end{equation}

\textbf{(i)} Since boundary conditions~\eqref{eq:BC.new} are regular,
Proposition~\ref{prop:sine.type} implies that $\{\l_n^0\}_{n \in \bZ}$ is an
incompressible sequence lying in the strip $\Pi_h$. For $n \in \bZ$ set
$Q_n := G_{\l_n^0}$, i.e. $Q_{21} = Q_{n,21} = g_{\l_n^0}$. Since $b c_0 \ne 0$,
definition~\eqref{eq:Kgmux} of $g_{\mu}$ yields
\begin{equation} \label{eq:Jl=int.mu}
 J(\l, Q_n)
 = i b_2 b \int_0^1 g_0(t) e^{i (b_1 - b_2) (\l - \l_n^0) t} dt,
 \qquad J(\l_n^0, Q_n) = i b_2 b c_0 =: \a \ne 0,
 \qquad n \in \bZ.
\end{equation}
Set $\delta := |\a|/2 > 0$. Since $F_0'(\l)$ is uniformly bounded in the strip
$\Pi_{h+1}$, i.e. $|F_0'(\l)| \le M_0$, $\l \in \Pi_{h+1}$, then
\begin{equation} \label{eq:wt.D0<d}
 |F_0(\l)| \le |F_0(\l_n^0)| + M_0 |\l - \l_n^0| = M_0 |\l - \l_n^0| < \delta,
 \qquad |\l - \l_n^0| < \eps_1, \qquad n \in \bZ.
\end{equation}
with some $\eps_1 \in (0,1)$. At the same time since $I(\mu) := \int_0^1 g_0(t)
e^{i (b_1 - b_2) \mu t} dt$ is continuous at $\mu$, it follows
from~\eqref{eq:Jl=int.mu} that for some $\eps_2 > 0$ we have
\begin{equation} \label{eq:Jl>d}
 |J(\l, Q_n) - \a| < \delta, \qquad |\l - \l_n^0| < \eps_2,
 \qquad n \in \bZ.
\end{equation}
Setting $\eps_0 := \min\{\eps_1, \eps_2\}$, taking into account that $2 \delta =
|\a|$ and combining~\eqref{eq:wt.D0<d}--\eqref{eq:Jl>d} we arrive at
\begin{equation} \label{eq:wtD0<Jl}
 |F_0(\l)| < |J(\l, Q_n)|, \qquad |\l - \l_n^0| < \eps_0, \qquad n \in \bZ.
\end{equation}
Relations~\eqref{eq:Delta.a=d=0},~\eqref{eq:Delta0=Jl} and~\eqref{eq:wtD0<Jl}
now imply that for $n \in \bZ$ determinant $\Delta_{Q_n}(\l)$ has no zeros in
$\bD_{\eps_0}(\l_n^0)$ which implies desired inequality~\eqref{eq:lQnn}.

\textbf{(ii)} Let $\eps > 0$, $\mu \in \bR$ and $Q = G_{\mu}$. For $M>0$ we set
$\cI_{\mu, M} := \{n \in \bZ : |\Re \l_n^0 - \mu| \le M\}$. Similarly to the
proof of Proposition~\ref{prop:lambdan.ordering} we can apply Riemann-Lebesgue
lemma and Rouch\'e theorem to the relation~\eqref{eq:Delta0=Jl} with $Q_{21} =
g_{\mu}$. This yields relations~\eqref{eq:card.IQ}--\eqref{eq:ln-ln0<eps.hole}
with $\cI_{Q, \eps} = \cI_{\mu, M_{\eps}}$ for some $M_{\eps} =
M_{\eps}(g_0, B, A) > 0$ that does not depend on $\mu$. Since
$\{\Re \l_n^0\}_{n \in \bZ}$ is an incompressible non-decreasing sequence, it is
clear that $\cI_{\mu, M_\eps} \subset \{n_\mu + 1, \ldots, n_\mu + N_{\eps}\}$,
for some $n_{\mu} \in \bZ$ and $N_{\eps} = N_{\eps}(g_0, B, A) \in \bN$ that
does not depend on $\mu$, which finishes the proof.

\textbf{(iii)} Since $b=0$, it follows from~\eqref{eq:Delta.a=d=0} that
$\Delta_Q(\l) = \Delta_0(\l)$, $\l \in \bC$, whenever $Q_{12} = 0$. This implies
that spectra of the operators $L(Q)$ and $L(0)$ coincide and finishes the proof.

\textbf{(iv)} First let us show that $\cX$ is compact. Let $\eps > 0$ and let us
build a finite $\eps$-net for $\cX$. Set $R = \|g_0\|_p e^{(b_2-b_1)h} / \eps$.
It is clear that $\cX_R = \{G_{\mu}/\mu : \mu \in \Pi_h, 1 \le |\mu| \le R\}$ is
a compact in $\LL{p}$ since $G_{\mu}/\mu$ is a continuous function of $\mu$ to
$L^p$ and the set of $\mu$ in the definition of $\cX_R$ is a compact in $\bC$.
On the other hand definition~\eqref{eq:Kgmux} of $G_{\mu}$ implies that
$\|G_{\mu} / \mu\|_p \le \|g_0\|_p e^{(b_2-b_1)h} / |\mu| < \eps$, for
$|\mu| > R$. Hence adding zero to the $\eps$-net of $\cX_R$ gives the $\eps$-net
for $\cX$.

Proposition~\ref{prop:ln-wtln<Delta} implies that $|\l_n - \l_n^0| \ge C^{-1}
|\Delta_Q(\l_n^0)|$, $|n| > N$, for some $C, N > 0$, not dependent on $Q$.
Let $|n| > N$ and $Q := Q_n = G_{\l_n^0}/\l_n^0$. Definition~\eqref{eq:Kgmux} of
$G_{\mu}$ and $c_0$, and relation~\eqref{eq:Delta.a=d=0} imply
\begin{equation} \label{eq:ln-ln0>C}
 |\l_n - \l_n^0| \ge C^{-1} \abs{\Delta_{Q_n}(\l_n^0) - \Delta_0(\l_n^0)}
 = C^{-1} \abs{b_2 b e^{i b_2 \l_n^0}} \cdot \abs{\int_0^1
 \frac{g_0(t) e^{-i (b_1 - b_2) \l_n^0 t}}{\l_n^0}
 e^{i (b_1 - b_2) \l_n^0 t} dt}
 = C^{-1} \abs{b_2 b e^{i b_2 \l_n^0} c_0}/|\l_n^0|.
\end{equation}
On the other hand $\|Q_n\|_p \le \|g_0\|_p e^{(b_2-b_1)h} / |\l_n^0|$ and
$\abs{e^{i b_2 \l_n^0}} \ge e^{-b_2 h}$. Hence setting $$\eps_0 = C^{-1}
\abs{b_2 b c_0} \|g_0\|_p^{-1} e^{(b_1-2b_2)h} > 0$$ and
combining~\eqref{eq:ln-ln0>C} with the estimate on $\|Q_n\|_p$, we arrive at the
desired estimate~\eqref{eq:lQnn-ln0>eps.Qn}.
\end{proof}
\begin{remark} \label{rem:history}
\textbf{(i)} Emphasize that the proofs of all results in this Section including
the proofs of Theorems~\ref{th:ln-wtln<Q-wtQ}
and~\ref{th:ln-wtln<Delta.Q-wtQ.Lp} rely on Bessel type
inequalities~\eqref{eq:sum.int.g<g_Fourier},~\eqref{eq:sum.int.nu.g<g_Fourier}
(see Proposition~\ref{th:p.bessel_for_ordinary_Fourier}) for ordinary Fourier
transform, not for its maximal version described in Theorem~\ref{th:p.bessel},
which proof relies on the deep Carleson-Hunt result~\ref{th:Carleson-Hunt}.

\textbf{(ii)} Theorem~\ref{th:ln-wtln<Q-wtQ} remains valid if $\cK$ is compact
in $\LL{1}$ and bounded in $\LL{p}$ which is slightly wider class of sets than
compacts in $\LL{p}$.

\textbf{(iii)} The case of Dirac system $(b_1=-b_2=1)$ and $\wt{Q} = 0$ has been
extensively studied in multiple recent papers by Sadovnichaya, Savchuk and
Shkalikov. In particular, estimate~\eqref{eq:sum.ln-wtln} was established
earlier in~\cite[Theorems 4.3, 4.5]{SavShk14} with the constant $C$ that depends
on $Q$, while estimate~\eqref{eq:ln-wtln<Q-wtQ.hole} of
Theorem~\ref{th:ln-wtln<Q-wtQ} with $\wt{Q}=0$ was established
in~\cite{SavSad18}.

Weighted estimates~\eqref{eq:weight.ln-wtln<Q-wtQ}
and~\eqref{eq:weight.ln-wtln<Q-wtQ.hole}, as well as
estimate~\eqref{eq:sum.ln-wtln}, which establish stability property of the
spectrum under perturbation $Q \to \wt{Q}$, are new even for Dirac system.

\textbf{(iv)} Emphasize that the uniform
estimates~\eqref{eq:ln-ln0<M}--\eqref{eq:ln-ln0<eps} in
Proposition~\ref{prop:incompress.uniform} are valid in the case of regular BC,
which generalizes Theorem 3 in~\cite{Sad16} even in the case of Dirac system
($b_1 = -b_2 = 1$).

\textbf{(v)} In a very recent preprint~\cite{Rzep20} L. Rzepnicki obtained sharp
asymptotic formulas for deviations $\l_n - \l_n^0 = \delta_n + \rho_n$ in the
case of Dirichlet BVP for Dirac system with $Q \in \LL{p}$, $1 \le p < 2$.
Namely, $\delta_n$ is explicitly expressed via Fourier coefficients and Fourier
transforms of $Q_{12}$ and $Q_{21}$, while $\{\rho_n\}_{n \in \bZ} \in
l^{p'/2}(\bZ)$, i.e. has ``twice'' better convergence to zero than what
formula~\eqref{eq:sum.ln-wtln} guarantees for $\l_n - \l_n^0$. Similar
result was obtained for eigenfunctions.

\textbf{(vi)} We mention also the papers~\cite{CGHL04},~\cite{ClaGes06}, and
\cite{BroKlMalMogWo19} where different spectral properties of $j$-selfadjoint
Dirac operators were investigated.
\end{remark}
\section{Stability property of eigenfunctions}
\label{sec:eigenfunction.stabil}
Throughout the section we will use the following notation for the
``maximal'' Fourier transform of the potential matrix $Q$. Namely, let us set
for $x \in [0,1]$, $\l \in \bC$ and $k \in \{1, 2\}$,
\begin{equation} \label{eq:def.sFk}
 \sF_k(x,\l) := \sF_k[Q](x,\l) := \sup_{s \in [0,x]}
 \abs{\int_0^s Q_{jk}(t) e^{i (b_k - b_j) \l t} \,dt}, \qquad j = 2/k.
\end{equation}
Note that this notation is generally valid for any matrix-function $W \in
\LL{1}$.
\subsection{Estimates of Fourier transforms of transformation operators}
\label{subsec:maxim.Fourier}
In this subsection we study ``Fourier'' transforms of the kernels of the
corresponding transformation operators from representation~\eqref{eq:e=(I+K)e0}
of the form $\int_0^x K^{\pm}_{jk}(x, t) e^{i \l b_k t} dt$. The motivation
comes from the formula~\eqref{eq:phi.jkx} for the entries of the fundamental
matrix solution of the system~\eqref{eq:system} where these integrals appear.
We will estimate these integrals with the ``maximal'' Fourier transforms
$\sF_1(x, \l)$, $\sF_2(x, \l)$ of the potential matrix.

In what follows we heavily use notation~\eqref{eq:def.ak.alphak}, in particular,
$a_k = b_k^{-1}$, $k \in \{1, 2\}$. Note also that
\begin{equation} \label{eq:alphaj}
 \a_j = a_j/(a_j-a_k), \quad b_k \a_j^{-1} = b_k - b_j,
 \quad b_k \a_j^{-1} \a_k = -b_j, \qquad k \in \{1,2\}, \quad j = 2/k.
\end{equation}

As a first step we study ``Fourier'' transforms of the auxiliary kernels $R$
from the representation~\eqref{2.51op} for the kernels of the transformation
operators $K^{\pm}$. The first auxiliary result estimates generalized
``Fourier'' transforms with an arbitrary bounded function $f$ instead of the
exponential function in the integral.
\begin{lemma} \label{lem:int.Rjk.f}
Let $Q \in \LL{1}$, and let $R = (R_{jk})_{j,k=1}^2 \in
\(X_{1,1}^0(\Omega) \cap X_{\infty,1}^0(\Omega)\) \otimes \bC^{2 \times 2}$ be a
(unique) solution of the system of integral
equations~\eqref{eq:Rkk=int.Qkj.Rjk}--\eqref{eq:Rjk=Qjk-int.Qjk.Rkk}. Let
$x \in [0,1]$ be fixed and let $f \in L^{\infty}(\bR)$ be such that $f(t) = 0$
for $t \notin [0,x]$. Let us set
\begin{equation} \label{eq:def.Fjk}
 F_{jk}(s; f) := \sup_{\genfrac{}{}{0pt}{2}{u \in [0,s]}{v \in [-u,x]}}
 \abs{\int_0^u R_{jk}(s, t) f(t+v) \,dt},
 \qquad s \in [0,x], \quad j, k \in \{1, 2\}.
\end{equation}
Then the following estimates hold for $s \in [0,x]$, $k \in \{1, 2\}$,
$j = 2/k$:
\begin{align}
\label{eq:Fkk<int.Fjk}
 & F_{kk}(s; f) \le |b_k| \int_0^s |Q_{kj}(t)| \cdot F_{jk}(t; f)
 \,dt, \\
\label{eq:Fjk<int.Fjk}
 & F_{jk}(s; f) \le |b_j| \sup_{\genfrac{}{}{0pt}{2}{u \in [0,s]}{v \in [-u,x]}}
 \abs{\a_j \int_0^u Q_{jk}(\a_k s + \a_j t) \cdot f(t+v) \,dt } + 2 |b_j b_k|
 \cdot \|Q_{jk}\|_{L^1[0,s]} \cdot
 \int_0^s |Q_{kj}(t)| \cdot F_{jk}(t; f) \,dt.
\end{align}
\end{lemma}
\begin{proof}
Let us set
\begin{equation} \label{eq:def.Jjk}
 J_{jk}(s, u, v) := \int_0^u R_{jk}(s, t) f(t+v) \,dt,
 \qquad 0 \le u \le s \le x, \quad v \in \bR, \quad j, k \in \{1,2\}.
\end{equation}
Note that since $f(t) = 0$, $t \notin [0,x]$, then $J_{jk}(s, u, v) = 0$ for
$0 \le u \le s \le x$ and $v \notin [-u, x]$. Hence
definition~\eqref{eq:def.Fjk} of $F_{jk}(s; f)$ implies that
\begin{equation} \label{eq:Jjk<Fjk}
 |J_{jk}(s, u, v)| \le F_{jk}(s; f), \qquad 0 \le u \le s \le x,
 \quad v \in \bR, \quad j, k \in \{1,2\}.
\end{equation}
Let $k \in \{1, 2\}$ and $j = 2/k$. It follows from~\eqref{eq:Rkk=int.Qkj.Rjk}
that
\begin{align}
\nonumber
 J_{kk}(s, u, v) &= \int_0^u R_{kk}(s, t) f(t+v) dt \\
\nonumber
 & = -\frac{i}{a_k} \int_0^u f(t+v) \,dt \int_{s-t}^s Q_{kj}(\xi)R_{jk}
 \bigl(\xi, \xi - s + t\bigr) \,d\xi \\
\nonumber
 &= -i b_k \int_{s-u}^s Q_{kj}(\xi) \,d\xi \int_{s-\xi}^{u} R_{jk}
 \bigl(\xi, \xi - s + t\bigr) f(t+v) \,dt \\
\nonumber
 &= -i b_k \int_{s-u}^s Q_{kj}(\xi) \,d\xi \int_{0}^{u+\xi-s} R_{jk}
 \bigl(\xi, \eta) f(\eta + s - \xi + v) \,d\eta \\
\label{eq:Jkk=int.Qkj.Jjk}
 &= -i b_k \int_{s-u}^s Q_{kj}(\xi)
 J_{jk}(\xi, u+\xi-s, v+s-\xi) \,d\xi, \qquad 0 \le u \le s \le x,
 \quad v \in \bR.
\end{align}
Changing notation and taking into account~\eqref{eq:Jjk<Fjk},
relation~\eqref{eq:Jkk=int.Qkj.Jjk} yields that
\begin{equation} \label{eq:Jkk<C1.int.Q.Jjk}
 |J_{kk}(\xi, u, v)| \le |b_k| \int_0^{\xi} |Q_{kj}(t)| \cdot
 F_{jk}(t; f) \,dt, \quad 0 \le u \le \xi \le x, \quad v \in \bR,
\end{equation}
Taking supremum over $v \in [-u, x]$ and $u \in [0, \xi]$
in~\eqref{eq:Jkk<C1.int.Q.Jjk} now yields~\eqref{eq:Fkk<int.Fjk}.

Taking into account notation~\eqref{eq:def.ak.alphak},
relation~\eqref{eq:Rjk=Qjk-int.Qjk.Rkk} yields for $0 \le u \le s \le x$ and
$v \in \bR$ that
\begin{align}
\label{eq:Jjk=J1+J2}
 J_{jk}(s, u, v) &= \int_0^u R_{jk}(s, t) f(t+v) \,dt
 = -\frac{i}{a_j} \cdot \(J_{1,jk}(s, u, v) + J_{2,jk}(s, u, v)\), \\
\label{eq:J1jk.def}
 J_{1,jk}(s, u, v) &:= \frac{-a_j}{a_k- a_j} \int_0^u Q_{jk}
 \(\frac{a_k s - a_j t}{a_k - a_j}\) f(t+v) \,dt
 = \a_j \int_0^u Q_{jk}(\a_k s + \a_j t) \cdot f(t+v) \,dt, \\
\label{eq:J2jk.def}
 J_{2,jk}(s, u, v) &:= \int_0^u f(t+v) \,dt
 \int_{\a_k s + \a_j t}^s Q_{jk}(\xi) \,
 R_{kk}\bigl(\xi,v_{\xi}+t\bigr) \,d\xi, \quad v_{\xi} := \g_k (\xi-s).
\end{align}
Changing order of integration in~\eqref{eq:J2jk.def} and then making the change
of variable $\eta = t + v_{\xi}$ we arrive at
\begin{align}
 J_{2,jk}(s, u, v) &= \int_{\a_k s}^s Q_{jk}(\xi) \,d\xi
 \int_0^{\min\{u,\, \xi - v_{\xi}\}}
 R_{kk}\bigl(\xi,v_{\xi}+t\bigr) \cdot f(t+v) dt \nonumber \\
\label{eq:J2=int.Q.int.Rkk}
 &= \int_{\a_k s}^s Q_{jk}(\xi) \,d\xi
 \int_{v_{\xi}}^{\min\{u + v_{\xi},\, \xi\}}
 R_{kk}(\xi, \eta) \cdot f(\eta - v_{\xi} + v) \,d\eta.
\end{align}
Combining~\eqref{eq:def.Jjk},~\eqref{eq:def.Fjk},~\eqref{eq:J2=int.Q.int.Rkk}
and~\eqref{eq:Jkk<C1.int.Q.Jjk} we get:
\begin{align}
\nonumber
 |J_{2,jk}(s, u, v)| &\le \int_{\a_k s}^s |Q_{jk}(\xi)|
 \(|J_{kk}\bigl(\xi, v_{\xi}, v - v_{\xi})| +
 |J_{kk}\bigl(\xi, \min\{u + v_{\xi}, \xi\},
 v - v_{\xi})|\) \,d\xi \\
\nonumber
 & \le 2 |b_k| \int_{0}^s |Q_{jk}(\xi)| \,d\xi
 \int_{0}^{\xi} |Q_{kj}(t)| \cdot F_{jk}(t; f) \,dt \\
\nonumber
 & = 2 |b_k| \int_{0}^s |Q_{kj}(t)| \cdot F_{jk}(t; f) \,dt
 \int_t^{s} |Q_{jk}(\xi)| \,d\xi \\
\label{eq:J2<int.Q.int.Q.Jjk}
 & \le 2 |b_k| \cdot \|Q_{jk}\|_{L^1[0,s]}
 \int_{0}^s |Q_{kj}(t)| \cdot F_{jk}(t; f) \,dt,
 \qquad 0 \le u \le s \le x, \quad v \in \bR.
\end{align}
Inserting~\eqref{eq:J1jk.def} and~\eqref{eq:J2<int.Q.int.Q.Jjk}
into~\eqref{eq:Jjk=J1+J2} we arrive at the following inequality for
$0 \le u \le s \le x$, $v \in \bR$
\begin{equation} \label{eq:Jjk<int.Fjk}
 |J_{jk}(s, u, v)|
 \le \abs{b_j \a_j \int_0^u Q_{jk}(\a_k s + \a_j t) \cdot f(t+v) \,dt }
 + 2 |b_j b_k| \cdot \|Q_{jk}\|_{L^1[0,s]}
 \int_0^s |Q_{kj}(t)| \cdot F_{jk}(t; f) \,dt.
\end{equation}
Taking supremum over $v \in [-u,x]$ and $u \in [0,s]$ in~\eqref{eq:Jjk<int.Fjk}
yields~\eqref{eq:Fjk<int.Fjk}.
\end{proof}
\begin{corollary} \label{cor:int.R<int.Q}
Let $Q \in \bU_{1,r}^{2 \times 2}$ for some $r > 0$, and let $R =
(R_{jk})_{j,k=1}^2 \in \(X_{1,1}^0(\Omega) \cap X_{\infty,1}^0(\Omega)\) \otimes
\bC^{2 \times 2}$ be a (unique) solution of the system of integral
equations~\eqref{eq:Rkk=int.Qkj.Rjk}--\eqref{eq:Rjk=Qjk-int.Qjk.Rkk}. Then the
following uniform estimate holds for $k, l \in \{1,2\}$.
\begin{equation} \label{eq:int.R<int.Q}
 \sup_{s \in [0,x]} \abs{\int_0^s R_{lk}(x, t) e^{i \l b_k t} \,dt}
 \le C e^{(b_2 - b_1) |\Im \l| x} \cdot \sF_k(x,\l),
 \quad x \in [0,1], \quad \l \in \bC,
\end{equation}
where $C = C(B, r) > 0$ does not depend on $Q$, $x$ and $\l$, and
$\sF_k(x,\l)$ is defined in~\eqref{eq:def.sFk}.
\end{corollary}
\begin{proof}
Let $\l \in \bC$, $x \in [0,1]$, $k \in \{1,2\}$ and $j = 2/k$ be fixed for the
rest of the proof. We will apply Lemma~\ref{lem:int.Rjk.f} with $f(t) =
e^{i b_k \l t}$, $t \in [0,x]$, and $f(t) = 0$, $t \notin [0,x]$. Let us
estimate the first summand in the r.h.s. of~\eqref{eq:Fjk<int.Fjk} for such $f$.
To this end, for $0 \le u \le s \le x$ and $v \in [-u, x]$ we have
\begin{align}
\nonumber
 \abs{\a_j \int_0^u Q_{jk}(\a_k s + \a_j t) \cdot f(t+v) \,dt }
 & = \abs{\a_j \int_{\max\{0,-v\}}^{\min\{u,x-v\}} Q_{jk}(\a_j t + \a_k s)
 e^{i b_k \l (t+v)} \,dt} \\
\nonumber
 & = \abs{\int_{\a_j \max\{0,-v\} + \a_k s}^{\a_j \min\{u,x-v\} + \a_k s}
 Q_{jk}(\xi) \exp\(i b_k \l (\a_j^{-1} \xi - \a_j^{-1} \a_k s + v)\)
 \,d\xi} \\
\nonumber
 & \le 2 \abs{\exp\(i \l (b_j s + b_k v)\)} \sup_{t \in [0,s]}
 \abs{\int_0^t Q_{jk}(\xi) \exp\(i (b_k - b_j) \l \xi \) \,d\xi} \\
\label{eq:int.Qjkftv}
 & \le 2 e^{|(b_j - b_k) \Im \l| x} \cdot \sF_k(x,\l) =:
 2 \sR_k(x,\l).
\end{align}
Here we used the change of variable $\xi = \a_j t + \a_k s$, iden titties
$b_k \a_j^{-1} \xi = (b_k-b_j) \xi$ and $-b_k \a_j^{-1} \a_k s = b_j s$,
and definition~\eqref{eq:def.sFk} of $\sF_k(x,\l)$.
Inserting~\eqref{eq:int.Qjkftv} into~\eqref{eq:Fjk<int.Fjk} now yields
\begin{equation} \label{eq:Fjk<sE+int.Fjk}
 F_{jk}(s; f) \le |b_j| \cdot \sR_k(x,\l)
 + 2 |b_j b_k| r \int_0^s |Q_{kj}(t)| \cdot F_{jk}(t; f) \,dt,
 \qquad s \in [0,x].
\end{equation}
Applying Gr\"onwall's inequality to~\eqref{eq:Fjk<sE+int.Fjk} and taking into
account that $\int_0^s |Q_{kj}(u)| \,du \le \|Q_{kj}\|_{L^1[0,1]} \le r$,
implies that
\begin{align}
\nonumber
 \abs{\int_0^u R_{jk}(s, t) e^{i \l b_k t} \,dt}
 = \abs{J_{jk}(s,u,0)} & \le F_{jk}(s; f) \le |b_j| \cdot \sR_k(x,\l)
 \cdot \exp\(2 |b_j b_k| r \int_0^s |Q_{kj}(u)| \,du\), \\
\label{eq:int0u.Rjk<E}
 & \le C_j \cdot \sR_k(x,\l),
 \qquad 0 \le u \le s \le x,
\end{align}
where $C_j := |b_j| \exp(2 |b_1 b_2| r^2)$. Setting $s=x$ and taking
supremum over $u \in [0,x]$ in~\eqref{eq:int0u.Rjk<E}
implies~\eqref{eq:int.R<int.Q} for $l \ne k$.

Inserting the estimate $F_{jk}(t; f) \le C_j \cdot \sR_k(x,\l)$,
$t \in [0, x]$, from~\eqref{eq:int0u.Rjk<E} into~\eqref{eq:Fkk<int.Fjk} yields
that
\begin{align}
\nonumber
 \abs{\int_0^u R_{kk}(s, t) e^{i \l b_k t} \,dt}
 & = \abs{J_{kk}(s,u,0)} \le F_{kk}(s; f)
 \le |b_k| \cdot C_j \int_0^s |Q_{kj}(t)| \cdot \sR_k(x,\l) \,dt \\
\label{eq:Fkk<C.Ek}
 & = |b_k| \cdot C_j \cdot \|Q_{kj}\|_{L^1[0,s]} \cdot \sR_k(x,\l)
 \le C \cdot \sR_k(x,\l), \qquad 0 \le u \le s \le x,
\end{align}
where $C = |b_1 b_2| r \cdot \exp(2 |b_1 b_2| r^2)$. Setting $s=x$ and taking
supremum over $u \in [0,x]$ in~\eqref{eq:Fkk<C.Ek}
implies~\eqref{eq:int.R<int.Q} for $l = k$ and completes the proof.
\end{proof}
We also need to estimate Fourier transforms of the auxiliary functions
$P_k^{\pm}$ from the representation~\eqref{eq:P.R.in.X0}--\eqref{2.51op}.
\begin{lemma} \label{lem:P.Fourier.Q}
Let $Q \in \bU_{1,r}^{2 \times 2}$ for some $r > 0$ and let $K^{\pm}$ be the
kernels from the integral representation~\eqref{eq:e=(I+K)e0}. Let
matrix-function $P^{\pm}$ be given by~\eqref{eq:P.R.in.X0}--\eqref{2.51op} for
$K^{\pm}$. Then the following uniform estimate holds for $x \in [0,1]$ and
$\l \in \bC$
\begin{equation} \label{eq:int.Pk<Ek+E1+E1}
 \abs{\int_0^x P^{\pm}_k(t) e^{i b_k \l (x-t)} \,dt} \le C \cdot
 e^{2 (b_2 - b_1) |\Im \l| x} \(\sF_1(x,\l) + \sF_2(x,\l)\),
 \qquad k \in \{1, 2\},
\end{equation}
where $C = C(B, r) > 0$ does not depend on $Q$, $x$ and $\l$, and
$\sF_1(x,\l)$, $\sF_2(x,\l)$ are defined in~\eqref{eq:def.sFk}.
\end{lemma}
\begin{proof}
Recall that $\g_1 = a_1 a_2^{-1}$, $\g_2 = a_2 a_1^{-1}$. It follows
from~\eqref{2.52op} that
\begin{align}
\label{eq:P1=R12+int}
 P_1^{\pm}(x) &= \mp \g_2 R_{12}(x,0) - \int_0^x \( R_{11}(x,t)
 P_1^{\pm}(t) \pm \g_2 R_{12}(x, t) P_2^{\pm}(t)\) dt, \\
\label{eq:P2=R21+int}
 P_2^{\pm}(x) &= \mp \g_1 R_{21}(x,0) - \int_0^x \( \mp \g_2 R_{21}(x,t)
 P_1^{\pm}(t) + R_{22}(x, t) P_2^{\pm}(t)\) dt.
\end{align}

Let $k \in \{1, 2\}$ and $j = 2/k$ be fixed for the rest of the proof.

Let us set $\g := \max\{|\g_1|, |\g_2|\}$. It follows
from~\eqref{eq:P1=R12+int}--\eqref{eq:P2=R21+int} that for $x \in [0,1]$ and
$\l \in \bC$ we have
\begin{multline} \label{eq:int.Pk.e}
 \abs{\int_0^x P^{\pm}_k(t) e^{i b_k \l (x-t)} \,dt}
 \le \g \abs{\int_0^x R_{kj}(t,0) e^{i b_k \l (x-t)} \,dt}
 + \g \sum_{l=1}^2 \abs{\int_0^x P_l^{\pm}(s) \,ds \int_s^x R_{kl}(t,s)
 e^{i b_k \l (x-t)} \,dt} \\
 = \g \abs{Z_{kj}(x, 0, \l)}
 + \g \sum_{l=1}^2 \abs{\int_0^x P_l^{\pm}(s) Z_{kl}(x, s, \l) \,ds}, \qquad
 Z_{kl}(x, s, \l) := \int_s^x R_{kl}(t,s) e^{i b_k \l (x-t)} \,dt.
\end{multline}

It follows from~\eqref{eq:Rkk=int.Qkj.Rjk} and Corollary~\ref{cor:int.R<int.Q}
that for $0 \le s \le x \le 1$ and $\l \in \bC$ we have
\begin{align}
\nonumber
 \abs{Z_{kk}(x, s, \l)}
 & = \abs{b_k \int_s^x e^{i b_k \l (x-t)} \,dt \int_{t-s}^t
 Q_{kj}(\xi) R_{jk}(\xi,\xi-t+s) \,d\xi} \\
\nonumber
 & = \abs{b_k \int_0^x Q_{kj}(\xi) \,d\xi \int_{\max\{s,\xi\}}^{\min\{x,s+\xi\}}
 R_{jk}(\xi,\xi-t+s) e^{i b_k \l (x-t)} \,dt} \\
\nonumber
 & = \abs{b_k \int_0^x Q_{kj}(\xi) \,d\xi
 \int_{\max\{\xi+s-x,0\}}^{\min\{s,\xi\}}
 R_{jk}(\xi,\eta) e^{i b_k \l (x+\eta-\xi-s)} \,d\eta} \\
\nonumber
 & \le 2 |b_k| \int_0^x \abs{Q_{kj}(\xi) e^{i b_k \l (x-\xi-s)}}
 \sup_{u \in [0,\xi]}
 \abs{\int_0^u R_{jk}(\xi,\eta) e^{i b_k \l \eta} \,d\eta} \,d\xi \\
\nonumber
 & \le 2 |b_k| \cdot C \int_0^x \abs{Q_{kj}(\xi)} \cdot
 e^{|b_k \Im \l| x} \cdot
 e^{|(b_j - b_k) \Im \l| \xi}
 \cdot \sF_k(\xi, \l) \,d\xi \\
\label{eq:Zkk.est}
 & \le 2 |b_k| \cdot C \cdot \|Q_{kj}\|_1 \cdot
 e^{2 (b_2 - b_1) |\Im \l| x} \cdot \sF_k(x, \l).
\end{align}

Relation~\eqref{eq:Rjk=Qjk-int.Qjk.Rkk} implies for $0 \le s \le x \le 1$ and
$\l \in \bC$
\begin{align}
\label{eq:Zkj=1+2}
 Z_{kj}(x, s, \l) & = \int_s^x R_{kj}(t,s) e^{i b_k \l (x-t)} \,dt
 = \frac{i}{a_j} \(Z_{kj,1}(x, s, \l) - Z_{kj,2}(x, s, \l)\), \\
\label{eq:def.Zkj1}
 Z_{kj,1}(x, s, \l) & := \frac{a_j}{a_j-a_k} \int_s^x Q_{kj}(\a_j t + \a_k s)
 e^{i b_k \l (x-t)} \,dt, \\
\label{eq:def.Zkj2}
 Z_{kj,2}(x, s, \l) & := \frac{a_j}{a_k} \int_s^x e^{i b_k \l (x-t)} \,dt
 \int_{\a_j t + \a_k s}^t
 Q_{kj}(\xi)R_{jj}\bigl(\xi,\g_j(\xi-t)+s\bigr) \,d\xi.
\end{align}
Making a change of variable $\xi = \a_j t + \a_k s$ in~\eqref{eq:def.Zkj1}
combined with identities $\a_j = a_j/(a_j-a_k)$,
$b_k \a_j^{-1} \xi = (b_k-b_j) \xi$ and $b_k \a_j^{-1} \a_k s = -b_j s$,
and definition~\eqref{eq:def.sFk} of $\sF_j(x, \l)$
we get for $0 \le s \le x \le 1$ and $\l \in \bC$
\begin{align}
\nonumber
 \abs{Z_{kj,1}(x, s, \l)}
 & = \abs{\int_{s}^{\a_j x + \a_k s} Q_{kj}(\xi)
 \exp\(i b_k \l \bigl(x + \a_j^{-1} (-\xi + \a_k s)\bigr)\) \,d\xi} \\
\label{eq:Zkj1.est}
 & = \abs{e^{i \l (b_k x - b_j s)} \int_{s}^{\a_j x + \a_k s} Q_{kj}(\xi)
 e^{i (b_j - b_k) \l \xi} \,d\xi}
 \le 2 e^{(b_2 - b_1) |\Im \l| x} \sF_j(x,\l).
\end{align}
Changing order of integration in~\eqref{eq:def.Zkj2}, then making a change of
variable $\eta = \g_j(\xi-t) + s$ combined with the identities
$t = \xi - \g_k \eta + \g_k s$, $\g_j(\xi - \a_j^{-1} \xi - \g_k s) + s = \xi$
and $b_k \g_k = b_j$, and applying Corollary~\ref{cor:int.R<int.Q} yields
\begin{align}
\nonumber
 \abs{Z_{kj,2}(x, s, \l)}
 & = \abs{\g_j \int_s^x Q_{kj}(\xi) \,d\xi
 \int_{\xi}^{\min\{x, \a_j^{-1} \xi + \g_k s\}}
 R_{jj}\bigl(\xi,\g_j(\xi-t)+s\bigr) e^{i b_k \l (x-t)} \,dt} \\
\nonumber
 & = \abs{\int_s^x Q_{kj}(\xi) \,d\xi
 \int_{s}^{\min\{\g_j(\xi-x)+s, \xi\}} R_{jj}(\xi,\eta)
 \exp\bigl(i b_k \l (x - \xi + \g_k \eta - \g_k s)\bigr) \,d\eta} \\
\nonumber
 & \le 2 \int_s^x \abs{Q_{kj}(\xi)}
 \abs{e^{i \l (b_k (x-\xi) - b_j s)}}
 \sup_{u \in [0,\xi]} \abs{\int_0^u R_{jj}(\xi,\eta) e^{i b_j \l \eta} \,d\eta}
 \,d\xi \\
\nonumber
 & \le 2C \int_s^x \abs{Q_{kj}(\xi)} \cdot
 e^{|(b_k - b_j) \Im \l| x} \cdot
 e^{(b_2 - b_1) |\Im \l| \xi} \cdot \sF_j(\xi,\l) \,d\xi \\
\nonumber
 & \le 2C \cdot e^{2 (b_2 - b_1) |\Im \l| x}
 \cdot \sF_j(x,\l) \int_s^x |Q_{kj}(\xi)| \,d\xi \\
\label{eq:Zkj2.est}
 & \le 2Cr \cdot e^{2 (b_2 - b_1) |\Im \l| x} \cdot \sF_j(x,\l),
 \qquad 0 \le s \le x \le 1, \quad \l \in \bC.
\end{align}
Inserting~\eqref{eq:Zkj1.est} and~\eqref{eq:Zkj2.est}
into~\eqref{eq:Zkj=1+2} we arrive at
\begin{equation} \label{eq:Zkj.est}
 |Z_{kj}(x, s, \l)| \le C_j \cdot e^{2 (b_2 - b_1) |\Im \l| x}
 \cdot \sF_j(x,\l), \qquad 0 \le s \le x \le 1, \quad \l \in \bC,
\end{equation}
where $C_j = 2 |b_j| (1+Cr)$. Finally, inserting~\eqref{eq:Zkk.est}
and~\eqref{eq:Zkj.est} into~\eqref{eq:int.Pk.e}, using
estimate~\eqref{eq:norm.whP.P} from the proof of Theorem~\ref{th:K-wtK<Q-wtQ}
for $p=1$, and inequalities $\|Q_{kj}\|_1 \le \|Q\|_1 \le r$, yields
\begin{align}
\nonumber
 \abs{\int_0^x P^{\pm}_k(t) e^{i b_k \l (x-t)} \,dt} & \le \g \cdot
 \|P_k^{\pm}\|_{L^1[0,x]} \sup_{s \in [0,x]} |Z_{kk}(x, s, \l)| + \g \cdot
 (1 + \|P_j^{\pm}\|_{L^1[0,x]}) \sup_{s \in [0,x]} |Z_{kj}(x, s, \l)| \\
\nonumber
 & \le \g \cdot C^{(1)}(B, r) \(\|Q\|_1 \sup_{s \in [0,x]} |Z_{kk}(x, s, \l)| +
 (1 + \|Q\|_1) \sup_{s \in [0,x]} |Z_{kj}(x, s, \l)|\) \\
\label{eq:int.Pk<F}
 & \le C_3 \cdot e^{2 (b_2 - b_1) |\Im \l| x} \cdot
 \(\sF_k(x, \l) + \sF_j(x,\l)\),
\end{align}
where $C_3 = C_3(B,r) > 0$ does not depend on $Q$, $x$ or $\l$, which finishes
the proof.
\end{proof}
Now we are ready to prove the first main result of the section about kernels
$K^{\pm}$ of the transformation operators.
\begin{theorem} \label{th:int.K<int.Q}
Let $Q \in \bU_{1,r}^{2 \times 2}$ for some $r > 0$, and let $K^{\pm}$ be the
kernels of the corresponding transformation operators from
representation~\eqref{eq:e=(I+K)e0}. Then the following uniform estimate holds
\begin{equation} \label{eq:int.Kjk<Ek+E1+E1}
 \abs{\int_0^x K^{\pm}_{jk}(x, t) e^{i b_k \l t} \,dt} \le C \cdot
 e^{2 (b_2-b_1) |\Im \l| x} \(\sF_1(x,\l) + \sF_2(x,\l)\),
 \qquad x \in [0,1], \quad \l \in \bC, \quad j, k \in \{1, 2\},
\end{equation}
where $C = C(B, r) > 0$ does not depend on $Q$, $x$ and $\l$, and
$\sF_k(x,\l)$ is defined in~\eqref{eq:def.sFk}. In other words,
\begin{equation} \label{eq:int.Kjk<sum.Fk}
 \abs{\int_0^x K^{\pm}_{jk}(x, t) e^{i b_k \l t} \,dt} \le C \cdot
 e^{2 (b_2-b_1) |\Im \l| x} \sum_{l \ne m} \sup_{s \in [0,x]}
 \abs{\int_0^s Q_{lm}(t) e^{i (b_m-b_l) \l t} dt},
 \qquad x \in [0,1], \quad \l \in \bC, \quad j, k \in \{1, 2\},
\end{equation}
\end{theorem}
\begin{proof}
It follows from~\eqref{2.51op} that
\begin{equation} \label{eq:Kjk=Rjk+int}
 K_{jk}^{\pm}(x,t) = R_{jk}(x,t) + \delta_{jk} P_k^{\pm}(x-t) +
 \int_t^x R_{jk}(x,s)P_k^{\pm}(s-t)\,ds,
 \qquad 0 \le t \le x \le 1, \quad j,k \in \{1,2\},
\end{equation}
where $P_k^{\pm} \in L^1[0,1]$, $k \in \{1, 2\}$. This in turn yields for
$x \in [0,1]$, $\l \in \bC$ and $j, k \in \{1,2\}$,
\begin{equation} \label{eq:int.K.e=int.R.f}
 \abs{\int_0^x K_{jk}^{\pm}(x,t) e^{i b_k \l t} \,dt} \le
 \abs{\int_0^x R_{jk}(x,t) e^{i b_k \l t} \,dt} +
 \abs{\int_0^x P_k^{\pm}(t) e^{i \l (x-t)} \,dt} +
 \abs{\int_0^x R_{jk}(x,t) \int_0^t P_k^{\pm}(s) e^{i b_k \l (t-s)} \,ds \,dt}.
\end{equation}
Lemmas~\ref{lem:P.Fourier.Q} and~\ref{lem:wh.estim.Xinf} imply for
$x \in [0,1]$, $\l \in \bC$ and $j, k \in \{1,2\}$ with $l = 2/k$
\begin{align}
\nonumber
 \abs{\int_0^x R_{jk}(x,t) \int_0^t P_k^{\pm}(s) e^{i b_k \l (t-s)} \,ds \,dt}
 & \le \sup_{t \in [0, x]} \abs{\int_0^t P_k^{\pm}(s) e^{i b_k \l (t-s)} \,ds}
 \int_0^x \abs{R_{jk}(x,t)} \,dt \\
\nonumber
 & \le C \sup_{t \in [0, x]} \( e^{2 (b_2-b_1) |\Im \l| t}
 \(\sF_1(t,\l) + \sF_2(t,\l)\)\) \|R_{jk}\|_{X_{\infty,1}(\Omega)} \\
\label{eq:int.Rjk.int.Pk<E}
 & \le C_1 e^{2 (b_2-b_1) |\Im \l| x} \(\sF_1(x,\l) + \sF_2(x,\l)\),
\end{align}
where $C_1 = C_1(B, r)$ does not depend on $Q$, $x$ and $\l$.
Putting~\eqref{eq:int.R<int.Q},~\eqref{eq:int.Pk<Ek+E1+E1}
and~\eqref{eq:int.Rjk.int.Pk<E} in~\eqref{eq:int.K.e=int.R.f} we arrive
at~\eqref{eq:int.Kjk<Ek+E1+E1}.
\end{proof}
\subsection{Stability of Fourier transforms of transformation operators}
\label{subsec:Fourier.stabil}
Alongside equation~\eqref{eq:system} we consider similar Dirac type
equation with the same matrix $B$ but with a different potential matrix $\wt{Q}
\in \LL{1}$. In this subsection we apply results of the previous subsection to
study stability of ``Fourier'' transforms of the kernels of the corresponding
transformation operators from representation~\eqref{eq:e=(I+K)e0}. Namely, we
establish analogue of Lipshitz property for the deviation
\begin{equation} \label{eq:int.whK}
\int_0^x (K^{\pm}_Q - K^{\pm}_{\wt{Q}})_{jk}(x, t) e^{i \l b_k t} dt.
\end{equation}
Theorem~\ref{th:int.wh.K<int.wh.Q} will play crucial role in the study of
deviations of the root functions of operators $L(Q)$ and $L(\wt{Q})$.

Below we systematically use notation~\eqref{eq:wh.Q.K.R.P.def}. Let us recall
it:
\begin{align}
\label{eq:wh.Q.K}
 \wh{Q} &:= Q - \wt{Q}, \qquad
 K^{\pm} := K^{\pm}_Q, \quad \wt{K}^{\pm} := K^{\pm}_{\wt{Q}}, \quad
 \wh{K}^{\pm} := K^{\pm} - \wt{K}^{\pm}, \\
\label{eq:wh.R.P}
 R &:= R_Q, \quad \wt{R} := R_{\wt{Q}}, \quad \wh{R} := R - \wt{R}, \qquad
 P^{\pm} := P^{\pm}_Q, \quad \wt{P}^{\pm} := P^{\pm}_{\wt{Q}}, \quad
 \wh{P}^{\pm} := P^{\pm} - \wt{P}^{\pm}.
\end{align}

To estimate the deviation~\eqref{eq:int.whK} we first need to extend auxiliary
results of the previous subsection about $R$ and $P^{\pm}$ to $\wh{R}$ and
$\wh{P}^{\pm}$.
\begin{lemma} \label{lem:int.wh.Rjk.f}
Let $Q, \wt{Q} \in \LL{1}$ and let $R, \wt{R} \in \(X_{1,1}^0(\Omega) \cap
X_{\infty,1}^0(\Omega)\) \otimes \bC^{2 \times 2}$ be (unique) solutions of the
system of integral
equations~\eqref{eq:Rkk=int.Qkj.Rjk}--\eqref{eq:Rjk=Qjk-int.Qjk.Rkk} for $Q$ and
$\wt{Q}$ respectively. Let $x \in [0,1]$ be fixed and let $f \in
L^{\infty}(\bR)$ be such that $f(t) = 0$ for $t \notin [0,x]$. Let us set
\begin{equation} \label{eq:def.wh.Jjk.Fjk}
 \wh{F}_{jk}(s; f) := \sup_{\genfrac{}{}{0pt}{2}{u \in [0,s]}{v \in [-u,x]}}
 \abs{\int_0^u (R - \wt{R})_{jk}(s, t) f(t+v) \,dt},
 \qquad s \in [0,x], \quad j, k \in \{1, 2\}.
\end{equation}
Then the following estimates hold for $s \in [0,x]$, $k \in \{1, 2\}$,
$j = 2/k$:
\begin{align}
\label{eq:wh.Fkk<int.wh.Fjk}
 \wh{F}_{kk}(s; f) & \le |b_k| \int_0^s \(|Q_{kj}(t)| \cdot \wh{F}_{jk}(t; f) +
 \bigabs{\wh{Q}_{kj}(t)} \cdot \wt{F}_{jk}(t; f)\) \,dt, \\
\label{eq:wh.Fjk<int.wh.Fjk}
 \wh{F}_{jk}(s; f) & \le \Theta_{jk}(s;f) + 2|b_jb_k| \cdot
 \|Q_{jk}\|_{L^1[0,s]} \int_0^s |Q_{kj}(t)| \cdot \wh{F}_{jk}(t; f) \,dt, \\
\nonumber
 \Theta_{jk}(s; f)
 & := |b_j| \sup_{\genfrac{}{}{0pt}{2}{u \in [0,s]}{v \in [-u,x]}}
 \abs{\a_j \int_0^u \wh{Q}_{jk}(\a_k s + \a_j t) \cdot f(t+v) \,dt} \\
\label{eq:Theta.jk}
 & + 2 |b_j| \int_0^s \bigabs{\wh{Q}_{jk}(t)} \cdot \wt{F}_{kk}(t; f) \,dt
 + 2|b_jb_k| \cdot \|Q_{jk}\|_{L^1[0,s]} \int_0^s
 \bigabs{\wh{Q}_{kj}(t)} \cdot \wt{F}_{jk}(t; f) \,dt.
\end{align}
Here $\wh{Q} = Q - \wt{Q}$ and $\wt{F}_{jk}(s; f)$, $j, k \in \{1, 2\}$, is
defined in~\eqref{eq:def.Fjk} with $\wt{R}_{jk}$ in place of $R_{jk}$.
\end{lemma}
\begin{proof}
We will follow the schema of the proof of Lemma~\ref{lem:int.Rjk.f} step by
step and will use the following identity
\begin{equation} \label{eq:XY-wt.XY}
 X \cdot Y - \wt{X} \cdot \wt{Y} =
 X \cdot (Y - \wt{Y}) + (X - \wt{X}) \cdot \wt{Y} =
 X \cdot \wh{Y} + \wh{X} \cdot \wt{Y},
\end{equation}
every time we encounter a product difference of the form $X \cdot Y - \wt{X}
\cdot \wt{Y}$, where $X$, $Y$ are some objects associated with
equation~\eqref{eq:systemIntro} for $Q$ and $\wt{X}$, $\wt{Y}$ are the same
objects associated with equation~\eqref{eq:systemIntro} for $\wt{Q}$.

Recall that $\wh{R} = R - \wt{R}$. Let us set
\begin{equation} \label{eq:def.wh.Jjk}
 \wh{J}_{jk}(s, u, v) := \int_0^u \wh{R}_{jk}(s, t) f(t+v) \,dt,
 \qquad 0 \le u \le s \le x, \quad v \in \bR, \quad j, k \in \{1,2\}.
\end{equation}
Let $k \in \{1, 2\}$ and $j = 2/k$. Similarly to~\eqref{eq:Jkk=int.Qkj.Jjk}
with account of~\eqref{eq:XY-wt.XY} we have for $0 \le u \le s \le x$ and
$v \in \bR$,
\begin{align}
\nonumber
 \wh{J}_{kk}(s, u, v) &= \int_0^u \wh{R}_{kk}(s, t) f(t+v) dt \\
\nonumber
 & = -i b_k \int_0^u f(t+v) \,dt \int_{s-t}^s
 \(Q_{kj}(\xi) \, R_{jk}(\xi, \xi -s + t) -
 \wt{Q}_{kj}(\xi) \, \wt{R}_{jk}(\xi, \xi -s + t)\) \,d\xi \\
\nonumber
 & = -i b_k \int_0^u f(t+v) \,dt \int_{s-t}^s
 \(Q_{kj}(\xi) \, \wh{R}_{jk}(\xi, \xi -s + t) +
 \wh{Q}_{kj}(\xi) \, \wt{R}_{jk}(\xi, \xi -s + t)\) \,d\xi \\
\label{eq:wh.Jkk=int.Qkj.wh.Jjk}
 & = -i b_k \int_{s-u}^s
 \(Q_{kj}(\xi) \, \wh{J}_{jk}(\xi, u+\xi-s, v+s-\xi) +
 \wh{Q}_{kj}(\xi) \, \wt{J}_{jk}(\xi, u+\xi-s, v+s-\xi) \) \,d\xi.
\end{align}
Similarly to~\eqref{eq:Jkk<C1.int.Q.Jjk} we get
\begin{equation} \label{eq:wh.Jkk<C1.int.Q.wh.Jjk}
 |\wh{J}_{kk}(\xi, \eta, v)| \le |b_k| \int_0^{\xi} \(\bigabs{Q_{kj}(t)} \cdot
 \wh{F}_{jk}(t; f) + \bigabs{\wh{Q}_{kj}(t)} \cdot \wt{F}_{jk}(t; f) \) \,dt,
 \qquad 0 \le \eta \le \xi \le x, \quad v \in \bR,
\end{equation}
which yields~\eqref{eq:wh.Fkk<int.wh.Fjk}.

Taking into account~\eqref{eq:def.ak.alphak} and~\eqref{eq:XY-wt.XY},
relation~\eqref{eq:Rjk=Qjk-int.Qjk.Rkk} yields for $0 \le u \le s \le x$ and
$v \in \bR$,
\begin{align}
\label{eq:wh.Jjk=J1+J2}
 \wh{J}_{jk}(s, u, v) & = \int_0^u \wh{R}_{jk}(s, t) f(t+v) \,dt
 = -\frac{i}{a_j} \cdot \(\wh{J}_{1,jk}(s, u, v) + \wh{J}_{2,jk}(s, u, v)\), \\
\label{eq:wh.J1jk.def}
 \wh{J}_{1,jk}(s, u, v) & :=
 \a_j \int_0^u \wh{Q}_{jk}(\a_k s + \a_j t) \cdot f(t+v) \,dt, \\
\nonumber
 \wh{J}_{2,jk}(s, u, v) & := \int_0^u f(t+v) \,dt
 \int^s_{\a_k s + \a_j t}
 \(Q_{jk}(\xi) \, R_{kk}(\xi, \g_k (\xi-s) + t)
 - \wt{Q}_{jk}(\xi) \, \wt{R}_{kk}(\xi, \g_k (\xi-s) + t)\)\,d\xi \\
\label{eq:wh.J2jk.def}
 & = \int_0^u f(t+v) \,dt
 \int^s_{\a_k s + \a_j t}
 \(Q_{jk}(\xi) \, \wh{R}_{kk}(\xi, \g_k (\xi-s) + t)
 + \wh{Q}_{jk}(\xi) \, \wt{R}_{kk}(\xi, \g_k (\xi-s) + t)\)\,d\xi.
\end{align}
For brevity we denote,
\begin{equation}
 \xi_1 := \g_k(\xi - s), \quad
 \xi_2 := \min\{u + \g_k (\xi-s), \xi\}, \quad
 v_{\xi} = v - \g_k (\xi - s).
\end{equation}
Similarly to~\eqref{eq:J2=int.Q.int.Rkk} and~\eqref{eq:J2<int.Q.int.Q.Jjk}, it
follows from~\eqref{eq:wh.J2jk.def},~\eqref{eq:wh.Jkk<C1.int.Q.wh.Jjk}
and~\eqref{eq:Jjk<Fjk}, that for $0 \le u \le s \le x$ and $v \in \bR$,
\begin{align}
\nonumber
 |\wh{J}_{2,jk}(s, u, v)| &\le \int_{\a_k s}^s \sum_{l=1}^2 \(
 \bigabs{Q_{jk}(\xi)} \cdot \bigabs{\wh{J}_{kk}\bigl(\xi, \xi_l, v_{\xi})}
 + \bigabs{\wh{Q}_{jk}(\xi)} \cdot
 \bigabs{\wt{J}_{kk}\bigl(\xi, \xi_l, v_{\xi})}\) \,d\xi \\
\nonumber
 & \le 2 |b_k| \int_{0}^s |Q_{jk}(\xi)| \,d\xi
 \int_{0}^{\xi} \(\bigabs{Q_{kj}(t)} \cdot \wh{F}_{jk}(t; f)
 + \bigabs{\wh{Q}_{kj}(t)} \cdot \wt{F}_{jk}(t; f) \) \,dt
 + 2 \int_{0}^s \bigabs{\wh{Q}_{jk}(\xi)} \cdot \wt{F}_{kk}(\xi; f) \,d\xi \\
\label{eq:wh.J2<int.Q.int.Q.Jjk}
 & \le 2 |b_k| \cdot \|Q_{jk}\|_{L^1[0,s]}
 \int_{0}^s \(\bigabs{Q_{kj}(t)} \cdot \wh{F}_{jk}(t; f)
 + \bigabs{\wh{Q}_{kj}(t)} \cdot \wt{F}_{jk}(t; f) \) \,dt
 + 2 \int_{0}^s \bigabs{\wh{Q}_{jk}(t)} \cdot \wt{F}_{kk}(t; f) \,dt.
\end{align}
Inserting~\eqref{eq:wh.J1jk.def} and~\eqref{eq:wh.J2<int.Q.int.Q.Jjk}
into~\eqref{eq:wh.Jjk=J1+J2} we arrive
at~\eqref{eq:wh.Fjk<int.wh.Fjk}--\eqref{eq:Theta.jk}.
\end{proof}
\begin{corollary} \label{cor:int.whR<int.whQ}
Let $Q, \wt{Q} \in \bU_{1,r}^{2 \times 2}$ for some $r > 0$, and let $R, \wt{R}
\in \(X_{1,1}^0(\Omega) \cap X_{\infty,1}^0(\Omega)\) \otimes \bC^{2 \times 2}$
be (unique) solutions of the system of integral
equations~\eqref{eq:Rkk=int.Qkj.Rjk}--\eqref{eq:Rjk=Qjk-int.Qjk.Rkk} for $Q$ and
$\wt{Q}$ respectively. Then the following uniform estimate holds for
$x \in [0,1]$, $\l \in \bC$ and $j, k \in \{1, 2\}$,
\begin{equation} \label{eq:int.whR<int.whQ}
 \sup_{s \in [0,x]} \abs{\int_0^s (R - \wt{R})_{jk}(x, t) e^{i \l b_k t}
 \,dt} \le C \cdot e^{(b_2-b_1) |\Im \l| x} \(\sF_k[Q - \wt{Q}](x,\l) +
 \|Q - \wt{Q}\|_1 \sF_k[\wt{Q}](x,\l)\),
\end{equation}
where $C = C(B, r) > 0$ does not depend on $Q$, $\wt{Q}$, $x$ and $\l$, and
$\sF_k[W](x,\l)$ is defined in~\eqref{eq:def.sFk} for $W \in \LL{1}$.
\end{corollary}
\begin{proof}
The proof will follow the proof of Corollary~\ref{cor:int.R<int.Q}. Let
$x \in [0,1]$, $k \in \{1,2\}$ and $j = 2/k$ be fixed for the rest of the proof.
We will apply Lemma~\ref{lem:int.Rjk.f} with $f(t) = e^{i b_k \l t}$,
$t \in [0,x]$, and $f(t) = 0$, $t \notin [0,x]$. Let us estimate
$\Theta_{lk}(s; f)$ in~\eqref{eq:wh.Fjk<int.wh.Fjk}--\eqref{eq:Theta.jk} for
such $f$. Similarly to~\eqref{eq:int.Qjkftv} we have
$0 \le u \le s \le x$, $v \in [-u, x]$ and $\l \in \bC$
\begin{equation} \label{eq:int.whQjkftv}
 \abs{\a_j \int_0^u \wh{Q}_{jk}(\a_k s + \a_j t) \cdot f(t+v) \,dt }
 \le 2 e^{(b_2-b_1) |\Im \l| x} \cdot \sF_k[\wh{Q}](x,\l) =:
 2 \wh{\sR}_k(x, \l).
\end{equation}
Further, estimates~\eqref{eq:int0u.Rjk<E}--\eqref{eq:Fkk<C.Ek} on
$\wt{F}_{lk}(t; f)$ in place of $F_{lk}(t; f)$ yield for $l \in \{1, 2\}$ and
$m = 2/l$
\begin{equation} \label{eq:int.whQ.wtF}
 \int_0^s \bigabs{\wh{Q}_{ml}(t)} \cdot \wt{F}_{lk}(t; f) \,dt \le
 C_1 \cdot \|\wh{Q}\|_1 \cdot \wt{\sR}_k(x, \l), \qquad
 \wt{\sR}_k(x, \l) := e^{(b_2-b_1) |\Im \l| x} \cdot \sF_k[\wt{Q}](x,\l).
\end{equation}
Here and below constants $C_1$, $C_2$, $C_3$, only depend on $B$, $r$ and do not
depend on $Q$, $\wt{Q}$, $x$ and $\l$. Taking into account that
$\max\{\|Q_{jk}\|_{L_1[0,1]}, \|Q_{kj}\|_{L_1[0,1]}\} \le r$,
inserting~\eqref{eq:int.whQjkftv}--\eqref{eq:int.whQ.wtF}
in~\eqref{eq:wh.Fkk<int.wh.Fjk}--\eqref{eq:Theta.jk} implies
\begin{align}
\label{eq:whFkk<sE+int.whFjk}
 \wh{F}_{kk}(s; f) & \le C_2 \( \|\wh{Q}\|_1 \cdot \wt{\sR}_k(x,\l)
 + \int_0^s |Q_{kj}(t)| \cdot \wh{F}_{jk}(t; f) \,dt\),
 \qquad s \in [0,x], \quad \l \in \bC, \\
\label{eq:whFjk<sE+int.whFjk}
 \wh{F}_{jk}(s; f) & \le C_3 \(\wh{\sR}_k(x,\l) + \|\wh{Q}\|_1 \cdot
 \wt{\sR}_k(x,\l) + \int_0^s |Q_{kj}(t)| \cdot
 \wh{F}_{jk}(t; f) \,dt \), \qquad s \in [0,x], \quad \l \in \bC.
\end{align}
The proof is finished in the same way as in Corollary~\ref{cor:int.R<int.Q}
using Gr\"onwall's inequality.
\end{proof}
Following the proof of Lemma~\ref{lem:P.Fourier.Q} and using
Corollary~\ref{cor:int.whR<int.whQ} and identity~\eqref{eq:XY-wt.XY} in the same
way it was done in the proof Lemma~\ref{lem:int.wh.Rjk.f}, we can obtain the
following result on the auxiliary matrix-function $\wh{P}^{\pm} =
P^{\pm} - \wt{P}^{\pm}$.
\begin{lemma} \label{lem:whP.Fourier.Q}
Let $Q, \wt{Q} \in \bU_{1,r}^{2 \times 2}$ for some $r > 0$, and let $K^{\pm}$,
$\wt{K}^{\pm}$ be the kernels of the corresponding transformation operators from
representation~\eqref{eq:e=(I+K)e0} for $Q$ and $\wt{Q}$ respectively. Let
matrix-functions $P^{\pm}$, $\wt{P}^{\pm}$, $k \in \{1, 2\}$, be given
by~\eqref{eq:P.R.in.X0}--\eqref{2.51op} for $K^{\pm}$ and $\wt{K}^{\pm}$
respectively. Then the following uniform estimate holds for $x \in [0,1]$ and
$\l \in \bC$
\begin{equation} \label{eq:int.whPk<Ek+E1+E1}
 \abs{\int_0^x \bigl(P^{\pm}_k(t) - \wt{P}^{\pm}_k(t)\bigr)
 e^{i b_k \l (x-t)} \,dt} \le C \cdot e^{2 (b_2 - b_1) |\Im \l| x} \sum_{j=1}^2
 \(\sF_j[Q - \wt{Q}](x,\l) + \|Q - \wt{Q}\|_1 \sF_j[\wt{Q}](x,\l)\),
 \qquad k \in \{1, 2\},
\end{equation}
where $C = C(B, r) > 0$ does not depend on $Q$, $\wt{Q}$, $x$ and $\l$, and
$\sF_j[W](x,\l)$ is defined in~\eqref{eq:def.sFk} for $W \in \LL{1}$.
\end{lemma}
Now we are ready to prove Theorem~\ref{th:int.wh.K<int.wh.Q.intro}, our main
result on stability of Fourier transforms of kernels $K_Q^{\pm}$ of the
transformation operators. We will formulate it again using compact
notation~\eqref{eq:def.sFk}.
\begin{theorem} \label{th:int.wh.K<int.wh.Q}
Let $Q, \wt{Q} \in \bU_{1,r}^{2 \times 2}$ for some $r > 0$, and let $K^{\pm} :=
K^{\pm}_Q$, $\wt{K}^{\pm} := K^{\pm}_{\wt{Q}}$ be the kernels of the
corresponding transformation operators from
representation~\eqref{eq:e=(I+K)e0} for $Q$ and $\wt{Q}$ respectively. Then the
following uniform estimate holds for $x \in [0,1]$ and $\l \in \bC$
\begin{equation} \label{eq:int.whK<F1+F2}
 \sum_{j,k=1}^2 \abs{\int_0^x \bigl(K^{\pm}_Q - K^{\pm}_{\wt{Q}}\bigr)_{jk}(x,t)
 e^{i b_k \l t} \,dt} \le C \cdot e^{2 (b_2 - b_1) |\Im \l| x} \sum_{j=1}^2
 \(\sF_j[Q - \wt{Q}](x,\l) + \|Q - \wt{Q}\|_1 \sF_j[\wt{Q}](x,\l)\),
\end{equation}
where $C = C(B, r) > 0$ does not depend on $Q$, $\wt{Q}$, $x$ and $\l$, and
$\sF_j[W](x,\l)$ is defined in~\eqref{eq:def.sFk} for $W \in \LL{1}$.
\end{theorem}
\begin{proof}
The proof is similar to the proof of Theorem~\ref{th:int.K<int.Q}.
First we subtract formula~\eqref{eq:Kjk=Rjk+int} for $\wt{Q}$ from the same
formula for $Q$ and apply identity~\eqref{eq:XY-wt.XY} to the product of
differences $R_{jk}(x, s) P_k^{\pm}(s-t) - \wt{R}_{jk}(x, s)
\wt{P}_k^{\pm}(s-t)$ in the integral. This yields the following inequality with
account of notations~\eqref{eq:wh.Q.K}--\eqref{eq:wh.R.P},
\begin{multline} \label{eq:int.whK.exp}
 \abs{\int_0^x \wh{K}_{jk}^{\pm}(x,t) e^{i b_k \l t} \,dt}
 \le \abs{\int_0^x \wh{R}_{jk}(x,t) e^{i b_k \l t} \,dt}
 + \abs{\int_0^x \wh{P}_k^{\pm}(t) e^{i \l (x-t)} \,dt} \\
 + \abs{\int_0^x R_{jk}(x,t) \int_0^t \wh{P}_k^{\pm}(s)
 e^{i b_k \l (t-s)} \,ds \,dt}
 + \abs{\int_0^x \wh{R}_{jk}(x,t) \int_0^t \wt{P}_k^{\pm}(s)
 e^{i b_k \l (t-s)} \,ds \,dt}.
\end{multline}
The first summand in r.h.s of~\eqref{eq:int.whK.exp} involving Fourier
transform of $\wh{R}_{jk}(x,t)$ is estimated using
Corollary~\ref{cor:int.whR<int.whQ}, while the second summand is estimated using
Lemma~\ref{lem:whP.Fourier.Q}. The last two summands are estimated similarly
to inequality~\eqref{eq:int.Rjk.int.Pk<E}. More specifically, for the third
summand we apply Lemma~\ref{lem:whP.Fourier.Q} again to estimate ``maximal''
Fourier transform of $\wh{P}_k^{\pm}$ and Lemma~\ref{lem:wh.estim.Xinf} to
estimate $\|R_{jk}\|_{X_{\infty,1}(\Omega)}$, while for the fourth summand
we apply Lemma~\ref{lem:P.Fourier.Q} to estimate ``maximal''
Fourier transform of $\wt{P}_k^{\pm}$ and again Lemma~\ref{lem:wh.estim.Xinf} to
estimate $\|\wh{R}_{jk}\|_{X_{\infty,1}(\Omega)}$.
\end{proof}
\subsection{Stability property of the fundamental matrix}
Throughout the section $\Phi_Q(x, \l)$ and $\Phi_{\wt{Q}}(x, \l)$ will denote
the fundamental matrix solutions of the system~\eqref{eq:system} for $Q$ and
$\wt{Q}$ respectively and notations~\eqref{eq:Phi.def},~\eqref{eq:Phi0.def} will
be used. For reader's convenience we recall them here,
\begin{align}
\label{eq:PhiQ}
 \Phi_Q(\cdot, \l) &= \begin{pmatrix}
 \varphi_{11}(\cdot, \l) & \varphi_{12}(\cdot, \l)\\
 \varphi_{21}(\cdot, \l) & \varphi_{22}(\cdot, \l)
 \end{pmatrix}
 = \begin{pmatrix} \Phi_1(\cdot, \l) & \Phi_2(\cdot, \l) \end{pmatrix},
 \qquad \l \in \bC, \\
\label{eq:PhiwtQ}
 \Phi_{\wt{Q}}(\cdot, \l) &= \begin{pmatrix}
 \wt{\varphi}_{11}(\cdot, \l) & \wt{\varphi}_{12}(\cdot, \l) \\
 \wt{\varphi}_{21}(\cdot, \l) & \wt{\varphi}_{22}(\cdot, \l)
 \end{pmatrix} = \begin{pmatrix}
 \wt{\Phi}_1(\cdot, \l) & \wt{\Phi}_2(\cdot, \l) \end{pmatrix},
 \qquad \l \in \bC, \\
\label{eq:Phi0}
 \Phi^0(x, \l) &= \begin{pmatrix}
 e^{i b_1 x \l} & 0 \\ 0 & e^{i b_2 x \l} \end{pmatrix}
 =: \begin{pmatrix}
 \varphi_{11}^0(x, \l) & \varphi_{12}^0(x, \l)\\
 \varphi_{21}^0(x,\l) & \varphi_{22}^0(x,\l)
 \end{pmatrix} =: \begin{pmatrix}
 \Phi_1^0(x, \l) & \Phi_2^0(x, \l) \end{pmatrix},
 \qquad x \in [0,1], \quad \l \in \bC.
\end{align}
The following uniform estimate of the deviation of fundamental matrix will play
important role in studying deviations of root vectors.
\begin{theorem} \label{th:int.phi<int.Q}
Let $Q, \wt{Q} \in \bU_{1, r}^{2 \times 2}$ for some $r > 0$. Then the following
uniform estimate holds
\begin{multline} \label{eq:PhiQ-PhiwtQ}
 \abs{\Phi_Q(x,\l) - \Phi_{\wt{Q}}(x,\l)}
 \le C \cdot e^{2 (b_2 - b_1) |\Im \l| x} \sum_{k=1}^2
 \(\sF_k[Q - \wt{Q}](x,\l) + \|Q - \wt{Q}\|_1 \sF_k[\wt{Q}](x,\l)\) \\
 = C \cdot e^{2 (b_2 - b_1) |\Im \l| x} \sum_{j \ne k}
 \(\sup_{s \in [0,x]} \abs{\int_0^s (Q_{jk}(t) - \wt{Q}_{jk}(t))
 e^{i (b_k-b_j) \l t} dt} \right. \\
 \left. + \|Q - \wt{Q}\|_1 \sup_{s \in [0,x]}
 \abs{\int_0^s \wt{Q}_{jk}(t) e^{i (b_k-b_j) \l t} dt}\),
 \qquad x \in [0,1], \quad \l \in \bC,
\end{multline}
where $C = C(B, r) > 0$ does not depend on $Q$, $\wt{Q}$, $x$ and $\l$. In
particular,
\begin{equation} \label{eq:PhiQ-Phi0}
 \abs{\Phi_Q(x,\l) - \Phi_0(x,\l)}
 \le C \cdot e^{2 (b_2 - b_1) |\Im \l| x} \sum_{j \ne k}
 \sup_{s \in [0,x]} \abs{\int_0^s Q_{jk}(t) e^{i (b_k-b_j) \l t} dt},
 \qquad x \in [0,1], \quad \l \in \bC.
\end{equation}
\end{theorem}
\begin{proof}
It follows from
representations~\eqref{eq:phi-wt.phi.jkx}--\eqref{eq:whK.def} in
Lemma~\ref{lem:phi-wt.phi} with account of formula~\eqref{eq:Kjlk}, that
each deviation $\varphi_{jk}(x,\l) - \wt{\varphi}_{jk}(x,\l)$,
$j,k \in \{1, 2\}$, is a linear combinations of all eight
deviations~\eqref{eq:int.whK}. Applying Theorem~\ref{th:int.wh.K<int.wh.Q} with
account of notation~\eqref{eq:def.sFk} completes the proof.
\end{proof}
Recall that for $f \in L^s([0,1]; \bC^n)$, $s \in (0, \infty]$, $n \in \bN$, we
denote $\|f\|_s := \|f\|_{L^s([0,1]; \bC^n)}$.
\begin{proposition} \label{eq:Phix.compact}
Let $\cK$ be compact in $\LL{1}$ and $h \ge 0$. Then the following uniform
estimates hold:

\textbf{(i)} For any $\delta > 0$ there exists $M_{\delta} =
M_{\delta}(\cK, B, h) > 0$ such that
\begin{equation} \label{eq:phi-phi0<delta}
 \norm{\varphi_{jk}(\cdot, \l)}_s + \norm{\varphi_{kk}(\cdot, \l)
 -\varphi_{kk}^0(\cdot, \l)}_s < \delta, \qquad
 s \in (0, \infty], \quad |\l| > M_{\delta}, \quad |\Im \l| \le h,
 \quad k \in \{1,2\}, \quad j = 2/k, \quad Q \in \cK.
\end{equation}

\textbf{(ii)} There exists $\delta_0 = \delta_0(B, h) > 0$ such that for $M_{\delta_0}$ defined in part (i) we have
\begin{equation} \label{eq:Phik.norm}
 \norm{\varphi_{kk}(\cdot, \l)}_s \ge \delta_0, \qquad s \in (0, \infty],
 \quad |\l| > M_{\delta_0}, \quad |\Im \l| \le h, \quad k \in \{1,2\},
 \quad Q \in \cK,
\end{equation}

\textbf{(iii)} Let $r > 0$ then there exists $C_2 = C_2(r, B, h) > 0$ such that
\begin{equation} \label{eq:phijk<C2}
 \abs{\varphi_{jk}(x, \l)} + \abs{\frac{d}{d\l}\varphi_{jk}(x, \l)} \le C_2,
 \qquad x \in [0,1], \quad |\Im \l| \le h, \quad j, k \in \{1, 2\},
 \quad Q \in \bU_{1,r}^{2 \times 2}.
\end{equation}
\end{proposition}
\begin{proof}
\textbf{(i)} The proof is immediate from
represantations~\eqref{eq:phi.jkx}--\eqref{eq:Kjlk},
Lemma~\ref{lem:fourier.coef.alter} and inequality $\|f\|_s \le \|f\|_{\infty}$
valid for $f \in L^{\infty}[0,1]$ and $s \in (0, \infty]$.

\textbf{(ii)} It is clear that
\begin{equation} \label{eq:exp.Ch}
 \abs{\varphi_{kk}^0(x, \l)} = \abs{e^{i b_k \l x}} \ge C, \qquad x \in [0,1],
 \quad |\Im \l| \le h, \quad k \in \{1,2\},
\end{equation}
with $C := \exp\(-\max\{|b_1|, |b_2|\} \cdot h\)$. Inequality~\eqref{eq:exp.Ch}
implies that
\begin{equation} \label{eq:norm.Phik0}
 \norm{\varphi_{kk}^0(\cdot, \l)}_s \ge C,
 \qquad s \in (0, \infty], \quad |\Im \l| \le h, \quad k \in \{1,2\}.
\end{equation}
Setting $\delta_0 = C/2$ and combining inequalities~\eqref{eq:norm.Phik0}
and~\eqref{eq:phi-phi0<delta} we arrive at~\eqref{eq:Phik.norm}, since
$C - \delta_0 = \delta_0$.

\textbf{(iii)} The proof is immediate from
represantations~\eqref{eq:phi.jkx}--\eqref{eq:Kjlk} in
Lemma~\ref{lem:phi.jk=e+int} and Theorem~\ref{th:K-wtK<Q-wtQ}. Here is how it
looks like for the derivative $\frac{d}{d\l}\varphi_{jk}(x, \l)$. We have for
$x \in [0,1]$, $|\Im \l| \le h$, $j, k \in \{1, 2\}$ and $Q \in
\bU_{1,r}^{2 \times 2}$:
\begin{multline}
 \abs{\frac{d}{d\l}\varphi_{jk}(x, \l)}
 \le \abs{i b_k x \delta_{jk} e^{i \l b_k x}}
 + \sum_{l=1}^2 \int_0^x \abs{i b_l t K_{jl,k}(x,t) e^{i \l b_l x}} dt \\
 \le |b_k| e^{|b_k| h} + \sum_{l=1}^2 |b_l| e^{|b_k| h}
 \|K_{jl,k}\|_{X_{\infty,1}(\Omega; \bC^{2 \times 2})}
 \le b_0 e^{b_0 h} (1 + 2 C \|Q\|_1) \le b_0 e^{b_0 h} (1 + 2 C r),
\end{multline}
where $b_0 := \max\{|b_1|, |b_2|\}$.
\end{proof}
We need similar result for balls $\bU_{p,r}^{2 \times 2}$, $p \in (1, 2]$.
\begin{proposition} \label{eq:Phix.holes}
Let $Q \in \bU_{p,r}^{2 \times 2}$ for some $p \in (1, 2]$, $r > 0$, and let
$h \ge 0$. Then the following uniform estimates hold:

\textbf{(i)} For any $\delta > 0$ there exists a set $\cJ_{Q, \delta} \subset
\bZ$ such that for $s \in (0, \infty]$, $k \in \{1, 2\}$ and $j = 2/k$ we have
\begin{align}
\label{eq:card.cJQdelta}
 \card(\bZ \setminus \cJ_{Q,\delta}) \le N_{\delta}
 := C_0 \cdot (r/\delta)^{p'},& \qquad 1/p' + 1/p = 1, \\
\label{eq:phijk.s<delta}
 \norm{\varphi_{jk}(\cdot, \l)}_s + \norm{\varphi_{kk}(\cdot, \l)
 -\varphi_{kk}^0(\cdot, \l)}_s < \delta, &
 \qquad \l \in \Pi_{h,Q,\delta}, \quad
 \Pi_{h,Q,\delta} := \bigcup_{n \in \cJ_{Q,\delta}} [n, n+1] \times [-h, h].
\end{align}
Here $C_0 = C_0(p, h, b) > 0$ does not depend on $Q$, $r$ and $\delta$.

\textbf{(ii)} There exists a constant $\delta_0 = \delta_0(B, h) > 0$ such that
\begin{align} \label{eq:Phik.norm.holes}
 \norm{\varphi_{kk}(\cdot, \l)}_s \ge \delta_0,
 \qquad \l \in \Pi_{h,Q,\delta}, \quad \delta \in (0, \delta_0],
 \quad s \in (0, \infty], \quad k \in \{1, 2\},
\end{align}
where the set $\cJ_{Q, \delta}$ also satisfies
inequalities~\eqref{eq:card.cJQdelta}--\eqref{eq:phijk.s<delta}.
\end{proposition}
\begin{proof}
\textbf{(i)} The proof for $s = \infty$ is immediate from
the estimate~\eqref{eq:PhiQ-Phi0} of Theorem~\ref{th:int.phi<int.Q} and
Lemma~\ref{lem:max.fourier.tail.Lp} applied with $g = Q_{jk}$,
$b = b_k - b_k$ and $\delta' = \delta C^{-1} e^{2 (b_1-b_2) h} / 2$, where
$C$ is taken from the inequality~\eqref{eq:PhiQ-Phi0}. Namely, we need to set
$\cJ_{Q, \delta} = \cI_{Q_{12}, \delta'} \cap \cI_{Q_{21}, \delta'}$. Inequality
$\|f\|_s \le \|f\|_{\infty}$ valid for $f \in L^{\infty}[0,1]$ and $s \in (0, \infty]$, finishes the proof.

\textbf{(ii)} The proof is immediate by combining part (i) and
inequality~\eqref{eq:norm.Phik0} as it was done in the proof of
Proposition~\eqref{eq:Phix.compact}(ii) with the same $\delta_0 =
\exp\(-\max\{|b_1|, |b_2|\} \cdot h\)/2$.
\end{proof}
Recall that for $F \in L^s([0,1]; \bC^{2 \times 2})$, $s \in (0, \infty]$, we
denote $\|F\|_s := \|F\|_{L^s([0,1]; \bC^{2 \times 2})}$. Combining
Theorems~\ref{th:int.phi<int.Q} and~\ref{th:p.bessel} we obtain an important
stability property of the fundamental matrix.
\begin{proposition} \label{prop:sum.phi-wt.phi.x.p}
Let $Q, \wt{Q} \in \bU_{p, r}^{2 \times 2}$ for some $p \in (1, 2]$ and $r > 0$.
Let $\L = \{\mu_n\}_{n \in \bZ}$ be an incompressible sequence of density $d$
lying in the strip $\Pi_h$. Then for some $C = C(p, r, B, h, d) > 0$ that does
not depend on $Q$, $\wt{Q}$ and $\L$ the following uniform estimates hold
\begin{align}
\label{eq:sum.wh.phi.p'}
 \sum_{n \in \bZ} \norm{\Phi_Q(\cdot, \mu_n) -
 \Phi_{\wt{Q}}(\cdot, \mu_n)}_{\infty}^{p'}
 & \le C \cdot \|Q - \wt{Q}\|_p^{p'},
 \qquad Q, \wt{Q} \in \bU_{p, r}^{2 \times 2}, \quad 1/p + 1/p' = 1, \\
\label{eq:sum.wh.phi.p.vu}
 \sum_{n \in \bZ} (1+|n|)^{p-2}\norm{\Phi_Q(\cdot, \mu_n) -
 \Phi_{\wt{Q}}(\cdot, \mu_n)}_{\infty}^{p}
 & \le C \cdot \|Q - \wt{Q}\|_p^{p},
 \qquad Q, \wt{Q} \in \bU_{p, r}^{2 \times 2}.
\end{align}
\end{proposition}
\begin{proof}
First note that notations~\eqref{eq:def.sFk} and~\eqref{eq:sF.f.def} imply for
any $W \in \LL{1}$ and $\l \in \bC$
\begin{equation} \label{eq:sFk=sF}
 \sup_{x \in [0,1]} \sF_k[W](x, \l) = \sF_k[W](1, \l)
 = \sF[W_{jk}]((b_k-b_j)\l), \qquad k \in \{1, 2\}, \quad j = 2/k.
\end{equation}
It follows from Theorem~\ref{th:int.phi<int.Q}, inequalities $|a+b|^{p'} \le
2^{p'-1}(|a|^{p'} + |b|^{p'})$ and $|\Im \mu_n| \le h$, and
relation~\eqref{eq:sFk=sF} that
\begin{multline} \label{eq:sum.phijk<sum.sF}
 \sum_{n \in \bZ} \norm{\Phi_Q(\cdot, \mu_n) -
 \Phi_{\wt{Q}}(\cdot, \mu_n)}_{\infty}^{p'}
 = \sum_{n \in \bZ} \max_{j,k \in \{1, 2\}} \sup_{x \in [0,1]}
 \abs{\varphi_{jk}(x, \mu_n) - \wt{\varphi}_{jk}(x, \mu_n)}^{p'} \\
 \le 2^{p'-1} C^{p'} e^{2 (b_2 - b_1) |\Im \mu_n| p'}
 \sum_{n \in \bZ} \sum_{k=1}^2 \(\Bigl(\sF_k[Q - \wt{Q}](1,\mu_n)\Bigr)^{p'} +
 \|Q - \wt{Q}\|_1^{p'} \(\sF_k[\wt{Q}](1,\mu_n)\)^{p'}\) \\
 \le 2^{p'-1} C^{p'} e^{2 (b_2 - b_1) h p'}
 \sum_{j \ne k} \sum_{n \in \bZ}
 \(\sF[Q_{jk} - \wt{Q}_{jk}]^{p'}\bigl((b_k-b_j) \mu_n\bigr) +
 \|Q - \wt{Q}\|_1^{p'} \sF[\wt{Q}_{jk}]^{p'}\bigl((b_k-b_j) \mu_n\bigr)\),
\end{multline}
where $\sF[g]^{p'}(\l) := \(\sF[g](\l)\)^{p'}$ and $C = C(B, r) > 0$ does not
depend on $Q$, $\wt{Q}$ and the sequence $\L$.

Let $k \in \{1, 2\}$, $j = 2/k$ and $\tau := |b_k - b_j| = |b_1 - b_2|$. Since
$\L = \{\mu_n\}_{n\in \bZ}$ is an incompressible sequence of density $d$ lying
in the strip $\Pi_h$, then $\L_k = \{(b_k - b_j) \mu_n\}_{n\in \bZ}$ is an
incompressible sequence of density $d/\tau$ lying in the strip $\Pi_{h\tau}$.
Hence, applying Theorem~\ref{th:p.bessel} for the sequence $\L_k$, first with
$g = Q_{jk} - \wt{Q}_{jk} \in L^p[0,1]$, $p \in (1,2]$, and then with
$g = \wt{Q}_{jk} \in L^p[0,1]$, we arrive at
\begin{align}
\label{eq:sum.sF.whQ}
 \sum_{n \in \bZ} \sF[Q_{jk} - \wt{Q}_{jk}]^{p'}\bigl((b_k - b_j)\mu_n\bigr)
 & \le C\bigl(p, h\tau, d/\tau\bigr) \cdot \|Q_{jk} - \wt{Q}_{jk}\|_p^{p'}
 \le C_1 \|Q - \wt{Q}\|_p^{p'}, \\
\label{eq:sum.sF.wtQ}
 \sum_{n \in \bZ} \|Q - \wt{Q}\|_1^{p'} \sF[\wt{Q}_{jk}]^{p'}((b_k - b_j)\mu_n)
 & \le C\bigl(p, h\tau, d/\tau\bigr) \cdot \|Q - \wt{Q}\|_1^{p'}
 \|\wt{Q}_{jk}\|_p^{p'} \le C_1 r^{p'} \|Q - \wt{Q}\|_p^{p'},
\end{align}
where $C_1 = C_1(p, B, h, d) = C(p, h \tau, d/\tau) > 0$.
Inserting~\eqref{eq:sum.sF.whQ}--\eqref{eq:sum.sF.wtQ}
into~\eqref{eq:sum.phijk<sum.sF} we arrive at the desired
estimate~\eqref{eq:sum.wh.phi.p'}. Estimate~\eqref{eq:sum.wh.phi.p.vu} is
obtained similarly by using inequality~\eqref{eq:sum.int.nu.g<g}.
\end{proof}
The following auxiliary result of mostly algebraic nature will be needed to
study root functions stability in the next subsection.
\begin{lemma} \label{lem:Fx-wtFx}
Let $Q, \wt{Q} \in \bU_{1,r}^{2 \times 2}$ for some $r > 0$. Let also
\begin{equation}
 F(x, \l) := \sum_{j=1}^2 \(\a_j + \sum_{k,l=1}^2 \b_{jkl}
 \varphi_{kl}(1, \l)\) \Phi_j(x, \l), \quad\text{and}\quad
 \wt{F}(x, \l) := \sum_{j=1}^2 \(\a_j + \sum_{k,l=1}^2 \b_{jkl}
 \wt{\varphi}_{kl}(1, \l)\) \wt{\Phi}_j(x, \l),
\end{equation}
where $\a := \col(\a_1, \a_2) \in \bC^2$ and $\b :=
(\b_{jkl})_{j,k,l=1}^2 \in \bC^{2 \times 2 \times 2}$. Let $h \ge 0$. Then
there exists $C = C(r, B, \a, \b, h)$ that does not depend on $Q$ and
$\wt{Q}$ such that
\begin{equation} \label{eq:Fl-wtF.wtl<C}
 \norm{F(\cdot, \l) - \wt{F}(\cdot, \wt{\l})}_{\infty}
 \le C \cdot \(|\l - \wt{\l}| +
 \norm{\Phi_Q(\cdot, \l) - \Phi_{\wt{Q}}(\cdot, \l)}_{\infty}\),
 \qquad \l, \wt{\l} \in \Pi_h.
\end{equation}
\end{lemma}
\begin{proof}
We will split the desired difference into two parts:
\begin{equation} \label{eq:Fl-wtF.wtl}
 \norm{F(\cdot, \l) - \wt{F}(\cdot, \wt{\l})}_{\infty}
 \le \norm{F(\cdot, \l) - \wt{F}(\cdot, \l)}_{\infty}
 + \norm{\wt{F}(\cdot, \l) - \wt{F}(\cdot, \wt{\l})}_{\infty},
 \qquad \l, \wt{\l} \in \Pi_h.
\end{equation}
The second summand is trivially estimated using
Proposition~\eqref{eq:Phix.compact}(iii). Indeed, estimate~\eqref{eq:phijk<C2}
is valid for $\wt{F}(x, \l)$ but with $C_2$ that also depends on matrices $\a$
and $\b$. Hence
\begin{equation} \label{eq:Fxl-Fxwtl}
 \abs{\wt{F}(x, \l) - \wt{F}(x, \wt{\l})} = \abs{\int_{\l}^{\wt{\l}}
 \frac{d}{dz}\wt{F}(x, z) dz} \le \max_{z \in [\l, \wt{\l}]}
 \abs{\frac{d}{dz}F(x, z)} \cdot |\l - \wt{\l}| \le C_2|\l - \wt{\l}|,
 \qquad x \in [0,1], \quad \l, \wt{\l} \in \Pi_h.
\end{equation}

Recall that $\|f\|_{\bC^2} = \max\{|f_1|, |f_2|\}$ for $f = \col(f_1, f_2) \in
\bC^2$. For the first summand the part of the estimate~\eqref{eq:phijk<C2}
regarding $|\varphi_{jk}(x,\l)|$ implies
\begin{align}
\nonumber
 \abs{F(x, \l) - F(x, \wt{\l})}
 &\le \sum_{j=1}^2 \(|\a_j| \cdot
 \abs{\Phi_j(x, \l) - \wt{\Phi}_j(x, \l)}
 + \sum_{k,l=1}^2 |\b_{jkl}| \cdot \abs{\varphi_{kl}(1, \l) \Phi_j(x, \l)
 - \wt{\varphi}_{kl}(1, \l) \wt{\Phi}_j(x, \l)}\) \\
\nonumber
 & \le \sum_{j=1}^2 \(|\a_j| \cdot
 \abs{\Phi_j(x, \l) - \wt{\Phi}_j(x, \l)} \right .\\
\nonumber
 & \qquad \left. + \sum_{k,l=1}^2 |\b_{jkl}|
 \(\abs{\varphi_{kl}(1, \l) - \wt{\varphi}_{kl}(1, \l)} \cdot
 \abs{\Phi_j(x, \l)} + \abs{\wt{\varphi}_{kl}(1, \l)} \cdot
 \abs{\Phi_j(x, \l) - \wt{\Phi}_j(x, \l)}\)\) \\
\nonumber
 & \le C_1 \sum_{j=1}^2 \(
 \abs{\Phi_j(x, \l) - \wt{\Phi}_j(x, \l)} +
 \abs{\Phi_j(1, \l) - \wt{\Phi}_j(1, \l)}\) \\
\label{eq:Fxl-wtFxl}
 & \le 4 C_1 \norm{\Phi_Q(\cdot, \l) - \Phi_{\wt{Q}}(\cdot, \l)}_{\infty},
 \qquad x \in [0,1], \quad \l, \wt{\l} \in \Pi_h,
\end{align}
where $C_1 > 0$ only depends on $\a$, $\b$ and $C_2$ from~\eqref{eq:phijk<C2}.
Inserting~\eqref{eq:Fxl-Fxwtl} and~\eqref{eq:Fxl-wtFxl}
into~\eqref{eq:Fl-wtF.wtl} we arrive at~\eqref{eq:Fl-wtF.wtl<C} with
$C = \max\{4 C_1, C_2\}$, which finishes the proof.
\end{proof}
\subsection{Stability property of the eigenfunctions}
Now we are ready to formulate and prove main results of this section. The
following result in the case of $p=1$ generalizes~\cite[Theorem~4]{Sad16} for
the case of Dirac-type system and extends it with the new results for
$p \in (1,2]$.
\begin{theorem} \label{th:eigenfunc.compact}
Let $\cK$ be compact in $\LL{p}$ for some $p \in [1, 2]$ and $Q, \wt{Q} \in
\cK$. Let boundary conditions~\eqref{eq:BC} be strictly regular. Let also
$s \in (0, \infty]$. Then there exist systems $\{f_{Q,n}\}_{n \in \bZ}$ and
$\{f_{\wt{Q},n}\}_{n \in \bZ}$ of root vectors of the operators $L(Q)$ and
$L(\wt{Q})$ such that the following uniform relations hold
\begin{align}
\label{eq:fn.norm}
 \|f_{Q,n}\|_s = \|f_{\wt{Q},n}\|_s &= 1,
 \qquad |n| > N, \qquad Q, \wt{Q} \in \cK, \\
\label{eq:fn-wt.fn<Q.L1}
 \sup_{|n| > N} \norm{f_{Q,n} - f_{\wt{Q},n}}_{\infty}
 & \le C \cdot \|Q - \wt{Q}\|_1, \qquad Q, \wt{Q} \in \cK, \quad p = 1, \\
\label{eq:fn-wt.fn<Q.Lp}
 \sum_{|n| > N} \norm{f_{Q,n} - f_{\wt{Q},n}}_{\infty}^{p'} & \le C \cdot
 \|Q - \wt{Q}\|_p^{p'}, \qquad Q, \wt{Q} \in \cK, \quad p \in (1, 2],
 \quad 1/p' + 1/p = 1, \\
\label{eq:fn-wt.fn<Q.Hardy}
 \sum_{|n| > N} (1+|n|)^{p-2} \norm{f_{Q,n} - f_{\wt{Q},n}}_{\infty}^{p}
 & \le C \cdot \|Q - \wt{Q}\|_p^{p},
 \qquad Q, \wt{Q} \in \cK, \quad p \in (1, 2].
\end{align}
Here constants $N \in \bN$ and $C > 0$ do not depend on $Q$, $\wt{Q}$ and $s$.
If $p=1$ then also
\begin{align} \label{eq:fn-wt.fn.to.0}
 \sup_{Q, \wt{Q} \in \cK} \norm{f_{Q,n} - f_{\wt{Q},n}}_{\infty} \to 0
 \quad\text{as}\quad |n| \to \infty.
\end{align}
\end{theorem}
\begin{proof}
Let $\L_0 = \{\l_n^0\}_{n \in \bZ}$ be the sequence of zeros of the
characteristic determinant $\Delta_0$. Let $\L = \{\l_n\}_{n \in \bZ}$ and
$\wt{\L} = \{\wt{\l}_n\}_{n \in \bZ}$ be canonically ordered sequences of zeros
of characteristic determinants $\Delta := \Delta_Q$ and $\wt{\Delta} :=
\Delta_{\wt{Q}}$ (i.e. eigenvalues of the operators $L(Q)$ and $L(\wt{Q})$)
respectively (see definition~\ref{def:canon.order} that involves sequences $\L$
and $\L_0$). Proposition~\ref{prop:incompress.uniform}(i) ensures that there
exist constants $h = h(\cK, B, A) > 0$ and $d = d(\cK, B, A) > 0$, not dependent
on $Q$ and $\wt{Q}$, such that $\L_Q$ and $\L_{\wt{Q}}$ an incompressible
sequences of density $d$ lying in the strip $\Pi_h$.
Proposition~\ref{prop:ln-wtln<Delta} implies that for some constants
$N_0 \in \bN$ and $C_0 > 0$, not dependent on $Q$ and $\wt{Q}$, uniform
inequality~\eqref{eq:ln-wtln<C.Delta} holds,
\begin{multline} \label{eq:ln-wtln<C3}
 |\l_n - \wt{\l}_n| \le C_0 |\wt{\Delta}(\l_n)| \le
 C_0 |\Delta(\l_n) - \wt{\Delta}(\l_n)| \\
 \le C_0 \cdot J \cdot \max_{j,k \in \{1, 2\}}
 \abs{\varphi_{jk}(1, \l_n) - \wt{\varphi}_{jk}(1, \l_n)}
 \le C_1 \norm{\Phi_Q(\cdot, \l_n) - \Phi_{\wt{Q}}(\cdot, \l_n)}_{\infty},
 \qquad |n| > N_0.
\end{multline}
Here we applied formula~\eqref{eq:Delta} for $\Delta$ and $\wt{\Delta}$, and set
$J := |J_{32}| + |J_{13}| + |J_{42}| + |J_{14}|$ and $C_1 = C_0 \cdot J > 0$.

Let $\eps > 0$ be fixed (we will choose it later). By
Proposition~\ref{prop:ln-wtln<Delta} we can choose $N_{\eps} \ge N_0$ that only
depends on $\cK$, $A$ and $B$ that guarantees uniform
inequality~\eqref{eq:ln-ln0<eps}, i.e.
\begin{equation} \label{eq:ln.wtln-ln0<eps}
 |\l_n - \l_n^0| < \eps, \quad |\wt{\l}_n - \l_n^0| < \eps,
 \qquad |n| > N_{\eps}.
\end{equation}

Let $\delta \in (0, \delta_0]$ be fixed (we will choose it later) and
$M_{\delta}$ be chosen to satisfy uniform inequality~\eqref{eq:phi-phi0<delta}
of Proposition~\ref{eq:Phix.compact}(i) on $\varphi_{jk}(\cdot, \l)$. Since
$\delta \le \delta_0$ we can assume that $M_{\delta} > M_{\delta_0}$. Hence
inequality~\eqref{eq:Phik.norm} is also satisfied. Since $\l_n^0 \to \infty$ as
$n \to \infty$ we can choose $N_{\delta,\eps} \ge N_{\eps}$ such that $|\l_n^0|
> M_{\delta} - \eps$ for $|n| > N_{\delta,\eps}$.
Combining this with inequality~\eqref{eq:ln.wtln-ln0<eps} ensures that
$|\l_n| > M_{\delta}$ and $|\wt{\l}_n| > M_{\delta}$ for $|n| > N_{\delta,\eps}$. Due
to the choice of $M_{\delta}$ and inequalities $|\Im \l_n| \le h$,
$|\Im \wt{\l}_n| \le h$, Proposition~\ref{eq:Phix.compact} implies that for
$s \in (0, \infty]$, $k \in \{1,2\}$, $j = 2/k$, we have with account of
notations~\eqref{eq:PhiQ}--\eqref{eq:Phi0},
\begin{align}
\label{eq:phi.ln-phi0.ln<delta}
 \norm{\varphi_{kk}(\cdot, \l_n) - \varphi_{kk}^0(\cdot, \l_n)}_s
 & < \delta, \qquad \norm{\varphi_{jk}(\cdot, \l_n)}_s < \delta,
 \qquad C_2 \le \norm{\varphi_{kk}(\cdot, \l_n)}_s \le C_3,
 \qquad |n| > N_{\delta,\eps}, \\
\label{eq:wt.phi.ln-wt.phi0.ln<delta}
 \norm{\wt{\varphi}_{kk}(\cdot, \l_n) - \varphi_{kk}^0(\cdot, \l_n)}_s
 & < \delta, \qquad \norm{\wt{\varphi}_{jk}(\cdot, \l_n)}_s < \delta,
 \qquad C_2 \le \norm{\wt{\varphi}_{kk}(\cdot, \l_n)}_s \le C_3,
 \qquad |n| > N_{\delta,\eps},
\end{align}
where $C_2 = \delta_0, C_3 > 0$ do not depend on $Q$, $\wt{Q}$ and $n$.

Since boundary conditions are regular, one can transform them to the
form~\eqref{eq:BC.new} with coefficients $a, b, c, d$ satisfying relation
$ad \ne bc$.

\textbf{(i)} In this step we assume that $|b| + |c| \ne 0$. Without loss of
generality it suffices to consider the case $b \ne 0$. It is easy to verify that
the vector-functions
\begin{align}
\label{eq:fnx}
 f_n(\cdot) &:= F(\cdot, \l_n), \quad F(x, \l) := (b + a \varphi_{12}(\l))
 \Phi_1(x, \l) - (1 + a \varphi_{11}(\l)) \Phi_2(x, \l), \\
\label{eq:wtfnx}
 \wt{f}_n(\cdot) &:= \wt{F}(\cdot, \wt{\l}_n), \quad
 \wt{F}(x, \l) := (b + a \wt{\varphi}_{12}(\l))
 \wt{\Phi}_1(x, \l) - (1 + a \wt{\varphi}_{11}(\l)) \wt{\Phi}_2(x, \l),
\end{align}
are (possibly zero) eigenfunctions of the operators $L(Q)$ and $L(\wt{Q})$
respectively, corresponding to the eigenvalues $\l_n$ and $\wt{\l}_n$ (see also
the proof of Theorem 1.1 in~\cite{LunMal16JMAA}). Let us show that with
appropriate choice of $\delta > 0$ the following inequalities hold:
\begin{equation} \label{eq:C4<fns<C5}
 C_4 \le \|f_n\|_s \le C_5, \qquad C_4 \le \|\wt{f}_n\|_s \le C_5,
 \qquad |n| > N_{\delta,\eps}, \quad s \in (0, \infty],
\end{equation}
where $C_4, C_5 > 0$ do not depend on $Q$ and $\wt{Q}$. The estimate from above
trivially follows from Proposition~\ref{eq:Phix.compact}(iii), and is valid
uniformly for all $n \in \bZ$, $Q, \wt{Q} \in \cK$ and $s \in (0, \infty]$. Let
$F(x,\l) = \col(F_1(x,\l), F_2(x,\l))$. Since $b \ne 0$,
inequalities~\eqref{eq:phi.ln-phi0.ln<delta} imply that
\begin{multline} \label{eq:fns>C}
 \|f_n\|_s = \|F(\cdot, \l_n)\|_s \ge \|F_1(\cdot, \l_n)\|_s \ge
 |b + a \varphi_{12}(\l_n)| \cdot \|\varphi_{11}(\cdot, \l_n)\|_s -
 |1 + a \varphi_{11}(\l_n)| \cdot \|\varphi_{12}(\cdot, \l_n)\|_s \\
 \ge (|b| - |a| \delta) C_2 - (1 + |a| \delta) \delta
 = |b| C_2 - \delta(1 + |a| (C_2 + \delta))
 \ge |b| C_2/2 =: C_4, \qquad |n| > N_{\delta,\eps},
\end{multline}
with appropriate $\delta = \delta(C_2(\cK, B), a, b) = \delta(\cK, B, A) > 0$
that does not depend on $Q$ and $\wt{Q}$. Inequality~\eqref{eq:C4<fns<C5} for
$\|\wt{f}_n\|_s$ is proved in the same way.

Now inequalities~\eqref{eq:C4<fns<C5} allow us to set $f_{Q,n} :=
f_n / \|f_n\|_s$ and $f_{\wt{Q},n} := \wt{f}_n / \|\wt{f}_n\|_s$ for
$|n| > N_{\delta,\eps}$. Vectors $f_{Q,n}$ and $f_{\wt{Q},n}$ are proper,
normalized eigenfunctions of the operators $L(Q)$ and $L(\wt{Q})$ respectively,
corresponding to simple eigenvalues $\l_n$ and $\wt{\l}_n$, and
satisfy~\eqref{eq:fn.norm} with $N = N_{\delta,\eps}$. To estimate
$\norm{f_{Q,n} - f_{\wt{Q},n}}_{\infty}$ we first observe that for any
$u, v \in L^{\infty}([0,1]; \bC^2)$ and $s \in (0, \infty]$ we have,
\begin{equation} \label{eq:u-v.norm.s}
 \norm{\frac{u}{\|u\|_s} - \frac{v}{\|v\|_s}}_{\infty} =
 \frac{\norm{\, \|v\|_s \cdot u - \|u\|_s \cdot v\,}_{\infty}}{\|u\|_s
 \cdot \|v\|_s} \le
 \frac{\|v\|_s \cdot \|u-v\|_{\infty} + \bigl|\,\|u\|_s - \|v\|_s\,\bigr|
 \cdot \|v\|_{\infty}}{\|u\|_s \cdot \|v\|_s}
 \le \frac{2 \|v\|_{\infty} \cdot \|u-v\|_{\infty}}{\|u\|_s \cdot \|v\|_s}.
\end{equation}
Here in the last step we applied inequalities $\|v\|_s \le \|v\|_{\infty}$
and $\bigl|\,\|u\|_s - \|v\|_s\,\bigr| \le \|u - v\|_s \le \|u - v\|_{\infty}$.
Setting $u = f_n$ and $v = \wt{f}_n$ in~\eqref{eq:u-v.norm.s}, taking into
account inequalities~\eqref{eq:C4<fns<C5} and then applying
Lemma~\ref{lem:Fx-wtFx} to functions $F(\cdot, \l)$ and $\wt{F}(\cdot, \wt{\l})$
from~\eqref{eq:fnx}--\eqref{eq:wtfnx} with account of~\eqref{eq:ln-wtln<C3} we
arrive at
\begin{multline} \label{eq:fn-wtfn}
 \norm{f_{Q,n} - f_{\wt{Q},n}}_{\infty}
 = \norm{\frac{f_n}{\|f_n\|_s} - \frac{\wt{f}_n}{\|\wt{f}_n\|_s}}_{\infty}
 \le 2 C_5 C_4^{-2} \norm{f_n - \wt{f}_n}_{\infty}
 = C_6 \norm{F(\cdot, \l_n) - \wt{F}(\cdot, \wt{\l}_n)}_{\infty} \\
 \le C_7 |\l_n - \wt{\l}_n|
 + C_7 \norm{\Phi_Q(\cdot, \l_n) - \Phi_{\wt{Q}}(\cdot, \l_n)}_{\infty}
 \le C_8 \norm{\Phi_Q(\cdot, \l_n) - \Phi_{\wt{Q}}(\cdot, \l_n)}_{\infty},
 \qquad |n| > N_{\delta,\eps},
\end{multline}
where $C_6 = 2 C_5 C_4^{-2} > 0$, $C_7 = C_6 C(r, B, \a, \b, h) =
C_6 C(\cK, B, A) > 0$ and $C_8 = C_7 (C_1 + 1) > 0$. Here $C(r, B, \a, \b, h)$
is from Lemma~\ref{lem:Fx-wtFx}, where $r = \sup_{Q \in \cK} \|Q\|_1 < \infty$
and $\a \in \bC^2$ and $\b \in \bC^{2 \times 2 \times 2}$ are derived from
formula~\eqref{eq:fnx} and only depend on $a$ and $b$.

Now inequality~\eqref{eq:fn-wtfn} and Proposition~\ref{prop:sum.phi-wt.phi.x.p}
imply estimates~\eqref{eq:fn-wt.fn<Q.Lp}--\eqref{eq:fn-wt.fn<Q.Hardy} since
$\L = \{\l_n\}_{n \in \bZ}$ is an incompressible sequence of density $d$ lying
in the strip $\Pi_h$, while Proposition~\ref{eq:Phix.compact}(i) implies
relation~\eqref{eq:fn-wt.fn.to.0} since $\norm{f_{Q,n} - f_{\wt{Q},n}}_{\infty}
\le \norm{f_{Q,n} - f_{0,n}}_{\infty} + \norm{f_{\wt{Q},n} - f_{0,n}}_{\infty}$.
Finally, estimate~\eqref{eq:fn-wt.fn<Q.L1} is immediate from
inequality~\eqref{eq:fn-wtfn} combined with Lemma~\ref{lem:phi.jk=e+int} and
Theorem~\ref{th:K-wtK<Q-wtQ}, which finishes the proof in the case $|b|+|c|>0$.

\textbf{(ii)} Now assume that $b=c=0$. In this case $a d \ne 0$ and
$\Delta_0(\l) = (d + e^{i b_2 \l})(1 + a e^{i b_1 \l})$. Let $\L_0^1 =
\{\l_n^0\}_{n \in \cI_1}$ and $\L_0^2 = \{\l_n^0\}_{n \in \cI_2}$ be the
sequences of zeros of the first and second factor, respectively, where
$\cI_1 \sqcup \cI_2 = \bZ$, i.e. $\L_0 = \L_0^1 \sqcup \L_0^2$. Clearly, these
sequences constitute arithmetic progressions lying on the lines parallel to the
real axis. Lemma~\ref{lem:ln0.exp.asymp} implies that the arithmetic
progressions $\L_0^1$ and $\L_0^2$ are separated,
i.e., $|\l_n^0 - \l_m^0| > 2 \kappa$, $n \in \cI_1$, $m \in \cI_2$ for some
$\kappa > 0$. This implies the following estimates,
\begin{equation} \label{eq:1+ae=1}
 \abs{1 + a e^{i b_1 \l_n^0}} \ge \tau, \quad n \in \cI_1, \qquad
 \abs{d + e^{i b_2 \l_m^0}} \ge \tau, \quad m \in \cI_2.
\end{equation}
for some $\tau = \tau(A,B) > 0$. It follows
from inequalities~\eqref{eq:1+ae=1},~\eqref{eq:phi.ln-phi0.ln<delta}
and~\eqref{eq:ln.wtln-ln0<eps} that are all valid for $n \in \cI_1$ and
$|n| > N_{\delta,\eps}$, that
\begin{multline} \label{eq:1+a.phi11}
 |1 + a \varphi_{11}(\l_n)| \ge \abs{1 + a e^{i b_1 \l_n^0}}
 - |a| \cdot \abs{e^{i b_1 \l_n^0} - e^{i b_1 \l_n}}
 - |a| \cdot \abs{e^{i b_1 \l_n} - \varphi_{11}(\l_n)} \\
 \ge \tau - |a| \abs{e^{i b_1 \l_n^0}} \(e^{|b_1 (\l_n - \l_n^0)|} - 1\)
 - |a| \delta
 \ge \tau - |a| \(e^{|b_1| h} \(e^{|b_1| \eps} - 1\) + \delta\) = \tau/2,
 \qquad n \in \cI_1, \quad |n| > N_{\delta,\eps},
\end{multline}
if we set $\delta = \frac{\tau}{4|a|} > 0$ and $\eps = \log(1 +
\delta e^{-|b_1| h}) / |b_1| > 0$. Making $\delta$ and $\eps$ smaller
if needed, we can similarly guarantee the inequality
\begin{equation} \label{eq:d+phi22}
 |d + \varphi_{22}(\l_n)| \ge \tau/2, \qquad n \in \cI_2, \quad
 |n| > N_{\delta,\eps}.
\end{equation}

As per the step (i) of the proof the vector-functions
\begin{equation} \label{eq:fn1.fn01}
 f_n(\cdot) := F(\cdot, \l_n), \qquad
 \wt{f}_n(\cdot) := \wt{F}(\cdot, \wt{\l}_n), \qquad n \in \cI_1,
\end{equation}
are (possibly zero) eigenfunctions of the operators $L(Q)$ and $L(\wt{Q})$
respectively, corresponding to the eigenvalues $\l_n$ and $\wt{\l}_n$,
$n \in \cI_1$. Here $F(\cdot, \l)$ and $\wt{F}(\cdot, \l)$ are defined
in~\eqref{eq:fnx}--\eqref{eq:wtfnx}. Similarly to
inequality~\eqref{eq:C4<fns<C5} in the step~(i) let us show that with
appropriate choice of $\delta > 0$ the following inequalities hold:
\begin{equation} \label{eq:C4<fn1s<C5}
 C_9 \le \|f_n\|_s \le C_{10}, \qquad C_9 \le \|\wt{f}_{1,n}\|_s \le C_{10},
 \qquad n \in \cI_1, \quad |n| > N_{\delta,\eps}, \quad s \in (0, \infty],
\end{equation}
where $C_9, C_{10} > 0$ do not depend on $Q$, $\wt{Q}$ and $s$. As in the
step~(i) we only need to focus on the estimate from below. To this end
inequalities~\eqref{eq:1+a.phi11} and~\eqref{eq:phi.ln-phi0.ln<delta} imply
\begin{multline} \label{eq:fns>C9}
 \|f_n\|_s = \|F(\cdot, \l_n)\|_s \ge \|F_2(\cdot, \l_n)\|_s \ge
 |1 + a \varphi_{11}(\l_n)| \cdot \|\varphi_{22}(\cdot, \l_n)\|_s
 - |b + a \varphi_{12}(\l_n)| \cdot \|\varphi_{21}(\cdot, \l_n)\|_s \\
 \ge \tau C_2 / 2 - (|b| + |a| \delta) \delta
 \ge \tau C_2 / 4 =: C_9, \qquad n \in \cI_1, \quad |n| > N_{\delta,\eps},
\end{multline}
with appropriate adjustment to $\delta$ if needed.
Inequality~\eqref{eq:C4<fn1s<C5} on $\|\wt{f}_n\|_s$ is established similarly.
As in the step~(i) we set $f_{Q,n} := f_n / \|f_n\|_s$, $f_{\wt{Q},n} :=
\wt{f}_n / \|\wt{f}_n\|_s$ for $n \in \cI_1$ and $|n| > N_{\delta,\eps}$.
Inequalities~\eqref{eq:C4<fn1s<C5} imply inequality~\eqref{eq:fn-wtfn}
for $n \in \cI_1$ and $|n| > N_{\delta,\eps}$.

Going over to the second branches $\{\l_n\}_{n \in \cI_2}$ and
$\{\wt{\l}_n\}_{n \in \cI_2}$ of eigenvalues, we note that the vector-functions
\begin{align}
\label{eq:fn2}
 f_n(\cdot) &:= G(\cdot, \l_n), \qquad G(x,\l) := (d + \varphi_{22}(\l))
 \Phi_1(x, \l) - \varphi_{21}(\l)\Phi_2(x, \l), \qquad n \in \cI_2, \\
\label{eq:wtfn2}
 \wt{f}_n(\cdot) &:= \wt{G}(\cdot, \wt{\l}_n), \qquad \wt{G}(x,\l) :=
 (d + \wt{\varphi}_{22}(\l)) \wt{\Phi}_1(x, \l) -
 \wt{\varphi}_{21}(\l) \wt{\Phi}_2(x, \l), \qquad n \in \cI_2,
\end{align}
are (possibly zero) eigenfunctions of the operators $L(Q)$ and $L(0)$
respectively, corresponding to the eigenvalues $\l_n$ and $\wt{\l}_n$,
$n \in \cI_2$. Using inequality~\eqref{eq:d+phi22} instead
of~\eqref{eq:1+a.phi11} and adjusting $\delta$ and $\eps$ if needed,
we can obtain inequalities~\eqref{eq:C4<fn1s<C5} but for $n \in \cI_2$
instead of $n \in \cI_1$. Applying Lemma~\ref{lem:Fx-wtFx} to functions
$G(\cdot, \l)$ and $\wt{G}(\cdot, \l)$ we similarly arrive at~\eqref{eq:fn-wtfn}
for $n \in \cI_2$ and $|n| > N_{\delta,\eps}$.

Now with inequality~\eqref{eq:fn-wtfn} being valid for all $|n| >
N_{\delta,\eps}$, the proof is finished in the same way as in the step~(i).
\end{proof}
Next we extend Theorem~\ref{th:eigenfunc.compact} to the case $\cK =
\bU_{p,r}^{2 \times 2}$. Similarly to Theorem~\ref{th:ln-wtln<Delta.Q-wtQ.Lp} we
cannot select a universal constant $N$ serving all potentials. Instead, we need
to sum over the sets of integers, the complements of which have uniformly
bounded cardinality.
\begin{theorem} \label{th:eigenfunc.holes}
Let $Q, \wt{Q} \in \bU_{p,r}^{2 \times 2}$ for some $p \in (1,2]$ and $r > 0$.
Let boundary conditions~\eqref{eq:BC} be strictly regular. Let also
$s \in (0, \infty]$. Then there exist systems $\{f_{Q,n}\}_{n \in \bZ}$ and
$\{f_{\wt{Q},n}\}_{n \in \bZ}$ of root vectors of the operators $L(Q)$ and
$L(\wt{Q})$ and the set $\cI_{Q,\wt{Q}} \subset \bZ$ such that the following
uniform relations hold
\begin{align}
\label{eq:card.IQ.wtQ.func}
 \card\(\bZ \setminus \cI_{Q, \wt{Q}}\) & \le N,
 \qquad Q, \wt{Q} \in \bU_{p, r}^{2 \times 2}, \\
\label{eq:fn.norm.holes}
 \|f_{Q,n}\|_s = \|f_{\wt{Q},n}\|_s &= 1,
 \qquad n \in \cI_{Q, \wt{Q}}, \quad Q, \wt{Q} \in \bU_{p, r}^{2 \times 2}, \\
\label{eq:fn-wt.fn<Q.Lp.holes}
 \sum_{n \in \cI_{Q, \wt{Q}}} \norm{f_{Q,n} - f_{\wt{Q},n}}_{\infty}^{p'}
 & \le C \cdot \|Q - \wt{Q}\|_p^{p'},
 \qquad Q, \wt{Q} \in \bU_{p, r}^{2 \times 2}, \qquad 1/p' + 1/p = 1, \\
\label{eq:fn-wt.fn<Q.Hardy.holes}
 \sum_{n \in \cI_{Q, \wt{Q}}} (1+|n|)^{p-2}
 \norm{f_{Q,n} - f_{\wt{Q},n}}_{\infty}^{p}
 & \le C \cdot \|Q - \wt{Q}\|_p^{p},
 \qquad Q, \wt{Q} \in \bU_{p, r}^{2 \times 2}.
\end{align}
Here constants $N \in \bN$ and $C > 0$ do not depend on $Q$, $\wt{Q}$ and $s$.

Moreover, for any $\eps > 0$ there exist a set $\cI_{\eps} := \cI_{Q,
\wt{Q}, \eps} \subset \bZ$ and a constant $N_{\eps} = N_{\eps}(p, r, A, B) \in
\bN$ that does not depend on $Q$ and $\wt{Q}$, such that the following uniform
estimates hold
\begin{align}
\label{eq:card.IQ.wtQ.eps.func}
 \card\(\bZ \setminus \cI_{Q, \wt{Q}, \eps}\) & \le N_{\eps},
 \qquad Q, \wt{Q} \in \bU_{p, r}^{2 \times 2}, \\
\label{eq:fn-wtfn<eps.whQ}
 \sup_{n \in \cI_{Q, \wt{Q}, \eps}} \norm{f_{Q,n} - f_{\wt{Q},n}}_{\infty}
 & \le \eps \|Q - \wt{Q}\|_{p}, \qquad Q, \wt{Q} \in \bU_{p, r}^{2 \times 2}.
\end{align}
\end{theorem}
\begin{proof}
Let $\L_0 = \{\l_n^0\}_{n \in \bZ}$, $\L = \L_Q = \{\l_n\}_{n \in \bZ}$ and
$\wt{\L} = \L_{\wt{Q}} = \{\wt{\l}_n\}_{n \in \bZ}$ be the same as in the proof
of Theorem~\ref{th:eigenfunc.compact}.
Proposition~\ref{prop:incompress.holes}(i) ensures that there exist constants
$h = h(p, r, B, A) > 0$ and $d = d(p, r, B, A) > 0$, not dependent on $Q$ and
$\wt{Q}$, such that $\L_Q$ and $\L_{\wt{Q}}$ an incompressible sequences of
density $d$ lying in the strip $\Pi_h$.
Theorem~\ref{th:ln-wtln<Delta.Q-wtQ.Lp}(ii) implies that for some constants
$N_0 \in \bN$ and $C_0 > 0$, not dependent on $Q$ and $\wt{Q}$, uniform
inequalities~\eqref{eq:card.IQ.wtQ}--\eqref{eq:ln-wtln<C.Delta.hole} hold,
\begin{align}
\label{eq:card.cI.Q.wtQ.eigenf}
 \card\(\bZ \setminus \cI_{Q, \wt{Q}}\) & \le N_0, \\
\label{eq:ln-wtln<C1.wh.Phi}
 |\l_n - \wt{\l}_n| \le C_0 |\wt{\Delta}(\l_n)|
 & \le C_1 \norm{\Phi_Q(\cdot, \l_n) - \Phi_{\wt{Q}}(\cdot, \l_n)}_{\infty},
 \qquad n \in \cI_{Q, \wt{Q}},
\end{align}
Here we applied trivial uniform estimate on $|\wt{\Delta}(\l_n)|$
from~\eqref{eq:ln-wtln<C3}, which is valid for $Q, \wt{Q} \in
\bU_{1,r}^{2 \times 2}$ and $n \in \bZ$.

Let $\eps > 0$ be fixed (we will choose it later). By
Proposition~\ref{prop:incompress.holes}(ii) there exists
$N_{\eps} = N_{\eps}(p, r, B, A)$ that do not depend on $Q$ and $\wt{Q}$, and
the sets $\cI_{Q, \eps}, \cI_{\wt{Q}, \eps} \subset \bZ$, such that,
\begin{align}
\label{eq:card.cIQ.cIwtQ}
 \card\(\bZ \setminus \cI_{Q, \eps}\) +
 \card\(\bZ \setminus \cI_{\wt{Q}, \eps}\) & \le N_{\eps}, \\
\label{eq:ln-ln0<eps.wtln-ln0<eps}
 |\l_n - \l_n^0| < \eps, \quad |\wt{\l}_n - \l_n^0| & < \eps,
 \qquad n \in \cI_{Q,\eps} \cap \cI_{\wt{Q}, \eps}.
\end{align}

Let $\delta \in (0, \delta_0]$ be fixed (we will choose it later) and
$\cJ_{Q,\delta}$ be chosen to satisfy uniform
inequalities~\eqref{eq:card.cJQdelta}--\eqref{eq:Phik.norm.holes} of
Proposition~\ref{eq:Phix.holes} on $\varphi_{jk}(\cdot, \l)$. Now set,
\begin{equation}
 \cI_{Q, \wt{Q}, \delta, \eps} := \left\{
 n \in \cI_{Q, \wt{Q}} \cap \cI_{Q, \eps} \cap \cI_{\wt{Q}, \eps} \ : \
 \l_n \in \Pi_{h, Q, \delta}, \ \wt{\l}_n \in \Pi_{h, \wt{Q}, \delta}
 \right\}.
\end{equation}
It follows from~\eqref{eq:card.cI.Q.wtQ.eigenf},~\eqref{eq:card.cIQ.cIwtQ},
\eqref{eq:card.cJQdelta} and the fact that $\Lambda_Q$ and $\Lambda_{\wt{Q}}$
are incomressible sequences of density $d$, not dependent on $Q$ and $\wt{Q}$,
that $\card\(\bZ \setminus \cI_{Q, \wt{Q}, \delta, \eps}\) \le N_{\delta,\eps}$,
with some $N_{\delta,\eps}$ that does not depend $Q$ and $\wt{Q}$.

Combining inequalities~\eqref{eq:phijk.s<delta}--\eqref{eq:Phik.norm.holes}
with Proposition~\ref{eq:Phix.compact}(iii), implies that for
$s \in (0, \infty]$, $k \in \{1,2\}$, $j = 2/k$, we have with account of
notations~\eqref{eq:PhiQ}--\eqref{eq:Phi0},
\begin{align}
\label{eq:phi.ln-phi0.ln<delta.holes}
 \norm{\varphi_{kk}(\cdot, \l_n) - \varphi_{kk}^0(\cdot, \l_n)}_s
 & < \delta, \qquad \norm{\varphi_{jk}(\cdot, \l_n)}_s < \delta,
 \qquad \delta_0 \le \norm{\varphi_{kk}(\cdot, \l_n)}_s \le C_2,
 \qquad n \in \cI_{Q, \wt{Q}, \delta, \eps}, \\
\label{eq:wt.phi.ln-wt.phi0.ln<delta.holes}
 \norm{\wt{\varphi}_{kk}(\cdot, \l_n) - \varphi_{kk}^0(\cdot, \l_n)}_s
 & < \delta, \qquad \norm{\wt{\varphi}_{jk}(\cdot, \l_n)}_s < \delta,
 \qquad \delta_0 \le \norm{\wt{\varphi}_{kk}(\cdot, \l_n)}_s \le C_2,
 \qquad n \in \cI_{Q, \wt{Q}, \delta, \eps},
\end{align}
where $C_2 > 0$ does not depend on $Q$, $\wt{Q}$ and $n$. Clearly,
inequalities~\eqref{eq:ln-wtln<C1.wh.Phi},~\eqref{eq:ln-ln0<eps.wtln-ln0<eps}
are valid for $n \in \cI_{Q, \wt{Q}, \delta, \eps}$ since
$\cI_{Q, \wt{Q}, \delta, \eps}$ is a subset of each of the three sets
$\cI_{Q, \wt{Q}}$, $\cI_{Q, \eps}$, $\cI_{\wt{Q}, \eps}$.

From this point, the proof of relations
\eqref{eq:card.IQ.wtQ.func}--\eqref{eq:fn-wt.fn<Q.Hardy.holes} is carried out in
the same way as in Theorem~\ref{th:eigenfunc.compact} with all
inequalities $|n| > N_{\delta, \eps}$ replaced by inclusion $n \in
\cI_{Q,\wt{Q},\eps,\delta}$.

Inequality~\eqref{eq:fn-wtfn<eps.whQ} easily follows
from~\eqref{eq:fn-wt.fn<Q.Lp.holes} by applying Chebyshev's inequality technique
used in the proof of Theorem~\ref{th:ln-wtln<Delta.Q-wtQ.Lp}(ii).
\end{proof}
In the sequel we need the following definition.
\begin{definition} \label{def:asymp.norm}
Let $\fF = \{f_n\}_{n \in \bZ}$ be a sequence of elements of a Banach space $X$.

\textbf{(i)} The sequence $\fF$ is called \textbf{almost normalized in $X$} if
$\|f_n\|_X \asymp 1$, $n \in \bZ$.

\textbf{(ii)} The sequence $\fF$ is called \textbf{asymptotically normalized in
$X$} if for some $N \in \bN$ we have $\|f_n\|_X = 1$, $|n| > N$.
\end{definition}
The following example shows that in some cases we can relax compactness
condition and even boundedness of $\cK$ and sum over all $n \in \bZ$
in~\eqref{eq:fn-wt.fn<Q.Lp}--\eqref{eq:fn-wt.fn<Q.Hardy}. Though due to
relation~\eqref{eq:u-v.norm.s} to relax boundedness of $\cK$ we can only
normalize eigenfunctions in $L^{\infty}([0,1]; \bC^2)$. Normalizing in
$L^s([0,1]; \bC^2)$ for any $s \in (0, \infty)$ would need, e.g., inclusion
$\cK \subset \bU_{1,r}^{2 \times 2}$ for some $r > 0$.
\begin{proposition} \label{prop:root.Q12=0}
Let $Q_{12} = \wt{Q}_{12} = 0$ and $Q_{21}, \wt{Q}_{21} \in L^p[0,1]$ for some
$p \in (1, 2]$. Let boundary conditions~\eqref{eq:BC.new} be strictly regular
with the parameter $b$ in them being zero. Then eigenvalues of operators $L(Q)$
and $L(\wt{Q})$ are simple and separated and for some systems
$\{f_n\}_{n \in \bZ}$ and $\{\wt{f}_n\}_{n \in \bZ}$ of their eigenfunctions the
following uniform relations hold
\begin{align}
\label{eq:fn.norm.easy}
 \|f_n\|_{\infty} = \|\wt{f}_n\|_{\infty} & = 1, \qquad n \in \bZ, \\
\label{eq:fn-wt.fn<Q.Lp.easy}
 \sum_{n \in \bZ} \norm{f_n - \wt{f}_n}_{\infty}^{p'} & \le C \cdot
 \|Q - \wt{Q}\|_p^{p'}, \qquad 1/p' + 1/p = 1, \\
\label{eq:fn-wt.fn<Q.Hardy.easy}
 \sum_{n \in \bZ} (1+|n|)^{p-2} \norm{f_n - \wt{f}_n}_{\infty}^{p} &
 \le C \cdot \|Q - \wt{Q}\|_p^{p},
\end{align}
where $C = C(p, B, A) > 0$ does not depend on $Q$ and $\wt{Q}$.
\end{proposition}
\begin{proof}
Since $Q_{12} = 0$ then explicit formula~\eqref{eq:Phi.Q12=0} holds. In
particular
\begin{equation} \label{eq:phi21.Q12=0}
 \varphi_{21}(x, \l) = -i b_2 e^{i b_2 \l x} \int_0^x Q_{21}(t)
 e^{i (b_1 - b_2) \l t} dt, \qquad x \in [0,1], \quad \l \in \bC.
\end{equation}
Since $b=0$, relation~\eqref{eq:Phi.Q12=0} implies that vector function
$\a_1 \Phi_1(\cdot, \l) + \a_2 \Phi_2(\cdot, \l)$ satisfy boundary
conditions~\eqref{eq:BC.new} if and only if
\begin{equation} \label{eq:U.C1.C2}
 \begin{pmatrix} 1 + a e^{i b_1 \l} & 0 \\ c e^{i b_1 \l} + \varphi_{21}(\l) &
 d + e^{i b_2 \l} \end{pmatrix} \begin{pmatrix} \a_1 \\ \a_2 \end{pmatrix} = 0.
\end{equation}
Since boundary conditions~\eqref{eq:BC.new} are strictly regular
then $ad \ne 0$ and~\eqref{eq:U.C1.C2} implies that eigenvalues of $L(Q)$ are
simple and separated, do not depend on $Q$ and form a union of two arithmetic
progressions lying on two lines parallel to the real axis,
\begin{equation} \label{eq:l1n.l2n.b=0}
 \l_{1,n} = \wt{\l}_{1,n} = \frac{\arg(-a^{-1}) + 2 \pi n}{b_1}
 + i\frac{\ln|a|}{b_1}, \qquad
 \l_{2,n} = \wt{\l}_{2,n} = \frac{\arg(-d) + 2 \pi n}{b_2}
 - i\frac{\ln|d|}{b_2}, \qquad n \in \bZ.
\end{equation}
For eigenvalues $\l_{2,n}$ we have $d + e^{i b_2 \l_{2,n}} = 0$ and hence it
follows from~\eqref{eq:U.C1.C2} that vector functions
\begin{equation} \label{eq:g2n.def}
 g_{2,n}(x) := \Phi_2(x, \l_{2,n}) = \Phi_2^0(x, \l_{2,n})
 = \col\(0, e^{i b_2 \l_{2,n} x}\)
\end{equation}
are the eigenfunctions of the operator $L(Q)$. Hence
\begin{equation} \label{eq:gn2}
 \sum_{n \in \bZ} \norm{g_{2,n} - \wt{g}_{2,n}}_{\infty}^{p'} = 0
 \quad\text{and}\quad
 \|g_{2,n}\|_{\infty} = \|\,\wt{g}_{2,n}\|_{\infty}
 = \sup_{x \in [0,1]}\abs{e^{i b_2 \l_{2,n} x}}
 = \max\{1, |d|\} =: C_2 > 0, \qquad n \in \bZ.
\end{equation}
For eigenvalues $\l_{1,n}$ we have $1 + a e^{i b_1 \l_{1,n}} = 0$. Since
boundary conditions are strictly regular, then for $d_n := d +
e^{i b_2 \l_{1,n}}$ we have $D^{-1} \le |d_n| \le D$, $n \in \bZ$, for some
$D > 1$ that only depends on $a, d, b_1, b_2$. Hence~\eqref{eq:U.C1.C2} implies
that vector functions
\begin{align}
\nonumber
 g_{1,n}(x) & := d_n \Phi_1(x, \l_{1,n}) - (c e^{i b_1 \l_{1,n}} +
 \varphi_{21}(\l_{1,n})) \Phi_2(x, \l_{1,n}) \\
\label{eq:gn1.def}
 & = \begin{pmatrix}
 d_n e^{i b_1 \l_{1,n} x} \\
 d_n \varphi_{21}(x, \l_{1,n}) + (c e^{i b_1 \l_{1,n}} +
 \varphi_{21}(\l_{1,n})) e^{i b_2 \l_{1,n} x}
 \end{pmatrix}, \qquad n \in \bZ,
\end{align}
are non-zero eigenfunctions of the operator $L(Q)$ corresponding to the
eigenvalues $\l_{1,n}$. It follows from~\eqref{eq:Phi.Q12=0} that
\begin{equation} \label{eq:gn1.norm}
 \norm{g_{1,n}}_{\infty} \ge |d_n| \sup_{x \in [0,1]}\abs{e^{i b_1 \l_{1,n} x}}
 \ge D^{-1} \max\{1, |a|^{-1}\} =: C_1 > 0, \qquad n \in \bZ,
\end{equation}
It is also clear that $\|g_{1,n}\|_{\infty} \le C(A, B) (\|Q_{21}\|_1 + 1)$.
Therefore, relations~\eqref{eq:gn2} and~\eqref{eq:gn1.norm} imply that
\begin{equation} \label{eq:gQn.def}
 \{g_{Q,n}\}_{n \in \bZ} := \{g_n\}_{n \in \bZ} := \{g_{1,n}\}_{n \in \bZ} \cup
 \{g_{2,n}\}_{n \in \bZ}
\end{equation}
is almost normalized sequence of eigenfunctions of the operator $L(Q)$.

Since $\l_{1,n} = \wt{\l}_{1,n}$ it follows from~\eqref{eq:gn1.def}
and~\eqref{eq:Phi.Q12=0} that
\begin{equation} \label{eq:gn1-wt.x}
 g_{1,n}(x) - \wt{g}_{1,n}(x) := d_n \(\varphi_{21}(x, \l_{1,n})
 - \wt{\varphi}_{21}(x, \l_{1,n})\) - e^{i b_2 \l_{1,n} x}
 \(\varphi_{21}(1, \l_{1,n}) - \wt{\varphi}_{21}(1, \l_{1,n})\)
\end{equation}
Hence~\eqref{eq:phi21.Q12=0}, definition~\eqref{eq:sF.f.def} of $\sF[f]$
and~\eqref{eq:l1n.l2n.b=0} imply for $n \in \bZ$,
\begin{align} \label{eq:gn1-wt}
 \norm{g_{1,n} - \wt{g}_{1,n}}_{\infty} &\le |b_2| \cdot
 \sup_{x \in [0,1]}\abs{e^{i b_2 \l_{1,n} x}} \cdot
 \(|d_n| + \abs{e^{i b_2 \l_{1,n}}}\) \cdot
 \sF[Q_{21} - \wt{Q}_{21}]((b_1 - b_2)\l_{1,n}) \\
 & \le |b_2| \cdot \max\{1, |a|^{-b_2/b_1}\} \cdot (D + |a|^{-b_2/b_1}) \cdot
 \sF[Q_{21} - \wt{Q}_{21}]((b_1 - b_2)\l_{1,n})
\end{align}
Combining Theorem~\ref{th:p.bessel} and~\eqref{eq:gn1-wt} we arrive at
\begin{align} \label{eq:sum.gn1}
 \sum_{n \in \bZ} \norm{g_{1,n} - \wt{g}_{1,n}}_{\infty}^{p'}
 \le C(p, b_1, b_2, a, d) \cdot \|Q_{21} - \wt{Q}_{21}\|_p^{p'}.
\end{align}
Now set
\begin{equation} \label{eq:fnj.def}
 f_{j,n} := \frac{g_{j,n}}{\|g_{j,n}\|_{\infty}}, \qquad
 \wt{f}_{j,n} := \frac{\wt{g}_{j,n}}{\|\wt{g}_{j,n}\|_{\infty}},
 \qquad n \in \bZ, \quad j \in \{1,2\}.
\end{equation}
Clearly $f_{j,n}$ and $\wt{f}_{j,n}$ are eigenfunctions of the operators $L(Q)$
and $L(\wt{Q})$ that satisfy~\eqref{eq:fn.norm.easy}. Further,
similarly to~\eqref{eq:u-v.norm.s} for any non-zero elements $u, v$ of some Banach
space we have
\begin{equation} \label{eq:u-v.norm}
 \norm{\frac{u}{\|u\|} - \frac{v}{\|v\|}}
 \le \frac{2\|u-v\|}{\|u\|}.
\end{equation}
Combining~\eqref{eq:u-v.norm},~\eqref{eq:gn2} and~\eqref{eq:gn1.norm} we arrive
at
\begin{equation} \label{eq:fn-wtfn.norm}
 \|f_{j,n} - \wt{f}_{j,n}\| \le 2C_j^{-1} \|g_{j,n} - \wt{g}_{j,n}\|,
 \qquad n \in \bZ, \quad j \in \{1,2\}.
\end{equation}
Inequality~\eqref{eq:fn-wt.fn<Q.Lp.easy} now immediately follows
from~\eqref{eq:fn-wtfn.norm},~\eqref{eq:gn2} and~\eqref{eq:sum.gn1}.
Inequality~\eqref{eq:fn-wt.fn<Q.Hardy.easy} is derived similarly.
\end{proof}
The following example shows that in a special case $Q_{12}=0$, $b=0$,
$a=1$, stability property~\eqref{eq:fn-wt.fn<Q.Lp.easy} of the eigenfunctions of
the operator $L(Q)$ is equivalent to the abstract
inequality~\eqref{eq:sum.int.g<g} from Theorem~\ref{th:p.bessel} with
a sequence $\{\mu_n\}_{n \in \bZ}$ being an arithmetic progression.
\begin{example} \label{ex:stabil.equiv}
Let $p > 0$, $Q_{12} = 0$, $Q_{21} \in L^1[0,1]$, and let boundary
conditions~\eqref{eq:BC.new} be strictly regular with $b=0$ and $a=1$.
Set $\mu_n = 2 (1 - b_2/b_1) \pi n$, $n \in \bZ$. Assume that
\begin{equation} \label{eq:sum.int.Q21.01}
 \sum_{n \in \bZ} \abs{\int_0^1 Q_{21}(t) e^{i \mu_n t} dt}^{p} < \infty,
\end{equation}
which holds whenever $Q_{21} \in L^p[0,1]$, $p \in (1,2]$, due to the classical
Hausdorff-Young theorem for Fourier coefficients
(see~\cite[Theorem XII.2.3]{Zig59_v2}). Formula~\eqref{eq:phi21.Q12=0} for
$\varphi_{21}(x, \l)$ easily implies equivalences
\begin{align}
\label{eq:sup.int.Q21.equiv.phi}
 \sum_{n \in \bZ} \sup_{x \in [0,1]}
 \abs{\int_0^x Q_{21}(t) e^{i \mu_n t} dt}^{p} < \infty
 & \quad \Leftrightarrow \quad \sum_{n \in \bZ} \sup_{x \in [0,1]}
 \abs{\varphi_{21}(x, \mu_n)}^{p} < \infty, \\
\label{eq:int.Q21.equiv.phi}
 \sum_{n \in \bZ} \abs{\int_0^1 Q_{21}(t) e^{i \mu_n t} dt}^{p} < \infty
 & \quad \Leftrightarrow \quad \sum_{n \in \bZ}
 \abs{\varphi_{21}(1, \mu_n)}^{p} < \infty.
\end{align}
Recall that the sequence of eigenvalues of the operator $L(Q)$
(of BVP~\eqref{eq:system},~\eqref{eq:BC.new}) in this special case is the union
of arithmetic progressions~\eqref{eq:l1n.l2n.b=0}. In particular, they are
simple and separated. Furthemore, the sequence $\{g_{Q,n}\}_{n \in \bZ}$ defined
in~\eqref{eq:g2n.def},~\eqref{eq:gn1.def},~\eqref{eq:gQn.def} is almost
normalized sequence of eigenfunctions of the operator $L(Q)$. Combining
relation~\eqref{eq:gn1-wt.x} with the pre-condition~\eqref{eq:sum.int.Q21.01}
now implies the following equivalence
\begin{equation} \label{eq:sum.int.Q21.equiv}
 \sum_{n \in \bZ} \sup_{x \in [0,1]}
 \abs{\int_0^x Q_{21}(t) e^{i \mu_n t} dt}^{p} < \infty
 \quad \Leftrightarrow \quad
 \sum_{n \in \bZ} \|g_{Q,n} - g_{0,n}\|_{\infty}^p < \infty.
\end{equation}
The first condition in~\eqref{eq:sum.int.Q21.equiv} is equivalent to the
abstract inequality~\eqref{eq:sum.int.g<g} from Theorem~\ref{th:p.bessel} for
the function $g = Q_{21}$ and the sequence $\mu_n = 2 (1 - b_2/b_1) \pi n$.
\end{example}
\begin{corollary} \label{th:root.Lp.stabil.compact}
Let $Q \in \LL{p}$, $p \in (1, 2]$, $p'=p/(p-1)$, and let boundary
conditions~\eqref{eq:BC} be strictly regular. Then the systems
$\{f_n\}_{n \in \bZ}$ and $\{f_n^0\}_{n \in \bZ}$ of root vectors of the
operators $L(Q)$ and $L(0)$ can be chosen asymptotically normalized in
$L^{p'}([0,1]; \bC^2)$ and satisfying the following uniform estimates
\begin{align}
\label{eq:sum.fn-fn0}
 \sum_{n \in \bZ} \norm{f_n - f_n^0}_{\infty}^{p'} & \le \infty, \\
\label{eq:sum.fn-fn0.weight}
 \sum_{n \in \bZ} (1+|n|)^{p-2} \norm{f_n - f_n^0}_{\infty}^{p} & \le \infty.
\end{align}
\end{corollary}
\begin{definition} \label{def:p-close}
Two sequences $\{f_n\}_{n \in \bZ}$ and $\{g_n\}_{n \in \bZ}$ in a Banach space
$X$ are called $\theta$-close in $X$ if $\{\|f_n-g_n\|_X\}_{n \in \bZ} \in
l^{\theta}(\bZ)$, i.e.
\begin{equation} \label{eq:fn-gn.X.p}
 \sum_{n \in \bZ} \|f_n-g_n\|_X^{\theta} < \infty.
\end{equation}
If $X$ is a Hilbert space and $\theta=2$ it is a classic definition of
quadratically close sequences in a Hilbert space (see~\cite[Subsection IV.2.4]
{GohKre65}
\end{definition}
\begin{corollary} \label{cor:p-close}
Assume conditions of Theorem~\ref{th:root.Lp.stabil.compact}. Then systems
$\{f_n\}_{n \in \bZ}$ and $\{f_n^0\}_{n \in \bZ}$ of root vectors of the
operators $L(Q)$ and $L(0)$ can be chosen to be asymptotically normalized and
$p'$-close in $L^{p'}([0,1]; \bC^2)$.
\end{corollary}
\begin{proof}
Since $\|f\|_{p'} \le \|f\|_{\infty}$ for $f \in L^{\infty}([0,1]; \bC^2)$,
inequality~\eqref{eq:sum.fn-fn0} yields estimate~\eqref{eq:fn-gn.X.p} with
$\theta=p'$ and $X = L^{p'}([0,1]; \bC^2)$.
\end{proof}
\section{Criterion for Bari basis property} \label{sec:bari}
The Bari-Markus property of quadratic closeness of systems of root vectors of
the operators $L(Q)$ and $L(0)$ when $Q \in \LL{2}$ was
studied in numerous papers (Mityagin, Baskakov references). But the question
whether system of root vectors of the operator $L(Q)$ forms a proper Bari
basis was never investigated to the best of our knowledge.
\begin{definition} \label{def:bari}
A sequence of vectors in a Hilbert space $\fH$ forms a \textbf{Bari basis} if it
is quadratically close to an orthonormal complete sequence of vectors.
\end{definition}
Our considerations are largely based on the following abstract criterion
for Bari basis property.
\begin{proposition}~\cite[Theorem VI.3.2]{GohKre65} \label{prop:crit.bari}
A complete system $\fF = \{f_n\}_{n \in \bZ}$ of unit vectors in a Hilbert space
$\fH$ forms a Bari basis if and only if there exists a sequence $\{g_n\}_{n \in
\bZ}$ biorthogonal to $\fF$ that is quadratically close to $\fF$.
\end{proposition}
In the sequel we will need the following slightly more practical version of this
criterion. Throughout the section for any vector $f \ne 0$ in a Hilbert space we
will denote as $\wh{f}$ the normalization of $f$, $\wh{f} := f / \|f\|$.
\begin{lemma} \label{lem:crit.bari.not.norm}
Let $\fF = \{f_n\}_{n \in \bZ}$ be a complete system of vectors in a Hilbert
space $\fH$. Let also $\{g_n\}_{n \in \bZ}$ be ``almost biorthogonal'' to $\fF$.
Namely, $(f_n, g_m) = 0$, $n \ne m$, $(f_n, g_n) \ne 0$, $n, m \in \bZ$. Then
normalization of $\fF$, $\wh{\fF} := \{\wh{f}_n\}_{n \in \bZ}$, forms a Bari
basis in $\fH$ if and only if
\begin{equation} \label{eq:sum.fn.gn-1}
 \sum_{n \in \bZ} \( \frac{\|f_n\|^2 \cdot \|g_n\|^2}{|(f_n, g_n)|^2} - 1\)
 < \infty.
\end{equation}
\end{lemma}
\begin{proof}
Put $f_n' = \wh{f}_n = f_n / \|f_n\|$ and $g_n' = \|f_n\| / (f_n, g_n) \cdot
g_n$. It is clear that $\fF' := \{f_n'\}_{n \in \bZ}$ is complete system of unit
vectors in $\fH$ and $\fG' := \{g_n'\}_{n \in \bZ}$ is biorthogonal to $\fF'$,
$(f_n', g_m') = \delta_{n,m}$, $n,m \in \bZ$. Further, we have
\begin{equation} \label{eq:|fn'-gn'|}
 \|f_n' - g_n'\|^2 = \|f_n'\|^2 - (f_n', g_n') - \ol{(f_n', g_n')} + \|g_n'\|^2
 = \|g_n'\|^2 - 1 = \frac{\|f_n\|^2 \cdot \|g_n\|^2}{|(f_n, g_n)|^2} - 1.
\end{equation}
Hence, condition~\eqref{eq:sum.fn.gn-1} holds if and only if systems $\fF'$ and
$\fG'$ are quadratically close. Proposition~\ref{prop:crit.bari} now finishes
the proof.
\end{proof}
First we reduce the Bari basis property of the perturbed operator $L(Q)$ to that
of the unperturbed operator $L(0)$.
\begin{lemma} \label{lem:bari.LQ=L0}
Let $Q \in \LL{2}$ and let boundary conditions~\eqref{eq:BC} be strictly
regular. Then normalized systems of root vectors of the operators $L(Q)$ and
$L(0)$ form Bari basis in $\LLV{2}$ only simultaneously.
\end{lemma}
\begin{proof}
In accordance with Corollary~\ref{cor:p-close} we can choose almost normalized
systems of root vectors $\{f_n\}_{n \in \bZ}$ and $\{f_n^0\}_{n \in \bZ}$ of the
operators $L(Q)$ and $L(0)$, respectively, that are quadratically close, i.e.
\begin{equation} \label{eq:|fn-fn0|.in.l2}
 \bigl\{\|f_n - f_n^0\|\bigr\}_{n \in \bZ} \in l^2(\bZ).
\end{equation}
Since vector systems $\{f_n\}_{n \in \bZ}$ and $\{f_n^0\}_{n \in \bZ}$ are
almost normalized, it follows from inclusion~\eqref{eq:|fn-fn0|.in.l2} and
inequality~\eqref{eq:u-v.norm} that $\bigl\{\|\wh{f}_n -
\wh{f}_n^0\|\bigr\}_{n \in \bZ} \in l^2(\bZ)$. This inclusion in turn implies
that normalized root vector systems $\{\wh{f}_n\}_{n \in \bZ}$ and
$\{\wh{f}_n^0\}_{n \in \bZ}$ form Bari basses in $\LLV{2}$ only
simultaneously.
\end{proof}
Lemma~\ref{lem:bari.LQ=L0} immediately implies that the system of root
vectors of the operator $L(Q)$ forms a Bari basis if boundary
conditions~\eqref{eq:BC} are self-adjoint. In this case the system of root
vectors of the unperturbed operator $L(0)$ forms an orthonormal basis in
$\LLV{2}$. Below we will show that the opposite is true for Dirac operators
($-b_1 = b_2 = 1$) and a wide class of Dirac-type operators. In particular,
it is valid in the special case of quasi-periodic boundary conditions. To treat
this case we need the following simple estimate.
\begin{lemma} \label{lem:ejx.ej.Ej}
Let $h \ge 0$ and $\l \in \Pi_h$. Denote
\begin{equation} \label{eq:ejx.ej.Ej}
 e_j(x) := e_{j,\l}(x) := e^{i b_j \l x}, \quad
 e_j := e_{j, \l} := e^{i b_j \l}, \quad
 E_j^{\pm} := E_{j,\l}^{\pm} := \int_0^1 |e_j(x)|^{\pm 2} dx
 = \int_0^1 e^{\mp 2 b_j \Im \l x} dx.
\end{equation}
Then the following estimates hold:
\begin{align}
\label{eq:Ej+Ej->1}
 & E_j^+ E_j^- \ge 1 + \frac{(b_j \Im \l)^2}{3}, \qquad
 \text{in particular} \quad
 E_j^+ E_j^- > 1, \quad \text{if} \quad \Im \l \ne 0, \\
\label{eq:E1+E2-E1-E2+}
 & \frac{\sqrt{E_1^+ E_2^-}}{\sqrt{E_1^- E_2^+}}
 = |e_2^{-1} e_1| + O(|\Im \l|^2)
 = 1 + (b_2 - b_1) \cdot \Im \l + O(|\Im \l|^2).
\end{align}
\end{lemma}
\begin{proof}
It is clear that
\begin{equation} \label{eq:Ej=f}
 E_j^{\pm} = f(\mp 2 b_j \Im \l), \quad\text{where}\quad
 f(x) := \frac{e^x - 1}{x} = 1 + \frac{x}{2} + O(x^2), \quad |x| < h.
\end{equation}
It follows from Taylor expansion of $e^x$ that
\begin{equation} \label{eq:fx.f-x>1}
 f(x)f(-x) = \frac{e^x - 1}{x} \cdot \frac{e^{-x} - 1}{-x}
 = \frac{e^x + e^{-x} - 2}{x^2}
 = 2 \sum_{k=1}^{\infty} \frac{x^{2k-2}}{(2k)!}
 \ge 1 + \frac{x^2}{12}, \quad x \in \bR.
\end{equation}
Estimate~\eqref{eq:Ej+Ej->1} now immediately follows from~\eqref{eq:Ej=f}
and~\eqref{eq:fx.f-x>1}.

Further, it follows from~\eqref{eq:Ej=f} that $E_j^{\pm} = 1 \mp b_j y +
O(y^2)$, where we set for brevity $y := \Im \l$. Hence
\begin{equation} \label{eq:E1pm.E2mp}
 \sqrt{E_1^{\pm} E_2^{\mp}}
 = \sqrt{\(1 \mp b_1 y + O(y^2)\) \(1 \pm b_2 y + O(y^2)\)}
 = \sqrt{1 \pm (b_2 - b_1) y + O(y^2)} = 1 \pm \frac{b_2 - b_1}{2} y + O(y^2).
\end{equation}
Further note that
\begin{equation} \label{eq:e2-1e1}
 |e_2^{-1} e_1| = \abs{e^{i(b_1-b_2)\l}} = e^{(b_2-b_1) \Im\l}
 = 1 + (b_2 - b_1) \cdot \Im \l + O(|\Im \l|^2).
\end{equation}
Relation~\eqref{eq:E1+E2-E1-E2+} now immediately follows from~\eqref{eq:e2-1e1}
and~\eqref{eq:E1pm.E2mp}.
\end{proof}
Notation~\eqref{eq:ejx.ej.Ej} will be used throughout this section.
\begin{proposition} \label{prop:crit.bari.period}
Let $Q \in \LL{2}$ and let boundary conditions~\eqref{eq:BC}
be of the form
\begin{equation} \label{eq:quasi.per.bc}
 y_1(0) + a y_1(1) = d y_2(0) + y_2(1) = 0, \quad a d \ne 0,
\end{equation}
and are strictly regular. Then the normalized system of root vectors of the
operator $L(Q)$ forms a Bari basis in $\LLV{2}$ if and only if $|a| =
|d| = 1$.
\end{proposition}
\begin{proof}
Due to Lemma~\ref{lem:bari.LQ=L0} it is sufficient to consider the case $Q=0$.
If $|a|=|d|=1$ then it's easy to verify that the operator $L(0)$ with boundary
conditions~\eqref{eq:quasi.per.bc} is self-adjoint. Hence its normalized system
of root vectors forms an orthonormal basis in $\LLV{2}$ and Bari basis in
particular.

Now assume that the normalized system of root vectors of the operator $L(0)$
forms a Bari basis in $\LLV{2}$. According to the proof of
Lemma~\ref{lem:ln0.exp.asymp} the eigenvalues of the operator $L(0)$ are simple
and split into two separated arithmetic progressions $\L_0^1$ and $\L_0^2$,
where
\begin{equation} \label{eq:Lam1.Lam2}
 e^{i b_2 \l} = -d, \quad \l \in \L_0^1 = \{\l_{1,n}^0\}_{n \in \bZ},
 \qquad\text{and}\qquad
 e^{i b_1 \l} = -a^{-1}, \quad \l \in \L_0^2 = \{\l_{2,n}^0\}_{n \in \bZ}.
\end{equation}
It is easy to verify that vectors
\begin{equation}
 f_{1,n}^0(x) = \col\(0, \ e^{i b_2 \l_{1,n}^0 x}\), \quad
 \ol{g_{1,n}^0(x)} = \col\(0, \ e^{-i b_2 \l_{1,n}^0 x}\), \qquad n \in \bZ,
\end{equation}
are the eigenvectors of the operators $L(0)$ and $L^*(0)$ respectively
corresponding to the eigenvalues $\l_{1,n}^0$ and $\ol{\l_{1,n}^0}$, and the
vectors
\begin{equation}
 f_{2,n}^0(x) = \col\(e^{i b_1 \l_{2,n}^0 x}, \ 0\), \quad
 \ol{g_{2,n}^0(x)} = \col\(e^{-i b_1 \l_{2,n}^0 x}, \ 0\), \qquad n \in \bZ,
\end{equation}
are the eigenvectors of the operators $L(0)$ and $L^*(0)$ respectively
corresponding to the eigenvalues $\l_{2,n}^0$ and $\ol{\l_{2,n}^0}$. It is clear
that
\begin{equation} \label{eq:fjn.gkm=delta}
 \(f_{j,n}^0, g_{k,m}^0\) = \delta_{j,n}^{k,m}, \qquad j,k \in \{1,2\},
 \quad n,m \in \bZ.
\end{equation}
Thus the union system $\fF := \{f_{1,n}^0\}_{n \in \bZ} \cup \{f_{2,n}^0\}_{n
\in \bZ}$ is the system of root vectors of the operator $L(0)$ and $\fG :=
\{g_{1,n}^0\}_{n \in \bZ} \cup \{g_{2,n}^0\}_{n \in \bZ}$ is biorthogonal to it.

Since normalization of the system $\fF$ forms a Bari basis in $L^2([0,1];
\bC^2)$ then according to Lemma~\ref{lem:crit.bari.not.norm} we have
\begin{equation} \label{eq:sumj.sumn.alp}
 \sum_{j=1,2} \sum_{n \in \bZ} \a_{j,n} < \infty, \qquad
 \a_{j,n} := \frac{\bigl\|f_{j,n}^0\bigr\|^2 \cdot
 \bigl\|g_{j,n}^0\bigr\|^2}{\bigl(f_{j,n}^0, g_{j,n}^0\bigr)} - 1,
 \quad j \in \{1,2\}, \ \ n \in \bZ.
\end{equation}
Let $j=1$, $n \in \bZ$ be fixed and $\l = \l_{1,n}^0$. Then taking into account
Lemma~\ref{lem:ejx.ej.Ej} and formula~\eqref{eq:fjn.gkm=delta} we have
\begin{equation} \label{eq:alp1n>Im}
 \a_{1,n} = \bigl\|f_{1,n}^0\bigr\|^2 \cdot \bigl\|g_{1,n}^0\bigr\|^2 - 1
 = E_2^+ E_2^- - 1 \ge \frac{(b_2 \Im \l_{1,n}^0)^2}{3}.
\end{equation}
It follows from~\eqref{eq:l1n.l2n.bc=0} that $b_2 \Im \l_{1,n}^0 = -\ln|d|$.
Since the series in~\eqref{eq:sumj.sumn.alp} converges
formula~\eqref{eq:alp1n>Im} implies that $\ln|d|=0$, which means that $|d|=1$.

Similarly considering the case $j=2$ we conclude that $|a|=1$, which finishes
the proof.
\end{proof}
The following intermediate result plays the crucial role in proving
Theorem~\ref{th:crit.bari.dirac}.
\begin{proposition} \label{prop:crit.bari.b.ne.0}
Let boundary conditions~\eqref{eq:BC} be of the form~\eqref{eq:BC.new} and be
strictly regular. Let $\{\l_n^0\}_{n \in \bZ}$ be the sequence of
the eigenvalues of the operator $L(0)$. Then the normalized system of root
vectors of the operator $L(0)$ forms a Bari basis in $\LLV{2}$ if and
only if the following conditions hold
\begin{equation} \label{eq:sum1.sum2<inf}
 b_1 |c| + b_2 |b| = 0, \quad \sum_{n \in \bZ} \abs{\Im \l_n^0}^2 < \infty,
 \quad \sum_{n \in \bZ} (|z_n| - \Re z_n) < \infty,
 \quad z_n := \(1 + d e^{- i b_2 \l_n^0}\)\ol{\(1 + a e^{i b_1 \l_n^0}\)}.
\end{equation}
\end{proposition}
\begin{proof}
\textbf{(i)} First let $b=c=0$. Then it follows from~\eqref{eq:Delta_0_in_roots}
that $z_n = 0$. It also follows from~\eqref{eq:l1n.l2n.bc=0} that $\Im \l_n^0
\to 0$ if and only if $|a|=|d|=1$. Hence condition~\eqref{eq:sum1.sum2<inf} is
equivalent to $|a|=|d|=1$, i.e. that the operator $L(0)$ is selfadjoint. This in
turn is equivalent to Bari basis property of the system of root vectors of the
operator $L(0)$ due to Proposition~\ref{prop:crit.bari.period}.

\textbf{(ii)} Now let $|b|+|c|>0$. Without loss of generality we can assume that
$b \ne 0$. By definition of strictly regular boundary conditions there exists
$n_0 \in \bN$ such that eigenvalues $\l_n^0$ of $L(0)$ for $|n| > n_0$ are
geometrically and algebraically simple and separated from each other. Let
$\fF := \{f_n\}_{n \in \bZ}$ be a system of root vectors of the operator $L(0)$
and $\fG := \{g_n\}_{n \in \bZ}$ be the corresponding system for the adjoint
operator $L^*(0)$ (zero superscript is omitted for convenience). Clearly, $\fG$
is almost biorthogonal to $\fF$. Hence Lemma~\ref{lem:crit.bari.not.norm}
implies that normalization of $\fF$ forms a Bari basis in $\LLV{2}$ if
and only if condition~\eqref{eq:sum.fn.gn-1} holds.

According to the proof of Theorem~1.1 in~\cite{LunMal16JMAA} vector-functions
$f_n(\cdot)$ and $g_n(\cdot)$ for $|n| > n_0$ are of the following form,
\begin{align}
\label{eq:fn0x.bari}
 f_n(x) &:= \col\(b e^{i b_1 \l_n^0 x},
 - (1 + a e^{i b_1 \l_n^0}) e^{i b_2 \l_n^0 x}\), \\
\label{eq:gn0x.bari}
 \ol{g_n(x)} &:= \col\((1 + d e^{-i b_2 \l_n^0}) e^{-i b_1 \l_n^0 x},
 - k b e^{-i b_2 \l_n^0 x}\), \qquad k := -b_2 b_1^{-1} > 0.
\end{align}
Let $|n| > n_0$ be fixed and set $\l = \l_n^0$. Taking into account
notation~\eqref{eq:ejx.ej.Ej} and performing straightforward calculations we
see that
\begin{align}
\label{eq:|fn|2}
 \|f_n\|^2 &= |b|^2 E_1^+ + |1 + a e_1|^2 E_2^+, \\
\label{eq:|gn|2}
 \|g_n\|^2 &= \abs{1 + d e_2^{-1}}^2 E_1^- + k^2 |b|^2 E_2^-, \\
\label{eq:fn.gn}
 (f_n, g_n) &= b\((1 + d e_2^{-1}) + k (1 + a e_1)\).
\end{align}
Since boundary conditions~\eqref{eq:BC.new} are strictly regular, it follows
from the proof of Theorem~1.1 in~\cite{LunMal16JMAA} that the following estimate
holds
\begin{equation}
 (f_n, g_n) \asymp \Delta'(\l_n^0) \asymp 1, \quad |n| > n_0.
\end{equation}
Here for $a_n, b_n \in \bC$, $n \in S \subset \bZ$, notation $a_n
\asymp b_n$, $n \in S$, means that $C_1 |b_n| \le |a_n| \le C_2 |b_n|$,
$n \in S$, for some $C_2 > C_1 > 0$. Hence condition~\eqref{eq:sum.fn.gn-1} is
equivalent to
\begin{equation} \label{eq:sum.fn.gn-fngn}
 \sum_{|n| > n_0} \( \|f_n\|^2 \cdot \|g_n\|^2 - |(f_n, g_n)|^2 \)
 < \infty.
\end{equation}
With account of~\eqref{eq:|fn|2}--\eqref{eq:fn.gn} we get
\begin{multline} \label{eq:fn2.gn2-fn.gn2=tau.sum}
 \|f_n\|^2 \cdot \|g_n\|^2 - |(f_n, g_n)|^2
 = \(|b|^2 \cdot E_1^+ + |1 + a e_1|^2 \cdot E_2^+\) \cdot
 \(\abs{1 + d e_2^{-1}}^2 E_1^- + k^2 |b|^2 E_2^-\) \\
 - |b|^2 \abs{(1 + d e_2^{-1}) + k (1 + a e_1)}^2
 = \tau_{1,n} + \tau_{2,n} + \tau_{3,n}, \qquad |n| > n_0,
\end{multline}
where
\begin{align}
\label{eq:tau1.tau2}
 \tau_{1,n} &:= |b|^2 \cdot \abs{1 + d e_2^{-1}}^2 \cdot
 (E_1^+ E_1^- - 1), \qquad
 \tau_{2,n} := k^2 |b|^2 \cdot |1 + a e_1|^2 \cdot (E_2^+ E_2^- - 1), \\
\label{eq:tau3}
 \tau_{3,n} &:= k^2 |b|^4 E_1^+ E_2^- +
 \abs{1 + d e_2^{-1}}^2 \cdot |1 + a e_1|^2 \cdot E_2^+ E_1^- -
 2 k |b|^2 \cdot \Re z_n,
\end{align}
where $z_n$ is defined in~\eqref{eq:sum1.sum2<inf}.

According to Proposition~\ref{prop:sine.type}(i), $\l_n^0 \in \Pi_h$, $n \in
\bZ$, for some $h \ge 0$. Hence it follows from~\eqref{eq:Ej+Ej->1} that
\begin{equation} \label{eq:Ej.Ej-1.asymp}
 0 \le E_j^+ E_j^- - 1 \asymp |\Im \l_n^0|^2, \quad n \in \bZ.
\end{equation}
Since $b \ne 0$ and $k > 0$ relations~\eqref{eq:tau1.tau2}
and~\eqref{eq:Ej.Ej-1.asymp} now imply that
\begin{equation} \label{eq:tau1.tau2.asymp}
 0 \le \tau_{1,n} \asymp \abs{1 + d e_2^{-1}}^2 \cdot |\Im \l_n^0|^2,
 \qquad 0 \le \tau_{2,n} \asymp |1 + a e_1|^2 \cdot |\Im \l_n^0|^2,
 \qquad |n| > n_0.
\end{equation}
Combining~\eqref{eq:tau1.tau2.asymp} with Lemma~\ref{lem:ln0.exp.asymp} we get
\begin{equation} \label{eq:tau1+tau2.asymp}
 0 \le \tau_{1,n} + \tau_{2,n} \asymp |\Im \l_n^0|^2, \quad |n| > n_0.
\end{equation}
With account of~\eqref{eq:Ej.Ej-1.asymp} we get
\begin{align} \label{eq:tau3.estim}
 k^2 |b|^4 E_1^+ E_2^- + \abs{1 + d e_2^{-1}}^2 \cdot |1 + a e_1|^2
 \cdot E_2^+ E_1^-
 & \ge 2 k |b|^2 \cdot\abs{\(1 + d e_2^{-1}\) \ol{(1 + a e_1)}}
 \sqrt{E_1^+ E_1^- \cdot E_2^- E_2^+} \nonumber \\
 & \ge 2 k |b|^2 \cdot \Re z_n,
 \quad n \in \bZ.
\end{align}
Hence $\tau_{3,n} \ge 0$, $n \in \bZ$. Combining this
with~\eqref{eq:tau1+tau2.asymp} and~\eqref{eq:fn2.gn2-fn.gn2=tau.sum} we see
that condition~\eqref{eq:sum.fn.gn-fngn} holds if and only if
\begin{equation} \label{eq:sum.Im.sum.tau3}
 \sum_{|n| > n_0} \abs{\Im \l_n^0}^2 < \infty, \qquad
 \sum_{|n| > n_0} \tau_{3,n} < \infty.
\end{equation}
Similar to calculations done in~\eqref{eq:tau3.estim} one can verify that
\begin{align}
\label{eq:tau4}
 \tau_{3,n} &= |\tau_{4,n}|^2 + \tau_{5,n}, \qquad
 \tau_{4,n} := \sqrt{E_1^+ E_2^-} \cdot k |b|^2
 - \sqrt{E_2^+ E_1^-} \cdot z_n, \\
\label{eq:tau5}
 \tau_{5,n} &:= 2 k |b|^2 \cdot \(\sqrt{E_1^+ E_1^- \cdot E_2^- E_2^+} - 1\)
 \cdot \Re z_n.
\end{align}
It follows from~\eqref{eq:Ej.Ej-1.asymp} that for some $C > 0$,
\begin{equation}
 0 \le \sqrt{E_1^+ E_1^- \cdot E_2^- E_2^+} - 1 \le
 C \cdot |\Im \l_n^0|^2, \quad n \in \bZ.
\end{equation}
Combining the last relation with~\eqref{eq:e.bj.ln.asymp.1} we see that for
some $\wt{C} > 0$,
\begin{equation} \label{eq:tau5.estim}
 |\tau_{5,n}| \le \wt{C} \cdot |\Im \l_n^0|^2, \quad |n| > n_0.
\end{equation}
Hence if the first series in~\eqref{eq:sum.Im.sum.tau3} converges then so is the
series $\sum_{|n| > n_0} |\tau_{5,n}|$. Hence in view of~\eqref{eq:tau4}
condition~\eqref{eq:sum.Im.sum.tau3} is equivalent to
\begin{equation} \label{eq:sum.Im.sum.tau4}
 \sum_{|n| > n_0} \abs{\Im \l_n^0}^2 < \infty, \qquad
 \sum_{|n| > n_0} |\tau_{4,n}|^2 < \infty.
\end{equation}
With account of~\eqref{eq:E1+E2-E1-E2+} we get from~\eqref{eq:tau4} that
\begin{equation}
 \tau_{4,n} = \sqrt{E_2^+ E_1^-} \(\(|e_2^{-1} e_1| + O(|\Im \l_n^0|^2)\)
 \cdot k |b|^2 - z_n\), \quad |n| > n_0.
\end{equation}
Since $\sqrt{E_2^+ E_1^-} \asymp 1$, $n \in \bZ$, it is clear now that
condition~\eqref{eq:sum.Im.sum.tau4} is equivalent to
\begin{equation} \label{eq:sum.Im.sum.tau6}
 \sum_{|n| > n_0} \abs{\Im \l_n^0}^2 < \infty, \qquad
 \sum_{|n| > n_0} |\tau_{6,n}|^2 < \infty, \qquad
 \tau_{6,n} := k |b|^2 |e_2^{-1} e_1| - z_n.
\end{equation}
It follows from~\eqref{eq:Delta_0_in_roots} that
\begin{align}
\nonumber
 |\tau_{6,n}|^2
 & = k^2 |b|^4 |e_2^{-1} e_1|^2 + |z_n|^2
 - 2 k |b|^2 |e_2^{-1} e_1| \cdot \Re z_n \\
\nonumber
 & = k^2 |b|^4 |e_2^{-1} e_1|^2 + |bc|^2 |e_2^{-1} e_1|^2
 - 2 k |b|^2 |e_2^{-1} e_1| \cdot \Re z_n \\
\nonumber
 & = |b|^2 |e_2^{-1} e_1| \cdot \Bigl(|e_2^{-1} e_1| \cdot
 \bigl(k^2 |b|^2 + |c|^2\bigr) - 2 k \Re z_n\Bigr) \\
\label{eq:tau6}
 & = |b|^2 |e_2^{-1} e_1| \cdot \Bigl(|e_2^{-1} e_1| \cdot
 \bigl(k |b| - |c|\bigr)^2 + 2 k \(|z_n| - \Re z_n\)\Bigr).
\end{align}
Since $b \ne 0$ and $|e_2^{-1} e_1| \asymp 1$ it follows from~\eqref{eq:tau6}
that~\eqref{eq:sum.Im.sum.tau6} is equivalent to~\eqref{eq:sum1.sum2<inf}
which finishes the proof.
\end{proof}
\begin{corollary} \label{cor:fn=gn.crit}
Assume conditions of Proposition~\ref{prop:crit.bari.b.ne.0}. Let $|n| > n_0$ be
fixed and let $f_n$, $g_n$ be eigenvectors defined
via~\eqref{eq:fn0x.bari}--\eqref{eq:gn0x.bari}. Then
\begin{equation} \label{eq:fn2.gn2=fngn2}
 \|f_n\|^2 \cdot \|g_n\|^2 = |(f_n, g_n)|^2,
\end{equation}
if and only if
\begin{equation} \label{eq:Im=0.kb=ad}
 \Im \l_n^0 = 0 \quad\text{and}\quad k |b|^2 =
 \(1 + d e^{- i b_2 \l_n^0}\)\ol{\(1 + a e^{i b_1 \l_n^0}\)}.
\end{equation}
\end{corollary}
\begin{proof}
\textbf{(i)} First let condition~\eqref{eq:Im=0.kb=ad} hold. Since $\Im \l_n^0 =
0$ then $E_1^{\pm} = E_2^{\pm} = 1$. Hence it follows from~\eqref{eq:tau1.tau2},
\eqref{eq:tau4} and~\eqref{eq:tau5} that $\tau_{1,n} = \tau_{2,n} = \tau_{5,n} =
0$, while
\begin{equation} \label{eq:tau4.simple}
 \tau_{4,n} := k |b|^2 - \(1 + d e_2^{-1}\)\ol{(1 + a e_1)}.
\end{equation}
Second condition in~\eqref{eq:Im=0.kb=ad} now implies that $\tau_{4,n} = 0$.
Hence $\tau_{3,n} = 0$ and formula~\eqref{eq:fn2.gn2-fn.gn2=tau.sum}
implies~\eqref{eq:fn2.gn2=fngn2}.

\textbf{(ii)} Now let condition~\eqref{eq:fn2.gn2=fngn2} hold. Since $\tau_{j,n}
\ge 0$, $j \in \{1, 2, 3\}$, formula~\eqref{eq:fn2.gn2-fn.gn2=tau.sum} and
condition~\eqref{eq:fn2.gn2=fngn2} implies that $\tau_{j,n} = 0$, $j \in \{1, 2,
3\}$. If $\Im \l_n^0 \ne 0$ then by Lemma~\ref{lem:ejx.ej.Ej}, $E_j^+ E_j^- - 1
> 0$, $j \in \{1, 2\}$. Since $\tau_{1,n} = \tau_{2,n} = 0$, $b \ne 0$, $k > 0$
it follows from~\eqref{eq:tau1.tau2} that $1 + d e_2^{-1} = 1 + a e_1 = 0$. This
contradicts Lemma~\ref{lem:ln0.exp.asymp} since boundary
conditions~\eqref{eq:BC.new} are strictly regular. Hence $\Im \l_n^0 = 0$, which
again implies that $E_1^{\pm} = E_2^{\pm} = 1$. This in turn implies that
$\tau_{5,n} = 0$ and formula~\eqref{eq:tau4.simple} for $\tau_{4,n}$. Since
$\tau_{3,n} = \tau_{5,n} = 0$, then $\tau_{4,n} = 0$.
Formula~\eqref{eq:tau4.simple} now implies second condition
in~\eqref{eq:Im=0.kb=ad} which finishes the proof.
\end{proof}
Now we are ready to formulate the main result of this section.
\begin{theorem} \label{th:crit.bari.dirac}
Let $Q \in \LL{2}$ and let boundary conditions~\eqref{eq:BC.new} be strictly
regular. Let either $b_1/b_2 \in \bQ$ or $abcd=0$. Then the normalized system
of root vectors of the operators $L(Q)$ forms a Bari basis in $\LLV{2}$ if and
only if the operator $L(0)$ is self-adjoint. The latter holds if and only if
the matrix $\begin{pmatrix} a & \mu b \\ \mu^{-1} c & d \end{pmatrix}$ with
$\mu = \sqrt{-b_2/b_1}$ is unitary.
\end{theorem}
\begin{proof}
Due to Lemma~\ref{lem:bari.LQ=L0} it is sufficient to consider the case $Q=0$.
Since boundary conditions~\eqref{eq:BC} are regular we can transform them to the
form~\eqref{eq:BC.new}.

\textbf{(i)} If the operator $L(0)$ with boundary conditions~\eqref{eq:BC.new}
is self-adjoint then its normalized system of root vectors forms an orthonormal
basis in $\LLV{2}$ and Bari basis in particular. Further, note that
boundary conditions~\eqref{eq:BC.new} are self-adjoint if and only if
$A_{12} B A_{12}^* = A_{34} B A_{34}^*$, where $A_{12} = \(\begin{smallmatrix}
1 & b \\ 0 & d \end{smallmatrix}\)$ and $A_{34} = \(\begin{smallmatrix}
a & 0 \\ c & 1 \end{smallmatrix}\)$. Straightforward calculations show that it
is equivalent to unitarity property of the matrix
$\begin{pmatrix} a & \mu b \\ \mu^{-1} c & d \end{pmatrix}$

\textbf{(ii)} Now assume that the normalized system of root vectors of the
operator $L(0)$ forms a Bari basis in $\LLV{2}$.

If $b=c=0$ then Proposition~\ref{prop:crit.bari.period} yields that $|a|=|d|=1$,
in which case operator $L(0)$ is self-adjoint. This finishes the proof in this
case.

Now let $|b|+|c| \ne 0$. Proposition~\ref{prop:crit.bari.b.ne.0} implies that
relations~\eqref{eq:sum1.sum2<inf} take place. Since $|b|+|c| \ne 0$, first
condition in~\eqref{eq:sum1.sum2<inf} implies that $b \ne 0$ and $c \ne 0$.

First, let $b_1 / b_2 \in \bQ$. In this case $b_1 = -m_1 \b$, $b_2 = m_2
\b$, where $\b > 0$, $m_1, m_2 \in \bN$. Set $m = m_1 + m_2$. Since $ad
\ne bc$, $\Delta_0(\cdot)$ is a polynomial at $e^{i \b \l}$ of degree $m$
with non-zero roots $e^{i\mu_k}$, $\mu_k \in \bC$, $k \in \{1, \ldots, m\}$.
Hence, zeros $\{\l_n^0\}_{n \in \bZ}$ of $\Delta_0(\cdot)$ form a union of
arithmetic progressions $\left\{\frac{\mu_k + 2 \pi n}{\b}\right\}_{n \in
\bZ}$, $k \in \{1, \ldots, m\}$. If $\Im \mu_k \ne 0$, for some $k \in \{1,
\ldots, m\}$, then
\begin{equation}
 \sum_{n \in \bZ} \abs{\Im \frac{\mu_k + 2 \pi n}{\b}}^2 =
 \sum_{n \in \bZ} \abs{\Im \frac{\mu_k}{\b}}^2 = \infty,
\end{equation}
which contradicts the first
relation in~\eqref{eq:sum1.sum2<inf}. Hence $\Im \l_n^0 = 0$, $n \in \bZ$. This
implies that $E_j^{\pm} = \int_0^1 |e^{i b_j \l_n^0 x}|^{\pm 2} dx = 1$ and
\begin{equation}
 \tau_{4,n} = k |b|^2 - z_n
 = k |b|^2 - \(1 + d e^{- i b_2 \l_n^0}\)\ol{\(1 + a e^{i b_1 \l_n^0}\)}.
\end{equation}
According to the proof of Proposition~\ref{prop:crit.bari.b.ne.0},
condition~\eqref{eq:sum1.sum2<inf} is equivalent to~\eqref{eq:sum.Im.sum.tau4},
i.e. $\sum_{n \in \bZ} |\tau_{4,n}|^2 < \infty$.
It is clear that $e^{-i b_2 \l_n^0} = (e^{-i \mu_k})^{m_2}$ for
some $k = k_n$. Hence $e^{-i b_2 \l_n^0}$ attains a finite set of values when
$n \in \bZ$. Similarly $e^{i b_1 \l_n^0}$ attains a finite set of values when
$n \in \bZ$. Hence, $\tau_{4,n}$ attains a finite set of values when $n \in \bZ$
and each value is attained infinite times. Now
condition~\eqref{eq:sum.Im.sum.tau4} implies that $\tau_{4,n} = 0$, $n \in \bZ$.
Corollary~\ref{cor:fn=gn.crit} yields condition~\eqref{eq:fn2.gn2=fngn2} for
eigenvectors $f_n$ and $g_n$ introduced in the proof of
Proposition~\ref{prop:crit.bari.b.ne.0}. Taking into account
formula~\eqref{eq:|fn'-gn'|} we see that normalized eigenvectors $f_n'$ and
$g_n'$ of the operators $L(0)$ and $L^*(0)$ corresponding to the common
eigenvalue $\l_n^0 = \ol{\l_n^0}$ are equal for all $n \in \bZ$. It follows
easily from this that $L(0) = L^*(0)$.

Next, let's assume that $a=d=0$ (without extra condition on $b_1 / b_2$). In
this case $\Delta_0(\l) = e^{i b_2 \l} - bc e^{i b_1 \l}$. Hence $\l_n^0 =
\frac{\mu + 2 \pi n}{b_2-b_1}$, $n \in \bZ$, where $bc = e^{i \mu}$. Relation
$\sum_{n \in \bZ} |\Im \l_n^0|^2 < \infty$ implies that $\Im \l_n^0 = 0$, $n \in
\bZ$. Further, since $a=d=0$ then $\tau_{4,n} = k |b|^2 - 1 = \const$. Since
$\sum_{n \in \bZ} |\tau_{4,n}|^2 < \infty$ then $\tau_{4,n} = 0$, $n \in \bZ$.
Now the same reasoning as above implies that the operator $L(0)$ is
self-adjoint. Along the way, we see that this is the case if and only if
$|b|^2=k^{-1}$ and $|c|=|b|^{-1}$.

Finally, let $b_1 / b_2 \notin \bQ$, $a=0$, $bcd \ne 0$ (the case $d=0$,
$abc \ne 0$ can be treated similarly). Clearly, $\{z_n\}_{n \in \bZ}$ is a
bounded sequence, hence $\sum_{n \in \bZ} (|z_n| - \Re z_n) < \infty$ implies
that $\Im z_n \to 0$ as $n \to \infty$. Since $a = 0$ then $\Delta_0(\l) =
1 + d e^{-i b_2 \l} - bc e^{i b_1 \l}$. Hence, $z_n =
1 + d e^{-i b_2 \l_n^0} = bc e^{i b_1 \l_n^0}$. Let $\l_n^0 = \a_n + i \b_n$,
$d = r_2 e^{i \psi_2}$, $bc = r_1 e^{i \psi_1}$, where $\a_n, \b_n \in \bR$,
$r_1, r_2 > 0$ and $\psi_1, \psi_2 \in [-\pi, \pi)$. Then
\begin{equation}
 \Im z_n = r_2 e^{b_2 \b_n} \sin(\psi_2 - b_2 \a_n) =
 r_1 e^{-b_1 \b_n} \sin(\psi_1 + b_1 \a_n) \to 0
 \quad\text{as}\quad n \to \infty.
\end{equation}
This implies that $\sin(\psi_2 - b_2 \a_n) \to 0$ and $\sin(\psi_1 + b_1 \a_n)
\to 0$ as $n \to \infty$. Since $b_1/b_2 \notin \bQ$ and $\a_n =
\frac{2 \pi n}{b_2-b_1}(1 + o(1))$, this contradicts Weyl's equidistribution
theorem and implies that in this case condition~\eqref{eq:sum1.sum2<inf} never
holds.
\end{proof}
\textbf{Acknowledgement.} The publication has been prepared with the support of
the ``RUDN University Program 5-100''.

\end{document}